\documentclass[11pt, a4paper]{amsart}
\usepackage[margin=2.5cm,marginpar=2cm]{geometry}
\usepackage{amsmath}
\usepackage{amssymb}
\usepackage{xcolor}
\usepackage{mathrsfs}
\usepackage{mathtools}
\usepackage{stmaryrd}
\usepackage{longtable}
\usepackage[longtable]{multirow}
\usepackage{array}
\newcolumntype{V}[1]{>{\centering\arraybackslash}m{#1}}
\usepackage{makecell}
\usepackage{comment}
\usepackage{tikz-cd}
\usepackage{rotating}
\usepackage{tablefootnote}
\usepackage{adjustbox}
\usepackage[lowtilde]{url}
\usepackage{hyperref}
\usepackage[
  backend=biber,
  style=alphabetic,
  sorting=nyt,
  maxnames=99,
  maxalphanames=99,
  minalphanames=1,
  url=true,
  doi=true,
  eprint=true
]{biblatex}
\usepackage{xstring}
\DeclareLabelalphaNameTemplate{
  \namepart[use=true, strwidth=1, compound=true]{prefix}
  \namepart[strwidth=1]{family}
}
\DeclareLabelalphaTemplate{
  \labelelement{
    \field[final]{shorthand}
    \field{label}
    \field[strwidth=1,strside=left,noalphaothers]{labelname}
    \field[strwidth=3,strside=left,ifnames=0]{labeltitle}
  }
  \labelelement{
    \field[strwidth=2,strside=right]{year}
  }
}

\newtoggle{arxivbibentry}
\newtoggle{journalwitharxivbibentry}
\newcommand{\setarxivbibtoggle}{%
  \togglefalse{arxivbibentry}%
  \togglefalse{journalwitharxivbibentry}%
  \iffieldundef{eprinttype}{}{%
    \edef\bibeprinttype{\thefield{eprinttype}}%
    \IfSubStr{\bibeprinttype}{Xiv}{\toggletrue{arxivbibentry}}{%
      \IfSubStr{\bibeprinttype}{xiv}{\toggletrue{arxivbibentry}}{}%
    }%
  }%
  \iftoggle{arxivbibentry}{}{%
    \iffieldundef{doi}{}{%
      \edef\bibtmp{\thefield{doi}}%
      \IfSubStr{\bibtmp}{arXiv}{\toggletrue{arxivbibentry}}{}%
    }%
  }%
  \iftoggle{arxivbibentry}{}{%
    \iffieldundef{note}{}{%
      \edef\bibtmp{\thefield{note}}%
      \IfSubStr{\bibtmp}{arXiv}{\toggletrue{arxivbibentry}}{}%
    }%
  }%
  \iftoggle{arxivbibentry}{%
    \iffieldundef{journaltitle}{%
      \iffieldundef{journal}{}{%
        \togglefalse{arxivbibentry}%
        \toggletrue{journalwitharxivbibentry}%
      }%
    }{%
      \togglefalse{arxivbibentry}%
      \toggletrue{journalwitharxivbibentry}%
    }%
  }{}%
}
\AtEveryBibitem{%
  \setarxivbibtoggle
  \clearfield{urldate}%
  \iftoggle{journalwitharxivbibentry}{%
    \clearfield{url}%
    \clearfield{urlraw}%
  }{%
    \iftoggle{arxivbibentry}{%
      \clearfield{doi}%
      \clearfield{url}%
      \clearfield{urlraw}%
      \clearfield{number}%
      \iffieldundef{note}{}{%
        \edef\bibtmp{\thefield{note}}%
        \IfSubStr{\bibtmp}{arXiv}{\clearfield{note}}{}%
      }%
    }{%
      \clearfield{doi}%
      \ifentrytype{online}{}{%
        \clearfield{url}%
        \clearfield{urlraw}%
      }%
    }%
  }%
}
\renewbibmacro*{url+urldate}{%
  \usebibmacro{url}%
}
\renewbibmacro*{doi+eprint+url}{%
  \iftoggle{journalwitharxivbibentry}{%
    \iftoggle{bbx:doi}{\printfield{doi}}{}%
    \iffieldundef{eprint}{}{%
      \setunit{\addspace}\newblock
      \usebibmacro{eprint}%
    }%
  }{%
    \iftoggle{arxivbibentry}{%
      \usebibmacro{eprint}%
    }{%
      \ifentrytype{online}{%
        \iftoggle{bbx:url}{\usebibmacro{url}}{}%
      }{%
      }%
    }%
  }%
}
\usepackage{bbm}
\usepackage[new]{old-arrows}
\usepackage[capitalise,noabbrev,nameinlink]{cleveref}
\usepackage[textwidth=1.3in,textsize=tiny]{todonotes}
\usepackage{enumitem}
\usepackage{epigraph}

\makeatletter
\let\origsection\section
\renewcommand\section{\@ifstar{\starsection}{\nostarsection}}

\newcommand\nostarsection[1]
{\sectionprelude\origsection{#1}\sectionpostlude}

\newcommand\starsection[1]
{\sectionprelude\origsection*{#1}\sectionpostlude}

\newcommand\sectionprelude{%
  \vspace{1em}
}

\newcommand\sectionpostlude{%
  \vspace{1em}
}
\makeatother

\renewcommand{\arraystretch}{1.5}

\newcommand*{\isoarrow}[1]{\arrow[#1,"\rotatebox{90}{\(\sim\)}"]}
\newcommand*{\isovert}{\rotatebox{90}{\(\sim\)}}

\newcommand{\C}{\mathbb{C}}
\newcommand{\bQ}{\mathbb{Q}}
\newcommand{\A}{\mathcal{A}}

\newcommand{\J}{\mathcal{J}}
\newcommand{\g}{\mathfrak{g}}
\newcommand{\orb}{\mathbb{O}}
\newcommand{\orbk}{\orb_{K}}
\newcommand{\torbk}{\widetilde{\orb}_{K}}
\newcommand{\horbk}{\widehat{\orb}_{K}}

\newcommand{\Z}{\mathbb{Z}}
\newcommand{\gr}{\operatorname{gr}}

\newcommand{\fD}{\mathfrak{D}}

\newcommand{\fg}{\mathfrak{g}}

\newcommand{\fh}{\mathfrak{h}}

\newcommand{\fk}{\mathfrak{k}}

\newcommand{\fl}{\mathfrak{l}}

\newcommand{\fo}{\mathfrak{o}}
\newcommand{\fP}{\mathfrak{P}}

\newcommand{\fq}{\mathfrak{q}}

\newcommand{\fr}{\mathfrak{r}}
\newcommand{\fS}{\mathfrak{S}}
\newcommand{\fs}{\mathfrak{s}}

\newcommand{\ft}{\mathfrak{t}}

\newcommand{\fu}{\mathfrak{u}}

\newcommand{\ckfg}{{\check{\fg}}}

\newcommand{\ckfh}{{\check{\fh}}}
\newcommand{\ckorb}{\check{\orb}}

\newcommand{\fso}{\mathfrak{so}}
\newcommand{\fsp}{\mathfrak{sp}}
\newcommand{\fmp}{\widetilde{\mathfrak{sp}}}

\newcommand{\cC}{\mathcal{C}}
\newcommand{\cA}{\mathcal{A}}
\newcommand{\cF}{\mathcal{F}}
\newcommand{\cK}{\mathcal{K}}
\newcommand{\cL}{\mathcal{L}}
\newcommand{\cM}{\mathcal{M}}
\newcommand{\cN}{\mathcal{N}}

\newcommand{\cP}{\mathcal{P}}

\newcommand{\cS}{\mathcal{S}}
\newcommand{\cE}{\mathcal{E}}
\newcommand{\cJ}{\mathcal{J}}
\newcommand{\cV}{\mathcal{V}}

\newcommand{\cZ}{\mathcal{Z}}

\newcommand{\cW}{\mathcal{W}}
\newcommand{\cU}{\mathcal{U}}

\newcommand{\sD}{\mathscr{D}}

\newcommand{\Ug}{\cU\fg}
\newcommand{\Zg}{\cZ\fg}
\newcommand{\Abar}{\bar{A}}

\newcommand{\hb}{\hbar}
\newcommand{\series}{\llbracket \hb \rrbracket}

\newcommand{\WV}{W^{\scriptscriptstyle V}}
\newcommand{\cWV}{\cW^{\scriptscriptstyle V}}
\newcommand{\cWVbar}{\overline{\cW}^{\scriptscriptstyle V}}
\newcommand{\MV}{M^{\scriptscriptstyle V}}
\newcommand{\PBPG}{\PBP_{\GR}(\ckorb)}
\newcommand{\PBPGext}{\PBP_{\GR}^{\mathrm{ext}}(\ckorb)}
\newcommand{\MgK}{\cM_f(\g,K)}
\newcommand{\tGamma}{\widetilde{\Gamma}}

\DeclareMathOperator{\Lie}{Lie}
\DeclareMathOperator{\codim}{codim}
\DeclareMathOperator{\Sing}{Sing}
\DeclareMathOperator{\Ind}{Ind}
\DeclareMathOperator{\Ad}{Ad}
\DeclareMathOperator{\ad}{ad}
\DeclareMathOperator{\AV}{AV}
\DeclareMathOperator{\AC}{AC}
\DeclareMathOperator{\Ann}{Ann}
\DeclareMathOperator{\supp}{supp}
\DeclareMathOperator{\Vect}{Vec}
\DeclareMathOperator{\AOD}{AOD}
\DeclareMathOperator{\IrrAC}{IrrAC}
\DeclareMathOperator{\Unip}{Unip}

\DeclareMathOperator{\der}{Der}
\DeclareMathOperator{\Rep}{Rep}
\DeclareMathOperator{\sq}{sq}
\DeclareMathOperator{\Diff}{Diff}
\DeclareMathOperator{\Pic}{Pic}
\DeclareMathOperator{\Mod}{Mod}
\DeclareMathOperator{\Coh}{Coh}
\DeclareMathOperator{\Aut}{Aut}

\DeclareMathOperator{\Mp}{\widetilde{Sp}}

\DeclareMathOperator{\res}{Res}
\DeclareMathOperator{\PBP}{PBP}
\DeclareMathOperator{\Cov}{\mathcal{C}ov}
\DeclareMathOperator{\Sat}{Sat}

\DeclareMathOperator{\spr}{Spr}
\DeclareMathOperator{\sgn}{sgn}

\newcommand{\AODK}{\AOD_K}
\newcommand{\cAOD}{\mathcal{AOD}}
\newcommand{\UnipOv}{\Unip_{\ckorb}}
\newcommand{\cUnipOv}{\mathcal{U}nip_{\ckorb}}

\newcommand{\q}{\mathfrak{q}}

\newcommand{\SL}{\operatorname{SL}}
\newcommand{\SO}{\operatorname{SO}}
\newcommand{\Sp}{\operatorname{Sp}}

\newcommand{\Spin}{\operatorname{Spin}}

\newcommand{\Id}{\operatorname{Id}}
\newcommand{\pr}{\operatorname{pr}}
\newcommand{\smallhalf}{{\scriptscriptstyle 1/2}}
\newcommand{\mC}{\widetilde{C}}

\newcommand{\R}{\mathbb{R}}

\newcommand{\T}{\mathcal{T}}

\newcommand{\TT}{\mathscr{T}}

\newcommand{\tatp}{\TT^+(\Q, Y)}
\newcommand{\tat}{\TT(\Q,Y)}
\newcommand{\PicAlg}{\mathscr{PA}}

\newcommand{\quan}{\mathcal{Q}}
\newcommand{\per}{\operatorname{Per}}
\newcommand{\cm}{\C^\times}
\newcommand{\OO}{\mathcal{O}}
\newcommand{\Q}{\OO_\hb}

\newcommand{\rf}{\mathfrak{r}}
\newcommand{\GR}{G_\R}

\newcommand{\KR}{K_\R}

\newcommand{\spec}{\operatorname{Spec}}

\newcommand{\slf}{\mathfrak{sl}}
\newcommand{\Dcal}{\mathcal{D}}
\newcommand{\Hom}{\operatorname{Hom}}

\newcommand{\GL}{\operatorname{GL}}
\newcommand{\PSL}{\operatorname{PSL}}

\newcommand{\Irr}{\operatorname{Irr}}

\newcommand{\aut}{{\operatorname{Aut}}}
\newcommand{\Pe}{\cP_\epsilon}
\newcommand{\Pee}{\cP^{sp}_{\epsilon,\,\epsilon'}}

\newtheorem{Thm}{Theorem}[subsection]
\newtheorem{Prop}[Thm]{Proposition}
\newtheorem{Cor}[Thm]{Corollary}
\newtheorem{Lem}[Thm]{Lemma}
\theoremstyle{definition}
\newtheorem{Ex}[Thm]{Example}
\newtheorem{defi}[Thm]{Definition}
\newtheorem{Rem}[Thm]{Remark}

\Crefname{defi}{Definition}{Definition}
\numberwithin{equation}{section}

\title{Special unipotent representations and the coadjoint orbit method}
\author{Shilin Yu}
\dedicatory{Dedicated to my parents}

\begin{document}

\begin{abstract}
For any linear real reductive group or metapletic group, this article gives a geometric classification of special unipotent representations in the weak Arthur/Adams-Barbasch-Vogan (ABV) packets attached to quasi-distinguished nilpotent orbits in the Langlands or metaplectic dual Lie algebra in terms of their associated cycles.
We provide a uniform construction of the Harish-Chandra modules of these representations via deformation quantization of admissible vector bundles over certain Lagrangian subvarieties of the affinizations of the universal covers of special nilpotent orbits in question, which aligns with the coadjoint orbit method philosophy of Kirillov, Kostant, and Vogan. As consequences, all such Harish-Chandra modules have irreducible associated cycles, and are unitarizable by results from the theory of mixed Hodge modules. Conjecturally, the unitarity of all ABV packets can be reduced to the case when the dual orbits are distinguished.
Our approach highlights the application of Lusztig's conjecture on the geometry of special pieces, which has been proven for classical Lie algebras by Kraft and Procesi, and for all cases by Juteau, Levy, Sommers and the author.
\end{abstract}

\maketitle

\epigraph{\it All nilpotent orbits are equal, but some orbits are more equal than others.}{}

\tableofcontents

\section{Introduction}\label{sec:intr}

\subsection{Background and goals} \label{subsec:backgrounds}

The classification of irreducible unitary representations of real reductive groups is a long-standing central problem in the representation theory of Lie groups. The foundational work of Vogan and his collaborators has provided a comprehensive framework for understanding the unitary dual of such groups, culminating in the \texttt{atlas} algorithm of Adams–van Leeuwen–Trapa–Vogan \cite{ALTV}, which determines the unitarity of representations of reductive groups and has been implemented in the \texttt{atlas} software package \cite{atlas}. Another powerful approach arises from Howe’s theory of theta lifting \cite{Howe79, Howe89}, which has been an effective tool for constructing unitary representations of classical groups (see \Cref{subsec:previous_work} for further references). More recently, Saito’s theory of mixed Hodge modules \cite{Saito88, Saito90} has been applied by Schmid–Vilonen \cite{SchmidVilonen} and Davis–Vilonen \cite{DavisVilonen:MHM_real_groups, DavisVilonen:unitary} to provide a Hodge-theoretic perspective on the unitarity problem.

Two major sources of inspiration and examples for irreducible unitary representations have guided the field for decades. One is geometric, rooted in the coadjoint orbit method philosophy initiated by Kirillov and Kostant; the other is arithmetic, arising from the theory of automorphic representations developed by Langlands, Arthur, and others. The present article aims to weave these perspectives together, integrating symplectic and algebro-geometric methods with Hodge-theoretic approaches to shed new light on the unitarity problem.

The coadjoint orbit method proposes that unitary representations of a (connected) Lie group can be realized as “quantizations’’ of its coadjoint orbits—symplectic manifolds naturally endowed with the Kirillov–Kostant–Souriau symplectic structure. This idea mirrors the broader paradigm of quantizing classical mechanical systems into quantum ones: while no universal mathematical definition of quantization exists, it has proven remarkably effective in many concrete contexts. 
Vogan \cite{Vogan_AV} reformulated the orbit method for reductive groups and conjectured that one can ``quantize'' certain \emph{admissible vector bundles} over a single nilpotent orbit to obtain irreducible (and conjecturally unitarizable) Harish–Chandra modules, provided a certain boundary codimension condition is satisfied. These Harish–Chandra modules were constructed via deformation quantization by Losev \cite{Losev:quan_symp_orbit} in the case of complex reductive groups, and by Leung–Yu \cite{LY} and Losev–Yu \cite{LosevYu} in the case of real reductive groups.

On the arithmetic side, Arthur conjectured the existence of what are now known as \emph{Arthur packets}—certain finite sets of representations of reductive algebraic groups over local and global fields \cite{Arthur84, Arthur89}—parameterized by Arthur parameters.
The representations occurring in these packets are now referred to as Arthur representations.
This framework predicts, in particular, the existence of a fundamental collection of representations called \emph{special unipotent representations} of any linear (real) reductive group $\GR$, attached to unipotent Arthur parameters.
Special unipotent representations were defined by Barbasch and Vogan \cite{BV85:unipotent} for complex semisimple groups (regarded as real groups), and subsequently extended to general Arthur parameters and general linear real reductive groups by Adams, Barbasch, and Vogan \cite{ABV}.

Arthur \cite[Section 4]{Arthur89} and Adams--Barbasch--Vogan \cite[Chapter 1, Problem F]{ABV} conjectured that representations in Arthur/ABV packets are unitary. 
Moreover, these special unipotent representations—and more generally the broader, still conjectural class of unipotent representations—are expected to form fundamental building blocks of the unitary dual of a linear reductive group \cite{Vogan:unitary}.

In the present article, we combine the orbit method with mixed Hodge theory to deduce the unitarity of all special unipotent representations of any linear real reductive group or of any metaplectic group $\GR$, attached to (quasi-)distinguished nilpotent orbits $\ckorb$ in $\ckfg^*$ (see \Cref{defn:dist_orbit,defn:quasi-dist-saturation}), where $\ckfg$ is the Langlands/metaplectic dual Lie algebra of the complexified Lie algebra $\fg$ of $\GR$. In this setting, Vogan's original conjecture remains valid even without the codimension condition: the corresponding Harish--Chandra modules are realized as deformation quantizations of admissible vector bundles over nilpotent orbits.
In particular, all such Harish-Chandra modules have irreducible associated cycles, and their unitarity follows from the unitarity criterion derived from mixed Hodge theory in \cite{DMB-Hodge}. Finally, the general case is reduced to this setting via the work of Adams, Ionov, Mason-Brown, and Vogan \cite{AIMV:ArthurPackets}.

The case of classical groups was previously established by Barbasch, Ma, Sun, and Zhu \cite{BMSZ:counting, BMSZ:construction_unitarity} via the method of theta lifting \cite{Howe79, Howe89}, together with parabolic and cohomological induction. They constructed all special unipotent representations of classical groups and computed their associated cycles, observing that those attached to (quasi-)distinguished nilpotent orbits $\ckorb$ in $\ckfg^*$ possess irreducible associated cycles. This key observation provides one of the starting points and motivations for the present work.

The goal of this article is to develop a uniform geometric approach to the unitarity conjecture of Arthur and Adams--Barbasch--Vogan for all linear reductive groups, one that remains faithful to the original philosophy of the coadjoint orbit method while revealing new connections with the intricate algebraic and symplectic geometry of special pieces, studied in \cite{Spaltenstein,Lusztig:unipotent}.
A central ingredient in our approach is the study of certain coverings of nilpotent orbits, in particular, the universal cover in our setting. This is closely related to the refined Barbasch--Vogan--Lusztig--Spaltenstein duality map defined in \cite{LMBM}, which was subsequently generalized in \cite{MBMY} to encompass coverings of more general, non-special nilpotent orbits. 
A key discovery underlying the present work, however, is that Lusztig’s special piece conjecture \cite{Lusztig:unipotent} (now a theorem by \cite{Kraft-Procesi:special,JLSY:SpecialPiece}) implies that the geometry of (the affinizations of) those covers of special orbits relevant to our application exhibits remarkably favorable properties not shared by general nilpotent orbit covers. These properties are essential for applying deformation quantization to construct the desired representations. 

\subsection{Summary of approach} \label{subsec:approach}

Let $G$ be a simply connected complex simple group defined over $\R$, with Lie algebra $\fg$, and let $\GR = G(\R)$ denote its group of real points, which is necessarily connected. In the main body of the paper, we will work with more general real reductive groups; see \Cref{subsec:basic_settings} for details. For the purposes of the introduction, however, it suffices to restrict attention to this case.

Fix a maximal compact subgroup $\KR$ of $\GR$, and let $K$ be its complexification. Let $\fk$ denote the (complex) Lie algebra of $K$. The real group $\GR$ then determines a Harish--Chandra symmetric pair $(\fg, K)$. We work in the category of Harish--Chandra (HC) $(\fg, K)$-modules, which is closely related to the representation theory of $\GR$. Let $\Irr(\fg, K)$ denote the set of equivalence classes of irreducible HC $(\fg, K)$-modules.

As mentioned in \Cref{subsec:backgrounds}, we will connect the coadjoint orbit method philosophy to the study of special unipotent representations of $\GR$. The orbit method has been successfully developed for nilpotent Lie groups \cite{Kirillov_nil}, solvable Lie groups of type I \cite{AK}, and compact Lie groups, using geometric quantization. Thus, as proved in \cite{Duflo}, the general case reduces to reductive groups, where substantial challenges arise, particularly for nilpotent orbits.
In \cite{Vogan_AV}, Vogan reformulated the orbit method for reductive groups and conjectured that one can quantize admissible vector bundles (see \Cref{subsec:AOD}) over nilpotent $K$-orbits $\orbk$ in $\fk^\perp := (\fg/\fk)^* \subset \fg^*$ to obtain irreducible (unitarizable) HC modules, with some additional conditions on $\orbk$. Namely, let $\orb := \Ad(\fg) \cdot \orbk$ be the $\Ad(\fg)$-saturation of $\orbk$, which is a nilpotent $\Ad(\fg)$-orbit in $\fg^*$. Vogan's conjecture treats the case when the boundary $\partial\orb$ has codimension at least 4 in the Zariski closure $\overline{\orb}$ of $\orb$. 


Our main technical tool is the prior work \cite{LY}, where general results concerning equivariant formal deformation quantization of Lagrangian subvarieties of conical symplectic varieties were developed.
The principal challenge in applying the results of \cite{LY} lies in the codimension condition in Vogan's formulation. Let $\orb$ be the Barbasch--Vogan dual orbit of a (quasi-)distinguished nilpotent orbit $\ckorb$ in $\ckfg^*$, and let $\partial \orb$ be its boundary. Then $\partial \orb$ generally has codimension $2$ in $\overline{\orb}$, while the setting of \cite{LY} requires $\partial \orb$ to have codimension at least $6$. Moreover, it was shown in \cite{LosevYu} that Vogan’s conjecture fails in some cases when the codimension equals $4$.

To overcome this obstacle, we instead consider the affinization $X$ of the universal cover $\widetilde{\orb}$ of $\orb$. The variety $X$ is a conical symplectic variety whose canonical quantization yields the special unipotent ideal attached to $\ckorb$ by \cite{LMBM}. A key new input is Lusztig’s conjecture on the geometry of special pieces \cite{Lusztig:unipotent}, proven for classical groups by Kraft and Procesi \cite{Kraft-Procesi:special} and, in full generality, by the author and collaborators \cite{JLSY:SpecialPiece}. It implies that the singular locus of $X$ has codimension at least $6$.

We can then apply the results of \cite{LY} to quantize admissible orbit data in $\orb$ (or rather their lifts to the smooth locus of $X$) and construct a collection of HC modules of special unipotent representations attached to $\ckorb$.
These modules are in bijection with the admissible orbit data in $\orb$, with the bijection given by taking associated cycles. In particular, these HC modules manifestly have irreducible associated cycles.

To complete the classification, it remains to show that the HC modules obtained by quantization exhaust the special unipotent representations attached to $\ckorb$. We do so by comparing the number of these HC modules with the number of admissible orbit data in $\orb$. For classical groups, this comparison has already been carried out in \cite{BMSZ:counting}. For exceptional groups, we use the \texttt{atlas} software to count special unipotent representations and to determine the number of admissible orbit data in each case, relying on the tables of component groups of nilpotent $K$-orbits in \cite{King}\footnote{Along the way, we found and corrected an error in \cite{King}; see \Cref{rem:Z_K}.}. 

It is worth noting that we employ \texttt{atlas} only to count the number of special unipotent representations, without verifying unitarity (which would be computationally much more intensive). Furthermore, the current \texttt{atlas} algorithm can compute associated varieties of representations, but not their associated cycles. Therefore, our result establishing the irreducibility of associated cycles for exceptional groups is genuinely new and cannot be verified directly using \texttt{atlas}.

Our conclusion about the unitarity of special unipotent representations relies crucially on the theory of mixed Hodge modules developed by Saito \cite{Saito88, Saito90}. Schmid and Vilonen \cite{SchmidVilonen} had the deep insight that mixed Hodge modules provide a powerful tool for detecting unitarity of HC modules. They observed that every irreducible HC module admits a canonical Hodge filtration coming from its Beilinson--Bernstein localization (\cite{BeilinsonBernstein}) as a twisted $\sD$-module on the complex flag variety. A unitarity criterion in terms of Hodge filtrations was proved in \cite{DavisVilonen:unitary} based on the conjecture in \cite{SchmidVilonen} and the algorithm computing Hodge filtrations (modulo conjectures) in \cite{AdamsTrapaVogan}.

It was conjectured in \cite[Remark 7.3.4]{LY} that the filtration of irreducible HC modules arising from quantization in \cite{LY} and \cite{LosevYu} should coincide with the Hodge filtration up to degree shift, when the nilpotent orbit $\orb$ of the complex group is birationally rigid. This conjecture has recently been proven by Davis and Mason-Brown \cite{DMB-Hodge} using mixed Hodge theory, as a byproduct of their proof of a general unitarity criterion \cite[Theorem 5.22]{DMB-Hodge} for HC modules with irreducible associated cycles (see \Cref{thm:unitarity_criterion}) and weakly unipotent annihilating ideals. 

As mentioned above, one of the key observations motivating the present paper is that the \emph{irreducibility of associated cycles} condition in \cite[Theorem~5.22]{DMB-Hodge} (see \Cref{thm:unitarity_criterion}(ii)) is also satisfied for special unipotent representations attached to (quasi-)distinguished nilpotent orbits $\ckorb$ of $\ckfg^*$. This holds even though the dual orbit $\orb=d(\ckorb)$ in $\fg^*$ very often has boundary of codimension~$2$, and therefore fails to meet the hypotheses in Vogan’s original conjecture and the main theorem, Theorem 1.4, of \cite{DMB-Hodge}. Moreover, our classification shows that the relevant HC modules are Hermitian (see \Cref{prop:Hermitian_UnipOv}), a property that should be understood prior to any discussion of unitarity and is not evident from the original definition of ABV packets (see \Cref{subsec:previous_work}).

\subsection{Previous work} \label{subsec:previous_work}

There are two different approaches to defining Arthur packets. One of them is due to Arthur himself \cite{Arthur:endoscopic} for the split forms of general linear groups and quasi-split forms of symplectic and orthogonal groups, using harmonic analysis together with both local and global methods. This approach was later extended to include quasi-split unitary groups by Mok \cite{Mok} and pure real forms of symplectic and orthogonal groups by Mœglin and Renard \cite{MoeglinRenard2020}. 

The other approach is due to Adams, Barbasch, and Vogan \cite{ABV} for the real points of any connected reductive algebraic group defined over $\R$. They defined what are now known as \emph{ABV packets} of representations attached to local Arthur parameters in terms of microlocal geometry and sheaf theory. They established a number of properties of these representations conjectured by Arthur except for their unitarity; see Problems A--F in \cite[Introduction]{ABV}. 

For the real reductive groups to which Arthur's analytic approach applies, it has been shown in \cite{AdamsArancibiaMezo,ArancibiaMezo:unitary,ArancibiaMezo:symplectic_orthogonal} that the Arthur packets coincide with the ABV packets. For more detailed references and history on the analytic approach and its relation to the ABV approach, we refer the reader to \cite[Section 1.3]{BMSZ:construction_unitarity}.

Special unipotent representations (of integral infinitesimal characters) of complex semisimple groups were classified by Barbasch and Vogan \cite{BV85:unipotent} and, in general, by Wong \cite{Wong:Richardson}. An alternative approach from the perspective of quantization of nilpotent orbit covers was obtained in \cite{LMBM}. The case of real groups was studied in \cite{ABV}. 

There is a large literature on the unitarity of special unipotent representations and more general representations. The unitary dual of general linear groups was completely understood by the work of Vogan \cite{Vogan:unitary_dual_GL}. For complex classical groups, Barbasch has classified the unitary dual in \cite{Barbasch:unitary_dual_complex_classical} and the unitarity of special unipotent representations follows from it. The unitarity of special unipotent representations for general complex reductive groups has recently been confirmed by Davis and Mason-Brown \cite{DMB-Hodge} using mixed Hodge theory.

For both complex and real classical groups, it was realized (see, e.g., \cite{Howe82, Li}) that Howe's theory of theta lifting \cite{Howe79, Howe89} is a powerful tool for constructing unitary representations. There has been much work on constructing special unipotent representations and more general unipotent representations using theta lifting; see \cite{Sahi, Przebinda91, Przebinda93, HuangZhu, HuangLi, Brylinski, He, Trapa, PaulTrapa, Barbasch:unipotent_dual_pair, Moeglin:unipotent_Howe, BarbaschWong:normality} for an incomplete list. 
Recently, Barbasch, Ma, Sun, and Zhu have constructed all special unipotent representations of classical groups via the method of theta lifting and parabolic induction \cite{BMSZ:counting, BMSZ:construction_unitarity} and confirmed their unitarity. They have also treated real spin groups via parabolic and cohomological induction and confirmed their unitarity \cite{BMSZ:spin}. 

The case of real exceptional groups is much less understood. Miller proved the unitarity of certain spherical special unipotent representations of exceptional groups in \cite{Miller} by realizing them as residues of Eisenstein series with the help of computer computations.
In the ongoing work \cite{AMVV}, Adams, Miller, van Leeuwen, and Vogan have checked the unitarity of all special unipotent representations of real exceptional groups case-by-case using the \texttt{atlas} software. 

\subsection{Main results} \label{subsec:main_results}

We assume the setting and notation in \cref{subsec:approach}, so that $\fg$ is a simple Lie algebra over $\C$. Let $\ckfg$ be its Langlands dual Lie algebra. Let $\fh$ and $\ckfh$ be the universal Cartan subalgebras of $\fg$ and $\ckfg$, respectively. Then $\ckfh$ is canonically identified with the linear dual $\fh^*$ of $\fh$. Let $\ckorb \subset \ckfg^*$ be a nilpotent coadjoint orbit (with respect to the coadjoint action of the adjoint group). We can regard $\ckorb$ as a nilpotent adjoint orbit in $\ckfg$ via the $\Ad(\ckfg)$-equivariant isomorphism $\ckfg \simeq \ckfg^*$ given by any nondegenerate $\Ad(\ckfg)$-invariant bilinear form on $\ckfg$. Pick $\check{e} \in \ckorb$ and complete it into an $\slf_2$-triple $(\check{e}, \check{h}, \check{f})$ with semisimple element $\check{h}$. The $\Ad(\ckfg)$-orbit of $\check{h}/2$ is uniquely determined by $\ckorb$ and hence gives rise to an element in $\ckfh/W = \fh^*/W$. Pick any representative $\lambda_{\ckorb} \in \fh^*$ of $\check{h}/2 \in \fh^*/W$. As in \cite[Section 5]{BV85:unipotent}, the element $\lambda_{\ckorb}$, or rather the $W$-orbit $W \cdot \lambda_{\ckorb}$, determines an algebraic character $\chi_{\lambda_{\ckorb}}$ of the center $\Zg$ of the universal enveloping algebra $\Ug$ of $\fg$ via the Harish-Chandra isomorphism $\Zg \simeq (S\fh)^W$. 

Let $\J_{\ckorb} := \J_{max}(\lambda_{\ckorb})$ be the maximal primitive ideal of $\Ug$ with infinitesimal character $\lambda_{\ckorb}$, which depends only on $\ckorb$ and is independent of the choice of $\check{e}$ and the $\slf_2$-triple. The \emph{Barbasch-Vogan (BV) dual orbit} $d(\ckorb) \in \cN_o$ of $\ckorb$ is defined to be the unique Zariski open $\Ad(\fg)$-orbit in the associated variety $\cV(\J_{\ckorb})$ of $\J_{\ckorb}$, guaranteed by \cite[Theorem 3.10]{Joseph:associated_variety}; see \cite[Appendix]{BV85:unipotent}.
Following \cite{BV85:unipotent}, we call $\lambda_{\ckorb}$ and $\J_{\ckorb}$ the \emph{special unipotent infinitesimal character} and \emph{special unipotent ideal}, respectively, attached to $\ckorb$.

The \emph{weak Arthur/ABV packet} attached to $\ckorb$ for $(\g,K)$ is the (finite) subset of $\Irr(\g,K)$ defined by
\begin{equation} \label{eq:UnipOv_defn}
	\UnipOv(\g, K) = \UnipOv(\GR) := \{M \in \Irr(\g,K) \,|\, \Ann(M) = \J_{\ckorb}\}.
\end{equation}
Note that since $\J_{\ckorb}$ is a maximal ideal, $M \in \UnipOv(\g, K)$ if and only if $M$ is annihilated by $\J_{\ckorb}$, equivalently, if and only if it has infinitesimal character $\lambda_{\ckorb}$ and its complex associated variety is the Zariski closure of the BV dual orbit $d(\ckorb)$ of $\ckorb$. The HC modules in $\UnipOv(\g, K)$ are called \emph{special unipotent representations} of $\GR$ attached to $\ckorb$. By \cite[Corollary 27.13]{ABV}, $\UnipOv(\g, K)$ is the union of those Arthur/ABV packets whose corresponding Arthur parameters are unipotent and have restrictions to $\SL_2(\C)$ determined by $\slf_2$-triples associated to $\ckorb$. 

In this article, we will apply the method in \cite{LY} to a smooth closed Lagrangian subvariety $Y$ of the smooth locus of the affinization $X$ of the universal cover $\widetilde{\orb}$ of a special nilpotent orbit $\orb$ in $\g^*$. Here $\orb$ is the BV dual orbit of any (quasi-)distinguished nilpotent orbit $\ckorb$ in $\ckfg^*$. Such $Y$ has an open dense $K$-stable subset $\torbk$ that is a finite $K$-equivariant cover of a nilpotent $K$-orbit $\orbk$ in $\orb_{\fk^\perp} := \orb \cap \fk^\perp$. 

As mentioned in \Cref{subsec:approach}, the main obstacle to applying the method in \cite{LY} is that the boundary $\partial \orb$ of $\orb$ generally has codimension $2$. By the (proven) conjecture of Lusztig on the geometry of special pieces (\Cref{thm:special_piece}), we are able to prove the following result (see \Cref{thm:codim_X} for the general case of quasi-distinguished nilpotent orbits).

\begin{Thm} \label{thm:main_codim_X_dist}
	Let $\ckorb$ be a distinguished nilpotent orbit in $\ckfg^*$ and let $\orb$ be the Barbasch--Vogan dual orbit of $\ckorb$. Let $\widetilde{\orb}$ be the universal cover of $\orb$ and let $X = \spec(\C[\widetilde{\orb}])$ be the affinization of $\widetilde{\orb}$. Then the singular locus $X^{sing}$ of $X$ has codimension at least $6$ in $X$.
\end{Thm}

We show that any $K$-equivariant admissible vector bundle $\cE$ over $\torbk$ can be extended uniquely to a $K$-equivariant twisted $\sD$-module $\overline{\cE}$ over $Y$, where the twist is half (or a square root) of the canonical line bundle $\omega_Y$ on $Y$ (\Cref{prop:extn_AOD}). We can then apply the method in \cite{LY} to construct a quantization of $\overline{\cE}$ over $Y$ to obtain a (reducible) HC $(\g,K)$-module; its decomposition into irreducible components gives all HC modules in $\UnipOv(\g, K)$. 

To summarize, let $\AODK(\orb)$ denote the set of (equivalence classes of) all {\it admissible orbit data} (see \Cref{subsec:AOD}) $(\orbk, \cV)$ with $\orbk$ a $K$-orbit in $\orb_{\fk^\perp}$ and $\cV$ a $K$-equivariant admissible vector bundle over $\orbk$. The main result of the article is as follows (\Cref{thm:Unip_AOD_bij}).


\begin{Thm}\label{thm:main_bij}
	Assume that $G$ is a simply connected complex simple group defined over $\R$ and $\GR = G(\R)$, with $(\g,K)$ the associated symmetric pair. Let $\ckorb$ be a (quasi-)distinguished nilpotent orbit in $\ckfg^*$ and let $\orb = d(\ckorb)$. Then the construction above gives a bijection 
	\[ \quan_{\ckorb}^K: \AODK(\orb) \xrightarrow{\ \scriptstyle\sim \ } \UnipOv(\g, K),\]
	whose inverse is given by taking the associated cycle of the HC module. 
	
	In particular, all HC modules in $\UnipOv(\g, K)$ have irreducible associated cycles in the sense of \Cref{defn:irred_AC}.
\end{Thm}

Note that the results also hold when $\GR$ is the metaplectic double cover $\Mp(2n,\R)$ of the real symplectic group $\Sp(2n,\R)$, in the framework of the metaplectic Barbasch--Vogan duality map (\cite{BMSZ:metaplecticBV}). This case and the case in which $\GR$ is a classical linear group were previously established in \cite[Theorem 5.3, (b)]{BMSZ:construction_unitarity} using theta lifting. Our approach is different and works uniformly for both classical and exceptional groups.


\Cref{thm:main_bij} has important implications for the unitarity of HC modules in $\UnipOv(\g, K)$. In \cite[Theorem 5.22]{DMB-Hodge}, it is shown that any irreducible Hermitian HC module with irreducible associated cycle, whose annihilator ideal is maximal and (very) weakly unipotent, is unitarizable. \Cref{thm:main_bij} already implies the irreducibility of the associated cycles of HC modules in $\UnipOv(\g, K)$. 
We will show in \Cref{prop:Hermitian_UnipOv} that all HC modules in $\UnipOv(\g, K)$ are Hermitian, again using a geometric argument based on \Cref{thm:main_bij}. Even this is not immediately obvious from the original definition of ABV packets, at least to the author and to the experts in the field whom he has consulted. The weak unipotence of $\J_{\ckorb}$ for distinguished or even nilpotent orbits $\ckorb$ was shown previously in \cite[Proposition 5.10]{BV85:unipotent}, and later generalized to arbitrary $\ckorb$ by Ma and the author in \cite[Theorem 2.19]{MaYu:weak_unipotence}. Therefore \cite[Theorem 5.22]{DMB-Hodge} implies the following theorem (\Cref{thm:unitarity_dist} and \cref{cor:unitarity_dist_reductive}).

\begin{Thm} \label{thm:main_unitarity_dist}
	All HC modules in the weak Arthur/ABV packet $\UnipOv(\g, K)$ attached to a (quasi-)distinguished nilpotent orbit $\ckorb$ in $\ckfg^*$ are unitarizable.
\end{Thm}

Finally, the recent work of Adams, Ionov, Mason-Brown, and Vogan \cite{AIMV:ArthurPackets} shows that the unitarity of general ABV packets can be reduced to that of unipotent ABV packets. More precisely, representations in general ABV packets can be obtained via parabolic and cohomological induction of special unipotent representations of Levi subgroups, and the inductions preserve unitarity. The unitarity of general ABV packets (\cref{thm:main_unitarity_ABV} below) then follows from that of special unipotent representations, where were established in different cases via completely different methods (classical groups v.s. exceptional groups, complex groups v.s. real groups), see the proof of \cite[Theorem 1.0.5]{AIMV:ArthurPackets}. 

On the other hand, Jeffrey Adams has informed the author that the method in \cite{AIMV:ArthurPackets} can also be used to reduce the unitarity of unipotent ABV packets to that of unipotent ABV packets attached to distinguished orbits of $\ckfg$.
Once established, this, together with \cref{thm:main_unitarity_dist} in the current article, will give a uniform proof of \cref{thm:main_unitarity_ABV}.

\begin{Thm} \label{thm:main_unitarity_ABV}
	All ABV packets consist of unitary representations.
\end{Thm}

\subsection{Organization of the paper}\label{subsec:organization}

We begin in \Cref{sec:preliminaries} by fixing the conventions on real reductive groups, nilpotent orbits and their covers, component groups, Slodowy slices, and associated cycles. The geometric input is established in \Cref{sec:geometry_special_pieces}: Lusztig's special pieces are used there to prove the codimension estimate for the affinization of covers of Barbasch--Vogan duals of quasi-distinguished orbits. With this input in place, \Cref{sec:adm_bundle} develops the language of Picard algebroids, twisted differential operators, admissible vector bundles, and admissible orbit data, including the modifications needed for metaplectic and spin groups. The quantization construction is then carried out in two stages. First, \Cref{sec:orbit_method_complex} treats the complex-group case and gives a priori constraints for the associated cycles of HC modules in $\UnipOv(\g, K)$ when $\ckorb$ is quasi-distinguished or $\g$ is classical metaplectic; then \Cref{sec:orbit_method_real} adapts the construction to real groups and proves the exhaustion statement by counting. The passage from construction to unitarity occupies \Cref{sec:unitarity}, where Hermitianness is proved and unitarity criterion from the mixed Hodge theory is applied. Finally, \Cref{sec:res_quant} studies the behaviour of the quantization functor/map under restriction/inflation1, while \Cref{sec:correction_component} supplies a component-group computation needed in the exceptional cases.

\subsection{Notation and conventions}\label{subsec:notations}

Unless otherwise specified, algebraic groups and varieties are defined over the field $\C$ of complex numbers. All group actions on varieties, sheaves, and modules are algebraic; for vector spaces and modules, this means locally finite.

Let $\phi: \widetilde{H} \to H$ be a finite surjective morphism of algebraic groups. A representation, sheaf, or any other object with a $\widetilde{H}$-action is called \emph{genuine} with respect to $\phi$ if the action does not factor through $H$. We use the superscript $^{gen}$ for the corresponding subset or full subcategory. 

For an algebraic group $G$, we denote by $G^\circ$ its identity component and by $\pi_0 G := G/G^\circ$ its component group. For a Lie group or algebraic group $G$, we write $\Lie(G)$ for its Lie algebra; we also use the corresponding fraktur letter when no ambiguity is possible, for example, $\fg=\Lie(G)$. If $V$ is a linear representation of $G$, then $V^G$ denotes the subspace of $G$-fixed vectors.

For an abelian or exact category $\cC$, we write $\cK\cC$ for its Grothendieck group and $\cK_{\geqslant 0}\,\cC$ for the submonoid of classes of objects of $\cC$. We write $\cK_{>0}\,\cC$ for the corresponding positive semigroup with the zero class removed. 
\vskip 1em

\subsection{Acknowledgements}
The author would like to thank Jeffrey Adams, Dougal Davis, Baohua Fu, Daniel Juteau, Paul Levy, Ivan Losev, Jia-Jun Ma, Lucas Mason-Brown, Dmytro Matvieievskyi, Kari Vilonen, Eric Sommers, Binyong Sun, Daniel Wong, and Chen-Bo Zhu for helpful discussions. The author is also grateful to the developers of the \texttt{atlas} software.

The author is indebted to Binyong Sun for persuading him several times to consider the problem of constructing special unipotent representations with irreducible associated cycles, which the author initially did not think possible, before realizing its relationship with Lusztig's special pieces conjecture. The author would also like to thank Jeffrey Adams for sharing his unpublished notes \cite{Adams:quasi-dist} containing a general definition of quasi-distinguished nilpotent orbits for all complex reductive groups that inspires our \cref{defn:quasi-dist-saturation}.

The work of the author has been partially supported by the Fundamental and Interdisciplinary Disciplines Breakthrough Plan of the Ministry of Education of China (FIDBP), by NSFC grants (Grant Nos. 12471028 and 12131018), and by the Natural Science Foundation of Fujian Province (Grant No. 2022J06005).

\section{Preliminaries}
\label{sec:preliminaries}

\subsection{Basic settings} \label{subsec:basic_settings}

We will work with the field $\C$ of complex numbers and the subfield $\R$ of real numbers. We follow the general setting in \cite[Section 8.1]{DavisVilonen:unitary} and define a {\it real reductive group} to be a real Lie group $\GR$ such that there exists a complex reductive group $G$ with an anti-holomorphic involutive automorphism $\varpi$, and a finite covering homomorphism from $\GR$ to a union of connected components of the group $G(\R):=G^\varpi$ of real points of $G$ (with respect to $\varpi$). Let $\g := \Lie(G)$ denote the (complex) Lie algebra of $G$, and let $\g_\R = \g^{\varpi}$ be the Lie algebra of $\GR$, which is a real form of $\g$.

Fix an anti-holomorphic involution $\varpi_c$ of $G$, commuting with $\varpi$, such that the $\varpi_c$-fixed point subgroup $U_{\R} := G^{\varpi_c}$ is a compact real form of $G$.
Then $\theta:=\varpi \circ \varpi_c$ is a Cartan involution of $G$. We use the same notation to denote the induced involution on $\g$. Then $U_\R \cap G(\R) = G(\R)^\theta$ is a maximal compact subgroup of $G(\R)$ with complexification $G^\theta$. Let $\KR$ denote the preimage of $G(\R)^\theta$ in $\GR$, and let $K$ denote its complexification. Then $\KR$ is a maximal compact subgroup of $\GR$ and $K$ is a finite cover of a union of connected components of $G^\theta$. Let $\fk$ denote the (complex) Lie algebra of $K$. Then it is naturally a (complex) Lie subalgebra of $\g$. Then $(\g,K)$ forms a symmetric Harish-Chandra pair. 

Most of our main results only concern the case when $G$ is a simply connected complex simple group and $\GR = G(\R)$. In this case, $\GR$ is not a complex group itself, since the complex group case has already been treated in \cite{DMB-Hodge}.
However, the more general setting above is necessary for several reasons. First, it is natural to consider the case of a {\it linear real reductive group}, i.e., a (closed) subgroup of finite index in $G(\R)$. This occurs in our treatment of the identity components $\SO_0(p,q)$ of the indefinite special orthogonal groups $\SO(p,q)$, see \Cref{cor:counting_classical_identity_component}. 

Also note that in the above general setting, we allow $G$ to be disconnected for technical reasons, since in \Cref{subsubsec:counting_classical} we will have to deal with disconnected orthogonal groups $G = \mathrm{O}(N)$ and $\GR = \mathrm{O}(p,q)$ with $N=p+q$, which is not of Harish-Chandra class. However, we will not deal with representations of $\mathrm{O}(p,q)$ directly; we only need this case to translate the counting result for $\mathrm{O}(p,q)$ in \cite[Proposition 4.11]{BMSZ:construction_unitarity} to $\SO(p,q)$ (see \Cref{lem:res_Opq} and \Cref{prop:counting_classical}). 

We will also deal with the metaplectic group $\Mp(2n,\R)$, the unique connected $2$-fold covering group of $\Sp(2n,\R)$, which is not a linear real reductive group. We will prove results analogous to those for linear classical groups in a somewhat uniform way. See \Cref{subsec:special_duality,subsec:special_duality_classical}, as well as \cref{cor:metaplectic_genuine} for more details.

Later we will consider the following situation: suppose we have $(G, \varpi)$ as well as two real reductive groups $\GR$ and $\GR'$ with homomorphisms $\GR \to G$ and $\GR' \to G$ and maximal compact subgroups $\KR$ and $\KR'$ as above respectively. Suppose there is a finite homomorphism $\phi: \GR \to \GR'$ of Lie groups intertwining $\GR \to G$ and $\GR' \to G$. Then $\phi$ induces a morphism $\phi: (\g, K) \to (\g, K')$ of symmetric pairs, which is the identity map of $\g$ and a finite homomorphism $\phi: K \to K'$ of algebraic groups onto a closed subgroup of finite index in $K'$ that intertwines the $\Ad$-actions of $K$ and $K'$ on $\g$. Then the pullback of $\GR'$-representations to $\GR$-representations corresponds to pullback of the corresponding HC $(\g,K')$-modules to $(\g,K)$-modules.

In the special case where $\phi: \GR \to \bar{G}_\R$, and hence $\phi: K \to \bar{K}$, is a finite surjective morphism, the pullback of representations via $\phi$ is just inflation.
Then we have natural embeddings $\UnipOv(\g, \bar{K}) \hookrightarrow \UnipOv(\g, K)$ by inflation. 
We write $\UnipOv(\g,K)^{gen} := \UnipOv(\g,K) \setminus \UnipOv(\g, \bar{K})$ for the subset of $\UnipOv(\g,K)$ consisting of isomorphism classes of \emph{genuine} (with respect to $\phi: K \to \bar{K}$) special unipotent irreducible HC $(\g,K)$-modules attached to $\ckorb$, that is, those modules that do not factor through $\bar{K}$. One important example is the 2-fold covering morphism $\Mp(2n,\R) \to \Sp(2n,\R)$.

\subsection{Nilpotent orbits}\label{subsec:nil_orbits}

We recall basic definitions and notation related to nilpotent orbits. We refer the reader to \cite{CM} for details.

Let $G$ be a connected complex reductive algebraic group with Lie algebra $\g=\Lie(G)$. Let $Z(G)$ be the center of $G$ and $G_{ad} := G/Z(G)$ be the adjoint group of $\g$. We work with the coadjoint action of $G$ on $\g^*$. Fix once and for all a nondegenerate $G$-invariant bilinear form $B$ on $\g$ (e.g. the Killing form when $\g$ is semisimple) and use it to identify $\g\cong\g^*$ as $G$-representations. Under this identification, we speak interchangeably of elements of $\g$ and of $\g^*$; in particular a coadjoint element $\xi\in\g^*$ will often be denoted by the corresponding element $e\in\g$. Sometimes we also allow $G$ to be disconnected with finitely many connected components, such as $G=\mathrm{O}(N)$.

The $G$-(co)adjoint orbits are the same as $G_{ad}$-(co)adjoint orbits. Kirillov, Kostant, and Souriau observed that any coadjoint orbit $\orb$ in $\fg^*$ carries a natural $G$-invariant algebraic symplectic form $\Omega=\Omega_\orb$, which is referred to as the \emph{KKS symplectic form} (see \cite[(5.14)(b)]{Vogan_AV}). Recall that a coadjoint $G$-orbit $\orb\subset\g^*$ is called nilpotent if (equivalently) its image in $\g$ under the fixed identification consists of nilpotent elements, or equivalently $0\in\overline{\orb}$ (see \cite[Theorem 5.7]{Vogan_AV}), where $\overline{\orb}$ denotes the closure of $\orb$ in $\g^*$ in the Zariski topology. 

We denote the set of nilpotent coadjoint $G$-orbits in $\g^*$ by $\cN_o$. It is well-known that $\cN_o$ is a finite set. Moreover, it is a partially ordered set under the closure order, i.e., $\orb' \preceq \orb$ (we say $\orb'$ is \emph{dominated} by $\orb$) if and only if $\orb' \subset \overline{\orb}$. 
Let $\cN^*(\g)$, or simply $\cN^*$, denote the nilpotent cone in $\g^*$, which is the $G$-stable closed subvariety of $\g^*$ consisting of all nilpotent elements of $\g^*$. Let $\cN$ be the similarly defined closed subset in $\g$. Then the identification $\g \simeq \g^*$ induces a $G$ (or $G_{ad}$)-equivariant isomorphism $\cN^* \simeq \cN$.

Given any $G$-orbit $\orb$ in $\cN$, the (finite) component groups $G_e/G_e^\circ$ of the centralizer $G_e$ of various $e \in \orb$ are all isomorphic to each other and the isomorphisms are unique up to composition with some inner automorphisms of the groups. Therefore we can define the {\it ($G$-equivariant) component group} $A_G(\orb):= G_e/G_e^\circ$ of $\orb$ by abuse of notation and terminology. Then $A_G(\orb)$ is isomorphic to the $G$-equivariant fundamental group $\pi_1^G(\orb)$ of $\orb$. When the group $G$ in question is clear, we will simply write $A(\orb) = A_G(\orb)$. When $G$ is a simply connected semisimple group, $A(\orb)$ is just the fundamental group $\pi_1(\orb)$. We will describe the component group $A(\orb)$ in more detail in \Cref{subsec:component_groups}.

By the Jacobson-Morozov Theorem, we can complete $e$ into an $\slf_2$–triple $(e,h,f)$, with $h \in \g$ semisimple and $f \in \g$ nilpotent. Such an $\slf_2$-triple is unique up to $G_{ad}$-conjugation. We then define the \emph{reductive centralizer} $Q:=Z_G(e,h,f)$, 
which is a reductive (indeed Levi) subgroup of $G_e$. The unipotent radical $U:=R_u(G_e)$ is the maximal connected normal unipotent subgroup of $G_e$. By \cite[Proposition 2.4]{BV85:unipotent} we have the semidirect product decomposition $G_e \cong Q \ltimes U$,
equivalently on Lie algebras $\fg_e=\fq \oplus \fu$ with $\fq=\Lie(Q)$ reductive and $\fu=\Lie(U)$ nilpotent; $Q$ acts on $U$ (and on $\fu$) via the adjoint action. Since $U$ is connected, we have $A_G(\orb) = G_e/G_e^\circ \simeq Q/Q^\circ$.

Now let $\GR \to G$ be a real reductive Lie group as in \Cref{subsec:basic_settings}, such that its Lie algebra $\g_\R$ is a real form of $\g$. We have a Cartan decomposition $\g = \fk \oplus \fs$ of $\g$ into $\pm 1$-eigenspaces of $\theta$, where $\fk = \g^\theta = \Lie(K)$ and $\fs = \g^{-\theta}$. The vector subspace $\fs$ can be identified $K$-equivariantly with $\fk^\perp = (\g/\fk)^*$ via a $\GR$-invariant nondegenerate bilinear form on $\g_\R$. 

Let $\cN_\theta^*:=\cN^* \cap \fk^\perp$ be the closed subvariety consisting of all nilpotent elements in $\fk^\perp \subset \g^*$, which is $K$-stable. We similarly define $\cN_\theta := \cN \cap \fs$. Then the isomorphism $\cN^* \simeq \cN$ restricts to a $K$-equivariant isomorphism $\cN^*_\theta \simeq \cN_\theta$. It is known that $\cN_\theta$ (and hence $\cN^*_\theta$) has finitely many $K$-orbits; see \cite{KostantRallis}. We will denote $K$-orbits in $\cN_\theta^* \simeq \cN_\theta$ by symbols such as $\orbk$, to distinguish them from nilpotent $G$-orbits $\orb$ in $\cN^* \simeq \cN$.

Let $\sigma = -\theta$. Then $\sigma$ is a $K$-equivariant anti-Poisson involution of $\g^*$. Let $\orb$ be a nilpotent orbit in $\g^*$. Recall that $\orb_{\fk^\perp} = \orb \cap \fk^\perp$. We can identify $\orb_{\fk^\perp}$ with its image $\orb \cap \fs$ under the isomorphism $\g^* \simeq \g$. Assume that $\orb_{\fk^\perp} \neq \varnothing$. Then both $\sigma$ and $\theta$ preserve $\orb$. In this case $\sigma$ restricts to an anti-Poisson involution of $\orb$ and $\orb_{\fk^\perp} = \orb^{\sigma}$ is a smooth closed Lagrangian subvariety of $\orb$, with respect to the KKS symplectic form $\Omega_\orb$ (\cite[Corollary 5.20]{Vogan_AV}).

Given a nilpotent $K$-orbit $\orbk$ in $\fs$ and any $e \in \orbk$, we can complete $e$ into a \emph{normal} $\slf_2$–triple $(e,h,f)$, i.e., $h \in \fk$ and $f \in \fs$. We can again define the reductive $K$-centralizer $K_Q:=Z_K(e,h,f) \subset K_e$. 
We again have a semidirect decomposition $K_e = K_Q \ltimes U^\theta$, where $U^\theta$ is the unipotent radical of $K_e$. We then define the $K$-component group $Z_K:=K_e/K_e^\circ \simeq K_Q/K_Q^\circ$ of $\orbk$.

By \cite{Sekiguchi:bijection}, \cite{Djokovic:bijection}, and \cite{Vergne}, there is a natural bijection, commonly known as the \emph{Kostant-Sekiguchi-Vergne (KSV) correspondence}, between the set of nilpotent $\GR$-orbits in $\g_\R^*$ and the set of nilpotent $K$-orbits in $\fk^\perp$ (see also \cite[Section 9.5]{CM}). 
For real exceptional groups $\GR$, the nilpotent $\GR$-orbits in $\g_\R^*$ are classified in \cite{Djokovic:centralizer_inner, Djokovic:centralizer_outer}, and hence we also have a classification of the nilpotent $K$-orbits in $\fk^\perp$.

Nilpotent $K$-orbits of real classical groups $\GR$ are classified by \emph{$ab$-diagrams}, which are Young diagrams of the corresponding partitions associated to the $G$-orbits with labels $a$ and $b$ filled in each box, alternating in each row, satisfying extra conditions. For details, see \cite{Kraft-Procesi:normal} and \cite{Ohta:Adm}. Also see \cite[Section 9.3]{CM} and \cite[Proposition 3.5]{GomezZhu:local_theta_lifting} for the classification of nilpotent $\GR$-orbits in $\g_\R^*$ in terms of signed Young diagrams. The two classifications are related by the KSV correspondence. The reductive centralizer $K_Q$ and $K_Q^\circ$ of a nilpotent $K$-orbit can be computed directly from the $ab$-diagrams using \cite{Ohta:Adm}, or one can first compute the reductive centralizer of the corresponding nilpotent $\GR$-orbit under the KSV correspondence, in terms of the signed Young diagrams using \cite[(3.7)]{GomezZhu:local_theta_lifting}, and then translate it to the $K$-orbit side using \cite[Remark 9.5.2]{CM}.

\subsection{Covers of nilpotent orbits} \label{subsec:orbit_covers}

For our purposes, we will also frequently work with covers of nilpotent orbits. By a \emph{$G$-equivariant cover} of a nilpotent $G$-orbit $\orb$ in $\fg^*$, we mean a finite \'etale $G$-equivariant morphism $\rho: \widetilde{\orb} \to \orb$, where $\widetilde{\orb}$ is a connected $G$-variety. Let $X := \spec(\C[\widetilde{\orb}])$ be the affinization of $\widetilde{\orb}$. The covering morphism $\rho: \widetilde{\orb} \to \orb$ induces a surjective morphism $\rho: X \to \overline{\orb}$. The KKS symplectic form on $\orb$ pulls back via $\rho$ to a $G$-invariant algebraic symplectic form on $\widetilde{\orb}$. This construction is studied by Brylinski and Kostant \cite{Brylinski-Kostant}. 


We will also consider covers of nilpotent $K$-orbits $\torbk$ in $\fk^\perp \simeq \fs$. Since we do not assume either the group $K$ or the orbit $\orbk$ to be connected, some explanations are necessary.

By a ($K$-equivariant) cover of $\orbk$, we mean a variety $\torbk$ equipped with a transitive $K$-action and a $K$-equivariant finite \'{e}tale (necessarily surjective) morphism $\tilde{\rho}: \torbk \to \orbk$.

A morphism between two covers $\hat{\rho}: \horbk \to \orbk$ and $\tilde{\rho}: \torbk \to \orbk$ is a $K$-equivariant morphism $\varphi: \horbk \to \torbk$ such that $\tilde{\rho} \circ \varphi = \hat{\rho}$, which is necessarily finite, \'{e}tale, and surjective. Let $\Cov_K(\orbk)$, or simply $\Cov(\orbk)$ when $K$ is obvious from the context, denote the category of all $K$-equivariant covers of $\orbk$. We can similarly define a $K$-equivariant cover $\varphi: \horbk \to \torbk$ of the cover $\torbk$ of $\orbk$ with covering morphism $\tilde{\rho}: \torbk \to \orbk$, as a $K$-equivariant finite \'{e}tale morphism $\varphi: \horbk \to \torbk$. Then $\hat{\rho} = \tilde{\rho} \circ \varphi: \horbk \to \orbk$ is also an object in $\Cov_K(\orbk)$ and $\varphi$ is precisely a morphism from $\hat{\rho}$ to $\tilde{\rho}$ in $\Cov_K(\orbk)$. 

We next record the corresponding deck-transformation group. Suppose $\tilde{\rho}: \torbk \to \orbk$ is an object of $\Cov_K(\orbk)$. We define the ($K$-equivariant) Galois group of $\tilde{\rho}$ as
	\[ \Gamma_{\tilde{\rho}} = \Aut_K(\tilde{\rho}) := \{ \sigma \in \Aut_K(\torbk) \,\mid\, \tilde{\rho} \circ \sigma = \tilde{\rho} \},\]
where $\Aut_K(\torbk)$ is the group of $K$-equivariant automorphisms of $\torbk$ as a $K$-variety. This can similarly be defined for coverings of orbit covers.
Fix a point $e \in \orbk$ and let $y$ be any point in the preimage of $e$ in $\torbk$. We have inclusions of groups $K_e^\circ \subset K_y \subset K_e$ and well-defined $K$-equivariant isomorphisms 
	\[ K/K_e \simeq \orbk, \quad k K_e \mapsto k \cdot e, \quad \text{and} \quad K/K_y \simeq \torbk, \quad k K_y \mapsto k \cdot y.\]
Then the Galois group $\Gamma_{\tilde{\rho}}$ can be identified with $N_{K_e}(K_y)/K_y$, where $N_{K_e}(K_y)$ is the normalizer of $K_y$ in $K_e$, via the action of $N_{K_e}(K_y)$ on $\torbk$ given by
\begin{equation}\label{eq:Aut_K}
	 \bar{n} \cdot (k \cdot y) := (k n^{-1}) \cdot y, \quad \forall\, n \in N_{K_e}(K_y), \, k \in K,
\end{equation}
where $\bar{n}$ denotes the image of $n$ in $N_{K_e}(K_y)/K_y$.

We now define the universal object in this category. We say that $\hat{\rho}: \horbk \to \orbk \in \Cov_K(\orbk)$ is a \emph{universal} $K$-equivariant cover of $\orbk$ if it is \emph{almost} an initial object in $\Cov_K(\orbk)$, i.e., for any $\tilde{\rho}: \torbk \to \orbk \in \Cov_K(\orbk)$, there exists at least one morphism $\varphi$ from $\hat{\rho}$ to $\tilde{\rho}$ in $\Cov_K(\orbk)$, but this morphism is not necessarily unique\footnote{An alternative standard approach is to define a similar category consisting of covers $\torbk$ with a fixed base point $y \in \torbk$. Then a universal cover is indeed an initial object in this category in the usual sense. We prefer a base-free definition.}.

In concrete terms, $\hat{\rho}: \horbk \to \orbk$ is a universal $K$-equivariant cover if and only if the $K$-stabilizer $K_z$ of some (and hence any) point $z \in \horbk$ is equal to $K_e^\circ$, where $e = \hat{\rho}(z)$. Such universal covers always exist and are isomorphic to one another (but the isomorphism is not necessarily unique). In the most concrete terms, they can be identified with $K/K_e^\circ$ ($K$-equivariantly).

Similarly, given any $\tilde{\rho}: \torbk \to \orbk \in \Cov(\orbk)$, we can define a universal cover $\varphi: \horbk \to \torbk$ of $\torbk$ as well. In this case, $\hat{\rho}=\tilde{\rho} \circ \varphi: \horbk \to \orbk$ is also a $K$-equivariant universal cover of $\orbk$.

The Galois groups of universal covers can be described in component-group terms. Note that for the universal cover $\hat{\rho}: \horbk \to \orbk$, its Galois group $\Gamma_{\hat{\rho}}$ can be naturally identified with the $K$-component group $Z_K = K_e/K_e^\circ$ by the discussion above. For a general cover $\tilde{\rho}: \torbk \to \orbk$, we can similarly identify the Galois group $\Gamma_{\varphi}$ of the universal cover $\varphi: \horbk \to \torbk$ of $\torbk$ with its component group $Z_K(\tilde{\rho}) := \pi_0(K_y) = K_y/K_y^\circ = K_y/K_e^\circ$. We have a natural injective homomorphism $Z_K(\tilde{\rho}) \simeq \Gamma_{\varphi} \hookrightarrow \Gamma_{\hat{\rho}}\simeq Z_K$, which can be identified with the obvious inclusion map $K_y/K_e^\circ \hookrightarrow K_e/K_e^\circ$. We have canonical identifications   
\begin{equation} \label{eq:Galois_normalizer}
    \Gamma_{\tilde{\rho}} \simeq N_{K_e}(K_y)/K_y \simeq N_{Z_K}(Z_K(\tilde{\rho})) / Z_K(\tilde{\rho}).
\end{equation}

Note that for a $K$-equivariant cover $\tilde{\rho}: \torbk \to \orbk$, the degree $\deg \tilde{\rho}$ of the finite morphism $\tilde{\rho}$ is well-defined since $K$ acts transitively on $\orbk$ and $\torbk$. We have $\deg \tilde{\rho} = |K_e/K_y|$ for any choice of $e \in \orbk$ and $y \in \tilde{\rho}^{-1}(e)$. 

The action of $\Gamma_{\tilde{\rho}}$ on $\torbk$ induces a morphism 
	\[ (\pr_2, \alpha): \Gamma_{\tilde{\rho}} \times \torbk \to \torbk \times \torbk, \]
where $\pr_2: \Gamma_{\tilde{\rho}} \times \torbk \to \torbk$ denotes the projection morphism to the second factor $\torbk$, and $\alpha: \Gamma_{\tilde{\rho}} \times \torbk \to \torbk$ denotes the action of $\Gamma_{\tilde{\rho}}$ on $\torbk$.

Finally, we recall the standard Galois condition for these possibly disconnected covers. Even though we are dealing with disconnected orbits and covers, the usual notion of Galois coverings can be extended to this case without difficulty, as reflected in the following definition/proposition. The proof is standard, so we omit it. 

\begin{defi} \label{defn:Galois} 
We say that a $K$-equivariant cover $\tilde{\rho}: \torbk \to \orbk$ is a {\it Galois cover} if the following equivalent conditions are satisfied:
\begin{enumerate}[label=(\roman*), itemindent=*, leftmargin=*]
    \item $|\Gamma_{\tilde{\rho}}|$ is equal to the degree of $\tilde{\rho}: \torbk \to \orbk$.
	\item The induced morphism $\Gamma_{\tilde{\rho}} \times \torbk \to \torbk \times \torbk$ is an isomorphism.
	\item The induced morphism from the quotient $\torbk/\Gamma_{\tilde{\rho}}$ of $\torbk$ by $\Gamma_{\tilde{\rho}}$ to $\orbk$ is a ($K$-equivariant) isomorphism.
    \item For one (and hence any) choice of $e \in \orbk$ and $y \in \tilde{\rho}^{-1}(e)$, $K_y$ is a normal subgroup of $K_e$, or equivalently, $Z_K(\tilde{\rho})$ is a normal subgroup of $Z_K$. In this case, $\Gamma_{\tilde{\rho}} \simeq K_e/K_y \simeq Z_K / Z_K(\tilde{\rho})$.
\end{enumerate}
\end{defi}

\subsection{\texorpdfstring{$\cm$}{C\textasciicircum{}*}-actions on nilpotent orbit covers} \label{subsec:orbit_cm}
Note that any nilpotent $G$-orbit $\orb$ in $\g$ and any nilpotent $K$-orbit $\orbk$ in $\fs$ are preserved by the dilation $\cm$-action on $\g$ and $\fs$ respectively. By \cite[Lemma 1.3]{Brylinski-Kostant}, the square of the dilation $\cm$-action lifts uniquely to any $G$-equivariant cover $\widetilde{\orb}$ of $\orb$. The same holds for any $K$-equivariant cover $\torbk$ of $\orbk$. We call such a $\cm$-action the \emph{Kostant-Brylinski $\cm$-action}. 

We briefly recall the construction from \cite[Section 2.4]{LY}. Fix $e \in \orbk$ and complete it into a normal $\slf_2$-triple $(e,h,f)$. Let $C_h := \exp(\C h)$ be the connected subgroup of $K$ generated by $h$. Then there is a well-defined surjective homomorphism of groups
\begin{equation}\label{eq:cm2C}
	s: \cm \to C_h, \quad \exp t \mapsto \exp(t h).
\end{equation}
The subgroup $C_h$ normalizes $K_e$ and $K_e^\circ$, hence also acts continuously on the discrete finite group $Z_K = K_e/K_e^\circ$ by conjugation. Since $C_h$ is connected, the $C_h$-action on $Z_K$ is trivial.

Let $y$ be any point in the preimage of $e$ in $\torbk$, so that $\torbk \simeq K/K_y$. The discussion above implies that $C_h$ normalizes both $K_e$ and $K_y$, and therefore there is a well-defined left action of $\cm$ on $\torbk \simeq K/K_y$ via \eqref{eq:cm2C} and \eqref{eq:Aut_K}, by right multiplication through $s$.
Since $[h,e]=2e$, the covering morphism $\tilde{\rho}: \torbk \to \orbk$ is $\cm$-equivariant, where we take the square of the dilation $\cm$-action on $\orbk$. This $\cm$-action on $\torbk$ is independent of the choice of $e$ and $y$, and is the unique $\cm$-action on $\torbk$ making $\tilde{\rho}$ $\cm$-equivariant. Therefore the $\cm$-action commutes with the $K$-action and with any automorphism $\sigma \in \Gamma_{\tilde{\rho}}$ of the covering $\tilde{\rho}$.

\subsection{Nilpotent orbits in classical Lie algebras}\label{subsec:nil_classical}


It is well-known that nilpotent orbits in classical Lie algebras are classified by certain partitions, which we briefly recall below. We refer to \cite[Section 5.1]{CM} for further details. Nilpotent orbits in $\slf_n$ are in one-to-one correspondence with the set of partitions of $n$ (\cite[Proposition 3.1.7 \& Theorem 5.1.1]{CM}). Next we recall the classification of nilpotent orbits in types $B$, $C$, and $D$ (\cite[Theorem 5.1.2--5.1.6]{CM}) in terms of partitions with constraints.

For a partition $\lambda$ of $N$, we always write its parts in a non-increasing sequence 
	\[ \lambda = [\lambda_1 \geqslant \lambda_2 \geqslant \cdots \geqslant \lambda_k \geqslant 0]. \]
For convenience, we sometimes allow $0$ as a part of $\lambda$.
For a part $a$ of $\lambda$, let $m(a) = m_\lambda(a)$ denote the multiplicity of $a$ in $\lambda$ and $h(a) = h_\lambda(a)$ denote the height of $a$ in $\lambda$, i.e., $h(a)=\sum_{b\geqslant a} m(b)$. We make the convention that $m(0) = 0$. Let $\#\lambda$ denote the number of nonzero parts of $\lambda$.

Let $V$ be a vector space of dimension $N$ endowed with a non-degenerate bilinear form satisfying $\langle w,v\rangle = (-1)^\epsilon \langle v,w\rangle$, where $\epsilon\in \{ 0,1\}$.
Let $G = G_\epsilon(V)$ be the stabilizer of this form and let $\g=\Lie(G)=\g_\epsilon(V)$, also denoted $\g_\epsilon(N)$, denote its Lie algebra.
(Note that $G$ is disconnected if $\epsilon=0$.)

An $\epsilon$-partition of $N$ is a partition $\lambda$ of $N$ such that $m(a)$ is even for all $a\equiv \epsilon\bmod 2$. Let $\Pe(N)$ denote the set of all $\epsilon$-partitions of $N$.
It is well-known that the nilpotent $G$-orbits in $\g$ are in one-to-one correspondence with the $\epsilon$-partitions of $N$. In all classical types, nilpotent $G$-orbits are the same as $G^\circ$-orbits, except that those $G$-orbits in type $D$ corresponding to partitions with all even parts (called the \emph{very even partitions/orbits}) split into two distinct orbits of $G^\circ = \SO(2n)$ that are interchanged by the adjoint action of $G=\mathrm{O}(2n)$ on $\g$. Thanks to \cref{rem:dist_very_even}, we will never encounter the very even orbits in our main results involving partitions of type $D$ (e.g., \Cref{prop:adm_Spin_pq} and \cref{prop:adm_SO*(2n)}). Other than this exception, we will write $\orb_\lambda$ for the nilpotent $G^\circ$-orbit corresponding to the $\epsilon$-partition $\lambda$.

\subsection{Component groups for complex groups} \label{subsec:component_groups}

The list of component groups of nilpotent orbits for adjoint and simply connected exceptional groups can be found in \cite{Alexeevski} (see also \cite[Section 8.4]{CM}). In this subsection, we shall describe the component groups and their canonical quotients for classical groups in terms of partitions.

As in \cref{subsec:nil_classical}, let $G = G_\epsilon(V)$, where $\epsilon\in \{ 0,1\}$ and $V$ is a vector space of dimension $N$ endowed with a non-degenerate bilinear form satisfying $\langle w,v\rangle = (-1)^\epsilon \langle v,w\rangle$. Suppose $\orb$ is a nilpotent orbit in $\fg=\fg_\epsilon(V)$.
Let $A(\orb)$ and $A^+(\orb)$ denote the component groups of the centralizer of $e\in\orb$ in $G^\circ$ and (possibly disconnected) $G$ respectively. These can be computed in terms of the partition $\lambda$ as follows. Let $\nu_1>\ldots > \nu_l$ be the parts of $\lambda$ which are not congruent to $\epsilon$ modulo 2. When $\epsilon = 1$, we make the convention that $\nu_l=0$; when $\epsilon = 0$, we have $\nu_i>0$ for all $i$.
Let $e \in\orb$ and complete it into an $\slf_2$-triple $(e,h,f)$. Then the reductive centralizer $Q \subset G_e$ in $G$ is always a product of orthogonal and symplectic groups. More precisely, a part $a$ of $\lambda$ such that $a \not\equiv \epsilon \bmod 2$ (resp. $a \equiv \epsilon\bmod 2$) contributes a factor ${\rm O}_{m(a)}$ (resp. $\Sp_{m(a)}$) to $Q$.

The symplectic factors are connected, hence only the orthogonal groups ${\rm O}_{m(\nu_i)}$ contribute to the component group $A^+(\orb)$ in $G$. Note that when $\epsilon = 1$, we follow the convention that $m(0) = 0$ and regard ${\rm O}_{m(\nu_l)} = {\rm O}_{0}$ as the trivial group, so it does not contribute to the component group. We denote the image in $A^+(\orb)$ of the non-trivial element of the component group of ${\rm O}_{m(\nu_i)}$ by $x_{\nu_i}$ and regard $x_{0}$ as the trivial element. Therefore $A^+(\orb)$ is isomorphic to $\fS_2^l$ when $\epsilon = 0$ and $\fS_2^{l-1}$ when $\epsilon = 1$, or uniformly, $A^+(\orb) \simeq \fS_2^{l-\epsilon}$.
In both orthogonal and symplectic cases, $A(\orb)$ is the subgroup of $A^+(\orb)$ generated by the elements $x_{\nu_i}x_{\nu_{i+1}}, 1\leqslant i \leqslant l-1$, and hence $A(\orb) \simeq \fS_2^{l-1}$.


\subsection{Slodowy slices} \label{subsec:Slodowy}

Let $e'$ be a nilpotent element in $\g$, and let $\orb' := G \cdot e'$ be the associated nilpotent orbit in $\g$. Complete $e'$ into an $\slf_2$-triple $(e',h',f')$. Let $Q' = Z_G(e',h',f')$. Recall that the \emph{Slodowy slice} (associated to this triple) is the affine subspace of $\g$ defined by
	\[ \cS = \cS_{e'} = e' + \mathfrak{z}_{\g}(f') \subset \g,\] 
where $\mathfrak{z}_{\g}(f')$ denotes the $f'$-centralizer of $\g$.
Suppose $\orb$ is a nilpotent orbit in $\g$ such that $\orb' \preceq \orb$. Let $\cS_\orb := \cS \cap \overline{\orb}$. Then $\cS_\orb$ is a transversal slice to $\orb'$ in $\overline{\orb}$ at $e'$. Both $\cS$ and $\cS_\orb$ are stable under the adjoint $Q'$-action.

Recall that the \emph{Kazhdan $\cm$-action} on $\fg$ is defined by 
\begin{equation}\label{eq:Kazhdan}
	z \cdot \xi = z^{-2} \Ad(\gamma(z))(\xi), \qquad z \in \cm, \ \xi \in \fg, 
\end{equation}
where $\gamma: \cm \to G$ is the cocharacter associated to the semisimple element $h'$. This action commutes with the adjoint $Q'$-action on $\g$, stabilizes the Slodowy slice $\cS$ and the slice $\cS_\orb$, and contracts them onto $e'$. 

Suppose $\rho: \widetilde{\orb} \to \orb$ is a $G$-equivariant connected finite covering of $\orb$. Let $X = \spec(\C[\widetilde{\orb}])$. Then $\rho$ induces a $G$-equivariant finite surjective morphism $\rho: X \to \overline{\orb}$. Let $\check{X}$ be any connected component of $\rho^{-1}(\cS_\orb)$. This is a transverse slice to the preimage of $\orb'$ in $X$.  
By the proof of \cite[Proposition 5.12]{MBMY}, the Kazhdan action lifts uniquely to a $\cm$-action on $X$ and again commutes with the $G$-action. This lifted action stabilizes $\check{X}$ and contracts it to a unique point $\tilde{e}' \in X$ such that $\rho(\tilde{e}') = e'$. This makes $\check{X}$ a conical symplectic singularity.

By the representation theory of $\slf_2 \simeq \mathrm{span}\{h',e',f'\}$, we have a $Q'$-invariant direct sum decomposition 
\begin{equation} \label{eq:zf_decomp}
	\mathfrak{z}_{\g}(f') = \bigoplus_{k \leqslant 0} \mathfrak{z}_{\g}(f')_k,
\end{equation}
where
\[
\mathfrak{z}_{\g}(f')_k = \{ x \in \mathfrak{z}_{\g}(f') : [h',x] = kx\}.
\] 
Let $\pr_0: \cS_\orb \to \mathfrak{q}'$ be the restriction of the 
$Q' \times \cm$-equivariant linear projection of $\cS = e' + \mathfrak{z}_{\g}(f')$ onto $\mathfrak{q}' = \mathfrak{z}_{\g}(f')_0$. Then 
	\[\phi := \pr_0 \circ \, \rho\,|_{\check{X}} : \check{X} \to \q' \simeq (\q')^*\] 
is a moment map for the Hamiltonian $Q'^\circ$-action on $\check{X}$. The $Q'^\circ$-action on $\check{X}$ induces a group homomorphism $\Phi: Q'^\circ \twoheadrightarrow \aut_{\cm}(\check{X})^\circ$. 

Finally, note that if $e' \in \fs$ and the $\slf_2$-triple $(e',h',f')$ is normal, then the decomposition \eqref{eq:zf_decomp} is $\theta$-stable and hence $(-\theta)$-stable.

\subsection{Associated varieties and cycles} \label{subsec:AC}

In this subsection, we recall the basic definitions of the associated varieties and cycles of HC $(\g,K)$-modules and set up notation. We refer the reader to \cite{Vogan_AV} and \cite[\S\,5.2]{DMB-Hodge} for details.


Let $\MgK$ denote the abelian category of all $(\g,K)$-modules of finite length.
For any $M \in \MgK$, we can pick a good $K$-stable filtration $F_\bullet M$ on $M$ and consider the associated graded $\gr^F M$, which is a finitely generated $(S\g,K)$-module, where $S\g$ is the symmetric algebra generated by $\g$. Thus $\gr^F M$ gives rise to a $K$-equivariant coherent sheaf on $\fk^\perp = (\g/\fk)^*$. We define the \emph{associated variety} $\AV(M)$ of $M$ to be the set-theoretic support $\supp(\gr^F M)$ of the coherent sheaf associated with $\gr^F M$ in $\fk^\perp$, which is the same as the $K \times \cm$-invariant closed subvariety of $\fk^\perp$ defined by the radical of the annihilator $\Ann_{S(\g/\fk)}(\gr^F M)$ of $\gr^F M$ in $S(\g/\fk)$ (\cite[pp. 25--26]{Matsumura:comm_ring}). The variety $\AV(M)$ is independent of the choice of the good filtration and is known to be a $K$-stable closed subvariety in $\cN_\theta$, hence a union of finitely many nilpotent $K$-orbits in $\fk^\perp$. 

Recall that for a two-sided ideal $I$ in $\Ug$, one can similarly define the notion of the \emph{associated variety} $\cV(I)$ of $I$ in $\fg^*$ as follows: let $F_\bullet \Ug$ denote the usual degree filtration of $\Ug$. Then, by the Poincar\'{e}-Birkhoff-Witt theorem, there is a natural isomorphism of $G$-equivariant graded algebras $\gr^F \Ug \xrightarrow{\sim} S(\fg)$, via which $\gr(I) := \gr^F(I)$ can be identified with a $G$-invariant homogeneous ideal in $S(\fg)$. Then the \emph{associated variety} of $I$ is the Zariski closed $G \times \cm$-invariant subvariety $\cV(I) := V(\gr(I)) \subset \fg^*$ corresponding to $\gr(I)$. 

Recall that we have the Harish-Chandra isomorphism $\Zg \simeq (S\fh)^W$, via which any $\lambda \in \fh^*$ determines a character $\chi_\lambda$ of $\Zg$. An ideal $I \subset \Ug$ is said to have \emph{infinitesimal character} (resp., \emph{generalized infinitesimal character}) $\chi_\lambda$ if $I \cap \Zg$ contains the maximal ideal $\ker \chi_\lambda$ (resp., some power of $\ker \chi_\lambda$). In this case, the associated variety $\cV(I)$ is always contained in the nilpotent cone $\cN^* \subset \fg^*$. An ideal $I \subset \Ug$ is \emph{primitive} if it is the annihilator ideal $\Ann(M) = \Ann_{\Ug}(M)$ of some irreducible $\Ug$-module $M$. By Schur's lemma, every primitive ideal $I$ has an infinitesimal character. In particular, its associated variety $\cV(I)$ is a closed $G \times \cm$-invariant subset of the nilpotent cone $\cN^*$. By \cite{Joseph:associated_variety}, there is a unique open dense nilpotent $G$-orbit $\orb$ in $\cV(I)$. Consequently, $\cV(I) = \overline{\orb}$ and hence is irreducible. 

The following fundamental fact about associated varieties was proved in \cite[Theorem 8.4]{Vogan_AV}. 

\begin{Thm} \label{thm:V_AV}
	Let $M$ be an irreducible $(\fg,K)$-module and set $I=\Ann(M)$, which is a primitive ideal in $U(\fg)$. Let $\orb$ be the open dense $G$-orbit in $\cV(I)$. Then the maximal $K$-orbits in $\AV(M)$ all lie in $\orb_{\fk^\perp}$. In particular,
	$$\dim \AV(M) = \frac{1}{2} \dim \orb.$$
\end{Thm}

Next, we define the notion of the associated cycle of an HC $(\g,K)$-module, which is a refinement of the associated variety. Let us first set up some notation. Let $\cC(\g,K)$ denote the abelian category of all finitely generated $(S(\g/\fk),K)$-modules $N$ (i.e., $N$ is a finitely generated $S(\g/\fk)$-module with a compatible locally finite $K$-action, see \cite[(2.1)(c)]{Vogan_AV}), or equivalently, the category of all coherent $K$-equivariant sheaves on $\fk^\perp$, with the additional condition that $\supp(N) \subset \cN_\theta$. 

The module $\gr^F M$ determines a class $[\gr^F M]$ in $\cK_{\geqslant 0}\cC(\g,K)$. By \cite[Proposition 2.2]{Vogan_AV}, this class is independent of the choice of the good filtration. Therefore the assignment $M \mapsto \gr^F M$ gives rise to a well-defined group homomorphism 
  \[ \gr: \cK\MgK\rightarrow \cK\cC(\g,K), \quad [M] \mapsto [\gr^F M], \] 
and restricts to a monoid homomorphism $\gr: \cK_{\geqslant 0}\MgK\rightarrow \cK_{\geqslant 0}\cC(\g,K)$.

Let $\orbk^1, \ldots, \orbk^m$ be the distinct nilpotent $K$-orbits in $\cN_\theta$. For each $i$, let $\Vect^K(\orbk^i)$ denote the exact category of $K$-equivariant algebraic vector bundles on $\orbk^i$. Recall from \cite[\S\,5.2]{DMB-Hodge} that an \emph{associated cycle} for $(\g,K)$ is a formal (finite) sum 
	\[\sum_{i=1}^n [\cV_i] \cdot \orbk^i, \quad [\cV_i] \in \cK_{\geqslant 0}\Vect^K(\orbk^i), \quad \forall\, 1 \leqslant i \leqslant m, \] 
satisfying the following condition: 
	\[\text{If } [\cV_i] \neq 0 \text{ for some } i, \text{ then } \,[\cV_j] = 0, \,\forall\, j \neq i, \text{ such that } \orbk^j \subset \partial \orbk^i = \overline{\orb}_{K}^i \backslash \orbk^i. \]
In other words, if we denote by $\{\orbk^{i_1}, \ldots, \orbk^{i_p}\}$ the subset of $\{ \orbk^1, \ldots ,\orbk^m\}$ consisting of all $K$-orbits $\orbk^i$ with $[\cV_i] \neq 0$, then $\orbk^{i_j}$ is not contained in $\overline{\orb}_{K}^{i_k}$ for all $k \neq j$.

Let $\AC(\g,K)$ denote the set of all associated cycles for $(\g,K)$. Then $\AC(\g,K)$ is an abelian monoid with the addition defined by
\[
	\left(\sum_{i=1}^n [\cV_i] \cdot \orbk^i \right) + \left(\sum_{i=1}^n [\cV'_i] \cdot \orbk^i \right) := \sum_{i=1}^n [\cW_i] \cdot \orbk^i,\]
where $[\cW_i]$ is defined as follows: if there is a $j$ such that $\orbk^i \subseteq \partial \orbk^j$ and $[\cV_j] + [\cV'_j] \neq 0$, then we define $[\cW_i]=0$; otherwise, we define $[\cW_i] = [\cV_i] + [\cV'_i]$.

Let $N \in \cC(\g,K)$ and let $\orbk^{(1)},\ldots ,\orbk^{(q)}$ be the (distinct) maximal $K$-orbits in $\supp(N)$ (with respect to the closure relation). 
By \cite[Lemma 2.9]{Vogan_AV}, there exists a finite filtration $0=N_0 \subset N_1 \subset \cdots \subset N_n$ of $N$ by finitely generated $(S(\g/\fk),K)$-submodules, such that each subquotient $N_k/N_{k-1}$ is generically reduced along each irreducible component $\overline{\orb}_K^{(j)}$ of $\supp(N)$, $1 \leqslant j \leqslant q$. Therefore the pullback $\cV_j^k:=N_k/N_{k-1}|_{\orbk^{(j)}}$ of $N_k/N_{k-1}$ to each $\orbk^{(j)}$, as a sheaf of $\C$-vector spaces, is naturally a $K$-equivariant coherent sheaf of $\OO_{\orbk^{(j)}}$-modules, and hence a $K$-equivariant algebraic vector bundle on $\orbk^{(j)}$.
We then define
	\[ \mathcal{AC}(N) :=  \sum_{j=1}^{q} \left( \sum_{k=1}^{n} [\cV_j^k] \right) \cdot \orbk^{(j)} \in \AC(\g,K). \]
By \cite[Theorem 2.13]{Vogan_AV}, $\mathcal{AC}(N)$ is independent of the choice of the filtration of $N$, and we therefore obtain a well-defined monoid homomorphism 
  \[ \mathcal{AC}: \cK_{\geqslant 0}\cC(\g,K) \rightarrow \AC(\g,K), \quad [N] \mapsto \mathcal{AC}(N).\]
Finally, we can define the monoid homomorphism
\[ \AC := \mathcal{AC} \circ \gr : \cK_{\geqslant 0}\MgK \rightarrow \AC(\g,K).\]
We will simply write $\AC(M) = \AC([M])$ for any $M \in \MgK$, and call it the \emph{associated cycle} of $M$. Clearly, if $\AC(M) = \sum_{j=1}^q [\cV_j] \cdot \orbk^{(j)}$ with $[\cV_j] \neq 0$, then $\AV(M) = \bigcup_{j=1}^q \overline{\orb}_K^{(j)}$.

The map $\AC$ is compatible with pullback of HC modules as follows: assume we have a morphism $\phi: (\g,K) \to (\g,K')$ of symmetric pairs as in the last part of \cref{subsec:basic_settings}. Then we have the natural functors defined by pullback via
	\[ \phi^*: \cC(\g,K') \to \cC(\g,K) \quad \text{and} \quad \phi^*: \cM_f(\g,K') \to \MgK.\]
These functors are exact and hence induce homomorphisms between the corresponding Grothendieck groups and submonoids, as well as the monoids $\AC(\g, \bullet)$, which are still denoted as $\phi^*$. When $\phi: K \to K'$ is injective, so that $K$ can be identified with a closed subgroup of $K'$ of finite index, pullback means restriction of modules, and hence we also denote $\phi^*$ as $\res_{K}^{K'}$.

The following lemma is standard.

\begin{Lem} \label{lem:AC_pullback}
	We have $\phi^* \circ \AC = \AC \circ \,\phi^*$, i.e., the following diagram commutes:
	\[
		\begin{tikzcd}
			\cK_{\geqslant 0}\cM_f(\g,K') \arrow[r, "\AC"] \arrow[d, "\phi^*"] & \AC(\g,K') \arrow[d, "\phi^*"] \\
			\cK_{\geqslant 0}\MgK \arrow[r, "\AC"] & \AC(\g,K).
		\end{tikzcd}
	\]
\end{Lem} 

The following definition is standard and gives the case of interest for our application.

\begin{defi} \label{defn:irred_AC}
	An associated cycle in $\AC(\g,K)$ of the form $[\cV] \cdot \orbk$, where $\orbk$ is a single nilpotent $K$-orbit in $\cN_\theta$ and $\cV \in \Vect^K(\orbk)$ is an irreducible $K$-equivariant vector bundle over $\orbk$, is called an \emph{irreducible associated cycle}. Let $\IrrAC(\g,K) \subset \AC(\g,K)$ denote the subset consisting of all irreducible associated cycles. 
	If $\AC(M) \in \IrrAC(\g,K)$, we say that the $(\g,K)$-module $M$ has an irreducible associated cycle.
\end{defi}

\begin{Rem}
	When $K$ is connected, any $(\g,K)$-module with irreducible associated cycle also has irreducible associated variety as a variety. It is not true in general when $K$ is not connected.
\end{Rem}

\section{Geometry of special pieces}
\label{sec:geometry_special_pieces}

\subsection{Special orbits and duality} \label{subsec:special_duality}

Let $\g$ be a complex reductive Lie algebra with Weyl group $W$, and let $\ckfg$ be its Langlands dual Lie algebra with the same Weyl group $W$. Let $\ckfh$ be the abstract Cartan subalgebra of $\ckfg$, which is identified as the linear dual $\fh^*$ of the abstract Cartan subalgebra $\fh$ of $\g$. 

For any nilpotent orbit $\orb \in \cN_o$, let $\mathrm{Loc}(\orb)$ be the set of isomorphism classes of irreducible $G_{ad}$-equivariant local systems over $\orb$. The Springer correspondence is an injective map
\[ 
\spr_\g : \Irr(W) \hookrightarrow \{ (\orb, \rho) \,|\, \orb \in \cN_o, \, \rho \in \mathrm{Loc}(\orb) \}. 
\]
We also abbreviate $\spr_\g$ to $\spr$ when the relevant Lie algebra $\g$ is clear from the context.

Let $\Irr(W)^{sp} \subset \Irr(W)$ denote the set of isomorphism classes of {\it special representations} of $W$ in the sense of Lusztig (\cite{Lusztig:irred_Weyl_I,Lusztig:irred_Weyl_II}). By definition, a special representation of $W$ is an irreducible representation of $W$ such that its fake degree coincides with its generic degree. 
For any $\sigma \in \Irr(W)^{sp}$, it is known that $\spr_\g(\sigma)$ is of the special form $(\orb_\sigma, \mathbbm{1})$, where $\mathbbm{1}$ denotes the trivial local system (\cite{Lusztig:irred_Weyl_I}). We say that $\orb_\sigma$ is the \emph{special (nilpotent) orbit} attached to the special representation $\sigma$. Let $\cN_o^{sp}$ denote the set of all special nilpotent orbits in $\g^*$. 

Recall that we used the map in \Cref{subsec:main_results}, defined by Barbasch and Vogan \cite{BV85:unipotent},
	\[ d = d_{\fg}: \check{\cN}_o \to \cN_o, \quad \ckorb \mapsto \orb = d(\ckorb), \]
where $\orb$ is the unique dense open nilpotent orbit in the associated variety of the maximal primitive ideal $\J_{\ckorb}$ with infinitesimal character $\chi_{\lambda_{\ckorb}}$ given by $\lambda_{\ckorb} = \frac{1}{2} \check{h}$ for some $\slf_2$-triple $(\check{e}, \check{h}, \check{f})$ in $\ckfg$ with $\check{e} \in \ckorb$. We call $d$ the {\it Barbasch-Vogan (BV) duality map} and $\orb$ the {\it Barbasch-Vogan (BV) dual orbit} of $\ckorb$.
The map $d$ is order-reversing with respect to the closure order. It satisfies 
\begin{equation} \label{eq:d^3}
	d_{\fg} \circ d_{\ckfg} \circ d_{\fg} = d_{\fg}. 
\end{equation}

It turns out that the subset $\cN_o^{sp} \subset \cN_o$ is exactly the image of $d$. Moreover, the restriction of $d$ to $\check{\cN}_o^{sp}$, the set of special nilpotent orbits in $\ckfg^*$, is an order-reversing bijection
\begin{equation} \label{eq:duality_sp}
	d = d_{\fg}: \check{\cN}_o^{sp} \xrightarrow{\sim} \cN_o^{sp}. 
\end{equation}
By \eqref{eq:d^3}, $d$ is an involution, i.e., $d_{\ckfg} \circ d_{\fg}$ is the identity map on $\check{\cN}_o^{sp}$.
Lusztig and Spaltenstein \cite{Spaltenstein} also defined the same bijection between the sets of special nilpotent orbits using the Springer correspondence and theory of two-sided cells of $\Irr(W)$.

Besides duality between the usual Langlands dual Lie algebras, there is also an analogous duality in the setting of representation theory of metaplectic groups $\Mp(2n,\R)$, the nonlinear two-fold covering group of the real symplectic group $\Sp(2n,\R)$, studied in \cite{BMSZ:metaplecticBV} and also in \cite{JLS:Duality}. In this case, the complex Lie algebra on both sides is $\fsp(2n,\C)$, and we write $\ckfg = \fg = \fmp(2n,\C)$ to distinguish it from the case when $\ckfg = \fso(2n+1,\C)$ and $\fg = \fsp(2n,\C)$. 
We identify the (abstract) Cartan subalgebra $\ckfh$ of $\ckfg$ with the dual $\fh^*$ of the (abstract) Cartan subalgebra $\fh$ of $\fg$ via half the trace form on the Lie algebra $\fg = \fsp(2n,\C)$. Then we can follow the same recipe as above \eqref{eq:UnipOv_defn} to define $\lambda_{\ckorb}$, $\cJ_{\ckorb}$ and the {\it metaplectic Barbasch-Vogan (BV) duality map} (\cite{BMSZ:metaplecticBV}), also denoted by $d$ when the type is clear from the context, unless otherwise specified. 

There is also a related notion of {\it metaplectic special orbit}, which appeared earlier in \cite{Moeglin:p-adic} (where it is called anti-special), \cite{JiangLiuSavin}, \cite{BMSZ:metaplecticBV}, and \cite{JLS:Duality} (where it is called alternative special). Again metaplectic special orbits are exactly the image of the metaplectic BV duality map (see \cite{BMSZ:metaplecticBV} and \cite{JLS:Duality}). We will review the classfication of metaplectic special orbits in \cref{subsec:special_duality_classical}.

\begin{Rem}
    We remark that there is a generalization of the usual and metaplectic BV duality maps to the setting of nonlinear covering groups, see \cite{GLLS:covering_BV}. 
\end{Rem}

The closure diagrams of (special) nilpotent orbits and the duality map for exceptional Lie algebras can be found in \cite{Spaltenstein} and \cite{JLS:Duality}. 
For $\g=\mathfrak{sl}_n$, all nilpotent orbits are special and the duality map $d$ is given by transposing the partitions corresponding to the orbits.

\subsection{Special orbits and duality maps for classical Lie algebras} \label{subsec:special_duality_classical}

We now describe the (metaplectic) special nilpotent orbits and the duality map for both the classical Lie algebras and the metaplectic type. Let $X \in \{B, C, D, \mC\}$. We attach to \(X\) two parity parameters
\[
(\epsilon,\epsilon') = (\epsilon_X,\epsilon'_X) =
\begin{cases}
(0,1), & X=B;\\
(1,0), & X=C;\\
(0,0), & X=D;\\
(1,1), & X=\mC.
\end{cases}
\]
where \(\mC\) denotes the {\it metaplectic} (\cite{BMSZ:metaplecticBV}) (or {\it alternative} (\cite{JLS:Duality})) type \(C\) theory. This defines a bijection of sets
\begin{equation} \label{eq:pairs_types}
	\{B,C,D,\mC\} \xrightarrow{\sim} \Z_2 \times \Z_2, \quad X \mapsto (\epsilon_X,\epsilon'_X). 
\end{equation}
For a given type $X$, we will consider the partitions in $\cP_X(N) = \cP_{\epsilon_X}(N)$ and the corresponding nilpotent orbits $\orb := \orb_{\lambda}$ in $\g_{\epsilon_X}(N)$ for $\lambda \in \cP_X(N)$.
Note that \(X=\mC\) has the same underlying symplectic nilpotent orbits as type \(C\), but with the metaplectic canonical quotient group of the component group described below. When $(\epsilon,\epsilon')$ is given, we require $N$ to be compatible with $(\epsilon,\epsilon')=(\epsilon_X,\epsilon'_X)$ in the sense that $N \equiv \epsilon'(1-\epsilon) \bmod{2}$.

We make one additional convention:
when $\epsilon = 1$ (i.e., we are in type $C$ or $\mC$), we need $h(0) \equiv \epsilon' \bmod 2$. We take it to be the smallest integer that is congruent to $\epsilon'$ modulo $2$ and larger than $\#\lambda$. 

Later, we will often call metaplectic special orbits of type $\mC$ simply special orbits, when the type $X$ (or the corresponding $(\epsilon_X,\epsilon'_X)$) is clear from the context, so that we can state results in a unified way.

Let $X \in \{B, C, D, \mC\}$ and $(\epsilon, \epsilon') = (\epsilon_X,\epsilon'_X)$ as defined above. 
Then a partition $\lambda \in \cP_X(N) = \Pe(N)$ is said to be {\it special} if and only if 
	\[ h(a) \equiv \epsilon' \bmod 2 \quad \text{whenever } a \text{ is a part of } \lambda \text{ with } a \equiv \epsilon \bmod 2. \]
Let $\cP_X^{sp}(N) = \Pee(N)$ denote the set of all special partitions $\lambda \in \Pe(N)$. Let $\g = \g_\epsilon(V)$ be of type $X$. Then we have the following characterization of special nilpotent orbits in $\g$ (see \cite[Section 2.2]{JLS:Duality}).

\begin{Lem} \label{lem:special_classical}
	A nilpotent orbit $\orb = \orb_\lambda$ in $\g = \g_\epsilon(N)$ of type $B$, $C$, $D$, or $\mC$ is special if and only if $\lambda$ is special, equivalently $\lambda \in \Pee(N)$.
\end{Lem}

For classical Lie algebras, the duality map $d$ can be computed explicitly in terms of partitions. General formulas for type $B$, $C$, and $D$ can be found in many sources in the literature, such as \cite[Theorem 5.1]{McGovern1994} and \cite[Theorem 12]{Sommers:Duality} (with $\nu = \varnothing$). The combinatorial formula for metaplectic BV duality of type $\mC$ can be found in \cite{BMSZ:metaplecticBV}. For any of the four types $X$, the duality interchanges $X$ with its {\it dual type} $d(X)=\check{X}$, corresponding to the pair 
	\[ (\epsilon_{\check{X}}, \epsilon'_{\check{X}}) = (\epsilon'_X, \epsilon_X) \]
under the bijection \eqref{eq:pairs_types}.
We will only need the explicit formulas for the restriction of $d$ to special orbits, which can be found in \cite[Sections 2.2 and 2.3]{JLS:Duality}. In this case, the duality map $d$ can be regarded as an order-reversing bijection between special partitions
	\[ d: \cP_{\epsilon',\,\epsilon}^{sp}(\check{N}) \xrightarrow{\sim} \Pee(N), \quad \eta \mapsto \lambda = d(\eta), \]
with $\check{N} := N + \epsilon - \epsilon'$. For each type of $\ckfg$, the formula for $d$ is given by:
\begin{equation}\label{eq:dBV_classical}
\begin{aligned}
	\check{X}=B:   &\qquad& d(\eta) &= \eta^-\!_C\!^*;  \\
	\check{X}=C:   &\qquad& d(\eta) &= \eta^+\!_B\!^*; \\
	\check{X}=D:   &\qquad& d(\eta) &= \eta^*\!_D;  \\
	\check{X}=\mC: &\qquad& d(\eta) &= \eta_D\!^*.  \\
\end{aligned}
\end{equation}

\subsection{Lusztig's canonical quotients} 
\label{subsec:Abar}

For any special nilpotent orbit $\orb$ in $\fg^*$, Lusztig defined in \cite{Lusztig:book} the \emph{canonical quotient group $\Abar(\orb)$} of $\orb$, which is a quotient group of the component group of $\orb$ with respect to the adjoint group of $\fg$. The list of Lusztig's canonical quotient groups for exceptional groups can be found in \cite{Lusztig:book}. In this subsection, we describe the canonical quotients for classical groups in terms of partitions.

We now recall the description of Lusztig's canonical quotient group $\Abar(\orb)$ in both the ordinary classical types and the metaplectic type, following
\cite[Section 5]{Sommers:Duality} (for the ordinary classical types) and \cite[Section 6.2]{JLS:Duality} (for all types). Given a type $X \in \{B, C, D, \mC\}$ and $(\epsilon, \epsilon') = (\epsilon_X,\epsilon'_X)$ as defined in \cref{subsec:special_duality_classical}, let $\orb = \orb_\lambda$ be a special nilpotent orbit in $\g = \g_\epsilon(N)$ of type $X$. Let $\nu_1>\cdots>\nu_l$ be the parts of $\lambda$ that are not congruent to $\epsilon$ modulo 2, as defined in \cref{subsec:component_groups}.
Consider the subsequence 
  \[\tau_1>\cdots>\tau_q\] 
of $\nu_1>\cdots>\nu_l$ consisting of those $\nu_i$ satisfying $h(\nu_i)\equiv\epsilon'\bmod 2$.
Recall that in types $C$ and $\mC$, our convention is that $\nu_l=0$. Thus $0$ is included among the $\tau_j$ precisely because we choose $h(0)$ to have parity $\epsilon'$: in type $C$, $h(0)$ is even; in
type $C'$, $h(0)$ is odd. Hence in both cases $\tau_q=0$. We also
make the convention that $\tau_0=\infty$.
Set 
	\[ N_\lambda := \ker(A(\orb_\lambda) \twoheadrightarrow \Abar(\orb_\lambda)). \] 
Then $N_\lambda$ is generated by the elements of the form $x_{\nu_{i}} x_{\tau_{j}}$ such that $\tau_{j-1} > \nu_{i} \geqslant \tau_{j}$ for $1 \leqslant j \leqslant q$.


\subsection{Quasi-distinguished nilpotent orbits}\label{subsec:quasi-dist}

We assume $G$ is a connected complex reductive algebraic group with Lie algebra $\g = \Lie(G)$. Let $L \subset G$ be a Levi subgroup with Lie algebra $\fl$ and let $\orb_L$ be a nilpotent $L$-orbit in $\fl$. The saturation of $\orb_L$ from $L$ to $G$ is the nilpotent $G$-orbit in $\fg$ given by
\[
   \Sat_L^G\orb_L := G \cdot \orb_L \subset \fg,
\]
where \(\orb_L\) is viewed as a subvariety in $\fg$ via the inclusion $\fl \subset \fg$. Note that the saturation is insensitive to the choice of the group $G$ and only depends on the Lie algebras.

We will need the following fundamental definition, which will later be generalized in \cref{defn:quasi-dist-saturation}. See \cite[Section 8.2]{CM} for details.

\begin{defi} \label{defn:dist_orbit}
	Let $\orb \subset \fg$ be a nilpotent $G$-orbit. We say that $\orb$ is distinguished if the only Levi subalgebra of $\g$ that meets $\orb$ is $\g$ itself, or equivalently, $\orb$ cannot be saturated from a nilpotent $L$-orbit of any proper Levi subgroup $L \subsetneq G$.
\end{defi}

The principal/regular orbit is always distinguished. We recall from \cite[Theorem 8.2.14]{CM} the following characterization of distinguished orbits (\Cref{defn:dist_orbit}) in classical Lie algebras.

\begin{Thm}\label{thm:dist_classical}
    We have the following classification of distinguished orbits in classical Lie algebras:
	\begin{enumerate}[itemindent=*, leftmargin=*]
		\item If $\g$ is of type $A$, then the only distinguished orbit is the principal/regular orbit.
		
		\item If $\g$ is of type $B$, $C$, or $D$, then an orbit $\orb = \orb_\lambda$ corresponding to an $\epsilon$-partition $\lambda$ is distinguished if and only if $\lambda$ has no repeated parts, i.e., $m(a) = 1$ for all $a \in \lambda$. This means that the partition of a distinguished orbit in types $B$ and $D$ has only odd parts, each occurring once, while the partition of a distinguished orbit in type $C$ has only even parts, each again occurring once. Note that distinguished orbits in type $D$ are never very even.
	\end{enumerate}
\end{Thm}

Choose $e \in \orb_L$. The inclusion $L_e \subset G_e$ of centralizers gives a homomorphism
\[
   \iota: A_L(\orb_L)\longrightarrow A_G(\Sat_L^G\orb_L),
\]
of component groups, defined up to inner automorphisms of $A_L(\orb_L)$. By \cite[Proposition~2.19]{MBMY}, $\iota$ descends to a well-defined homomorphism between canonical quotients
\[
   \bar\iota:\Abar(\orb_L)\longrightarrow \Abar(\Sat_L^G\orb_L),
\]
which is always injective by a case-by-case check, as noticed below \cite[(2.1.2)]{MBMY}. Again the homomorphism $\bar\iota$ is insensitive to the choice of the group $G$.

The following definition is inspired by \cite{Adams:quasi-dist}.

\begin{defi} \label{defn:quasi-dist-saturation}
	Let $\ckorb \subset \ckfg^*$ be a special nilpotent orbit. We say that $\ckorb$ is \emph{quasi-distinguished} if there is no proper Levi subgroup $\check L \subsetneq \check G$ and no special nilpotent orbit $\ckorb_{\check L} \subset \check\fl^*$ such that $\Sat_{\check L}^{\check G}\ckorb_{\check L} = \ckorb$
	and the induced homomorphism on canonical quotients $\bar\iota: \Abar_{\check L}(\ckorb_{\check L}) \longrightarrow \Abar_{\check G}(\ckorb)$ is an isomorphism.
\end{defi}

It is clear that distinguished orbits are always quasi-distinguished.
In the case of classical/metaplectic Lie algebras, the intrinsic notion of quasi-distinguished orbits is equivalent to the following notion from \cite[Section 10.1]{BMSZ:counting} defined in terms of partitions. This can be shown easily using \cite[Proposition~5.25]{MBMY} (which was only stated for classical Lie algebras, but the metaplectic version is similar).

\begin{defi} \label{defi:quasi-dist_classical}
	Let $\check{X} \in \{B, C, D, \mC\}$ and let $X$ be its dual type. A nilpotent orbit $\ckorb$ in $\ckfg$ of type $\check{X}$ is said to be {\it quasi-distinguished} if its associated partition 
	\[\eta = [\eta_1 \geqslant \eta_2 \geqslant \ldots \geqslant \eta_p] \in \cP_{\check{X}}(\check{N}) = \cP_{\epsilon_{\check{X}}}(\check{N})\]
	satisfies the following properties:
	\begin{enumerate}
		\item $\eta$ (or $\ckorb$) is of {\it good parity}: $\eta_i \not\equiv \epsilon_{\check{X}} \bmod 2$ for all; $i$\footnote{This was defined in \cite[p. 356]{Moeglin:multi1} and \cite[Definition 2.6]{BMSZ:counting}.}
		\item $\eta$ contains no consecutive equal parts $\eta_i = \eta_{i+1} > 0$ with $i \equiv \epsilon_{X} \bmod 2$.
	\end{enumerate}
\end{defi}

\begin{Rem} \label{rem:quasi-dist_classical}
Note that by \cref{defi:quasi-dist_classical}, if $\ckorb_\eta$ is quasi-distinguished, then the multiplicity of each part of $\eta$ is at most $2$, and if $\eta_i = \eta_{i+1}$, then $i \not\equiv \epsilon_X \bmod 2$. 

Also note that orbits of good parity, and hence quasi-distinguished orbits in the sense of \cref{defi:quasi-dist_classical}, are always (metaplectic) special by \cref{lem:special_classical}.
\end{Rem}

We now describe the construction of $\eta^\dagger$. Let $\eta = [\eta_1 \geqslant \eta_2 \geqslant \ldots \geqslant \eta_p] \in \cP_{\check{X}}(\check{N})$ be a quasi-distinguished partition. 
Clearly, $p$ is odd (resp. even) when $\ckfg$ is of type $B$ (resp. $D$). We extend this to the general assumption that $p$ is the least integer satisfying $p \geqslant \#\eta$ and $p \equiv \epsilon + \epsilon' \bmod 2$, by adding a $0$ to $\eta$ if necessary. Then the cases where $\ckfg$ is of type $B$ or $D$ automatically satisfy this condition. When $\ckfg$ is of type $C$ (resp. $\mC$), this means that we assume $p$ is odd (resp. even). 

Define the partition $\eta^\dagger$ by $\eta^\dagger_i = \eta_i - (-1)^{i+\epsilon}$ for $1 \leqslant i \leqslant p$. More concretely, we have for each type of $\ckfg$:
\begin{equation}\label{eq:eta_dagger}
\begin{aligned}
	\check{X} = B: &\qquad& \eta^\dagger = [\eta_1 - 1, \eta_2 + 1, \ldots, \eta_{2k} + 1, \eta_{2k+1} - 1], &\quad p = 2k + 1; \\
	\check{X} = C: &\qquad& \eta^\dagger = [\eta_1 + 1, \eta_2 - 1, \ldots, \eta_{2k} - 1, \eta_{2k+1} + 1], &\quad p = 2k + 1; \\
	\check{X} = D: &\qquad& \eta^\dagger = [\eta_1 + 1, \eta_2 - 1, \ldots, \eta_{2k-1} + 1, \eta_{2k} - 1], &\quad p = 2k; \\
	\check{X} = \mC: &\qquad& \eta^\dagger = [\eta_1 - 1, \eta_2 + 1, \ldots, \eta_{2k-1} - 1, \eta_{2k} + 1], &\quad p = 2k. \\
\end{aligned}
\end{equation}

The following lemma is an easy consequence of \eqref{eq:dBV_classical}, whose details are left to the reader. Note that the quasi-distinguished condition guarantees that the parts of $\eta^\dagger$ in each formula above indeed form a non-increasing sequence.

\begin{Lem}\label{lem:d_quasi-dis}
	Suppose $\ckorb$ is a quasi-distinguished nilpotent orbit in $\ckfg$ of type $\check{X} \in \{B, C, D, \mC\}$, whose associated partition is $\eta$. Then $\orb = d(\ckorb)$ corresponds to the partition $(\eta^\dagger)^*$.
\end{Lem}

\begin{Rem} \label{rem:dist_very_even}
	In type $D_n$ ($n \geqslant1$), a quasi-distinguished orbit $\ckorb$ and its dual orbit $\orb$ are never very even by \cref{defi:quasi-dist_classical}. Indeed, if $\ckorb$ corresponds to a quasi-distinguished partition $\eta$, then $\orb$ corresponds to $(\eta^\dagger)^*$, which must contain part $1$. 
\end{Rem}

\subsection{Special pieces}

Let $\g$ be a reductive Lie algebra over $\C$. 
For a special nilpotent orbit $\orb \subset \fg^*$, let $\partial_{sp} \orb$ denote the union of the closures $\overline{\orb}'_{sp}$ of all special orbits $\orb'_{sp}$ strictly dominated by $\orb$ under the closure relation. 

To each special orbit $\orb$ in $\fg^*$, we associate the \emph{special piece} $\cP(\orb) := \overline{\orb} \setminus \partial_{sp} \orb$. Then $\cP(\orb)$ is a locally closed subvariety of $\cN^*$. By \cite[Chapitre III. 1 \& 2]{Spaltenstein}, special pieces form a partition of $\cN^*$. This also holds for metaplectic special orbits and their special pieces, see \cite[Section 2.4]{JLS:Duality} as well as \cref{subsec:special_piece_classical}. In all cases, it turns out that the orbits in $\cP(\orb)$ consist exactly of the fiber $d_{\ckfg}^{-1} (d_{\fg}(\orb))$ of the duality map $d_{\ckfg}: \cN_o \to \check{\cN}_o$.

\subsubsection{Special pieces in classical/metaplectic Lie algebras} \label{subsec:special_piece_classical}
We now describe the structure of special pieces in classical/metaplectic Lie algebras. Let $X \in \{B, C, D, \mC\}$ and $(\epsilon, \epsilon') = (\epsilon_X,\epsilon'_X)$ as in \cref{subsec:special_duality_classical}.
Let $\g = \g_\epsilon(N)$ and $\orb = \orb_\lambda$ be a (metaplectic) special orbit corresponding to $\lambda \in \Pee(N)$. As in \cref{subsec:component_groups}, we set $\nu_1>\nu_2>\cdots>\nu_l$ to be the distinct parts of $\lambda$ that are not congruent to $\epsilon$ modulo $2$.
These are the parts of $\lambda$ that contribute to the
component group. Recall that when $\epsilon = 1$, i.e., $X = C$ or $\mC$, we allow the formal part \(0\), with \(x_0 = 1\).
Let 
\begin{equation}\label{eq:J}
	J = J_{\epsilon,\epsilon'}(\lambda) = 
\{ i \in \{ 1, \ldots, l - 1 \} \mid \nu_{i+1} = \nu_i - 2,\, h(\nu_i) \equiv \epsilon' \bmod 2 \}.
\end{equation}
Let $k = \lvert J\rvert$ and write $J = \{ i_1 < i_2 < \cdots < i_k\}$. For $i \in J$, put $s_i = \nu_i - 1$.
Then $s_i \equiv \epsilon \bmod 2$. Since $\lambda \in \Pee(N)$, the multiplicity of $s_i$ in $\lambda$ is even. Write $m_\lambda(s_i) = 2b_i$ for some non-negative integer $b_i$.
Then $\lambda$ can be written as 
\[[\mu_0, s_1+1, s_1^{2b_1}, s_1-1, \mu_1, s_2+1, \ldots, s_k-1, \mu_k],\]
where $\mu_i$ is a partition with all parts lying in $\{s_i - 1, \ldots, s_{i+1} + 1\}$ (we make the convention that $s_0 = \infty$ and $s_{k+1} = 0$).
It follows from \cite[Section 2.4]{JLS:Duality} that the orbits in $\cP(\orb)$ correspond to those obtained from $\lambda$ by replacing the subpartition $[s_j+1,s_j^{2b_j},s_j-1]$ by $[s_j^{2b_j+2}]$ for all $j$ in some subset of $\{ 1,\ldots, k\}$.


\subsubsection{Lusztig's conjecture on special pieces} \label{subsec:Lusztig}

Lusztig conjectured in \cite{Lusztig:green} that any special piece is rationally smooth and checked it for the minimal (non-zero) special nilpotent orbit. This was verified by Shoji in \cite{Shoji:Green} for $F_4$, and by Beynon and Spaltenstein in \cite{Beynon-Spaltenstein} for $E_6$, $E_7$, and $E_8$. For classical Lie algebras, Kraft and Procesi \cite{Kraft-Procesi:special} proved the stronger statement that $\cP(\orb)$ of any special orbit $\orb$ is the quotient of a smooth variety by a finite elementary abelian 2-group $H(\orb)$. Later Lusztig conjectured in \cite{Lusztig:unipotent} that the same statement also holds for exceptional Lie algebras and interpreted the group $H(\orb)$ as a normal subgroup of his canonical quotient group $\Abar(\orb)$ of the component group $A(\orb)$. We refer to \cite{Sommers:Duality} and \cite{JLSY:SpecialPiece} for detailed information about $\Abar(\orb)$ and $H(\orb)$, respectively. For a detailed historical overview, we refer the reader to \cite{FJLS23} and \cite{JLSY:SpecialPiece}.

In \cite{JLSY:SpecialPiece}, $H(\orb)$ is defined to be (the image of) the fundamental group of the smooth locus of a Slodowy slice in $\overline{\orb}$ to the unique minimal orbit in $\cP(\orb)$, so that $H(\orb)$ is identified as a subgroup of $\pi_1(\orb)$ up to conjugation. It is shown in \cite{JLSY:SpecialPiece} that the (isomorphic) image of $H(\orb)$ in $\Abar(\orb)$ is a well-defined subgroup of $\Abar(\orb)$ and coincides with Lusztig's definition of $H(\orb)$. For any metaplectic special orbit $\orb$ in $\fmp(2n,\C)$, there is also a version of $\Abar(\orb)$ and $H(\orb)$ in \cite{JLSY:SpecialPiece}. Note that in this case, $\Abar(\orb)$ is usually a quotient of $\pi_1(\orb)$ but not a quotient of $A(\orb)$.

The following refined version of Lusztig's special pieces conjecture has been proved in \cite{JLSY:SpecialPiece}. To set up notation, fix a (metaplectic) special nilpotent orbit $\orb$ and consider any subgroup $F \subset \Abar(\orb)$. Since $\Abar(\orb)$ is a quotient of the fundamental group $\pi_1(\orb)$ of $\orb$, the preimage of $F$ in $\pi_1(\orb)$ determines a connected $G_{sc}$-equivariant cover $\widetilde{\orb}_F$ of $\orb$ ($G_{sc}$ being the simply connected group with Lie algebra $\fg$)\footnote{In fact, all the relevant covers are actually equivariant with respect to the adjoint group $G_{ad}$, except for the case of metaplectic type.}, which depends only on the conjugacy class of $F$ in $\Abar(\orb)$. Let $X_F := \spec(\C[\widetilde{\orb}_F])$ be the affinization of $\tilde{\orb}_F$ and let $\pi_F: X_F \to \overline{\orb}$ be the surjective morphism induced by the covering morphism $\widetilde{\orb}_F \to \orb$. Set $\tilde{\cP}(\orb)_F := \pi_F^{-1}(\cP(\orb))$. Then it is an open subset of $X_F$.

\begin{Thm}[\cite{JLSY:SpecialPiece}] \label{thm:special_piece}
	Let $\orb$ be a (metaplectic) special nilpotent orbit in $\fg^*$. For any subgroup $F$ of $\Abar(\orb)$, $\tilde{\cP}(\orb)_F$ is smooth if and only if $F \cap H(\orb) = 1$. Moreover, if $\Abar(\orb)$ is a semidirect product of $F$ and $H(\orb)$, we have a Galois action of $H(\orb) \simeq \Abar(\orb)/F$ on $\widetilde{\orb}_F$ and hence on $X_F$, so that the restriction of $\pi_F$ to $\tilde{\cP}(\orb)_F$ is a quotient morphism by $H(\orb)$.
\end{Thm}

In exceptional types, either $H(\orb) = \Abar(\orb)$, or $H(\orb) = 1$ if the special piece is trivial, i.e., $\cP(\orb) = \orb$. Also note that, since all nilpotent orbits in type $A$ are special, the statement of Lusztig's special pieces conjecture is vacuous. Moreover, the only (quasi-)distinguished orbit in $\ckfg = \mathfrak{sl}_n$ is the regular orbit and its dual is the zero orbit in $\g$. Thus we ignore the type $A$ case. 

We now briefly recall the description of $H(\orb)$ for type $X \in \{B, C, D, \mC\}$ from \cite{JLSY:SpecialPiece}.
Retain the notation from \cref{subsec:component_groups} and \cref{subsec:special_piece_classical}. Let $\g = \g_\epsilon(N)$ and $\orb = \orb_\lambda$ be a (metaplectic) special orbit corresponding to $\lambda \in \Pee(N)$.
Let $J = J_{\epsilon,\epsilon'}(\lambda) = \{ i_1<i_2<\cdots <i_k\}$ be as in \eqref{eq:J}. Note that $\{ \nu_{i_j} \}_{i_j \in J}$ is a subset of $\{ \tau_i \}_{1 \leqslant i \leqslant q}$ by definition. Therefore, for each $1 \leqslant j \leqslant k$, we have $\nu_{i_j} = \tau_{t_j}$ for some $t_j$, with $t_1 < t_2 < \cdots < t_k$. By \cite[Corollary 1.9]{JLSY:SpecialPiece}, the subgroup $H(\orb)$ of $\Abar(\orb)$ is generated by the basis elements that are the images of 
\begin{equation}\label{eq:H(orb)_generator}
	x_{s_j+1} x_{s_j-1} = x_{\nu_{i_j}} x_{\nu_{i_j + 1}} = x_{\tau_{t_j}} x_{\tau_{t_j + 1}} \in A(\orb), \quad 1 \leqslant j \leqslant k.
\end{equation}

\begin{Lem}\label{lem:H_quasi-dist}
	Assume $\ckorb$ is a quasi-distinguished orbit in $\ckfg^*$ and let $\orb = d(\ckorb)$ be the (metaplectic) BV dual orbit in $\fg^*$. Then we have the equalities of groups
		\[ H(\orb) = \Abar(\orb) = \pi_1(\orb).\] 
	In particular, the subgroup $F$ of $\Abar(\orb)$ in \Cref{thm:special_piece} can only be the trivial subgroup and $\widetilde{\orb}$ is exactly the universal cover of $\orb$.
\end{Lem}

\begin{proof}

	We first consider the case of classical Lie algebras. Let $X \in \{B, C, D, \mC\}$ be the type of the Lie algebra $\g$ and let $(\epsilon, \epsilon') = (\epsilon_X, \epsilon'_X)$ be as defined in \Cref{subsec:special_duality_classical}. Assume $\fg = \fg_{\epsilon}(N)$, so that $\ckfg = \fg_{\epsilon'}(N + \epsilon - \epsilon')$.
	Let $G = G_\epsilon(N)^\circ$ be the corresponding connected linear classical group. Let $\eta$ be the partition associated to $\ckorb$ in $\ckfg^*$. Then by \cref{subsec:special_duality_classical}, $\lambda = d(\eta)$ is the partition associated to the orbit $\orb = d(\ckorb)$ in $\g^*$, with $\lambda \in \Pee(N)$ and $\eta \in \cP_{\epsilon',\,\epsilon}(N + \epsilon - \epsilon')$.
	
	By \Cref{lem:d_quasi-dis}, $\lambda = d(\eta) = (\eta^\dagger)^*$. As in \Cref{subsec:component_groups}, let $\nu_1 > \cdots > \nu_l$ be the parts of $\lambda$ that are not congruent to $\epsilon$ modulo 2 (including $0$ when $\epsilon=0$). It is not hard to see that $\nu_{i+1} = \nu_i - 2$ for all $1 \leqslant i \leqslant l-1$. Moreover, all parts of $\lambda$ have heights congruent to $\epsilon'$ modulo 2, since these heights are nothing but the parts of $\eta^\dagger$, and $\eta$ is of good parity by \cref{defi:quasi-dist_classical} (1). In particular, we have $h_\lambda(\nu_i) \equiv \epsilon' \bmod 2$ for all $i$.
	Combining the description of $H(\orb)$ in \eqref{eq:H(orb)_generator}, we deduce that $H(\orb) = \Abar(\orb) = A(\orb)$. 

	The equality $A(\orb) = \pi_1(\orb)$ for classical Lie algebras follows from \cite[Corollary 6.1.6]{CM}: when $G=\Sp(2n,\C)$ this is trivial; when $G=\SO(N)$, i.e., the dual type $\check{X} = C$ or $D$, the formulas in \eqref{eq:eta_dagger} tell us that $\lambda = (\eta^\dagger)^*$ must have part $1$ with even multiplicity, hence is not rather odd (defined above \cite[Corollary 6.1.6]{CM}).
	\footnote{Note that $A(\orb)$ in \cref{subsec:nil_classical} is defined with respect to the connected linear classical group $G$, not the adjoint group, but in this case component groups defined for different connected complex simple groups with the same Lie algebra $\fg$ all coincide with the fundamental group $\pi_1(\orb)$. The same holds for the exceptional groups.}

	For exceptional $\g$, let $G$ be the corresponding connected adjoint exceptional group. We can verify case by case that $H(\orb) = \Abar(\orb) = A(\orb)$ by \cite{Lusztig:unipotent} and \cite{JLSY:SpecialPiece}, and $A(\orb) = \pi_1(\orb)$ follows from tables in \cite{Alexeevski} (see also \cite[Section 8.4]{CM}).
\end{proof}

\subsubsection{Special pieces dual to quasi-distinguished orbits}
\label{subsec:special_piece_dist}

In this subsection, we study the geometry of special pieces associated to special orbits that are dual to quasi-distinguished orbits. Again all the statements for type $A$ are vacuous, so we ignore it. 

\begin{Prop}\label{prop:codim_sp_classical}
	Let $\fg$ be a complex simple Lie algebra of type $X \in \{B, C, D, \mC\}$, let $\ckfg$ be its Langlands/metaplectic dual Lie algebra of type $\check{X}$, and let $d: \check{\cN}_o \to \cN_o$ be the usual/metaplectic Barbasch-Vogan duality map. Let $\ckorb$ be a quasi-distinguished nilpotent orbit in $\ckfg^*$ and let $\orb = d(\ckorb)$ be the dual orbit in $\fg^*$. Then $\codim(\partial_{sp} \orb, \overline{\orb}) \geqslant 6$.
\end{Prop}

\begin{proof}

Let $X \in \{B, C, D, \mC\}$ be the type of the Lie algebra $\g$ and let $(\epsilon, \epsilon') = (\epsilon_X, \epsilon'_X)$ be as defined in \Cref{subsec:special_duality_classical}. Assume $\fg = \fg_{\epsilon}(N)$, so that $\ckfg = \fg_{\epsilon'}(N + \epsilon - \epsilon')$. 
We will consider all the minimal special degenerations of $\orb$ and show that their codimensions in $\overline{\orb}$ are all at least $6$. To list these minimal special degenerations, it is more convenient to work on the $\ckfg$ side. Namely, we consider all orbits $\ckorb'$ that are minimal among all special orbits that dominate $\ckorb$, then the orbits $\orb' = d(\ckorb')$ exhaust all the minimal special degenerations of $\orb$. The list of all minimal special degenerations in classical/metaplectic types of dimension 2 (resp. $\geqslant 4$) can be found in Table 1 (resp. Table 2) of \cite{JLS:Duality}. We refer to the Introduction of \cite{JLS:Duality} for the notation for the singularities mentioned below.

Recall from \cref{rem:quasi-dist_classical} that if $\ckorb_\eta$ is quasi-distinguished, then the multiplicity of each part of $\eta$ is at most $2$. This rules out all cases in \cite[Table 2]{JLS:Duality} involving the row $\mu$ since all such $\mu$ have at least 4 rows, which would imply that $\eta$ has at least one part of multiplicity greater than 2 and hence lead to contradiction. 

For the same reason, the integer $n$ in \cite[Table 1]{JLS:Duality} for the types $c$, $d$, and $e$ satisfy $n \geqslant 2$, whose corresponding singularities (of Slodowy slices) are $B_n$, $B_n$, and $[2B_n]^+$, respectively. By \cite[Figure 2]{JLS:Duality}, the singularities of their dual minimal special degenerations are all equivalent to $c_n^{sp}$, whose dimensions are $4n - 2 \geqslant 6$.

If the type $a$ case in \cite[Table 1]{JLS:Duality} occurs, then $\eta$ must have a part (which is $s+1$ in the notation of \cite[Table 1]{JLS:Duality}) congruent to $\epsilon_{\check{X}} = \epsilon'$ modulo 2, which is impossible since $\eta$ is of good parity by definition.

It remains to consider the type $b$ of \cite[Table 1]{JLS:Duality}. Again by \cref{rem:quasi-dist_classical}, if $\eta_i = \eta_{i+1}$, then $i \not\equiv \epsilon_X \equiv \epsilon'_{\check{X}} \bmod 2$. This means that $n=2$ in type $b$ of \cite[Table 1]{JLS:Duality} only occurs when $n = 2$ and the number $l = i-1$ of rows above the $i$-th row of $\eta$ as in \cite[Table 1]{JLS:Duality} is congruent to $\epsilon'_{\check{X}}$ modulo 2. The corresponding singularity is therefore $C_2$. When $n \geqslant 3$, the corresponding singularity is denoted as $C_n$ (resp. $C_n^*$) if the action of $\Abar(\ckorb')$ on the corresponding Slodowy slice $\cS_{\ckorb}$ transversal to the minimal special degeneration $\ckorb'$ in the closure of $\ckorb$ is trivial (resp. non-trivial). By \cite[Figure 2]{JLS:Duality}, the singularities of minimal special degenerations dual to $C_n$ are type $b_{n}^{sp}$, $d_{n+1}$, or $d_{n+1}/V_4$, all of dimension $4n-2 \geqslant 6$ (since $n \geqslant 2$), while the singularities dual to $C_n^*$ are $c_{n-1}^{sp}$ of dimension $4n-6 \geqslant 6$ (since $n \geqslant 3$). In both cases, the dual minimal special degenerations are of dimension $\geqslant 6$.

Therefore, in all cases, the minimal special degenerations of $\orb$ are of dimension $\geqslant 6$ and the proof is complete.
\end{proof}

For exceptional Lie algebras, the following proposition follows from a straightforward case-by-case check. The closure diagrams of nilpotent orbits of exceptional Lie algebras can be found in \cite{FJLS:generic_sing}. We list all non-distinguished quasi-distinguished orbits $\ckorb$ in \cref{tab:qD_exceptional} with $\orb = d(\ckorb)$, the fundamental groups $\pi_1(\orb)$ and Lusztig canonical quotient groups $\Abar(\orb)$ of $\orb$, the minimal special degenerations $\orb'_{sp}$ of $\orb$, the singularities of $\overline{\orb}$ along $\orb'_{sp}$, and their dimensions (we again refer the reader to \cite{JLS:Duality} for the notation for the singularities).

\begin{Prop} \label{prop:codim_sp_exceptional}
	Let $\fg$ be a complex exceptional simple Lie algebra, let $\ckfg$ be its Langlands dual Lie algebra, and let $d$ be the usual Barbasch-Vogan duality map. Let $\orb = d(\ckorb)$, where $\ckorb$ is a quasi-distinguished nilpotent orbit in $\ckfg^*$. Then exactly one of the following cases occurs:
	\begin{enumerate}
		\item \label{item:exc_codim_6}
			$\codim(\partial_{sp} \orb, \overline{\orb}) \geqslant 6$ (this is always the case when $\ckorb$ is distinguished), or
		\item \label{item:exc_codim_4}
			$\codim(\partial_{sp} \orb, \overline{\orb}) = 4$. In this case, we also have $\codim(\partial \orb, \overline{\orb}) = 4$.
	\end{enumerate}
\end{Prop}

\begin{table}[htbp]
\caption{Non-distinguished quasi-distinguished orbits in exceptional Lie algebras}\label{tab:qD_exceptional}
\centering
\renewcommand{\arraystretch}{1.22}
\setlength{\extrarowheight}{2.2pt}
\setlength{\tabcolsep}{4pt}
\begin{tabular}{|c|c|c|c|c|c|c|c|}
\hline
Type & $\ckorb$ & $\orb$ & $\pi_1(\orb)$ & $\Abar(\orb)$ & $\orb'_{sp}$ & Singularity & dim \\
\hline
\multirowcell{2}{$E_6$} & \multirowcell{2}{$D_4(a_1)$} & \multirowcell{2}{$D_4(a_1)$} & \multirowcell{2}{$\fS_3$} & \multirowcell{2}{$\fS_3$} & $A_3$ & $b_2^{sp}$ & $6$ \\ \cline{6-8}
 &  &  &  &  & $2A_2$ & $g_2^{sp}$ & $10$ \\ \hline
\multirowcell{2}{$E_7$} & \multirowcell{2}{$E_6(a_1)$} & \multirowcell{2}{$A_2+A_1$} & \multirowcell{2}{$\fS_2$} & \multirowcell{2}{$\fS_2$} & $A_2$ & $a_5^+$ & $10$ \\ \cline{6-8}
 &  &  &  &  & $(3A_1)''$ & $f_4^{sp}$ & $22$ \\ \hline
\multirowcell{2}{$E_7$} & \multirowcell{2}{$A_4+A_1$} & \multirowcell{2}{$A_4+A_1$} & \multirowcell{2}{$\fS_2$} & \multirowcell{2}{$\fS_2$} & $A_3+A_2+A_1$ & $a_2/\fS_2$ & $4$ \\ \cline{6-8}
 &  &  &  &  & $A_4$ & $a_2^+$ & $4$ \\ \hline
\multirowcell{2}{$E_8$} & \multirowcell{2}{$D_7(a_2)$} & \multirowcell{2}{$A_4+A_1$} & \multirowcell{2}{$\fS_2$} & \multirowcell{2}{$\fS_2$} & $D_4(a_1)+A_2$ & $a_2^+$ & $4$ \\ \cline{6-8}
 &  &  &  &  & $A_4$ & $a_4^+$ & $8$ \\ \hline
$E_8$ & $E_6(a_1)+A_1$ & $A_4+2A_1$ & $\fS_2$ & $\fS_2$ & $A_4+A_1$ & $a_2^+$ & $4$ \\ \hline
\end{tabular}
\end{table}

The following main theorem of this section strengthens \eqref{eq:codim_4} in \Cref{thm:birigid} (also see \Cref{prop:dual_dist}). 

\begin{Thm}\label{thm:codim_X}
	Let $\fg$ be a complex simple Lie algebra, let $\ckfg$ be its Langlands/metaplectic dual Lie algebra, and let $d$ be the usual/metaplectic Barbasch-Vogan duality map. Let $\orb = d(\ckorb)$, where $\ckorb$ is a quasi-distinguished nilpotent orbit in $\ckfg^*$. Let $\widetilde{\orb}$ be the universal cover of $\orb$ and $X = \spec(\C[\widetilde{\orb}])$. Then exactly one of the following two cases occurs:
    \begin{enumerate}[label=(\roman*)]
        \item $\codim(\partial \orb, \overline{\orb}) = 2$. In this case, $\codim(\partial_{sp} \orb, \overline{\orb}) \geqslant 6$ and $\codim(X^{sing}, X) \geqslant 6$;
        \item $\codim(\partial \orb, \overline{\orb}) \geqslant 4$.
    \end{enumerate}	
	Moreover, whenever $\codim(\partial_{sp} \orb, \overline{\orb}) \geqslant 6$, we have $\codim(X^{sing}, X) \geqslant 6$. This is always the case when $\ckorb$ is distinguished.
\end{Thm}

\begin{proof}
	The first part of the statement follows from \cref{prop:codim_sp_classical} and \cref{prop:codim_sp_exceptional}.
	Whenever $\codim(\partial_{sp} \orb, \overline{\orb}) \geqslant 6$, we have $\codim(X^{sing}, X) \geqslant 6$ by \Cref{thm:special_piece} and \Cref{lem:H_quasi-dist}. 
\end{proof}

\section{Admissible vector bundles and twisted \texorpdfstring{$\sD$}{D}-modules} \label{sec:adm_bundle}

The associated cycles of special unipotent representations studied in this paper are quite constrained. They turn out to be \emph{admissible orbit data} introduced in \cite[Section 7]{Vogan_AV}, whose definition will be recalled in \Cref{subsec:AOD}. We will prove this eventually in \Cref{thm:Unip_AOD_bij}. Standard references are \cite{Vogan_AV, Ohta:Adm, Schwartz}. However, we will treat admissible orbit data in terms of modules over certain Picard algebroids, which are more suitable for our purposes.

\subsection{Picard Lie algebroids and twisted \texorpdfstring{$\sD$}{D}-modules} \label{subsec:picard}

\subsubsection{Picard algebroids and TDOs} \label{subsec:pic_defn}

We recall the notion of Picard algebroid, its basic properties, and its relation with twisted differential operators (TDOs) and twisted $\mathscr{D}$-modules from \cite[Section 2]{BeilinsonBernstein}. The notion of a Picard algebroid is a special case of the notion of a Lie algebroid, whose definition can be found in \cite[Section 1.2]{BeilinsonBernstein}. Let $X$ be a smooth variety and $\T_X$ (resp. $\T_X^*$) be its tangent (resp. cotangent) sheaf. 

\begin{defi} \label{defn:pic}
	A {\it Picard Lie algebroid}, or simply a {\it Picard algebroid} over $X$, is a Lie algebroid $\TT$ over $X$ with $\OO_X$-linear anchor map $\eta: \TT \to \T_X$ whose kernel is $\OO_X \cdot 1_{\TT}$, generated by a central section $1_{\TT}$ of $\TT$. 
	A morphism of Picard algebroids is a morphism of Lie algebroids that preserves the distinguished sections $1_{\TT}$, which is necessarily an isomorphism.
\end{defi}

For any Picard algebroid $\TT$ over $X$, we have an exact sequence of sheaves of $\OO_X$-modules,
	\[ 0 \to \OO_X \xrightarrow{\iota_\TT} \TT \xrightarrow{\eta} \T_X \to 0, \]
where $\iota_\TT(f) = f \cdot 1_{\TT}$ for all $f \in \OO_X$, and all maps in the sequence are Lie algebra homomorphisms (we endow $\OO_X$ with the trivial Lie algebra structure).

Let $\PicAlg(X)$ denote the set of isomorphism classes of Picard algebroids over $X$. The Baer sum construction defines a natural vector-space structure over $\C$ on $\PicAlg(X)$. 
Namely, for any $\TT_i \in \PicAlg(X)$ and $r_i \in \C$, $i = 1,2$, the linear combination $\TT = r_1 \TT_1 + r_2 \TT_2$ is a Picard algebroid equipped with a morphism of Lie algebroids $s_{r_1, r_2}: \TT_1 \times_{\T_X} \TT_2 \to \TT$ such that $s_{r_1, r_2} (f_1, f_2) = r_1 f_1 + r_2 f_2$ for any $f_i \in \OO_X$, $i = 1,2$.
The zero vector in $\PicAlg(X)$ is given by the trivial Picard algebroid $\TT_{\operatorname{triv}} = \OO_X \oplus \T_X$.

To any given Picard algebroid $\TT$, we can associate a sheaf $\sD_\TT$ of \emph{twisted differential operators (TDOs)} in the sense of \cite[Definition 2.1.1]{BeilinsonBernstein}, defined as the quotient of the universal enveloping sheaf $\mathscr{U}(\TT)$ modulo the ideal generated by the central section $1 - 1_\TT$, where $1$ is the identity section of $\mathscr{U}(\TT)$ and $1_\TT$ is the central section of $\TT$. Then $\sD_\TT$ is equipped with an injective morphism of algebras $\iota_{\sD}: \OO_X \hookrightarrow \sD_{\TT}$, and an increasing exhaustive algebra filtration $\sD_{\TT, \leqslant i}$, $i \geqslant 0$. We say that $\sD_\TT$ is a {\it sheaf of TDOs with twist $\TT$}.

The morphism $\iota_{\sD}$ identifies $\OO_X$ with $\sD_{\TT, \leqslant 0}$.
The subsheaf $\sD_{\TT, \leqslant 1}$ of $\sD_\TT$ is naturally a Picard algebroid with Lie bracket given by the commutator in $\mathscr{U}(\TT)$.
The composite map $\TT \hookrightarrow \mathscr{U}(\TT) \to \sD_{\TT}$ gives an isomorphism between $\TT$ and $\sD_{\TT, \leqslant 1}$ as Picard algebroids. The constructions above show that $\TT \mapsto \sD_\TT$ gives an equivalence between the category of Picard algebroids and the category of sheaves of TDOs. Under this equivalence, the trivial Picard algebroid $\TT_{\operatorname{triv}}$ corresponds to $\sD_X$, the sheaf of usual (untwisted) differential operators on $X$.

To any line bundle $L$ over $X$, i.e., a locally free $\OO_X$-module of rank $1$, we can associate a sheaf of TDOs 
	\[ \sD_L := \Diff(L) \simeq L \otimes_{\OO_X} \sD_X \otimes_{\OO_X} L^{-1} \]
as the sheaf of local differential endomorphisms of $L$, equipped with the natural filtration $\sD_{L, \leqslant i}$ by orders of differential operators. Then $\TT(L) := \sD_{L, \leqslant 1}$ forms a Picard algebroid by the discussion above.
Let $\Pic(X)$ be the Picard group of $X$. Then the construction above gives a homomorphism of abelian groups
\begin{equation} \label{eq:Pic2PicAlg}
	\TT(\cdot): \Pic(X) \to \PicAlg(X), \quad L \mapsto \TT(L), 
\end{equation}
which transforms the tensor product of line bundles into the Baer sum of Picard algebroids. 


We now define the category of modules over Picard algebroids and TDOs. Given a quasi-coherent $\OO_X$-module $\cE$, a (left) $\TT$-module structure on $\cE$ is equivalent to a (left) $\sD_\TT$-module structure on $\cE$, such that the $\OO_X$-action on $\cE$ via $\iota_{\sD}: \OO_X \hookrightarrow \sD_{\TT}$ agrees with the $\OO_X$-module structure on $\cE$. Let $\Mod(\TT)$ and $\Mod(\sD_\TT)$ denote the abelian categories of $\TT$-modules and $\sD_\TT$-modules, respectively, so that we have a canonical equivalence $\Mod(\TT) \simeq \Mod(\sD_\TT)$ of categories. 

We are mainly interested in $\TT$-modules or $\sD_\TT$-modules $\cE$ that are coherent as $\OO_X$-modules. It is well-known (\cite[Lemma 2.3.1 (i)]{BeilinsonBernstein}) that the $\OO_X$-coherence of a $\sD_\TT$-module $\cE$ is equivalent to local freeness of $\cE$ as an $\OO_X$-module.  
Let $\Coh(\TT)$ and $\Coh(\sD_\TT)$\footnote{There is also a notion of a $\sD_\TT$-module that is coherent as a $\sD_\TT$-module (see \cite[Section 1.4]{HTT:D-modules}), which is different from $\OO_X$-coherence. We denote the category of coherent $\sD_\TT$-modules by $\Mod_c(\sD_\TT)$ (but we will never use it). Of course, $\OO_X$-coherence implies $\sD_\TT$-coherence.} denote the Serre subcategories of all $\OO_X$-coherent modules in $\Mod(\TT)$ and $\Mod(\sD_\TT)$, respectively. Then we have an equivalence $\Coh(\TT) \simeq \Coh(\sD_\TT)$ of categories.

\subsubsection{Inverse image functors} \label{subsubsec:pullback_pic}

Let $\varphi: Y \to X$ be a morphism of smooth varieties. Given a Picard algebroid $(\TT, \eta, 1_\TT)$ over $X$, one can define the pullback Picard algebroid $\varphi^{\circ} \TT$ over $Y$ as in \cite[Section 1.4.6 and 2.2]{BeilinsonBernstein}. Namely, let $\varphi^* \TT$ denote the usual inverse image of the $\OO_X$-module $\TT$ via $\varphi$ as an $\OO_Y$-module. Then $\varphi^{\circ} \TT$, as an $\OO_Y$-module, is defined to be the fiber product 
\begin{equation} \label{eq:pullbackTT}
	\varphi^{\circ} \TT := \varphi^* \TT \times_{\varphi^* \T_X} \T_Y 
\end{equation}
with respect to the morphisms 
\[ 
	\varphi^*\TT \xrightarrow{\varphi^*(\eta)} \varphi^* \T_X \xleftarrow{d \varphi} \T_Y, 
\] 
where $d \varphi: \T_Y \to \varphi^* \T_X$ is the differential of $\varphi$.
The anchor map is the projection map from $\varphi^{\circ} \TT$ to $\T_Y$ and $1_{\varphi^{\circ} \TT} = (\varphi^* 1_{\TT}, 0)$. The Lie bracket $[\cdot, \cdot]_{\varphi^{\circ} \TT}$ on $\varphi^{\circ} \TT$ is given by the formula
\[ 
	[ (f_1 \otimes l_1, \zeta_1),  (f_2 \otimes l_2, \zeta_2) ]_{\varphi^{\circ} \TT} := ( f_1 f_2 [l_1, l_2]_{\TT} + d\varphi(\zeta_1)(l_2) - d\varphi(\zeta_2)(l_1), [\zeta_1, \zeta_2]_{\T_Y} ), 
\]
for any local sections $f_i \in \OO_X$, $l_i \in \varphi^{-1}\TT$, $\zeta_i \in \T_Y$.
The construction gives a homomorphism of abelian groups
\[ 
	\varphi^\circ: \PicAlg(X) \to \PicAlg(Y), \quad \TT \mapsto \varphi^\circ \TT. 
\]
We also have the standard inverse image map $\varphi^*: \Pic(X) \to \Pic(Y)$ of line bundles given by $L \mapsto \varphi^* L$. We have the following commutative diagram:
\begin{equation} \label{diag:Pic2PicAlg_pullback}
	\begin{tikzcd}
		\Pic(X) \ar[r, "\varphi^*"] \ar[d] & \Pic(Y) \ar[d] \\
		\PicAlg(X) \ar[r, "\varphi^\circ"] & \PicAlg(Y)
	\end{tikzcd}
\end{equation}
where the vertical maps are given by \eqref{eq:Pic2PicAlg}.

Equivalently, one can define the pullback of the corresponding sheaf $\sD_\TT$ of TDOs as follows. Let 
\[ 
	\sD_\TT^{Y \to X} := \OO_Y \otimes_{\varphi^{-1} \OO_X} \varphi^{-1} \sD_\TT.
\]
This is a sheaf over $Y$ of $\OO_Y$-$\varphi^{-1} \sD_\TT$-bimodules. Then we define 
\[ 
\varphi^\circ \sD_\TT := \Diff_{\varphi^{-1} \sD_\TT} (\sD_\TT^{Y \to X}) 
\]
to be the sheaf over $Y$ of all $\OO_Y$-differential endomorphisms of $\sD_\TT^{Y \to X}$ that commute with the right $\varphi^{-1} \sD_\TT$-action. Then $\varphi^\circ \sD_\TT$ is a sheaf of TDOs over $Y$. It is not hard to see that $\varphi^\circ \sD_\TT$ corresponds to the Picard algebroid $\varphi^\circ \TT$ defined above, so that we have a canonical isomorphism $\varphi^\circ \sD_\TT \simeq \sD_{\varphi^\circ \TT}$ of sheaves of TDOs. 

For any $\TT$-module $\cE$ over $X$, the usual inverse image $\varphi^*\cE = \OO_Y \otimes_{\varphi^{-1} \OO_X} \varphi^{-1} \cE$ of $\cE$ as a $\OO_X$-module naturally admits the structure of a $\varphi^\circ \TT$-module. Regarding $\cE$ as a left $\sD_\TT$-module, we have 
	\[ \varphi^*\cE \simeq \sD_{\TT}^{Y \to X} \otimes_{\varphi^{-1} \sD_\TT} \varphi^{-1} \cE. \] 
This defines inverse image functors 
	\[ \varphi^*: \Mod(\TT) \to \Mod(\varphi^\circ \TT) \quad \text{and} \quad \varphi^*: \Mod(\sD_\TT) \to \Mod(\sD_{\varphi^\circ \TT}). \] 
When $\cE$ is $\OO_X$-coherent, $\varphi^*\cE$ is $\OO_Y$-coherent. Therefore $\varphi^*$ restricts to functors 
	\[ \varphi^*: \Coh(\TT) \to \Coh(\varphi^\circ \TT) \quad \text{and} \quad \varphi^*: \Coh(\sD_\TT) \to \Coh(\sD_{\varphi^\circ \TT}). \]

\subsubsection{Direct image functors} \label{subsubsec:pushforward_pic}

The definition of the direct image functor $\varphi_+$ of $\TT$-modules or left $\sD_\TT$-modules is more involved, see \cite[Section 1.3]{HTT:D-modules}. We need some preparation before giving the definition. For a sheaf of TDOs $\sD_\TT$ with twist $\TT$ over $X$, its opposite algebra $\sD_\TT^{op}$ is also a sheaf of TDOs with twist being the Picard algebroid $\TT^{op}$ defined in \cite[Section 2.4]{BeilinsonBernstein} by
\begin{equation}\label{eq:TTop}
	\TT^{op} = \TT(\omega_X) - \TT,
\end{equation}
where $\omega_X := \bigwedge^{\mathrm{top}}\T_X^*$ denotes the canonical line bundle of $X$ and the difference on the right-hand side is defined via the vector space structure of $\PicAlg(X)$. 

Note that, for any $\TT_i \in \PicAlg(X)$ and $\cE_i \in \Mod(\TT_i)$, $i = 1,2$, the tensor product $\cE_1 \otimes_{\OO_X} \cE_2$ is naturally a module over the Picard algebroid $\TT_1 + \TT_2$. Therefore, we have an equivalence of categories
\[ 
	\omega_X \otimes_{\OO_X} (\cdot): \Mod(\sD_\TT) \to \Mod(\sD_{\TT + \TT(\omega_X)}) 
\]
whose quasi-inverse is given by $\cE \mapsto \omega_X^{-1} \otimes_{\OO_X} \cE$. We have a similar equivalence for $Y$. Note that a left $\sD_{\TT + \TT(\omega_X)}$-module can be regarded as a right module over $\sD_{\TT + \TT(\omega_X)}^{op} \simeq \sD_{-\TT}$ by \eqref{eq:TTop}. Then we can define the direct image functor $\varphi_+: \Mod(\sD_{\varphi^\circ \TT}) \to \Mod(\sD_{\TT})$ so that it fits into the following commutative diagram
\[
	\begin{tikzcd}
		\Mod(\sD_{\varphi^\circ \TT}) \ar[r, "\varphi_+"] \ar[d, "\isovert", "\omega_Y \otimes_{\OO_Y} (\cdot)"'] & \Mod(\sD_{\TT}) \ar[d, "\omega_X \otimes_{\OO_X} (\cdot)", "\isovert"'] \\
		\Mod(\sD_{-\varphi^\circ \TT}^{op}) \ar[r] & \Mod(\sD_{-\TT}^{op})
	\end{tikzcd}
\]
where the lower horizontal functor is given by 
\[ 
	\Mod(\sD_{-\varphi^\circ \TT}^{op}) \to \Mod(\sD_{-\TT}^{op}), \quad \cE \mapsto \varphi_* (\cE \otimes_{\sD_{-\varphi^\circ \TT}} \sD_{-\TT}^{Y \to X}),
\]
where $\varphi_*$ denotes the usual sheaf-theoretic push-forward functor. Note that $\varphi^\circ(-\TT) = -\varphi^\circ \TT$, hence the $\sD_{-\varphi^\circ \TT}$-$\varphi^{-1} \sD_{-\TT}$-bimodule $\sD_{-\TT}^{Y \to X}$ is well-defined. We define the $\varphi^{-1} \sD_{\TT}$-$\sD_{\varphi^\circ \TT}$-bimodule $\sD_{\TT}^{X \leftarrow Y}$ by
\[ 
	\sD_{\TT}^{X \leftarrow Y} := \omega_Y \otimes_{\OO_Y} \sD_{-\TT}^{Y \to X} \otimes_{\varphi^{-1} \OO_X} \varphi^{-1} \omega_X^{-1}.
\]
Then the direct image functor $\varphi_+: \Mod(\sD_{\varphi^\circ \TT}) \to \Mod(\sD_{\TT})$ is defined as
\[ 
	\varphi_+ \cE = \varphi_*(\sD_{\TT}^{X \leftarrow Y} \otimes_{\sD_{\varphi^\circ \TT}} \cE). 
\]

We are mainly interested in the following situation later in the paper, especially in \Cref{sec:adm_bundle}. Let $\varphi: Y \to X$ be an \'{e}tale morphism. We write 
	\[ \TT^+_X := \frac{1}{2} \TT(\omega_X) \quad \text{and} \quad \TT^+_Y := \frac{1}{2} \TT(\omega_Y). \]
The canonical morphism $(d\varphi)^*: \varphi^*\T_X \to \T_Y$ dual to the differential morphism $d\varphi: \T_Y \to \varphi^*\T_X$ is an isomorphism. Also, we have $\dim X = \dim Y$. Therefore we have a canonical isomorphism $\varphi^*\omega_X \simeq \omega_Y$, which induces a canonical isomorphism $\varphi^\circ \TT^+_X \simeq \TT^+_Y$ of Picard algebroids. In fact, in this case we have $\varphi^\circ \TT^+_X \simeq \varphi^*\TT^+_X$ by \eqref{eq:pullbackTT}. Then, as in the case of untwisted sheaves of TDOs, one can show that the $\sD$-module push-forward functor 
\[
	\varphi_+: \Mod(\sD_{\TT^+_Y}) \simeq \Mod(\sD_{\varphi^\circ \TT^+_X}) \to \Mod(\sD_{\TT^+_X})
\] 
coincides with the sheaf-theoretic push-forward functor $\varphi_*$. If in addition $\varphi$ is a projective or finite \'{e}tale morphism, then $\varphi_+ = \varphi_*$ sends coherent $\OO_Y$-modules to coherent $\OO_X$-modules, and hence we have a functor
\[
	\varphi_*: \Coh(\sD_{\TT^+_Y}) \to \Coh(\sD_{\TT^+_X}).
\]

\subsubsection{Equivariance} \label{subsubsec:pic_equiv}

Let $K$ be an algebraic group with $\Lie(K) = \fk$. Let $X$ be a smooth variety equipped with an algebraic $K$-action $\alpha: K \times X \to X$. The differential of $\alpha$ is a Lie algebra homomorphism $d \alpha: \fk \to \Gamma(X, \T_X)$. The following definitions from \cite[Section 1.8]{BeilinsonBernstein} give the right setup for equivariance of Picard algebroids and their modules.

\begin{defi}\label{defn:pic_action}
	Let $(\TT, \eta, 1_\TT)$ be a Picard Lie algebroid over $X$. A {\it (weak) $K$-action} or a {\it weakly $K$-equivariant structure} $\alpha_\TT$ on $\TT$ is a $K$-action on $\TT$ by automorphisms of Picard algebroids such that it lifts the action $\alpha$ on $X$ and makes $\TT$ into a $K$-equivariant coherent sheaf of $\OO_X$-modules. In particular, it satisfies $k^* 1_\TT = 1_\TT$ for any $k \in K$.
\end{defi}

Let $\der(\TT)$ denote the sheaf of Lie algebras of derivations of $\TT$. A weak $K$-action $\alpha_\TT$ on $\TT$ induces a Lie algebra morphism $d\alpha_\TT: \fk \to \der(\TT)$ by differentiation.

\begin{defi}\label{defn:strong_pic_action}
	A {\it strong $K$-action} on $\TT$ is a pair $(\alpha_\TT, \nu_\fk)$, where $\alpha_\TT$ is a weak $K$-action on $\TT$ and $\nu_\fk: \fk \to \Gamma(X, \TT)$ is a morphism of Lie algebras, such that
	\begin{enumerate}[label=(\roman*)]
		\item
		$\nu_\fk$ is $K$-equivariant with respect to the adjoint $K$-action on $\fk$ and the $K$-action on $\Gamma(X, \TT)$ induced by $\alpha_\TT$;
		\item
		the differential $d\alpha_\TT$ of $\alpha_\TT$ coincides with the adjoint $\fk$-action $\ad_{\nu_\fk}$ induced by $\nu_\fk$.
	\end{enumerate}
A morphism between Picard algebroids with strong $K$-actions is required to be $K$-equivariant and intertwine the $\nu_\fk$'s, which is necessarily an isomorphism. 

Let $\PicAlg_K(X)$ denote the set of all isomorphism classes of Picard algebroids with compatible strong $K$-actions. As in the non-equivariant case, $\PicAlg_K(X)$ is a vector space under Baer sum. The zero vector is the trivial Picard algebroid $\TT_{triv} = \OO_X \oplus \T_X$ with the Lie algebra morphism $\nu_\fk: \fk \to \Gamma(X, \TT_{triv})$ given by the composition of 
\begin{equation} \label{eq:nu_triv_pic}
	\fk \to \Gamma(X, \T_X) \to \Gamma(X, \OO_X \oplus \T_X),
\end{equation} 
where $\fk \to \Gamma(X, \T_X)$ is given by the differentiation of the $K$-action on $X$ and $\Gamma(X, \T_X) \to \Gamma(X, \OO_X \oplus \T_X)$ is induced by the natural inclusion $\T_X \hookrightarrow \OO_X \oplus \T_X$ of sheaves.
\end{defi}

A strong $K$-action on $\TT$ is the same as a $K$-action on the associated sheaf of TDOs $\mathscr{D}_\TT$ by algebra automorphisms, such that the inclusion map $\iota_\TT: \OO_X \hookrightarrow \mathscr{D}_\TT$ is equivariant, together with a Lie algebra map $\nu_\fk: \fk \to \mathscr{D}_{\TT, \leqslant 1}$ that satisfies conditions similar to (i) and (ii) above. In this case, we also say that $\mathscr{D}_\TT$ carries a strong $K$-action.

Let $\Pic_K(X)$ denote the equivariant Picard group of $X$ consisting of all ($K$-equivariant) isomorphism classes of $K$-equivariant line bundles over $X$. Note that a $K$-action on a line bundle $L$ induces a strong $K$-action on the associated Picard algebroid $\TT(L)$. This enhances \eqref{eq:Pic2PicAlg} in the non-equivariant case to a homomorphism of abelian groups $\Pic_K(X) \to \PicAlg_K(X)$.

For a $K$-equivariant morphism $\varphi: Y \to X$ of smooth varieties with $K$-actions and $\TT \in \PicAlg_K(X)$, the pullback $\varphi^\circ \TT$ also carries an induced strong $K$-action, so that we also have a homomorphism $\varphi^\circ: \PicAlg_K(X) \to \PicAlg_K(Y)$ of vector spaces (see \cite[Section 2.4]{Yu:period} for more details). As in \eqref{diag:Pic2PicAlg_pullback}, we have a commutative diagram 
\begin{equation} \label{eq:Pic2PicAlg_pullback_equiv}
	\begin{tikzcd}
		\Pic_K(X) \ar[r, "\varphi^*"] \ar[d] & \Pic_K(Y) \ar[d] \\
		\PicAlg_K(X) \ar[r, "\varphi^\circ"] & \PicAlg_K(Y)
	\end{tikzcd}
\end{equation}

We now pass from equivariant Picard algebroids to equivariant modules.

\begin{defi}\label{defn:pic_module_action}
	Let $\TT$ be a Picard algebroid over $X$ equipped with a weak $K$-action $\alpha_\TT$. A {\it weakly $K$-equivariant} $\TT$-module, or {\it weak} $(\TT, K)$-module, is an $\TT$-module $\cE$ equipped with a $K$-equivariant structure (as an $\OO_X$-module) $\alpha_\cE$ that is compatible with $\alpha_\TT$, i.e., the multiplication morphism $\TT \otimes_{\OO_X} \cE \to \cE$ is $K$-equivariant, where the $K$-equivariant structure on $\TT \otimes_{\OO_X} \cE$ is induced by $\alpha_\TT$ and $\alpha_\cE$.

	Now assume $\TT$ is equipped with a strong $K$-action $(\alpha_\TT, \nu_\fk)$. 
	Given $\kappa \in (\fk^*)^{K}$, a {\it $(K, \kappa)$-equivariant $\TT$-module} is a weak $(\TT, K)$-module $\cE$ such that the differential of the $K$-action equals $\nu_\fk - \kappa \cdot \Id_\cE$. A {\it (strong) $(\TT, K)$-module} is a $(K, \kappa)$-equivariant $\TT$-module $\cE$ with $\kappa = 0$. We have equivalent notions of $(K, \kappa)$-equivariant $\sD_\TT$-modules and (strong) $(\sD_\TT, K)$-modules.
\end{defi}

Let $\Mod^\kappa(\TT,K)$ denote the abelian category of all $(K, \kappa)$-equivariant $\TT$-modules. All morphisms between objects in $\Mod^\kappa(\TT,K)$ are required to be $K$-equivariant.
Let $\Coh^\kappa(\TT,K)$ denote the Serre subcategory of $\Mod^\kappa(\TT,K)$ consisting of all $\OO_X$-coherent objects. When $\kappa = 0$, i.e., the $K$-action is strong, we will simply suppress $\kappa$ from the notations.
Note that, if $\beta: K \to \cm$ is a character of $K$ and $d\beta$ is the differential of $\beta$, then we have a natural equivalence of categories 
\begin{equation} \label{eq:equiv_mod_kappa}
	\Mod^\kappa(\TT,K) \simeq \Mod^{\kappa + d\beta}(\TT,K), \quad \cE \mapsto \cE \otimes_{\C} \C_{\beta}, 
\end{equation}
where $\C_{\beta}$ is the one-dimensional $K$-module associated to the character $\beta$. We have a similar equivalence of coherent categories
\begin{equation} \label{eq:equiv_coherent_kappa}
	\Coh^\kappa(\TT,K) \simeq \Coh^{\kappa + d\beta}(\TT,K), \quad \cE \mapsto \cE \otimes_{\C} \C_{\beta}, 
\end{equation}
The inverse image and direct image functors can both be enhanced to functors between the corresponding equivariant categories,
\[ \Mod^\kappa(\TT,K) \xrightleftharpoons[ \quad\varphi_+\quad ]{\quad\varphi^*\quad} \Mod^\kappa(\varphi^\circ \TT,K), \]
and similarly for the equivariant coherent categories when $\varphi$ is proper.

\subsection{Admissible representations}\label{subsec:adm_repn}

Let $H$ be an algebraic reductive group with identity component $H^\circ$ and (finite) component group $\pi_0 H = H/H^\circ$. Let $\delta: H \to \cm$ be a character of $H$ and let $d \delta: \fh = \Lie(H) \to \C$ be its differential. Set $\gamma = \frac{1}{2} d \delta$. 

\begin{defi}\label{defn:gamma_repn}
	A \emph{$\gamma$-representation of $H$} is a finite-dimensional representation $\pi: H \to \GL(V)$ such that $d \pi = \gamma \cdot \Id_V$. Let $\widehat{H}^\gamma$ denote the subset of $\widehat{H}$ consisting of equivalence classes of irreducible $\gamma$-representations.
\end{defi}

Let $\widetilde{H} = \widetilde{H}_\delta$ be the fiber product of the two homomorphisms $\delta: H \to \cm$ and $\cm \to \cm, \, z \mapsto z^2$, so that we have the Cartesian diagram
\begin{equation}\label{diag:H_tilde}
	\begin{tikzcd}
		\widetilde{H}_\delta \arrow{r} \arrow[d, "\tilde{\delta}"]  & H  \arrow[d, "\delta"]   \\
		\cm \arrow[r, "z \mapsto z^2"]  & \cm
	\end{tikzcd}
\end{equation}
Then $\widetilde{H}_\delta$ is a (possibly disconnected) two-fold covering of $H$. Let $\zeta$ denote the unique nontrivial element in the kernel of the covering map $\widetilde{H}_\delta \to H$. Clearly, we have $\tilde{\delta}(\zeta) = -1$ and $d \tilde{\delta} = \frac{1}{2} d \delta = \gamma$ as characters of $\Lie(\widetilde{H}_\delta) \simeq \fh$.

\begin{defi}\label{defn:adm}
	A representation $\pi: \widetilde{H}_\delta \to \GL(V)$ of $\widetilde{H}_\delta$ is said to be \emph{admissible}\footnote{The reader should not confuse the terminology here with the notion of admissible $(\g, K)$-modules.}
	if $\pi$ is trivial on the identity component $\widetilde{H}_\delta^\circ$ of $\widetilde{H}_\delta$ and $\pi(\zeta) = -\Id_V$. We denote the set of equivalence classes of irreducible admissible representations of $\widetilde{H}_\delta$ as $\Irr(\widetilde{H}_\delta)_{adm}$.
\end{defi}

Clearly, if there exists a nontrivial admissible representation of $\widetilde{H}_\delta$, then $\zeta$ does not lie in $\widetilde{H}_\delta^\circ$ and hence the homomorphism $\widetilde{H}_\delta \to H$ restricts to an isomorphism $\widetilde{H}_\delta^\circ \simeq H^\circ$. 
Any admissible representation of $\widetilde{H}_\delta$ factors through the quotient $\widetilde{H}_\delta \to \widetilde{H}_\delta/\widetilde{H}_\delta^\circ = \pi_0 \widetilde{H}_\delta$. The component group $\pi_0 \widetilde{H}_\delta$ is a central extension of $\pi_0 H$ by the cyclic subgroup $\{1, \zeta\}$ of order $2$. Therefore an admissible representation $\alpha: \widetilde{H}_\delta \to \GL(V)$ of $\widetilde{H}_\delta$ is the same as a representation $\alpha: \pi_0 \widetilde{H}_\delta \to \GL(V)$ of the finite group $\pi_0 \widetilde{H}_\delta$, such that $\alpha(\zeta) = -\Id_V$, which means that $\alpha$ is genuine with respect to the morphism $\widetilde{H}_\delta \to H$. We denote the set of equivalence classes of irreducible genuine representations of $\pi_0 \widetilde{H}_\delta$ as $\Irr(\pi_0 \widetilde{H}_\delta)^{gen}$.

We observe that tensoring any admissible representation of $\widetilde{H}_\delta$ with the character $\tilde{\delta}: \widetilde{H}_\delta \to \cm$ gives a representation of $\widetilde{H}_\delta$ that is trivial on $\zeta$ and hence factors through a $\gamma$-representation of $H$. Hence we have the following standard lemma (see \cite{Schwartz} and \cite[\S\,1, Proposition and Remark 4]{Ohta:Adm}).

\begin{Lem} \label{lem:adm_repn}
	The following conditions are equivalent:
	\begin{enumerate}
		\item 
		There exists a nontrivial admissible representation of $\widetilde{H}_\delta$.
		\item 
		The inverse image of $H^\circ$ via the covering map $\widetilde{H}_\delta \to H$ is disconnected.
		\item 
		There exists a (unique) character $\chi$ of $H^\circ$ such that $\chi(h)^2 = \delta(h)$ for all $h \in H^\circ$.
		\item 
		There exists a nontrivial $\gamma$-representation of $H$.
	\end{enumerate}
	When any of the above conditions is satisfied, the discussions above give rise to bijections 
	\begin{equation} \label{eq:bij_gamma_adm}
		\Irr(\widetilde{H}_\delta)_{adm} \simeq \Irr(\pi_0 \widetilde{H}_\delta)^{gen} \simeq \widehat{H}^\gamma.
	\end{equation}
\end{Lem}

We remark that in the case (3) of \Cref{lem:adm_repn}, the character $\chi$ is nothing but the restriction of $\tilde{\delta}$ to $\widetilde{H}_\delta^\circ \simeq H^\circ$.

We conclude this subsection by a few additional remarks. Assume that the conditions in \Cref{lem:adm_repn} are satisfied, so that we have a unique character $\chi$ of $H^\circ$ satisfying $\chi^2 = \delta|_{H^\circ}$. Note that $H^\circ$ is a normal subgroup of $H$, so $H$ acts on $H^\circ$ by conjugation, which in turn acts on the set of equivalence classes of representations of $H^\circ$ and $\fh$. Since $\delta$ is a character of $H$, it is clear that $d \delta$, $\delta$, $\gamma$, and $\chi$ are invariant under this $H$-action. Explicitly, this means that
\begin{equation}\label{eq:delta_inv}
	d\delta(\Ad_h(\xi)) = d\delta(\xi), \quad \delta(h^{-1} x h) = \delta(x), \quad \forall \, h \in H, \, \xi \in \fh, \, x \in H^\circ
\end{equation}
and
\begin{equation}\label{eq:chi_inv}
	\chi(h^{-1} x h) = \chi(x), \quad \forall \, h \in H, \, x \in H^\circ.
\end{equation}

\subsection{Admissible orbit data}\label{subsec:AOD}

This subsection starts from the pointwise definition of admissible orbit data and then gives base-point-free reformulations using universal covers. We next record the symmetry groups arising from the universal cover, classify admissible vector bundles by genuine representations, choose compatible $\cm$-actions, and finally collect the pullback identities needed for Galois covers.

\subsubsection{Definitions} \label{subsubsec:AOD_defn}

Let $\orb$ be a nilpotent $G$-orbit with $\orb_{\fk^\perp}\neq \varnothing$.
Let $\orbk$ be a $K$-orbit in $\orb_{\fk^\perp}$ and let $e$ be any point on $\orbk$. In the setting of \Cref{defn:gamma_repn}, let $H = K_e$ and $\delta = \delta_e: K_e \to \cm$ be the character corresponding to the top exterior power $\bigwedge^{top}(\fk / \fk_e)^*$ of the coadjoint representation of $K_e$ on the cotangent space $T^*_e \orbk \simeq (\fk / \fk_e)^*$ of $\orbk$ at $e$. Recall the following definition from \cite[Definition 7.13]{Vogan_AV}:

\begin{defi} \label{defn:adm_orbit_datum}
	Assume that $(\g, K)$ is a Harish-Chandra symmetric pair corresponding to a real reductive group $\GR$. A nilpotent $K$-orbit $\orbk$ in $\fk^\perp$ or $\fs$ is said to be \emph{$K$-admissible} if the conditions in \Cref{lem:adm_repn} are satisfied with respect to $H = K_e$ and $\delta$ as above. In this case, let $\gamma = \gamma_e$ and $\chi = \chi_e$ be as in \Cref{lem:adm_repn}.
	A $K$-equivariant (irreducible) \emph{admissible vector bundle} over $\orbk$ is of the form $\cV = K \times_{K_e} V$, where $V$ is a (irreducible) $\gamma$-representation of $K_e$, i.e., the Lie algebra $\fk_e$ acts on $V$ via the character $\frac{1}{2} \operatorname{tr}(\operatorname{ad}_{\fk/\fk_e})$. 
	
	An {\it ($K$-)admissible orbit datum} over the $K$-orbit $\orbk$ in $\fk^\perp$ is a pair $(\orbk, \cV)$, where $\cV$ is an irreducible $K$-equivariant admissible vector bundle over $\orbk$. If at least one $K$-equivariant admissible vector bundle over $\orbk$ exists, we say that $\orbk$ is {\it $K$-admissible} or admissible with respect to $\GR$. We will simply say that $\orbk$ is admissible if there is no ambiguity about the group $K$ or $\GR$. If $G$ is a simply connected complex semisimple group with a Cartan involution $\theta$, and $K = G^\theta$, we also call a $K$-admissible orbit $\orbk$ {\it linearly admissible}. 
\end{defi}

The independence of \Cref{defn:adm_orbit_datum} from the choice of $e \in \orbk$ is an immediate consequence of \eqref{eq:delta_inv} and \eqref{eq:chi_inv}. We can similarly define the notion of an admissible vector bundle over a $K$-equivariant cover $\tilde{\rho}: \torbk \to \orbk$ of $\orbk$. However, we would like to have more intrinsic equivalent definitions of admissibility that are independent of the choice of a point in $\torbk$. 

For this purpose, let $\tilde{\rho}: Y = \torbk \to \orbk$ be any $K$-equivariant cover of a nilpotent $K$-orbit $\orbk$ in $\fk^\perp$. Let $\varphi: Z=\horbk \to \torbk$ be a universal $K$-equivariant cover of $\torbk$, so that $\hat{\rho} := \tilde{\rho} \circ \varphi: \horbk \to \orbk$ is also a universal $K$-equivariant cover of $\orbk$. Recall from \Cref{subsubsec:pushforward_pic} that we have the Picard algebroid $\TT_Y^+ = \frac{1}{2} \TT(\omega_Y)$ over $Y$, where $\omega_Y$ is the canonical line bundle of $Y$. Since $\TT(\omega_Y)$ carries a canonical strong $K$-action, so does $\TT_Y^+$. We have a canonical $K$-equivariant isomorphism $\tilde{\rho}^* \omega_{\orbk} \simeq \omega_Y$ of line bundles, which induces a canonical isomorphism $\tilde{\rho}^* \TT_{\orbk}^+ \simeq \TT_Y^+$ of strongly $K$-equivariant Picard algebroids. 
We can consider the category $\Coh(\TT_Y^+, K)$ of all $\OO_Y$-coherent (strong) $(\TT_Y^+, K)$-modules and similarly for $\orbk$ and $Z$. As in \Cref{subsec:picard}, we can define inverse image and direct image functors between the relevant coherent categories associated to the three covering maps $\tilde{\rho}$, $\varphi$, and $\hat{\rho}$. 

We now state the following proposition giving equivalent definitions of admissibility.

\begin{Prop} \label{prop:adm_equivalent}
	The following conditions are equivalent:
	\begin{enumerate}[label=(\roman*)]
		\item $\orbk$ is $K$-admissible in the sense of \Cref{defn:adm_orbit_datum};
		\item the category $\Coh(\TT_{\orbk}^+, K)$ has non-zero objects;
		\item the category $\Coh(\TT_Y^+, K)$ has non-zero objects for any $K$-equivariant cover $\tilde{\rho}: Y = \torbk \to \orbk$;
		\item the canonical line bundle $\omega_Z$ of the $K$-equivariant universal cover $\hat{\rho}: Z=\horbk \to \orbk$ has a $K$-equivariant square root, that is, there exists a $K$-equivariant line bundle $\omega_Z^{\smallhalf}$ over $Z$ equipped with a $K$-equivariant isomorphism $\sq: (\omega_Z^{\smallhalf})^{\otimes 2} \xrightarrow{\,\sim\,} \omega_Z$ of line bundles, called a \emph{square isomorphism}.
	\end{enumerate}
	If any of the conditions is satisfied, then the abelian category of $K$-admissible vector bundles over any cover $\tilde{\rho}: Y = \torbk \to \orbk$ is equivalent to $\Coh(\TT_Y^+, K)$ in the obvious way.
\end{Prop}

\begin{proof}
	The proposition follows from the discussion in \cite[Section 2.4.2]{Yu:period}.
\end{proof}

\begin{Rem} \label{rem:omega_half}
    Note that if the pair $(\omega_Z^{\smallhalf}, \sq)$ in \Cref{prop:adm_equivalent} (iv) exists, then it is unique up to $K$-equivariant isomorphisms that are compatible with the square isomorphism $\sq$. More precisely, given two pairs $(\omega_Z^{\smallhalf}, \sq)$ and $((\omega_Z^{\smallhalf})', \sq')$ satisfying \Cref{prop:adm_equivalent} (iv), there are two $K$-equivariant isomorphisms $\phi: \omega_Z^{\smallhalf} \xrightarrow{\sim} (\omega_Z^{\smallhalf})'$ such that $\sq = \sq' \circ \,\phi^{\otimes 2}$, and these differ from each other by a sign.
	
	In particular, if the line bundle $\omega_Z^{\smallhalf}$ is fixed, the isomorphism $\sq$ in \Cref{prop:adm_equivalent} (iv) is unique up to multiplication by a nonzero scalar. From now on, we always fix a square isomorphism $\sq$. The choice of $\sq$ will not affect the constructions and definitions below.
\end{Rem}

We set up some notation before proceeding.
Suppose $\orb$ is a $G$-orbit in $\fg^*$. 
Given any $K$-orbit $\orbk$ in $\cN_\theta$, let $\AODK(\orbk)$ denote the set of $K$-equivariant isomorphism classes of admissible orbit data over $\orbk$, or equivalently, irreducible objects in $\Coh(\TT^+_{\orbk}, K)$.
Set
\begin{equation}\label{eq:AODO}
	\AODK(\orb):=\bigsqcup_{\orbk \textrm{ is a $K$-orbit in $\orb_{\fk^\perp}$}} \AODK(\orbk)
\end{equation}
We can regard $\AODK(\orb)$ naturally as a subset of $\AC(\g,K)$ defined in \Cref{subsec:AC}. Clearly $\AODK(\orb) \subset \IrrAC(\g,K)$. 
Then $\AODK(\orb)$ is precisely the set of isomorphism classes of irreducible objects in the abelian category $\cAOD(\orb, K) := \Coh(\TT^+_{\orb_{\fk^\perp}}, K)$. The inclusion into $\AC(\g,K)$ identifies $\cK\cAOD(\orb, K)$ with a subgroup of $\AC(\g,K)$, and identifies $\cK_{>0}\cAOD(\orb, K)$ with the subsemigroup generated by $\AODK(\orb)$. Thus any element of $\cK_{>0}\cAOD(\orb, K)$ is of the form
\[ \sum_{i=1}^r [\cV_i] \cdot \orbk^i, \]
where $\orbk^1, \ldots, \orbk^r$ are all the $K$-orbits in $\orb_{\fk^\perp}$ and each $\cV_i$ is either zero or a (not necessarily irreducible) $K$-equivariant admissible vector bundle over $\orbk^i$.

\subsubsection{Symmetries from the universal cover} \label{subsubsec:tilde_Gamma_K}

We now consider the symmetries of the universal cover. 
For a universal cover $\hat{\rho}: Z=\horbk \to \orbk$ of $\orbk$, we write $\Gamma_K(\orbk) := \Gamma_{\hat{\rho}}$ for its Galois group, as defined in \Cref{subsec:orbit_covers}. We also write $\Gamma_K$ for $\Gamma_K(\orbk)$ when there is no ambiguity about the $K$-orbit $\orbk$ in question. By the functoriality of the cotangent sheaf $\T^*_Z$ of $Z$, we have a natural $\Gamma_K$-action on $\T^*_Z$, and hence on $\omega_Z$, which commutes with the $K$-action. More precisely, there is a collection of $K$-equivariant isomorphisms $\tau_\psi: \omega_Z \xrightarrow{\sim} \psi^* \omega_Z$ for each $\psi \in \Gamma_K$\footnote{Note that $\psi^*\omega_Z$ carries a natural $K$-equivariant structure induced from that of $\omega_Z$, since $\psi$ is a $K$-equivariant automorphism of $Z$. The same applies to $\psi^*(\omega_Z^{\smallhalf})$.}, such that for any $\psi_1, \psi_2 \in \Gamma_K$, the composition of isomorphisms
\begin{equation} \label{eq:tau_psi}
	\omega_Z \xrightarrow{\,\tau_{\psi_2}\,} \psi_2^* \omega_Z \xrightarrow{\psi_2^*(\tau_{\psi_1})} \psi_2^* \psi_1^* \omega_Z \xrightarrow{\sim} (\psi_1 \circ \psi_2)^* \omega_Z 
\end{equation}
coincides with $\tau_{\psi_1 \circ \psi_2}$, where the last arrow arises from the natural equivalence $\psi_2^* \psi_1^* \xrightarrow{\sim} (\psi_1 \circ \psi_2)^*$ of functors (we will not repeat this fact in the following).

From now on, we always assume $\orbk$ is $K$-admissible. We want to study the symmetries of the $K$-equivariant line bundle $\omega_Z^{\smallhalf}$. It is not true in general that the action of $\Gamma_K$ on $\horbk$ can be lifted to an action on $\omega_Z^{\smallhalf}$. 
Instead, we define $\tGamma_K = \tGamma_K(\orbk)$ as the group consisting of all pairs $(\psi, \vartheta)$, where $\psi \in \Gamma_K$ and $\vartheta: \omega_Z^{\smallhalf} \xrightarrow{\sim} \psi^* (\omega_Z^{\smallhalf})$ is a $K$-equivariant isomorphism of line bundles such that the following condition holds:

\begin{enumerate}[label=($\ast$), ref=($\ast$), itemindent=*, leftmargin=*]
\item \label{cond:tGamma_K} 
The composition of isomorphisms
\[ 
	(\omega_Z^{\smallhalf})^{\otimes 2} \xrightarrow{\vartheta^{\otimes 2}} (\psi^* \omega_Z^{\smallhalf})^{\otimes 2} \xrightarrow{\,\sim\,} \psi^* \left[ (\omega_Z^{\smallhalf})^{\otimes 2} \right] \xrightarrow{\psi^*(\sq)} \psi^* \omega_Z 
\]
(where the middle arrow is the natural isomorphism expressing the compatibility of tensor products of coherent sheaves with the inverse-image functor, and we will not repeat this fact in the following) coincides with the composition
\[ (\omega_Z^{\smallhalf})^{\otimes 2} \xrightarrow{\, \sq \,} \omega_Z \xrightarrow{\,\tau_\psi\,} \psi^* \omega_Z. \]
\end{enumerate}

It is clear that condition ($\ast$) does not change if we rescale $\sq$ by any nonzero scalar, and hence is independent of the choice of $\sq$ (see \Cref{rem:omega_half}).
The multiplication in $\tGamma_K$ is given by
	\[ (\psi_1, \vartheta_1) \cdot (\psi_2, \vartheta_2) := (\psi_1 \circ \psi_2, \psi_2^* (\vartheta_1) \circ \vartheta_2), \]
where $\psi_2^* (\vartheta_1) \circ \vartheta_2$ is understood as the composition of isomorphisms
\begin{equation} \label{eq:mult_tilde_Gamma_K}
	\omega_Z^{\smallhalf} \xrightarrow{\vartheta_2} \psi_2^* \omega_Z^{\smallhalf} \xrightarrow{\psi_2^* (\vartheta_1)} \psi_2^* \psi_1^* \omega_Z^{\smallhalf} \xrightarrow{\,\sim\,} (\psi_1 \circ \psi_2)^* \omega_Z^{\smallhalf}.
\end{equation}
It is routine to check that this indeed defines a group structure on $\tGamma_K$.

There is a forgetful map $\tGamma_K \to \Gamma_K$ given by $(\psi, \vartheta) \mapsto \psi$, which is a surjective group homomorphism. 
To see the surjectivity, note that the composition of isomorphisms 
\[
    \sq': (\psi^* \omega_Z^{\smallhalf})^{\otimes 2} \xrightarrow{\,\sim\,} \psi^* \left[ (\omega_Z^{\smallhalf})^{\otimes 2} \right] \xrightarrow{\psi^*(\sq)} \psi^* \omega_Z  \xrightarrow{\tau_\psi^{-1}} \omega_Z
\]
is also a $K$-equivariant square isomorphism.
By the uniqueness of the square root pair as discussed in \Cref{rem:omega_half}, applied to $(\omega_Z^{\smallhalf}, \sq)$ and $(\psi^* \omega_Z^{\smallhalf}, \sq')$, there exist exactly two $K$-equivariant isomorphisms $\pm \vartheta$ (up to sign) such that $\sq = \sq' \circ (\pm \vartheta)^{\otimes 2}$. It is clear that $(\psi, \pm \vartheta) \in \tGamma_K$.
It follows that the forgetful map $\tGamma_K \to \Gamma_K$ is surjective and 
\[
    \ker(\tGamma_K \to \Gamma_K) = \{ (\Id_Z, \pm 1) \} \simeq \Z_2.
\]
Thus $\tGamma_K$ is a central extension of $\Gamma_K$ by $\Z_2$,
\begin{equation} \label{eq:extn_Gamma_K}
	1 \to \Z_2 \to \tGamma_K \to \Gamma_K \to 1. 
\end{equation}

By definition, $\tGamma_K$ acts on $\omega_Z^{\smallhalf}$ so that the square isomorphism $\sq: (\omega_Z^{\smallhalf})^{\otimes 2} \xrightarrow{\sim} \omega_Z$ is $K \times \tGamma_K$-equivariant, where the action of $\tGamma_K$ on $\omega_Z$ factors through the natural $\Gamma_K$-action on $\omega_Z$. Note that the $\tGamma_K$-action on $\omega_Z^{\smallhalf}$ induces a (strong) action of $\tGamma_K$ on $\TT(\omega_Z^{\smallhalf})$, which factors through a (strong) action of $\Gamma_K$, since the central kernel $\ker(\tGamma_K \to \Gamma_K) \simeq \Z_2$ acts on $\omega_Z^{\smallhalf}$ by scalars and hence trivially on the Picard algebroid. Then $\sq$ induces an isomorphism $2 \TT(\omega_Z^{\smallhalf}) \simeq \TT(\omega_Z)$, and hence an isomorphism 
\begin{equation} \label{eq:isom_T}
	\TT(\omega_Z^{\smallhalf}) \simeq \frac{1}{2}\TT(\omega_Z) = \TT^+_Z 
\end{equation} 
of strongly $K \times \Gamma_K$-equivariant Picard algebroids over $Z$. 
This makes $\omega_Z^{\smallhalf}$ naturally into a strongly $K \times \tGamma_K$-equivariant $\TT^+_Z$-module. These isomorphisms and the induced structures do not depend on the choice of the square isomorphism $\sq$ (cf. \Cref{rem:omega_half}).

We now give a more concrete description of the symmetries above, which depends on the choice of base points. 
Fix $e \in \orbk$ and $z \in \hat{\rho}^{-1}(e)$. Then $K_z = K_e^\circ$, and we have $K$-equivariant isomorphisms $K/K_e \simeq \orbk$, $k K_e \mapsto k \cdot e$, and similarly $K/K_z = K/K_e^\circ \simeq \horbk$, which fit into the following commutative diagram:
\begin{equation} \label{diag:cover_diagram}
	\begin{tikzcd}
		K/K_e^\circ \ar[r] \isoarrow{d} & K/K_e \isoarrow{d} \\
		\horbk \ar[r, "\hat{\rho}"]  & \orbk  \\
	\end{tikzcd}
\end{equation}
As in \Cref{subsec:orbit_covers}, we can identify $\Gamma_K = \Gamma_{\hat{\rho}}$ with the $K$-component group $Z_K = K_e / K_e^\circ$ of $\orbk$.

We now describe the induced actions on the canonical bundle and its square root in concrete terms. We have a $K$-equivariant isomorphism $\omega_Z \simeq K \times_{K_e^\circ} \C_\delta$, where the character $\delta: K_e \to \cm$ corresponds to the coadjoint action of $K_e$ on $\bigwedge^{top} (\fk/\fk_e)^*$ and the $K$-action on $K \times_{K_e^\circ} \C_\delta$ is induced by left multiplication. Let $K_e$ act on $K \times_{K_e^\circ} \C_\delta$ by
	\[ h \cdot [k, v] := [k h^{-1}, \delta(h) v], \quad \forall\, k \in K, \, h \in K_e, \, v \in \C_\delta. \]
It is routine to check that this action is well-defined and commutes with the left $K$-action. Its restriction to $K_e^\circ$ is given by
	\[ h \cdot [k, v] = [k h^{-1}, \delta(h) v] = [k, \delta(h^{-1}) \delta(h) v] = [k, v], \quad \forall\, k \in K, h \in K_e^\circ, \, v \in \C_\delta. \]
This means that $K_e^\circ$ acts trivially and hence the action of $K_e$ factors through an action of $Z_K = K_e / K_e^\circ$. This is compatible with the identification $\Gamma_K \simeq Z_K$ and the (abstractly defined) action of $\Gamma_K$ on $\omega_Z$ as above.

As in \Cref{subsec:adm_repn}, let $\widetilde{K}_e$ denote the double cover of $K_e$ associated to the character $\delta: K_e \to \cm$, so that we have a character $\tilde{\delta}: \widetilde{K}_e \to \cm$ whose restriction to $\widetilde{K}_e^\circ \simeq K_e^\circ$ is $\chi$.
Then we have a $K$-equivariant isomorphism $\omega_Z^{\smallhalf} \simeq K \times_{K_e^\circ} \C_\chi$. 
We let $\widetilde{K}_e$ act on $K \times_{K_e^\circ} \C_\chi$ by
	\[ h \cdot [k, v] := [k \bar{h}^{-1}, \tilde{\delta}(h) v], \quad \forall\, k \in K, h \in \widetilde{K}_e, \, v \in \C_\chi, \]
where $\bar{h}$ denotes the image of $h$ in $K_e$. Again this action is well-defined and commutes with the left $K$-action. 
As in the case of $K_e^\circ$, one can check that $\widetilde{K}_e^\circ$ acts trivially and hence the action of $\widetilde{K}_e$ factors through an action of the finite group $\widetilde{K}_e / \widetilde{K}_e^\circ \simeq \widetilde{K}_e/K_e^\circ$, which is a central extension of $Z_K = K_e / K_e^\circ$ by $\Z_2$. We write $\widetilde{Z}_K := \pi_0(\widetilde{K}_e)$. Then the discussion above gives an identification $\tGamma_K \simeq \widetilde{Z}_K$.

Note that the isomorphisms $\Gamma_K \simeq Z_K$ and $\tGamma_K \simeq \widetilde{Z}_K$ are compatible and depend on the choice of $z \in Z=\horbk$ only up to inner automorphisms. 
This implies the following easy fact.

\begin{Lem} \label{lem:bij_irr_gen}
    The induced bijection $\Irr(\widetilde{Z}_K)^{gen} \xrightarrow{\sim} \Irr(\tGamma_K)^{gen}$ is canonical and independent of the choice of the point $z \in Z = \horbk$.
\end{Lem}

\subsubsection{Classification of admissible vector bundles} \label{subsubsec:classification_adm_vb}

In addition to the original definition in \Cref{defn:adm_orbit_datum}, we give a more intrinsic classification of irreducible $K$-admissible vector bundles over a nilpotent $K$-orbit $\orbk$ (or its covers) in \Cref{prop:classification_AOD}. This classification is independent of the choice of base points in $\orbk$, which will be convenient for later use.

We introduce some notation before stating the result. Let $Z = \horbk$ as before, and let $\Coh(\sD_Z, K \times \tGamma_K)^{gen}$ denote the category of $\OO_Z$-coherent strong $(\sD_Z, K)$-modules with a compatible genuine $\tGamma_K$-action.
Let $\Rep(\tGamma_K)^{gen}$ denote the (semisimple) abelian category of all genuine $\tGamma_K$-representations.
For any $V \in \Rep(\tGamma_K)^{gen}$, let $\underline{V}_Z := \OO_Z \otimes_\C V$ and regard it as an object of $\Coh(\sD_Z, K \times \tGamma_K)^{gen}$, where the strong $K$-action is induced from the one on $\OO_Z$, and the genuine $\tGamma_K$-action is induced from that of the genuine $\tGamma_K$-representation $V$ and the $\tGamma_K$-action on $\OO_Z$ (which factors through $\Gamma_K$). By \cite[Lemma 2.7]{Yu:period}, this defines an equivalence of categories
\begin{equation} \label{eq:equiv_underline_V}
	\Rep(\tGamma_K)^{gen} \xrightarrow{\,\sim\,} \Coh(\sD_Z, K \times \tGamma_K)^{gen}, \quad V \mapsto \underline{V}_Z,
\end{equation}
whose quasi-inverse is given by 
\begin{equation} \label{eq:equiv_underline_V_inverse}
	\Coh(\sD_Z, K \times \tGamma_K)^{gen} \to \Rep(\tGamma_K)^{gen}, \quad \cV \mapsto \Gamma(Z, \cV)^K,
\end{equation}
by taking the $K$-invariant global sections. Since $Z$ is $K$-homogeneous, $\Gamma(Z, \cV)^K$ coincides with the space of flat sections of the $\sD_Z$-module $\cV$.

We have a second equivalence of categories
\begin{equation} \label{eq:equiv_twist_omega}
    \Coh(\sD_Z, K \times \tGamma_K)^{gen} \xrightarrow{\sim} \Coh(\TT^+_Z, K \times \Gamma_K), \quad \cE \mapsto \cE \otimes_{\OO_Z} \omega_Z^{\smallhalf}. 
\end{equation}
Note that, since both $\tGamma_K$-actions on $\cE$ and $\omega_Z^{\smallhalf}$ are genuine, the induced diagonal action of $\tGamma_K$ on the sheaf $\cE \otimes_{\OO_Z} \omega_Z^{\smallhalf}$ factors through a $\Gamma_K$-action. Therefore this functor is well-defined. 
This functor is an equivalence, with quasi-inverse
\begin{equation} \label{eq:equiv_twist_omega_inverse}
	\Coh(\TT^+_Z, K \times \Gamma_K) \xrightarrow{\sim} \Coh(\sD_Z, K \times \tGamma_K)^{gen}, \quad \cV \mapsto \cV \otimes_{\OO_Z} \omega_Z^{-\smallhalf},
\end{equation}
where $\omega_Z^{-\smallhalf} := (\omega_Z^{\smallhalf})^{-1}$ is the dual (or inverse) line bundle of $\omega_Z^{\smallhalf}$.


There is a third equivalence of categories
\begin{equation} \label{eq:equiv_pushforward}
	\Coh(\TT^+_Z, K \times \Gamma_K) \xrightarrow{\,\sim\,} \Coh(\TT^+_{\orbk}, K), \quad \cV \mapsto (\hat{\rho}_* \cV)^{\Gamma_K},
\end{equation}
where $\hat{\rho}_* \cV$ is the direct image of $\cV$ under $\hat{\rho}$ and $(\hat{\rho}_* \cV)^{\Gamma_K}$ is the subsheaf of $\Gamma_K$-invariant local sections of $\hat{\rho}_* \cV$.
Its quasi-inverse is given by $\cF \mapsto \hat{\rho}^* \cF$. See \cite[Section 1.8.9]{BeilinsonBernstein} for details.

We denote the composition of the three equivalences \eqref{eq:equiv_underline_V}, \eqref{eq:equiv_twist_omega}, and \eqref{eq:equiv_pushforward} above by 
\begin{equation} \label{eq:classification_adm_vb_cat}
	\Rep(\tGamma_K)^{gen} \xrightarrow{\,\sim\,} \Coh(\TT^+_{\orbk}, K), \quad V \mapsto \cF_V := \left[\hat{\rho}_* (\underline{V}_Z \otimes_{\OO_Z} \omega_Z^{\smallhalf})\right]^{\Gamma_K}.
\end{equation}
Passing to isomorphism classes of irreducible objects, we obtain the following alternative, equivalent classification of $K$-admissible orbit data.

\begin{Prop} \label{prop:classification_AOD}
	We have a bijection
	\begin{equation} \label{eq:classification_adm_vb_irr}
		\Irr(\tGamma_K)^{gen} \xrightarrow{\,\sim\,} \AODK(\orbk), \quad V \mapsto \cF_V,
	\end{equation}
	where $\cF_V = \left[\hat{\rho}_* (\underline{V}_Z \otimes_{\OO_Z} \omega_Z^{\smallhalf})\right]^{\Gamma_K}$.
\end{Prop}

\begin{Rem} \label{rem:classification_AOD_cover}
	We have a complete analogue of \Cref{prop:classification_AOD} when $\orbk$ is replaced by any $K$-equivariant cover $\tilde{\rho}: \torbk \to \orbk$. The proof is essentially the same, except that we need to replace $\hat{\rho}$ by the universal cover $\varphi: Z=\horbk \to \torbk$ of $\torbk$ and the Galois group $\Gamma_K = \Gamma_{\hat{\rho}}$ by $\Gamma_{\varphi}$ in the discussion above. 
\end{Rem}

\subsubsection{Strong \texorpdfstring{$\cm$}{C\textasciicircum{}*}-actions} \label{subsubsec:strong_cm_action}

As in \cite[Proposition 2.4.1]{LY}, we can define a $\cm$-action on any irreducible $K$-admissible vector bundle over $Y = \torbk$ that lifts the $\cm$-action on $\torbk$ defined in \Cref{subsec:orbit_cm}. Such a $\cm$-action is unique up to multiplication by a character of $\cm$. The goal of this subsection is to make uniform choices of the $\cm$-actions on admissible vector bundles over various covers of $\orbk$ that are compatible with various functors and constructions.

First, note that the $\cm$-action on any cover $Y = \torbk$ induces a canonical $\cm$-equivariant structure on the canonical line bundle $\omega_Y$ of $Y$, which in turn induces a strong $\cm$-action on the associated Picard algebroid $\TT(\omega_Y)$ and hence on $\TT^+_Y = \frac{1}{2} \TT(\omega_Y)$ as well. Therefore $\TT^+_Y$ is naturally a strong $K \times \cm$-equivariant Picard algebroid over $Y$ and it makes sense to ask for $(\cm, \kappa)$-equivariant structures on $K$-admissible vector bundles over $Y$ in the sense of \Cref{defn:pic_module_action}.

Now choose a $K \times \tGamma_K$-equivariant square root line bundle $\omega_Z^{\smallhalf}$ of the canonical line bundle $\omega_Z$ of $Z = \horbk$, together with a $K \times \tGamma_K$-equivariant square isomorphism $\sq: (\omega_Z^{\smallhalf})^{\otimes 2} \xrightarrow{\sim} \omega_Z$, in the sense of \Cref{prop:adm_equivalent} (iv). 
A natural first attempt is to define a $\cm$-action on $\omega_Z^{\smallhalf}$ that commutes with the $K$-action, making it into a strong $(\TT^+_Z, K \times \cm)$-module. 
It is also natural to require that the square map $\sq$ be $\cm$-equivariant with respect to this $\cm$-action. This would lift the identification $\TT(\omega_Z^{\smallhalf}) \simeq \TT^+_Z$ \eqref{eq:isom_T} into an equality in $\PicAlg_{K \times \cm}(Z)$.
This is not always true, however; see \Cref{ex:metaplectic_rep} below. 

We remedy this problem by passing to a double cover of $\cm$, or of $\Gamma_K \times \cm$, which is an analogue of the group $\tGamma_K$ in \Cref{subsubsec:tilde_Gamma_K}. To simplify the notation, we write 
\[
	\Gamma_K^{\cm} = \Gamma_K^{\cm}(\orbk) := \Gamma_K(\orbk) \times \cm.
\]
As in \Cref{subsubsec:tilde_Gamma_K}, for each $\psi \in \Gamma_K^{\cm}$, we write 
$\tau_\psi: \omega_Z \xrightarrow{\sim} \psi^* \omega_Z$
for the collection of $K$-equivariant isomorphisms defining the $\Gamma_K \times \cm$-equivariant structure on $\omega_Z$.
We now define $\tGamma_K^{\cm} = \tGamma_K^{\cm}(\orbk)$ to be the group consisting of pairs $(\psi, \vartheta)$, where $\psi \in \Gamma_K^{\cm}$ and $\vartheta: \omega_Z^{\smallhalf} \xrightarrow{\sim} \psi^* \omega_Z^{\smallhalf}$ is a $K$-equivariant isomorphism of line bundles satisfying the same condition \ref{cond:tGamma_K} as in \Cref{subsubsec:tilde_Gamma_K}. Again this condition is independent of the choice of $\sq$. The multiplication on $\tGamma_K^{\cm}$ is given by the same formula \eqref{eq:mult_tilde_Gamma_K}. Again, we have a forgetful map $\tGamma_K^{\cm} \to \Gamma_K^{\cm}$ given by $(\psi, \vartheta) \mapsto \psi$. The same argument as for $\tGamma_K$ in \Cref{subsubsec:tilde_Gamma_K} shows that there is a central extension
	\begin{equation} \label{eq:extn_Gamma_K_cm}
		1 \to \Z_2 \to \tGamma_K^{\cm} \to \Gamma_K^{\cm} \to 1,
	\end{equation}
such that, if we pull back this extension via the inclusion $\Gamma_K \simeq \Gamma_K \times \{1\} \hookrightarrow \Gamma_K^{\cm}$, we obtain the extension \eqref{eq:extn_Gamma_K}. In other words, $\tGamma_K$, as a natural normal subgroup of $\tGamma_K^{\cm}$, is nothing but the preimage of $\Gamma_K \times \{1\}$ in $\tGamma_K^{\cm}$. By definition, there are actions of $\tGamma_K^{\cm}$ on $\omega_Z$ and $\omega_Z^{\smallhalf}$ that extend the previously defined $\tGamma_K$-actions.

Before proceeding, we remark that the square isomorphism $\sq: (\omega_Z^{\smallhalf})^{\otimes 2} \xrightarrow{\sim} \omega_Z$ is $\tGamma_K^{\cm}$-equivariant, where the action of $\tGamma_K^{\cm}$ on $\omega_Z$ factors through the natural $\Gamma_K^{\cm}$-action on $\omega_Z$. Moreover, the $\tGamma_K^{\cm}$-action on $\omega_Z^{\smallhalf}$ induces an action on the associated Picard algebroid $\TT^+_Z \simeq \TT(\omega_Z^{\smallhalf})$, which factors through $\Gamma_K^{\cm}$. 
The line bundle $\omega_Z^{\smallhalf}$ can be considered as an object in $\Coh(\TT^+_Z, K \times \tGamma_K^{\cm})^{gen}$.

To better understand the central extension \eqref{eq:extn_Gamma_K_cm}, let $T$ denote the subgroup $\{1\} \times \cm$ of $\Gamma_K^{\cm}$ and let $\widetilde{T}$ denote the preimage of $T$ in $\tGamma_K^{\cm}$.
Then the projection homomorphism $\pr_{\cm}: \Gamma_K^{\cm} = \Gamma_K \times \cm$ to the second factor gives rise to an identification $T \simeq \cm$. Thus $\widetilde{T}$ is a central extension of $T \simeq \cm$ by $\{\pm 1\}$ via the quotient map $\tGamma_K^{\cm} \twoheadrightarrow \Gamma_K^{\cm}$. There are two cases:
\begin{enumerate}[label=(\alph*), itemindent=*, leftmargin=*]
	\item $\widetilde{T}$ is disconnected. Hence the restriction of $\widetilde{T} \twoheadrightarrow T$ to $\widetilde{T}^\circ$ is an isomorphism $\widetilde{T}^\circ \xrightarrow{\sim} T$, and $\widetilde{T} = \widetilde{T}^\circ \times \{\pm 1\} = T \times \{\pm 1\}\simeq \cm \times \{\pm 1\}$.

	\item $\widetilde{T}$ is a connected $2$-fold cover of $T$ via $\widetilde{T} \twoheadrightarrow T$, and hence can be identified with $\cm$ (abstractly) so that the covering map $\widetilde{T} \to T$ is identified with the square homomorphism $\cm \to \cm,\, z \mapsto z^2$.
\end{enumerate}
In both cases, the inclusions $\tGamma_K$ and $\widetilde{T}$ into $\tGamma_K^{\cm}$ induce an isomorphism of groups 
\begin{equation} \label{eq:Gamma_K_cm_iso}
	(\tGamma_K \times \widetilde{T}) / \{(1, 1), (-1, -1)\} \simeq \tGamma_K^{\cm}.
\end{equation}
We now use this quotient description to extend genuine representations of $\tGamma_K$ and lift the constructions in \Cref{subsubsec:classification_adm_vb} to $\cm$-equivariant or $\tGamma_K^{\cm}$-equivariant settings, as follows. 
First fix a character $\varepsilon: \widetilde{T} \to \cm$ of $\widetilde{T}$ that is genuine, that is, $\varepsilon(-1) = -1$. 
Note that this choice is not unique: two choices of $\varepsilon$ differ by tensoring with a character of $\widetilde{T}$ that factors through $T$, i.e., a character of $T$ regarded as a character of $\widetilde{T}$ by inflation. We will fix our choice of $\varepsilon$ for the rest of the article and write $\C_\varepsilon$ for the corresponding one-dimensional representation space. 
For any $V \in \Rep(\tGamma_K)^{gen}$, we can then identify $V$ with $V \boxtimes_\C \C_\varepsilon$, the latter being a $\tGamma_K \times \widetilde{T}$-representation which factors through $\tGamma_K^{\cm}$ by \eqref{eq:Gamma_K_cm_iso}. Indeed, the diagonal element $(-1, -1)$ acts trivially because both factors are genuine.


Let $\ft := \Lie(T)$. The covering morphism $\widetilde{T} \to T$ induces an isomorphism of Lie algebras $\Lie(\widetilde{T}) \xrightarrow{\sim} \ft$, and hence an isomorphism of dual spaces $\Lie(\widetilde{T})^* \simeq \ft^*$. Let $\kappa = d \varepsilon \in \ft^*$ be the differential of $\varepsilon$. Clearly, when case (a) above occurs, $\kappa$ integrates to an (algebraic) character of $T$, say, the restriction of the genuine character $\varepsilon$ of $\widetilde{T} \simeq T \times \Z_2$ to $T$. When case (b) occurs, $\kappa$ does not integrate to any character of $T$, but does integrate to a genuine character of the connected torus $\widetilde{T}$.

Now, in the definition of the equivalence functor \eqref{eq:equiv_underline_V}, we equip $\OO_Z$ with the standard strong $(\sD_Z, K \times \Gamma_K^{\cm})$-module structure, so that $\underline{V}_Z = \OO_Z \otimes_\C V \in \Coh(\sD_Z, K \times \tGamma_K)^{gen}$ carries a natural genuine $\tGamma_K^{\cm}$-action induced from the genuine action on $V$ and the strong $\Gamma_K^{\cm}$-action on $\OO_Z$. If we regard $\kappa \in \ft^*$ as a character of $\fk \oplus \ft$ by inflation, then $\underline{V}_Z$ becomes a $(K \times \tGamma_K^{\cm}, \kappa)$-equivariant $\sD_Z$-module. Any morphism in $\Coh(\sD_Z, K \times \tGamma_K)^{gen}$ is automatically $\tGamma_K^{\cm}$-equivariant. 
This gives a fully faithful embedding 
\[ \Coh(\sD_Z, K \times \tGamma_K)^{gen} \longhookrightarrow \Coh^{\kappa}(\sD_Z, K \times \tGamma_K^{\cm})^{gen}, \]
which is not essentially surjective. Different choices of $\varepsilon$ give different embeddings, which are intertwined by the equivalences \eqref{eq:equiv_coherent_kappa} with each other. By the above discussion concerning $\kappa$, we can set $\kappa = 0$ if and only if case (a) occurs.

\begin{Ex} \label{ex:metaplectic_rep} 
	The oscillator/metaplectic/Segal-Shale-Weil representations of the metaplectic group $\Mp(2n,\R)$ arise as quantizations of admissible vector bundles over minimal nilpotent $K$-orbits of $\fsp(2n,\R)$. In this example, case (a) occurs when $n$ is even and case (b) occurs when $n$ is odd. 
\end{Ex}

In the definition of the functor \eqref{eq:equiv_twist_omega}, $\cE \otimes_{\OO_Z} \omega_Z^{\smallhalf}$ carries a natural $\tGamma_K^{\cm}$-action induced from the genuine action on $\cE$ and the canonical genuine $\tGamma_K^{\cm}$-action on $\omega_Z^{\smallhalf}$. This action descends to a $\Gamma_K^{\cm}$-action, making $\cE \otimes_{\OO_Z} \omega_Z^{\smallhalf}$ a $(K \times \Gamma_K^{\cm}, \kappa)$-equivariant $\TT^+_Z$-module. In the definition of the functor \eqref{eq:equiv_pushforward}, $\hat{\rho}_* \cV$ is $(K \times \cm, \kappa)$-equivariant, and hence so is $(\hat{\rho}_* \cV)^{\Gamma_K}$. This upgrades the functor \eqref{eq:classification_adm_vb_cat} into the lower horizontal equivalence in the following commutative diagram:
\begin{equation} \label{diag:lift_graded}
	\begin{tikzcd}
		\Rep(\tGamma_K)^{gen} \ar[r, "\sim"] \ar[d, hook] & \Coh(\TT^+_{\orbk}, K) \ar[d, hook] \\
		\Rep(\tGamma_K^{\cm})^{gen} \ar[r, "\sim"] & \Coh^\kappa(\TT^+_{\orbk}, K \times \cm)
	\end{tikzcd}
\end{equation}
where the top horizontal equivalence is given by \eqref{eq:classification_adm_vb_cat}. The left vertical functor is given by $V \mapsto V \boxtimes_\C \C_\varepsilon$ and the right vertical functor is given by equipping any object in $\Coh(\TT^+_{\orbk}, K)$ with the $\cm$-action induced from $\varepsilon$ as above. These two vertical functors are fully faithful embeddings but not essentially surjective. 

\begin{Rem} \label{rem:strong_cm_action_Galois}
	The same construction works for any $K$-equivariant Galois cover $\tilde{\rho}: \torbk \to \orbk$ of $\orbk$. Again, we need to consider the universal cover $\varphi: Z=\horbk \to \torbk$ of $\torbk$ and $\Gamma_K$ should be replaced by the Galois group of $\varphi$.
\end{Rem}

\subsubsection{Pullback to Galois covers} \label{subsubsec:galois}

In this subsection, we consider the special case when $\tilde{\rho}: Y = \torbk \to \orbk$ is a $K$-equivariant Galois cover.
Recall that $\Gamma_{\tilde{\rho}}$ and $\Gamma_\varphi$ denote the Galois groups of the $K$-equivariant Galois covering $\tilde{\rho}: Y = \torbk \to \orbk$ and the universal cover $\varphi: \horbk \to \torbk$, respectively. Then $\Gamma_\varphi$ is naturally a normal subgroup of $\Gamma_K$ and $\Gamma_{\tilde{\rho}}$ can be identified with $\Gamma_K/\Gamma_\varphi$. Let $\tGamma_\varphi$ be the preimage of $\Gamma_\varphi$ under the quotient homomorphism $\tGamma_K \twoheadrightarrow \Gamma_K$. By \Cref{defn:Galois}, we have natural group isomorphisms 
    \[ \tGamma_K/\tGamma_\varphi \simeq \Gamma_K/\Gamma_\varphi \simeq \Gamma_{\tilde{\rho}}.\] 
Set $\Gamma_{\tilde{\rho}}^{\cm} := \Gamma_{\tilde{\rho}} \times \cm$. Then there are natural group isomorphisms 
\[ \tGamma_K^{\cm}/\tGamma_\varphi \simeq \Gamma_K^{\cm}/\Gamma_\varphi \simeq \Gamma_{\tilde{\rho}}^{\cm}.\]

Define
\begin{equation} \label{eq:cW_Y}
	\cW := \varphi_* \omega_Z^{\smallhalf} \in \Coh(\TT^+_Y, K \times \tGamma_K^{\cm})^{gen}.
\end{equation}
Then $\tGamma_\varphi$ also acts on $\cW$ via the action of $\tGamma_K \subset \tGamma_K^{\cm}$. 
For a fixed choice of the character $\varepsilon: \widetilde{T} \to \cm$ and $\kappa = d\varepsilon$, we lift any $V \in \Rep(\tGamma_K)^{gen}$ to an object $V \boxtimes_\C \C_\varepsilon$ in $\Rep(\tGamma_K^{\cm})^{gen}$ (still denoted as $V$) via a functor $\Rep(\tGamma_K)^{gen} \hookrightarrow \Rep(\tGamma_K^{\cm})^{gen}$ as in \Cref{subsubsec:strong_cm_action}.
Set
\begin{equation} \label{eq:cWV_Y}
	\cWV := (V \otimes_\C \cW)^{\Gamma_\varphi} \in \Coh^\kappa(\TT^+_Y, K \times \Gamma_{\tilde{\rho}}^{\cm}),
\end{equation}
where the diagonal $\tGamma_\varphi$-action on $V \otimes_\C \cW$ descends to $\Gamma_\varphi$.

\begin{Lem} \label{lem:galois_pullback_adm_vb}
	With the notation above, we have a natural isomorphism in $\Coh^\kappa(\TT^+_Y, K \times \Gamma_{\tilde{\rho}}^{\cm})$,
\begin{equation} \label{eq:cWV_FV_pullback}
	\cWV \simeq \tilde{\rho}^* \cF_V,
\end{equation}
where $\cF_V$ is as in \eqref{eq:classification_adm_vb_cat}. Moreover, we have a natural isomorphism in $\Coh^\kappa(\TT^+_{\orbk}, K \times \Gamma_K^{\cm})$,
\begin{equation} \label{eq:rho_cWV_FV}
	(\tilde{\rho}_* \cWV)^{\Gamma_{\tilde{\rho}}} \simeq \cF_V.
\end{equation}
\end{Lem}

\begin{proof}
	The isomorphism \eqref{eq:cWV_FV_pullback} is the descent isomorphism for the Galois cover $\tilde{\rho}$, see the paragraph around \eqref{eq:equiv_pushforward}. We have a chain of isomorphisms in $\Coh^\kappa(\TT^+_{\orbk}, K \times \Gamma_K^{\cm})$,
	\begin{equation} \label{eq:rho_cWV}
		\tilde{\rho}_* (V \otimes_\C \cW) \simeq \tilde{\rho}_* \varphi_* (V \otimes_{\C} \omega_Z^{\smallhalf}) \simeq \tilde{\rho}_* \varphi_* (\underline{V}_Z \otimes_{\OO_Z} \omega_Z^{\smallhalf}) \simeq \hat{\rho}_* (\underline{V}_Z \otimes_{\OO_Z} \omega_Z^{\smallhalf}),
	\end{equation}
	which follows from the definition $\cW=\varphi_*\omega_Z^{\smallhalf}$ and the equality $\hat{\rho}=\tilde{\rho}\circ\varphi$. Since $\Gamma_{\tilde{\rho}} \simeq \Gamma_K/\Gamma_\varphi$, we have natural isomorphisms
	\[ 
		(\tilde{\rho}_* \cWV)^{\Gamma_{\tilde{\rho}}} = \left[\tilde{\rho}_* (V \otimes_\C \cW)^{\Gamma_\varphi}\right]^{\Gamma_{\tilde{\rho}}} \simeq \left[ \tilde{\rho}_* (V \otimes_\C \cW) \right]^{\Gamma_K} \simeq \cF_V,
	\]
	which gives \eqref{eq:rho_cWV_FV}.
\end{proof}

\subsection{Admissibility vs specialness}\label{subsec:adm_special_orbit}

We discuss the relationship between admissibility and specialness of nilpotent orbits for linear classical/metaplectic groups. We first recall the admissibility criteria for classical types and extend item to the metaplectic case, then apply them to orbits dual to quasi-distinguished orbits. A main reference for classification of admissible orbits of linear classical groups is \cite{Ohta:Adm} (see also \cite{Schwartz}). In this subsection and \cref{subsec:adm_spin_genuine}, we will often need to compute the reductive $K$-centralizer $K_Q$ and $K_Q^\circ$ of a nilpotent $K$-orbit, which can be computed from the $ab$-diagrams associated to the $K$-orbit $\orbk$ as outlined in the last part of \cref{subsec:nil_orbits}.

Let $\lambda \in \cP_\epsilon(N)$ be an $\epsilon$-partition of $N$, with $\epsilon = 0$ or $1$. The following condition (P') was introduced on page 291 of \cite{Ohta:Adm}.

\begin{itemize}
    \item[(P')] For each even part of $\lambda$, the number of odd parts of $\lambda$ that are less than the even part is even, and for each odd part of $\lambda$, the number of even parts of $\lambda$ that are greater than the odd part is even.
\end{itemize}

We will need the following classification result for admissible orbits of classical groups.

\begin{Prop}\label{prop:adm_classical}
	Let $G_\R$ be a connected real classical Lie group of type $B$, $C$, or $D$, with complexified Lie algebra $\fg$. Let $K$ be the complexification of a maximal compact subgroup $K_\R$ of $G_\R$. Let $\orb$ be any nilpotent orbit in $\g^*$ and $\orbk$ be any nilpotent $K$-orbit in $\orb_{\fk^\perp}$. Then the following conditions are equivalent:
	\begin{enumerate}[label=(\roman*)]
		\item $\orbk$ is $K$-admissible;
		\item the partition $\lambda$ associated to $\orb$ satisfies the condition (P');
		\item $\orb$ is a special nilpotent orbit in $\fg$.
	\end{enumerate}
\end{Prop}

\begin{proof}
	The equivalence of (i) and (ii) is the content of \cite[Theorem 2]{Ohta:Adm}. The equivalence of (ii) and (iii) is an easy consequence of \cref{lem:special_classical}.
\end{proof}

For the metaplectic group $\Mp(2n,\R)$, we introduce the following definition.

\begin{defi}
	Let $\GR = \Mp(2n,\R)$, so that the complexification of any maximal compact subgroup of $\GR$ is $K \simeq \widetilde{\GL}(n,\C)$, the unique connected $2$-fold cover of $\GL(n,\C)$. We say that a nilpotent $K$-orbit $\orbk$ is {\it metaplectic admissible} if it is not admissible with respect to $\Sp(2n,\R)$ but is admissible with respect to $\Mp(2n,\R)$.
\end{defi}

Given $\lambda \in \cP_C(2n)$, we introduce the following notation: for each odd part $c$ of $\lambda$, let $m_{>c}(\lambda)$ denote the number of even parts of $\lambda$ that are greater than $c$. We have the following analogue of \cref{prop:adm_classical} for the metaplectic group.

\begin{Prop} \label{prop:metaplectic_adm}
	Let $\GR = \Mp(2n,\R)$. Let $\orb$ be any nilpotent orbit in $\fg^*$ and $\orbk$ be any nilpotent $K$-orbit in $\orb_{\fk^\perp}$. Then $\orbk$ is metaplectic admissible if and only if $\orb$ is a metaplectic special nilpotent orbit in $\fg^*$.
\end{Prop}

\begin{proof}
    Let $G = \Sp(2n, \C)$ and $\bar{G}_\R = \Sp(2n,\R)$,
	and let $\phi: \GR \to \bar{G}_\R$ be the $2$-fold covering morphism, so that we are in the situation of the last part of \Cref{subsec:basic_settings}. Then $\bar{K} \simeq \GL(n,\C)$ and $K \simeq \widetilde{\GL}(n, \C)$.

    Given any $\orbk \subset \orb_{\fk^\perp}$, fix a point $e \in \orbk$ and take a normal $\slf_2$-triple $(e, h, f)$. Let $\bar{K}_e \subset \bar{K}$ and $\bar{K}_Q \subset \bar{K}_e$ be the groups defined for $\bar{K}$ analogously to $K_e$ and $K_Q$ for $K$ (see \Cref{subsec:nil_orbits}). 
    Let $\lambda \in \cP_C(2n)$ be the partition corresponding to $\orb$. 
	Then $\bar{K}_Q^\circ$ is a product of (complex) orthogonal group factors (contributed from even parts of $\lambda$) and general linear group factors (contributed from odd parts of $\lambda$). Each odd part $c$ of $\lambda$ contributes a factor of $\GL(n_c)$ to $\bar{K}_Q^\circ$, where $2n_c = m(c)$. Therefore $\bar{K}_Q^\circ$ can be decomposed as a direct product
    \[ 
        \bar{K}_Q^\circ \simeq \prod_{a \text{ odd}} \GL(n_c) \times C,
    \]
    where $C$ is a direct product of the orthogonal group factors coming from even parts of $\lambda$.
    The embedding of each $\GL(n_c)$-factor into $\bar{K} \simeq \GL(n)$ is, up to conjugation, given by the standard block diagonal embedding 
		\[ \GL(n_c) \hookrightarrow \GL(n), \quad A \mapsto \mathrm{diag}(\underbrace{A, A, \cdots, A}_{c}, I_{n - c \cdot n_c}). \] 
	Since $a$ is odd, the preimage of each $\GL(n_c)$-factor of $\bar{K}_Q^\circ$ in $K \simeq \widetilde{\GL}(n)$ is connected and isomorphic to the $2$-fold cover $\widetilde{\GL}(n_c)$ of $\GL(n_c)$. 
    
    By \cite[Lemma 7]{Ohta:Adm}, one can deduce that the character $\gamma$ as defined in \cref{subsubsec:AOD_defn}, regarded as a character of $\fk_Q$, restricts to the zero character of the Lie algebra of $C$. Therefore $\orbk$ is admissible for $\Mp(2n,\R)$ if and only if the restriction of $\gamma$ to the Lie algebra of each $\GL(n_c)$-factor can be integrated to a character of $\widetilde{\GL}(n_c)$, and these characters all agree on the intersections of the various $\widetilde{\GL}(n_c)$, which is exactly $\ker(K \twoheadrightarrow \bar{K}) = \{\pm 1\}$. 
    Again by \cite[Lemma 7]{Ohta:Adm}, this is equivalent to the condition that the numbers $m_{>c}(\lambda)$ for all odd parts $a$ of $\lambda$ have the same parity. 
	
	On the other hand, $\orbk$ is admissible for $\Sp(2n,\R)$ if and only if $\lambda$ satisfies the condition (P'), i.e., the numbers $m_{>c}(\lambda)$ for all odd parts $a$ of $\lambda$ are even. Therefore $\orbk$ is metaplectic admissible if and only if the numbers $m_{>c}(\lambda)$ are odd for all odd parts $c$ of $\lambda$, and this is equivalent to the condition that $\lambda \in \cP_{1,1}^{sp}(2n)$, i.e., $\orb = \orb_\lambda$ is a metaplectic special nilpotent orbit in $\fg^*$.
\end{proof}

\begin{Prop}\label{prop:adm_dist}
	Let $\ckorb$ be any quasi-distinguished orbit in $\ckfg^*$ and $\orb = d(\ckorb)$ be its dual orbit in $\g^*$. Then any $K$-orbit $\orbk$ in $\orb_{\fk^\perp}$ is linearly admissible.
\end{Prop}

\begin{proof}
	For classical $\g$, we appeal to \Cref{prop:adm_classical}. Note that when $\g = \fso(N)$, admissibility for the classical groups automatically implies linear admissibility. The case of exceptional groups can be checked case by case using the tables in \cite{Nevins} (for groups of points of split simply connected algebraic groups of types $G_2$, $F_4$, $E_6$, and $E_7$) and \cite{Noel1, Noel2} (for all real exceptional groups, with appropriate reinterpretation as pointed out in \cite[Remark 8.5.4]{LY}). 
\end{proof}

\begin{Prop} \label{prop:non_adm_codim2}
	Let $G$ be a connected complex simple group that is not a spin group.
	Let $\ckorb$ be any quasi-distinguished orbit in $\ckfg$ and $\orb = d(\ckorb)$, where $d$ stands for the usual BV duality map (not the metaplectic one). Let $\orb'$ be a $G$-orbit in $\cP(\orb)$ such that $\codim(\orb', \overline{\orb}) = 2$. Then any $K$-orbit $\orbk' \subset \orb_{\fk^\perp}$ is not admissible for $K=(G^\theta)^\circ$.
\end{Prop}

\begin{proof}
	For linear classical $G$, $\orb'$ is non-special by \cref{prop:codim_sp_classical}, and hence $\orbk'$ is not admissible by \Cref{prop:adm_classical}. For exceptional $\g$, we can check it case by case as in the proof of \Cref{prop:adm_dist}.
\end{proof}

\begin{Rem}
	From \Cref{prop:adm_classical} and \Cref{prop:non_adm_codim2}, one can see a connection between specialness and admissibility. This has been pointed out by Vogan and Nevins; see the Main Theorem, the Conjecture, and the discussion between them in \cite{Nevins}.
\end{Rem}

\subsection{Genuine admissible orbit data of spin groups} \label{subsec:adm_spin_genuine}

We now consider the case when $G=\Spin(N)$. It suffices to only consider the case when $N \geqslant 7$, since otherwise $\fg$ is either abelian, or of type $A$ or $C$. Then $G$ is the universal covering group of $\bar{G} = \SO(N)$. The kernel of the covering morphism $G \to \bar{G}$ is isomorphic to $\Z_2$, denoted as $\{\pm 1\}$. For any anti-holomorphic involution $\varpi$ of $\fg$, regard it as an anti-holomorphic involution of $G=\Spin(N)$ or $\bar{G} = \SO(N)$, and set $\GR = G(\R) = G^\varpi$ (which is connected since $G$ is simply connected) and $\bar{G}_\R = (\bar{G}^\varpi)^\circ$. Then the restriction of the covering morphism $G \to \bar{G}$ to $\GR$ gives a 2-fold covering morphism $\GR \to \bar{G}_\R$. Let $\theta$ be the Cartan involution as in \cref{subsec:basic_settings}. Set $K = (G^\theta)^\circ = G^\theta$ and $\bar{K} = (\bar{G}^\theta)^\circ$.
There are two types of noncompact real forms: 
\begin{itemize}[itemindent=*, leftmargin=*]
	\item The special orthogonal Lie algebra $\fso(p,q)$, where $p+q=N$, $p,q \geqslant 1$. In this case, $\bar{G}_\R=\SO_0(p,q)$, $\GR=\Spin(p,q)$, and
		\[ \bar{K} \simeq \SO(p) \times \SO(q) \quad \text{and} \quad K \simeq (\Spin(p) \times \Spin(q))/\{(1, 1), (-1, -1)\}. \]
	\item The quaternionic orthogonal Lie algebra $\fso^*(2n)$, where $N=2n$ is even. In this case, $\bar{G}_\R=\SO^*(2n)$, $\GR=\Spin^*(2n)$, $\bar{K} \simeq \GL(n,\C)$, and $K \simeq \widetilde{\GL}(n, \C)$ is the unique connected $2$-fold cover of $\GL(n,\C)$.
\end{itemize}
In either case, we have an obvious inclusion map $\AOD_{\bar{K}}(\orb) \hookrightarrow \AODK(\orb)$ via inflation. The complement of the image of this inclusion is the set $\AODK(\orb)^{gen}$ of genuine admissible orbit data, which we study below in both cases.

\begin{Prop} \label{prop:adm_Spin_pq}
	Assume $\g = \fso(N)$ and let $\g_\R$ be a real form of $\g$ corresponding to a Cartan involution $\theta$. Let $\bar{G}_\R = \SO(p,q)$ and $G_\R = \Spin(p,q)$, with $p+q=N$. Set $G = \Spin(N)$, $\bar{G} = \SO(N)$, $K = G^\theta$ and $\bar{K} = (\bar{G}^\theta)^\circ$. Let $\ckorb = \ckorb_\eta$ be a quasi-distinguished orbit in $\ckfg$ and $\orb = d(\ckorb)$. Then 
	\begin{enumerate}[itemindent=*, leftmargin=*]
		\item If some nonzero part of $\eta$ has odd multiplicity (which must be $1$), then $\AODK(\orb)^{gen} = \varnothing$;
		\item Suppose all nonzero parts of $\eta$ have even multiplicities (which must be $2$). Then 
		\[
			|\AODK(\orb)^{gen}| =
			\begin{cases}
				4, & \text{if } p=q \text{ is even}; \\
				1, & \text{if } p=q \text{ is odd}; \\
				2, & \text{if } |p-q|=1; \\
				0, & \text{otherwise}.
			\end{cases}
		\]
		Moreover, in the first three cases, there is a unique $K$-orbit $\orbk \subset \orb_{\fk^\perp}$ such that $\AODK(\orbk)^{gen} \neq \varnothing$.
	\end{enumerate}
\end{Prop}

\begin{proof}
	In this case, $\ckfg$ is of type $C$ or $D$. By \Cref{lem:d_quasi-dis}, the partition associated to the orbit $\orb$ is $\lambda = d(\eta) = (\eta^\dagger)^*$. By \cref{defi:quasi-dist_classical}, it is easy to see that $\lambda$ has at least one odd part, say, $1$. 
	
	Given any $\orbk \subset \orb_{\fk^\perp}$, fix a point $e \in \orbk$ and take a normal $\slf_2$-triple $(e, h, f)$. Let $\bar{K}_e \subset \bar{K}$ and $\bar{K}_Q \subset \bar{K}_e$ be the groups defined for $\bar{K}$ analogously to $K_e$ and $K_Q$ for $K$ (see \Cref{subsec:nil_orbits}). By \Cref{prop:adm_dist}, we know that any such $\orbk$ is linearly admissible (for $\bar{K}$). The condition $\AODK(\orb)^{gen} = \varnothing$ is equivalent to the connectedness of the preimage of $\bar{K}_Q^\circ$ in $K$ (or $K_Q$), since this would mean that any $\gamma$-admissible representation $V$ of $K_Q$ must be trivial on $\ker(K_Q \to \bar{K}_Q) = \ker(K \to \bar{K})\simeq \Z_2$, and hence any such $V$ factors through $\bar{K}_Q$. 
	
	In our case, $\bar{K}_Q^\circ$ is a direct product of (complex) special orthogonal groups and general linear groups. Each odd member $c$ of $\lambda$ contributes a factor of $\SO(p_c) \times \SO(q_c)$ to $\bar{K}_Q^\circ$, where $p_c + q_c = m_\lambda(c)$ and $p_c$ (resp. $q_c$) is the number of rows of length $c$ in the $ab$-diagram of $\orbk$ that starts with $a$ (resp. $b$). 
	The preimage of each factor $\SO(p_c) \times \SO(q_c)$ in $K$ is $(\Spin(p_c) \times \Spin(q_c))/\{(1,1), (-1,-1)\}$.


	Suppose that $\eta$ has at least one nonzero part with odd multiplicity (which has to be $1$). Then by \cref{defi:quasi-dist_classical}, we can take this part to be $\lambda_i$, where $i$ is a positive odd integer. Since all parts of $\eta$ are of the same parity, we have $\eta_i \geqslant \eta_{i+1}+2$. Note that, if $i = \#\eta$, we make the convention that $\eta_{i+1}=0$ (this can only occur when $\ckfg$ is of type $C$). By \Cref{lem:d_quasi-dis}, $\eta^\dagger_i = \eta_i+1$ and $\eta^\dagger_{i+1} \leqslant \eta_{i+1}$, hence 
		\[ p_i + q_i = m_\lambda(i) = \eta^\dagger_i - \eta^\dagger_{i+1} \geqslant \eta_i - \eta_{i+1} + 1 \geqslant 3. \]
	This implies that $p_i \geqslant 2$ or $q_i \geqslant 2$. Therefore $\bar{K}_Q^\circ$ contains a factor of $\SO(t)$ with $t \geqslant 2$. It is easy to see that the preimage of this factor $\SO(t)$ in $K$ is $\Spin(t)$, which is connected. Therefore the preimage of $\bar{K}_Q^\circ$ in $K$ or $K_Q$ is connected, and hence $\AODK(\orbk)^{gen} = \varnothing$ for any $\orbk \subset \orb_{\fk^\perp}$. This confirms case (1).  
	
	From now on, assume that all nonzero parts of $\eta$ have even multiplicities (which equal $2$). Then $\eta$ is of the form $[t_1^2, t_2^2, \cdots, t_s^2]$ with $t_1 > t_2 > \cdots > t_s > 0$. Set $N=p+q$. We have two cases:
	\begin{enumerate}[label=(\Roman*), itemindent=*, leftmargin=*]
		\item $N$ is even: in this case, $\fg$ is of type $D$ and 
			\[ \eta^\dagger = [t_1+1, t_1-1, \cdots, t_s+1, t_s-1]. \] 
			Then all parts of $\lambda$ have multiplicity $2$.
		\item $N$ is odd: in this case, $\fg$ is of type $B$ and 
			\[ \eta^\dagger = [t_1+1, t_1-1, \cdots, t_s+1, t_s-1, 1]. \] 
			Then the largest part $\lambda_1 = 2s + 1$ of $\lambda$ is odd and has multiplicity $1$. The other odd parts of $\lambda$ all have multiplicity $2$.
	\end{enumerate}

	If there is one odd part $c$ of $\lambda$ with $m_\lambda(c) = 2$ and the $ab$-diagram of $\orbk$ satisfies $p_c=2$ or $q_c=2$, then as in case (1), $\bar{K}_Q^\circ$ contains a factor of $\SO(2)$ and hence $\AODK(\orbk)^{gen} = \varnothing$ for such $\orbk$. Therefore $\AODK(\orbk)^{gen} \neq \varnothing$ only if for any odd part $c$, we have both $p_c \leqslant 1$ and $q_c \leqslant 1$. 

	Assuem that, for all odd parts $c$ of $\lambda$, we have both $p_c \leqslant 1$ and $q_c \leqslant 1$. Then there is one unique $ab$-diagram configuration satisfying this condition, which forces $p=q$ when $\fg$ is of type $D$ and $|p-q|=1$ when $\fg$ is of type $B$. In either case, there is a unique $K$-orbit $\orbk \subset \orb_{\fk^\perp}$ attached to such an $ab$-diagram by \cite[Theorem 9.3.4]{CM}. For this $\orbk$, $\bar{K}_Q^\circ$ is a direct product of copies of $\SO(1)$ and hence trivial. Set 
		\[ \mathbf{p} = \sum_{c \text{ odd}} p_c \quad \text{and} \quad \mathbf{q} = \sum_{c \text{ odd}} q_c. \]
	Then $\bar{K}_Q = Z_{\bar{K}} \simeq \Z_2^{\mathbf{p}-1} \times \Z_2^{\mathbf{q}-1}$ and $K_Q = Z_K$ is a central $\Z_2$-extension of $\bar{K}_Q$. To describe this central extension, we identify $\Z_2^{\mathbf{p}-1}$ (resp. $\Z_2^{\mathbf{q}-1}$) with the diagonal subgroup of $\SO(\mathbf{p})$ (resp. $\SO(\mathbf{q})$), denoted as $\SO_{1^{\mathbf{p}}}$ (resp. $\SO_{1^{\mathbf{q}}}$). Let $\Spin_{1^{\mathbf{p}}}$ (resp. $\Spin_{1^{\mathbf{q}}}$) be the preimage of $\SO_{1^{\mathbf{p}}}$ (resp. $\SO_{1^{\mathbf{q}}}$) in $\Spin(\mathbf{p})$ (resp. $\Spin(\mathbf{q})$), which is a central $\Z_2$-extension of $\SO_{1^{\mathbf{p}}}$ (resp. $\SO_{1^{\mathbf{q}}}$). Then 
		\[ K_Q = Z_K \simeq (\Spin_{1^{\mathbf{p}}} \times \Spin_{1^{\mathbf{q}}}) / \{(1, 1), (-1, -1)\}, \]
	where $\{(1, 1), (-1, -1)\}$ is the diagonal subgroup of the product $\Z_2 \times \Z_2$ of the two central $\Z_2$-subgroups of $\Spin_{1^{\mathbf{p}}}$ and $\Spin_{1^{\mathbf{q}}}$.
	An irreducible genuine representation of $K_Q$ must be of the form $V_{\mathbf{p}} \otimes V_{\mathbf{q}}$ for some irreducible genuine representation $V_{\mathbf{p}}$ (resp. $V_{\mathbf{q}}$) of $\Spin_{1^{\mathbf{p}}}$ (resp. $\Spin_{1^{\mathbf{q}}}$) (the action of $(-1,-1) \in \Z_2 \times \Z_2$ on $V_{\mathbf{p}} \otimes V_{\mathbf{q}}$ is trivial, hence the representation factor through $K_Q$). The irreducible genuine representations of $\Spin_{1^n}$ for any positive integer $n$ are classified as follows:
	\begin{itemize}[itemindent=*, leftmargin=*]
		\item When $n=2m$ is even, there are two irreducible genuine representations of $\Spin_{1^n}$ up to isomorphism, which are the restrictions of the two half-spin representations $S_+^{n}$ and $S_-^{n}$ of $\Spin(2m)$, both of dimension $2^{m-1}$;
		\item When $n=2m+1$ is odd, there is only one irreducible genuine representation of $\Spin_{1^n}$ up to isomorphism, which is the restriction of the spinor representation $S^{n}$ of $\Spin(2m+1)$ of dimension $2^m$.
	\end{itemize}
	Returning to the classification of irreducible genuine representations of $K_Q$, we have the following three cases:
	\begin{enumerate}[label=(\roman*), itemindent=*, leftmargin=*]
		\item $p=q$ is even: in this case, $\mathbf{p} = \mathbf{q} = n$ is even. There are 4 irreducible genuine representations of $K_Q$ up to isomorphism, which are of the form $S_{\pm}^{\mathbf{p}} \otimes S_{\pm}^{\mathbf{q}}$, all of dimension $2^{n-2}$. Therefore $|\AODK(\orbk)^{gen}| = 4$;
		\item $p=q$ is odd: in this case, $\mathbf{p} = \mathbf{q} = n$ is odd. There is a unique irreducible genuine representation of $K_Q$ up to isomorphism, which is of the form $S^{\mathbf{p}} \otimes S^{\mathbf{q}}$ of dimension $2^{n-1}$. Therefore $|\AODK(\orbk)^{gen}| = 1$;
		\item $|p-q|=1$: in this case, $|\mathbf{p} - \mathbf{q}| = 1$. There are 2 irreducible genuine representations of $K_Q$ up to isomorphism: if $\mathbf{p} - \mathbf{q} = 1$, then $\mathbf{p}$ is odd and $\mathbf{q}$ is even, and $\Irr(K_Q) = \{S^{\mathbf{p}} \otimes S_{\pm}^{\mathbf{q}}\}$; if $\mathbf{q} - \mathbf{p} = 1$, then $\mathbf{p}$ is even and $\mathbf{q}$ is odd, and $\Irr(K_Q) = \{S_{\pm}^{\mathbf{p}} \otimes S^{\mathbf{q}}\}$.  
		All these representations have dimension $2^{\frac{\mathbf{p}+\mathbf{q}-1}{2}-1} = 2^{\min(\mathbf{p}, \mathbf{q})-1}$. Therefore $|\AODK(\orbk)^{gen}| = 2$.
	\end{enumerate}
\end{proof}

\begin{Prop} \label{prop:adm_SO*(2n)}
	Let $\bar{G}_\R = \SO^*(2n)$ and $G_\R = \Spin^*(2n)$. Set $G = \Spin(2n)$, $\bar{G} = \SO(2n)$, $K = G^\theta$ and $\bar{K} = (\bar{G}^\theta)^\circ$. Let $\ckorb = \ckorb_\eta$ be a quasi-distinguished orbit in $\ckfg$ and $\orb = d(\ckorb)$. Then $\AODK(\orb)^{gen} = \varnothing$.
\end{Prop}

\begin{proof}
	The proof is similar to that of \Cref{prop:adm_Spin_pq}.  
	The partition associated to the orbit $\orb$ is $\lambda = d(\eta) = (\eta^\dagger)^*$. 
	In the current setting, $\bar{K}_Q^\circ$ is a product of (complex) symplectic groups and general linear groups. Each odd part $c$ of $\lambda$ contributes a factor of $\GL(n_c)$ to $\bar{K}_Q^\circ$, where $2n_c = m_\lambda(c)$. By the same observation as in the first paragraph of the proof of \cref{prop:adm_Spin_pq}, we know $m_\lambda(1) \geqslant 2$ and is even. We write $m_\lambda(1)=2t$. 
	Therefore $\bar{K}_Q^\circ$ contains a factor of $\GL(t)$ with $t \geqslant 1$. The embedding of this $\GL(t)$ into $\bar{K} \simeq \GL(n)$ is, up to conjugation, given by the standard block diagonal embedding 
		\[ \GL(t) \hookrightarrow \GL(n), \quad A \mapsto \mathrm{diag}(A, I_{n-t}). \] 
	Thus the preimage of the $\GL(t)$-factor of $\bar{K}_Q^\circ$ in $K$ is connected and isomorphic to the metaplectic $2$-fold cover $\widetilde{\GL}(t)$ of $\GL(t)$. Therefore the preimage of $\bar{K}_Q^\circ$ in $K$ or $K_Q$ is connected, so $\AODK(\orb)^{gen} = \varnothing$ by the same argument as in \cref{prop:adm_Spin_pq}.
\end{proof}

\section{Coadjoint orbit method for complex groups} \label{sec:orbit_method_complex}

This section recalls the quantization tools needed below, specializes them to affinizations of nilpotent $G$-orbit covers, and then applies the resulting canonical quantizations to special unipotent ideals and associated cycles.

\subsection{Filtered and Hamiltonian quantizations} \label{subsec:filtered_hamiltonian_quant}

In this subsection, we will recall the definitions of filtered and Hamiltonian quantizations of graded Poisson algebras.

\begin{defi} \label{defn:graded_poisson}
	Let $d \in \Z_{>0}$. A \emph{(non-negatively) graded Poisson algebra} of degree $-d \in \Z_{<0}$ is a finitely generated commutative associative $\Z_{\geqslant 0}$-graded unital algebra 
		\[ A = \bigoplus_{i=0}^{\infty} A_i, \quad A_0 = \C \cdot 1, \]
	with unit $1$, equipped with a Poisson bracket $\{\cdot, \cdot\}$ of degree $-d$ such that
	  \[ \{A_i, A_j\}\subset A_{i+j-d}, \qquad \forall\,i,j \in \Z_{\geqslant 0}.\]
Note that finite generation implies that each $A_i$ is finite dimensional.
\end{defi}

\begin{defi}\label{defn:filtered_quant}
	A \emph{filtered quantization} of a graded Poisson algebra $A$ of degree $-d$ is a pair $(\cA,\varsigma)$ consisting of the following data:
	\begin{enumerate}[label=(\roman*)]
    	\item an associative algebra $\cA$ equipped with a complete multiplicative ascending filtration by subspaces
    	\[ \cA_{\leqslant 0} \subseteq \cA_{\leqslant 1} \subseteq \cA_{\leqslant 2} \subseteq \ldots, \quad \cA = \bigcup_{i=0}^{\infty} \cA_{\leqslant i} \]
    	such that
    	\[ [\cA_{\leqslant i}, \cA_{\leqslant j}] \subseteq \cA_{\leqslant i+j-d} \qquad \forall\,i,j \in \Z,\]
    	
    	\item an isomorphism of graded Poisson algebras
    	\[ \varsigma: \gr\cA \xrightarrow{\, \sim \,} A, \]
    	where $\gr\cA := \bigoplus_{i=0}^{\infty} \cA_{\leqslant i} / \cA_{\leqslant i-1}$ is the associated graded algebra of $\cA$, and
		the Poisson bracket on $\gr\cA$ is defined by
    	\[ \{a+\cA_{\leqslant i-1}, b+\cA_{\leqslant j-1}\} = [a,b]+\cA_{\leqslant i+j-d-1}, \qquad \forall\,a \in \cA_{\leqslant i}, \ b \in \cA_{\leqslant j}.\]
	\end{enumerate}
	An isomorphism of filtered quantizations $(\cA_1, \varsigma_1) \xrightarrow{\sim} (\cA_2, \varsigma_2)$ is an isomorphism of filtered algebras $\phi: \cA_1 \xrightarrow{\sim} \cA_2$ such that $\varsigma_1 = \varsigma_2 \circ \gr\phi$. Denote the set of isomorphism classes of quantizations of $A$ by $\mathrm{Quant}(A)$.
\end{defi}



Let $A$ be a graded Poisson algebra of degree $-d$. Suppose $G$ is an algebraic group that acts rationally on $A$ by graded Poisson automorphisms. Let $\der(A)$ denote the Lie algebra of derivations of $A$. The differentiation of the $G$-action on $A$ induces a Lie algebra homomorphism
\[\fg \to \der(A), \qquad \xi \mapsto \xi_A.\]
A \emph{classical co-moment map} (for the $G$-action on $A$) is a $G$-equivariant map $\phi: \g \to A_d$ satisfying
\[\{\phi(\xi), a\} = \xi_A(a), \qquad \xi \in \fg, \quad a \in A.\] 
In this case, we call the $G$-action on $A$ (together with $\phi$) Hamiltonian.

\begin{Ex} \label{ex:orbit_cover}
The most important example of a graded Poisson algebra with a Hamiltonian $G$-action for us is the algebra of regular functions $A=\C[\widetilde{\orb}]$ of a connected finite $G$-equivariant covering $\widetilde{\orb}$ of a nilpotent orbit $\orb$ in $\g^*$. The $G$-action on $A$ is induced by the $G$-action on $\widetilde{\orb}$, and the grading on $A$ is induced by the $\cm$-action on $\widetilde{\orb}$ as in \Cref{subsec:orbit_cm}. As mentioned in \cref{subsec:orbit_covers}, $\widetilde{\orb}$ is equipped with the pullback of the KKS symplectic form of $\orb$, which induces the Poisson bracket on $A$ of degree $-2$. 

The composition of the covering morphism $\rho: \widetilde{\orb} \to \orb$ with the inclusion $\orb \hookrightarrow \g^*$ is a moment map for the Hamiltonian $G$-action on $\widetilde{\orb}$. This induces a classical co-moment map $\phi: \g \to A_2$ for the Hamiltonian $G$-action on $A$.
\end{Ex}

A filtered quantization $(\cA,\varsigma)$ of $A$ is said to be \emph{$G$-equivariant} if $G$ acts rationally on $\cA$ by filtered algebra automorphisms and the isomorphism $\varsigma: \gr\cA \xrightarrow{\sim} A$ is $G$-equivariant (with respect to the $G$-action on $\gr\cA$ inherited from $\cA$ and that on $A$). In this case, we get a Lie algebra homomorphism
\[\fg \to \der(\cA), \qquad \xi \mapsto \xi_{\cA}.\]

\begin{defi}\label{defn:hamiltonian}
Suppose $A$ is a graded Poisson algebra equipped with a Hamiltonian $G$-action. A ($G$-)\emph{Hamiltonian} quantization of $A$ is a triple $(\cA,\varsigma,\Phi)$ consisting of
 \begin{itemize}
     \item[(i)] a $G$-equivariant filtered quantization $(\cA,\varsigma)$ of $A$, and
     \item[(ii)] a $G$-equivariant map $\Phi: \fg \to \cA_{\leq d}$, called a \emph{quantum co-moment map}, such that
     \[[\Phi(\xi), a] = \xi_{\cA}(a), \qquad \forall\,\xi \in \fg, \quad a \in \cA.\]
     %
 \end{itemize}
 An isomorphism $(\cA_1,\varsigma_1,\Phi_1) \xrightarrow{\sim} (\cA_2,\varsigma_2,\Phi_2)$ of Hamiltonian quantizations of $A$ is a $G$-equivariant isomorphism of filtered algebras $\varphi: \cA_1 \to \cA_2$ such that $\varsigma_1 = \varsigma_2 \circ \gr(\varphi)$ and $\Phi_2 = \varphi \circ \Phi_1$. Denote the set of isomorphism classes of Hamiltonian quantizations of $A$ by $\mathrm{Quant}^G(A)$.
\end{defi}

In particular, a quantum co-moment map $\Phi$ extends a $G$-equivariant filtered homomorphism $\Phi: \Ug \to \cA$ of algebras, where we filter $\Ug$ so that the associated graded algebra $S\fg = \C[\fg^*]$ has $\g$ sitting in degree $d$.

\begin{Rem}
	It is easy to see that the map $\gr\Phi: \fg \to A_d$ defined by 
	\[
		\gr\Phi(\xi) = \varsigma(\Phi(\xi) + \cA_{\leqslant d-1})
	\]
	is a classical co-moment map for $G$. If $A$ is already equipped with a Hamiltonian $G$-action and a classical co-moment map $\phi: \g \to A_d$, then one usually assumes the compatibility condition $\gr\Phi = \phi$ by default. However, we will mainly consider the case where $A=\C[X]$ is a canonical symplectic singularity and $G$ is a semisimple group; in this case, both the classical co-moment map and the quantum co-moment map are unique, so that the compatibility condition is automatically satisfied. See \Cref{lem:Hamiltonian_quant}.
\end{Rem}

The next subsection recalls the formal sheaf-theoretic version of quantization, which is the form in which period maps and even quantizations are most convenient.

\subsection{Formal quantizations}
\label{subsec:formal_quant}

We review deformation quantization in the algebraic context from \cite{BeKaledin}. Let $X$ be a smooth algebraic variety (over $\C$) equipped with an algebraic symplectic form $\Omega \in H^0(X, \Omega_X^2)$. Then $\Omega$ induces a Poisson bracket $\{ - , -\}$ on the structure sheaf $\OO_X$ of $X$.

\begin{defi}\label{defn:formal_quan}
A \emph{formal deformation quantization over $\C\series$}, or simply a {\it quantization}, of a smooth symplectic variety $(X, \Omega)$ is the following datum:
\begin{itemize}
	\item
		a sheaf $\OO_\hb$ of flat associative $\C\series$-algebras on $X$, complete and separated in the $\hb$-adic topology, such that the commutator $[a, b] = a \cdot b - b \cdot a$ lies in $\hb \OO_\hb$ for any $a, b \in \OO_\hb$, and
	\item
		an isomorphism $\varsigma: \OO_\hb / \hb \OO_\hb \xrightarrow{\sim} \OO_X$ of sheaves of commutative Poisson algebras, where the Poisson bracket on $\OO_\hb / \hb \OO_\hb$ is given by $\hb^{-1}[a, b] + \hb \OO_\hb$, $a, b \in \OO_\hb$.
\end{itemize}
\end{defi}

There is an obvious definition of isomorphic quantizations. Denote by $\quan(X, \Omega)$ or simply $\quan(X)$ the set of equivalence classes of quantizations of $X$. 

We also need the notions of graded and even quantization from \cite[Section 2.2]{Losev:isom_quant}.

\begin{defi}\label{defn:graded_quan}
	A symplectic variety $(X, \Omega)$ is called \emph{graded} of weight $d$ if it is equipped with a $\cm$-action, so that the symplectic form $\Omega$ is of weight $d$ with respect to the induced $\cm$-action on smooth algebraic forms. This means that the induced Poisson bracket is of degree $-d$. 
	
	A \emph{graded (formal) quantization} of a graded symplectic variety $(X, \Omega)$ of weight $d$ is a formal deformation quantization $\OO_\hb$ of $(X, \Omega)$ equipped with a pro-rational $\cm$-action by automorphisms of sheaves of algebras, such that 
	\begin{itemize}
		\item[(i)] the isomorphism $\varsigma: \OO_\hb / \hb \OO_\hb \xrightarrow{\sim} \OO_X$ is $\cm$-equivariant, and
		\item[(ii)] $z \cdot \hb = z^d \hb$ for any $z \in \cm$.
	\end{itemize}
	We require isomorphisms between graded quantizations to be $\cm$-equivariant. Let $\quan(X)^{gr}$ denote the set of isomorphism classes of graded quantizations.
\end{defi}

\begin{defi}\label{defn:even_quant}
	A graded quantization $\OO_\hb$ is said to be \emph{even} if there is a $\cm$-equivariant involutive anti-automorphism $\epsilon: \OO_\hb \to \OO_\hb$ that is the identity modulo $\hb$ and sends $\hb$ to $-\hb$. The involution $\epsilon$ is called the {\it parity involution} and the pair $(\OO_\hb,\epsilon)$ is called an \emph{even quantization}. 
\end{defi}

The parity involution $\epsilon$, if it exists, is unique up to $\cm$-equivariant automorphisms of the quantization $\OO_\hb$. An even quantization carries an extra action of $\Z_2 = \Z/2\Z$, which is generated by $\epsilon$. 

We give one more terminology that will be useful in the following.

\begin{defi}\label{defn:strongly_adm}
	A smooth symplectic variety $(X, \Omega)$ is said to be \emph{strongly admissible} if it satisfies $H^1(X,\OO_X) = H^2(X,\OO_X) = 0$, where $H^i(X,\OO_X)$ is the $i$-th sheaf cohomology group of $\OO_X$.
\end{defi}

Now we are ready to state the classification result for graded quantizations (\cite{BeKaledin}, \cite[Corollary 2.3.3]{Losev:isom_quant}). 

\begin{Thm}\label{thm:per_graded}
	Let $(X, \Omega)$ be a graded strongly admissible smooth symplectic variety. Then there is a bijection, called the \emph{noncommutative period map}, 
	\[ \per: \quan(X)^{gr} \xrightarrow{\sim} H^2(X,\C). \]
	Moreover, a graded quantization $\OO_\hb$ is even if and only if $\per(\OO_\hb) = 0$.
\end{Thm}

\subsection{Equivariant graded quantizations} \label{subsec:equiv_quant}

Let $(X, \Omega)$ be a graded smooth symplectic variety. Let $G$ be an algebraic group acting on $X$ via $\beta: G \times X \to X$ by symplectic automorphisms, such that the $G$-action commutes with the $\cm$-action. 
By a {\it $G$-equivariant graded quantization} of $X$, we mean a graded quantization $\OO_\hb$ of $X$ equipped with a $G$-action $\beta_\hb$ on $\OO_\hb$ by automorphisms of sheaves of $\C\series$-algebras that lifts the $G$-action on $X$, such that 
\begin{enumerate}[label=(\roman*)]
	\item
	it commutes with the $\cm$-action on $\OO_\hb$, and
	\item
	the composition $\OO_\hb \twoheadrightarrow \OO_\hb / \hb \OO_\hb \xrightarrow{\varsigma} \OO_X$ is $G$-equivariant.
\end{enumerate}

We consider the problem of lifting the $G$-action $\beta$ to an action $\beta_\hb$ on $\OO_\hb$. We will be mainly interested in the case of an even quantization.

\begin{defi}[{\cite[Definition 3.3.4]{LY}}]\label{defn:symm_beta}
	We say that a $G$-action $\beta_\hb$ on an even quantization $(\OO_\hb,\epsilon)$ of $(X,\Omega)$ is \emph{symmetrized} if it commutes with the $\cm \times \Z_2$-action on $\OO_\hb$.
\end{defi}

\begin{Prop}[{\cite[Corollary 3.3.2]{LY}}]
	Suppose the graded symplectic variety $X$ is strongly admissible and $(\OO_\hb, \epsilon)$ is an even quantization of $X$. Then the $G$-action $\beta$ lifts uniquely to a symmetrized $G$-action $\beta_\hb$ on $\OO_\hb$.
\end{Prop}

Now suppose $(X,\Omega)$ is equipped with a Hamiltonian action of an affine algebraic group $G$ that commutes with the $\cm$-action, and a classical moment map $\mu: X \to \fg^*$. Then $\mu$ corresponds to a classical co-moment map $\phi: \fg \to \C[X]=\Gamma(X, \OO_X)$ for the associated graded Poisson algebra $\C[X]$. By definition, $\mu$ and $\phi$ are $G$-equivariant. We assume further that $\mu$ is $\cm$-equivariant with respect to the $\cm$-action on $\fg^*$ of weight $d$. This is equivalent to saying that the image of $\phi$ lies in the degree $d$ part of $\C[X]$. We call the datum $(X, \Omega, \beta, \phi)$ a {\it graded (classical) Hamiltonian $G$-space}. 
There is a quantum version of a Hamiltonian action.

\begin{defi}\label{defn:Hamilt_action_quantization}
	A {\it $G$-equivariant graded Hamiltonian quantization} of $(X, \Omega)$ consists of
	\begin{itemize}
		\item 
			a $G$-equivariant graded quantization $\OO_\hb$ of $(X, \Omega)$ with $G$-action $\beta_\hb$, and
		\item
			a $G \times \cm$-equivariant $\C$-linear map $\Phi_\hb: \g \to \Gamma(X, \OO_\hb)$, i.e., its image lies in the degree $d$ part of $\Gamma(X, \OO_\hb)$,
	\end{itemize}
	such that the Lie algebra $\fg$-action on $\OO_\hb$ given by the adjoint action $\hb^{-1} [\Phi_\hb(\xi), \cdot]$ for any $\xi \in \fg$ agrees with the infinitesimal action of $\xi$ on $\OO_\hb$ induced by the $G$-equivariant structure. The map $\Phi_\hb$ is called a {\it quantum co-moment map} and the $G$-action is called a \emph{(graded) quantum Hamiltonian $G$-action}.
\end{defi}

For any $G$-equivariant graded Hamiltonian quantization of $(X, \Omega)$, let $\phi$ be the composition 
\[ \g \xrightarrow{\Phi_\hb} \Gamma(X, \OO_\hb) \to \Gamma(X, \OO_\hb / \hb \OO_\hb) \simeq \C[X]. 
\]
Then $\phi$ corresponds to a classical moment map $\mu: X \to \g^*$. We say that $\Phi_\hb$ lifts the classical co-moment map $\phi$ or the moment map $\mu$, and $(\OO_\hb, \beta_\hb, \Phi_\hb)$ is a ($G$-equivariant) graded Hamiltonian quantization of the Hamiltonian $G$-space $(X, \Omega, \beta, \mu)$. 

Conversely, given a Hamiltonian $G$-space $(X, \Omega, \beta, \phi)$, we want to construct a graded Hamiltonian quantization of it. We will only consider the case when the quantization is even. For general quantizations with $G$ restricted to be a reductive group, see \cite[Proposition 3.11]{BPW}.

\begin{defi}[{\cite[Definition 3.4.2]{LY}}]\label{defn:quan_moment_symm}
	Assume $(X,\Omega,\beta,\phi)$ is a graded Hamiltonian $G$-space and is equipped with a quantum Hamiltonian quantization $(\OO_\hb, \beta_\hb, \Phi_\hb)$. Suppose $\OO_\hb$ is an even quantization with a parity involution $\epsilon$ and the quantum $G$-action $\beta_\hb$ is symmetrized in the sense of \Cref{defn:symm_beta}. We say that the quantum co-moment map $\Phi_\hb$ is \emph{symmetrized} if $\epsilon(\Phi_\hb(\xi)) = \Phi_\hb(\xi)$ for all $\xi \in \fg$. In this case, we say that the datum $(\OO_\hb, \epsilon, \beta_\hb, \Phi_\hb)$ is a \emph{symmetrized (even) $G$-Hamiltonian quantization} and $(\beta_\hb,\Phi_\hb)$ is a \emph{symmetrized quantum Hamiltonian $G$-action}.
\end{defi}

\begin{Prop}[{\cite[Proposition 3.4.4]{LY}}]\label{prop:quan_moment_symm}
	Let $(X,\Omega, \beta, \phi)$ be a graded Hamiltonian $G$-space such that $X$ is strongly admissible. Then any even quantization $(\OO_\hb, \epsilon)$ of $(X,\Omega)$ admits a unique symmetrized quantum Hamiltonian $G$-action $(\beta_\hb, \Phi_\hb)$ that lifts $(\beta,\phi)$, so that $(\OO_\hb, \epsilon, \beta_\hb, \Phi_\hb)$ is a symmetrized $G$-Hamiltonian quantization of $(X,\Omega, \beta, \phi)$.
\end{Prop}

\subsection{Quantization of conical symplectic singularities} \label{subsec:quant_conical_symp_sing}

The notion of a \emph{symplectic singularity} was introduced in \cite{Beauville1}. A \emph{conical symplectic singularity} is a normal graded Poisson variety $X$ with symplectic singularities such that the $\cm$-action on $X$ contracts it to a point, which is automatically affine.
We will briefly recall the classification results of quantizations of conical symplectic singularities. For details, see \cite{Losev:quan_symp_orbit, LMBM}.

To obtain filtered quantizations of $A=\C[X]$ for a conical symplectic singularity $X$, we need to take a certain partial resolution of $X$ and take graded formal quantizations of the smooth locus of the resulting variety. Recall that a normal variety is $\bQ$-factorial if every Weil divisor has a
(nonzero) integer multiple that is Cartier. The following was proven by Losev in \cite{Losev:SRA}, based on \cite{BCHM}.

\begin{Prop} \label{prop:terminalization}
	Let $X$ be a Poisson variety with symplectic singularities. Then there is a birational projective morphism $\rho: \widetilde{X} \to X$ such that 
	\begin{itemize}
		\item[(i)] $\widetilde{X}$ is an irreducible normal Poisson variety (hence has symplectic singularities).
		\item[(ii)] $\widetilde{X}$ is $\bQ$-factorial.
		\item[(iii)] $\widetilde{X}$ has terminal singularities.
	\end{itemize}
\end{Prop}

\begin{Rem}
	Modulo condition (i), (iii) is equivalent to saying that the singular locus of $\widetilde{X}$ is of codimension $\geqslant 4$, see \cite[Main Theorem]{Namikawa1}.
\end{Rem}

The morphism $\rho: \widetilde{X} \to X$, or the variety $\widetilde{X}$, in \Cref{prop:terminalization} is called a {\it $\bQ$-factorial terminalization} of $X$. If $X$ is a conical symplectic singularity, then $\widetilde{X}$ also admits a $\cm$-action such that $\rho$ is $\cm$-equivariant (\cite[A.7]{Namikawa2}).

The following important properties of a $\bQ$-factorial terminalization, proven by \cite[Lemma 12]{Namikawa2} (see also \cite[Proposition 2.5]{Losev:quan_symp_orbit}), will be crucial in the construction of quantizations of $X$.

\begin{Prop} \label{prop:H_vanishing}
	Let $X$ be a conical symplectic singularity and $\rho: \widetilde{X} \to X$ be a $\bQ$-factorial terminalization. Then $\rho$ and the inclusion $\widetilde{X}^{reg} \hookrightarrow \widetilde{X}$ induce canonical isomorphisms of graded Poisson algebras
	\[ \C[X] \xrightarrow{\sim} \C[\widetilde{X}] \xrightarrow{\sim} \C[\widetilde{X}^{reg}] \]
	Moreover, $H^i(\widetilde{X}^{reg},\OO_{\widetilde{X}^{reg}}) = 0$ for $i = 1, 2$ and $H^i(\widetilde{X},\OO_{\widetilde{X}}) = 0$ for all $i > 0$.
\end{Prop}

Next we recall the parameter space for filtered quantizaitons of $A=\C[X]$.
Let $\rho: \widetilde{X} \to X$ be a $\bQ$-factorial terminalization. The \emph{Namikawa space} associated to $X$ is the finite-dimensional complex vector space
\[\fP = \fP^X := H^2(\widetilde{X}^{reg},\C). \]
There is a finite Coxeter group $W=W^X$, called the \emph{Namikawa Weyl group}, which acts linearly on $\fP^X$. The definition of the Namikawa space and Weyl group is independent of the choice of the $\bQ$-factorial terminalization $\rho$.

By \Cref{prop:H_vanishing} and \Cref{thm:per_graded}, graded formal quantizations of $\widetilde{X}^{reg}$ exist and we have a bijection 
\[ \per: \quan(\widetilde{X}^{reg})^{gr} \xrightarrow{\sim} H^2(\widetilde{X}^{reg},\C) = \fP^X. \]
Let $\OO_{\lambda, \hb}$ be a graded formal quantization of $\widetilde{X}^{reg}$ such that $\per(\OO_{\lambda, \hb}) = \lambda \in \fP^X$. Set $\cA_{\lambda, \hb} := \Gamma(\widetilde{X}^{reg},\OO_{\lambda,\hb})$, which is a flat $\C\series$-algebra, completed and separated in the $\hb$-adic topology, equipped with a pro-rational $\cm$-action. Let $\cA_{\lambda,\hb}^{fin} \subset \cA_{\lambda,\hb}$ be the $\cm$-finite sections of $\cA_{\lambda,\hb}$. Then the $\C$-algebra $\cA_{\lambda} := \cA_{\lambda,\hb}^{fin}/(\hb-1)$ comes with a natural filtration and an injective homomorphism $\varsigma: \gr\cA_{\lambda} \hookrightarrow \C[\widetilde{X}^{reg}] \simeq \C[X]$ of Poisson algebras. It turns out that $\cA_{\lambda}$ is a filtered quantization of $\C[X]$ and all filtered quantizations of $\C[X]$ are of this form by the following result of Losev \cite[Proposition 3.3, Theorem 3.4]{Losev:quan_symp_orbit}, which relies crucially on \Cref{prop:H_vanishing} (cf. \cite[Proposition 3.4]{BPW}).

\begin{Thm} \label{thm:quant_conical_symp_sing}
	With notation as above, the following hold:
	\begin{enumerate}[label=(\roman*)]
		\item $\varsigma: \gr \cA_{\lambda} \hookrightarrow \C[X]$ is an isomorphism of Poisson algebras, so that $(\cA_{\lambda}, \varsigma)$ is a filtered quantization of $\C[X]$.
		\item Any filtered quantization of $\C[X]$ is isomorphic to $\cA_{\lambda}$ for some $\lambda \in \fP^X$.
		\item $(\cA_{\lambda}, \varsigma)$ and $(\cA_{\lambda'}, \varsigma')$ are isomorphic as filtered quantizations of $\C[X]$ if and only if $\lambda' \in W^X \cdot \lambda$.
	\end{enumerate}
Thus there is a bijection
\[ \fP^X/W^X \simeq \mathrm{Quant}(\C[X]), \qquad W^X \cdot \lambda \mapsto \cA_{\lambda}. \]
\end{Thm}

For Hamiltonian actions of semisimple groups, the equivariant structure and the quantum co-moment map are uniquely determined.

The following lemma is a special case of \cite[Lemma 4.11.2]{LMBM}.

\begin{Lem} \label{lem:Hamiltonian_quant}
	Let $G$ be a connected semisimple algebraic group and suppose $A=\C[X]$ admits a Hamiltonian $G$-action. Then the following are true:
	\begin{itemize}
		\item[(i)] There is a unique classical co-moment map $\phi: \fg \to A_d$.
		\item[(ii)] Every filtered quantization $\cA \in \mathrm{Quant}(A)$ has a unique $G$-equivariant structure.
		\item[(iii)] For every $\cA \in \mathrm{Quant}(A)$, there is a unique quantum co-moment map $\Phi: \fg \to \cA_{\leq d}$.
	\end{itemize}
	In particular, there is a canonical bijection
	\[ \fP^X/W^X \xrightarrow{\sim} \mathrm{Quant}^G(A), \qquad W^X \cdot \lambda \mapsto (\cA_{\lambda}, \Phi).\]
\end{Lem}

For our later application, we will only consider the following special case of quantizations.

\begin{defi}[{\cite[Definition 5.0.1]{LMBM}}]\label{defn:can_quant}
	The \emph{canonical quantization} of $\C[X]$ is the Hamiltonian quantization $\cA_0$ corresponding to the parameter $0 \in \fP^X$. 
\end{defi}

In the setting of \cref{lem:Hamiltonian_quant}, we have a classical co-moment map $\phi: \fg \to A_d = \C[X]_d = \C[\widetilde{X}^{reg}]_d$ for the Hamiltonian $G$-action on $\widetilde{X}^{reg}$ restricted from that on $\widetilde{X}$, by \Cref{lem:Hamiltonian_quant} (i) and \Cref{prop:H_vanishing}. By \Cref{prop:quan_moment_symm}, we have a symmetrized even $G$-Hamiltonian quantization $(\OO_\hb, \epsilon, \beta_\hb, \Phi_\hb)$ of $X$ such that the quantum co-moment map $\Phi_\hb: \g \to \Gamma(X, \OO_\hb)$ lifts $\phi$. By \Cref{thm:quant_conical_symp_sing}, $\cA_0 = \Gamma(X, \OO_\hb)^{fin}/(\hb-1)$ is the canonical filtered quantization of $\C[X]$. Moreover, $\Phi_\hb$ descends to a quantum co-moment map $\Phi: \g \to \cA_0$. This $\Phi$ coincides with the one in \Cref{lem:Hamiltonian_quant} (iii). 

\subsection{Quantization of nilpotent $G$-orbit covers}
\label{subsec:quant_G_orbit_covers}

We will be mainly interested in the following example. Let $G$ be an algebraic semisimple group over $\C$ and let $\rho: \widetilde{\orb} \to \orb$ be a connected finite $G$-equivariant covering of a nilpotent orbit $\orb$ in $\g^*$. Let $A=\C[\widetilde{\orb}]$ as in \Cref{ex:orbit_cover} and $X = \spec A$, the affinization of $\widetilde{\orb}$. The covering morphism $\rho: \widetilde{\orb} \to \orb$ induces a surjective morphism $\rho: X \to \overline{\orb}$. The affinization $X$ inherits a natural Poisson structure from $\widetilde{\orb}$. Together with the $\cm$-action on $X$ induced by that on $\widetilde{\orb}$, it makes $X$ a conical symplectic singularity (see \cite[Lemma 2.5]{Losev:HC_bimodules} and \cite[\S P.3]{Namikawa:covers}). 

Let $\Gamma = \Gamma_{\rho}$ be the Galois group of the covering $\rho: \widetilde{\orb} \to \orb$. Then we have a $\Gamma$-action on $X$, which commutes with the $G \times \cm$-action and preserves the Poisson structure. Note that $\Gamma$ is a finite group and hence disconnected in general, so \Cref{lem:Hamiltonian_quant} cannot be applied directly. Nevertheless, by \cite[Proposition 5.1.1]{LMBM}, the $\Gamma$-action on $X$ lifts to a $\Gamma$-action on the canonical quantization $\cA_0$ of $\C[X]$ by automorphisms of filtered algebras, which also commutes with the $G$-action. Moreover, the image of the quantum co-moment map $\Phi: \g \to \cA_0$ lies in the $\Gamma$-invariant subalgebra $\cA_0^\Gamma$ of $\cA_0$.

We regard the quantum co-moment map $\Phi: \g \to \cA_0$ as a homomorphism $\Phi: \Ug \to \cA_0$ of algebras. Let $I_0(\widetilde{\orb}) \subset \Ug$ be the kernel of $\Phi$, which is called the \emph{unipotent ideal} associated to $\widetilde{\orb}$ in \cite{LMBM}. 
We record the following proposition.

\begin{Prop} \label{prop:unip_ideal_prim}
	$I_0(\widetilde{\orb})$ is a completely prime maximal primitive ideal.
\end{Prop}

\begin{proof}
	By Lemma 6.1.1 and Proposition 6.1.2 of \cite{LMBM}, $I_0(\widetilde{\orb})$ is a completely prime primitive ideal of $\Ug$. By \cite[Theorem 5.0.1]{MBM}, $I_0(\widetilde{\orb})$ is a maximal primitive ideal.
\end{proof}

Since $I_0(\widetilde{\orb})$ is a primitive ideal, it has an infinitesimal character, denoted as $\gamma_0(\widetilde{\orb})$.

In general, different covers of the same $\orb$ might give rise to the same unipotent ideal. We recall the classification result of all unipotent ideals from \cite{LMBM}. Let $\widetilde{\orb}$ and $\widehat{\orb}$ be $G$-equivariant covers of a common nilpotent orbit $\orb \subset \fg^*$. Set $\widetilde{X} = \spec(\C[\widetilde{\orb}])$, $\widehat{X} = \spec(\C[\widehat{\orb}])$.
Both $\widetilde{X}$ and $\widehat{X}$ contain finitely many $G$-orbits and $\widetilde{\orb}$ (resp. $\widehat{\orb}$) is the unique open orbit in $\widetilde{X}$ (resp. $\widehat{X}$). Moreover, every $G$-equivariant morphism $\widetilde{\orb} \to \widehat{\orb}$ extends uniquely to a finite $G$-equivariant morphism $\widetilde{X} \to \widehat{X}$, which is \'{e}tale over the open orbit $\widehat{\orb} \subset \widehat{X}$. We recall \cite[Definition 1.5.2]{LMBM} below.

\begin{defi} \label{defi:almost_etale}
A finite $G$-equivariant morphism $\widetilde{X} \rightarrow \widehat{X}$ is said to be \emph{almost \'{e}tale} if it is \'{e}tale over the open subset in $\widehat{X}$ obtained by removing all $G$-orbits of codimension $\geqslant 4$. 
\end{defi}

The following proposition is one part of \cite[Proposition 6.2.8]{LMBM}.

\begin{Prop} \label{prop:almost_etale_unip_ideal}
    Let $\widetilde{\orb} \to \widehat{\orb}$ be a $G$-equivariant morphism of covers of $\orb$ such that the induced finite morphism $\widetilde{X} \to \widehat{X}$ is almost \'{e}tale. Then $I_0(\widetilde{\orb}) = I_0(\widehat{\orb})$.
\end{Prop}

\begin{Rem} \label{rem:equiv_covers}
	In fact, \cite[Proposition 6.2.8]{LMBM} says that $I_0(\widetilde{\orb}) = I_0(\widehat{\orb})$ if and only if the two covers are equivalent in the sense of \cite[Definition 6.25]{LMBM}: define a binary relation $\ge$ on the set of all $G$-equivariant covers of $\orb$ such that $\widetilde{\orb} \ge \widehat{\orb}$ if and only if there is a $G$-equivariant morphism $\widetilde{\orb} \to \widehat{\orb}$ whose induced morphism $\widetilde{X} \rightarrow \widehat{X}$ is almost \'{e}tale. The symmetric closure of this relation gives an equivalence relation on the same set. The `if' direction is \cref{prop:almost_etale_unip_ideal}, while the `only if' direction is much harder and we do not need it for our application.
\end{Rem}

Now recall the following definition from \cite[Definition 7.29]{LMBM}.

\begin{defi}
    Let $\orb$ be a nilpotent orbit and $\widetilde{\orb}$ be a cover of $\orb$. We say that $\widetilde{\orb}$ is $2$-leafless if $\spec(\C[\widetilde{\orb}])$ has no codimension $2$ leaves.
\end{defi}

\begin{Rem} \label{rem:2_leafless_covers}
	If $\widehat{\orb}$ is a $2$-leafless cover of $\orb$, then every (connected) cover $\widetilde{\orb}$ of $\orb$ with a morphism $\widetilde{\orb} \to \widehat{\orb}$ is also $2$-leafless and hence the induced morphism $\spec(\C[\widetilde{\orb}]) \to \spec(\C[\widehat{\orb}])$ is almost \'{e}tale in the sense of \cref{defi:almost_etale}. By \cref{prop:almost_etale_unip_ideal}, this implies that $I_0(\widetilde{\orb}) = I_0(\widehat{\orb})$. In particular, this is true for the universal cover $\widetilde{\orb}$ of $\orb$.
\end{Rem}

Among all the $2$-leafless covers, there is an important subclass consisting of the so-called \emph{birationally rigid} $G$-equivariant covers of nilpotent orbits, a notion introduced in \cite[Section 2.4]{LMBM}. To compute the infinitesimal character $\gamma_0(\widetilde{\orb})$ for general covers $\widetilde{\orb}$, it suffices to compute it for birationally rigid covers by \cite[Proposition 8.1.1]{LMBM}.
We will not recall the original definition of birationally rigid covers and the relevant notion of {\it birational induction} of nilpotent orbit covers (\cite[Section 2.4]{LMBM}) here, but only need the following equivalent descriptions, which follow from Sections 4.4, 4.5, and 7.2 of \cite{LMBM}. The case of trivial covers of orbits was first studied in \cite[Proposition 4.5]{Losev:quan_symp_orbit}.

\begin{Thm} \label{thm:birigid}
	Let $X = \spec(\C[\widetilde{\orb}])$. The following conditions are equivalent:
	\begin{enumerate}
		\item $\widetilde{\orb}$ is birationally rigid.
		\item $X$ has terminal singularities, i.e., it is the $\bQ$-factorial terminalization of itself.
		\item $H^2(\widetilde{\orb}, \C) = 0$ and $\widetilde{\orb}$ is $2$-leafless, i.e.,
		\begin{equation} \label{eq:codim_4}
			\codim(X^{sing}, X) \geqslant 4,
		\end{equation}
		where $X^{sing}$ stands for the singular locus of $X$.
		\item The Namikawa space $\fP^X = 0$, i.e., the graded Poisson algebra $A = \C[X]$ has a unique filtered quantization (which is the canonical quantization).
	\end{enumerate}
\end{Thm}

The following proposition is a slight generalization of \cite[Proposition 9.1.1]{LMBM}, which relates special unipotent infinitesimal characters to $\gamma_0(\widetilde{\orb})$ for certain covers.

\begin{Prop}\label{prop:dual_dist}
	Let $\ckorb$ be a distinguished nilpotent orbit in $\ckfg^*$. Let $\orb = d(\ckorb)$ and $\widetilde{\orb}$ be the universal cover of $\orb$. Then 
	\begin{enumerate}[label=(\roman*), leftmargin=*]
		\item $\widetilde{\orb}$ is birationally rigid;
		\item $\gamma_0(\widetilde{\orb}) = \lambda_{\ckorb} = \frac{1}{2}\check{h}_{\ckorb}$.
	\end{enumerate}
\end{Prop}

\begin{proof}
	When $d$ is the usual BV duality map, this is exactly \cite[Proposition 9.1.1]{LMBM}. This is also true for the metaplectic case by similar arguments, whose details are left to the reader. For part (ii) in the metaplectic case, see \cref{prop:dual_quasi-dist} for a generalization.
\end{proof}

\begin{Rem} \label{rem:BVLS_refined}
	\Cref{prop:dual_dist} is a special and fundamental case of a refined version of the Barbasch-Vogan-Lusztig-Spaltenstein (BVLS) duality introduced in \cite{LMBM} and later extended in \cite{MBMY}. Namely, there is a map $\widetilde{\mathsf{D}}$ that enhances the duality map $d$ in the sense that, for any nilpotent orbit $\ckorb$ of $\ckfg^*$, $\widetilde{\mathsf{D}}(\ckorb)$ is an equivalence class of connected finite $G_{ad}$-equivariant covers of the BV dual orbit $\orb=d(\ckorb)$ of $\ckorb$ in the sense of \cref{rem:equiv_covers}. \cite[Proposition 9.2]{LMBM} says that the special unipotent ideal $\J_{\ckorb}$ is equal to the maximal primitive ideal $I_0(\widetilde{\mathsf{D}}(\ckorb))$ associated to (any member of) $\widetilde{\mathsf{D}}(\ckorb)$.
	When $\ckorb$ is distinguished, $\widetilde{\mathsf{D}}(\ckorb)$ contains the universal cover of $\orb$ as a member. 

	Inspired by the Sommers duality map (\cite{Sommers:Duality}) and the Achar duality map (\cite{Achar:duality}), the author and his collaborators later generalized both the refined BVLS duality map and Sommers duality map in \cite{MBMY}. To any pair $(\ckorb, \bar{C})$ of a nilpotent orbit $\ckorb$ of $\ckfg^*$ and a conjugacy class $\bar{C}$ of Lusztig's canonical quotient group $\Abar(\orb)$ of $\orb$, \cite{MBMY} assigns an actual $G$-equivariant cover $D(\ckorb, \bar{C})$ of the Sommers dual orbit $d_S(\ckorb, \bar{C})$ in $\fg^*$, where $G$ is the adjoint group of $\fg$. When the conjugacy class $\bar{C}$ is trivial, $D(\ckorb, 1)$ is a cover in the equivalence class $\widetilde{\mathsf{D}}(\ckorb)$, and when $\ckorb$ is distinguished, $D(\ckorb, 1)$ is exactly the universal cover of $\orb$.

	We also remark that the two refinements of the BVLS duality above can be easily extended to refine the metaplectic BV duality. Note that is this case, all dual covers are equivariant for the complex symplectic groups $\Sp(2n,\C)$, but not necessarily for the adjoint groups. In the classical and metaplectic cases, the refined BVLS duality map has also been realized in \cite{FHN} as taking the Coulomb branches of certain orthosymplectic quiver gauge theories in the framework of 3d mirror symmetry.  
\end{Rem}

\Cref{prop:dual_dist} (i) implies that the cover $\widetilde{\orb}$ is $2$-leafless, i.e., condition \eqref{eq:codim_4} holds. Recall that our \Cref{thm:codim_X} has strengthened the condition \eqref{eq:codim_4}. \Cref{thm:birigid} and \Cref{prop:dual_dist} (i) imply that, when $\ckorb$ is distinguished, there is a unique filtered quantization of $\C[X]$, which is the canonical quantization $\cA_0$. 

We now extend \cref{prop:dual_dist} to the case of quasi-distinguished nilpotent orbits.

\begin{Prop}\label{prop:dual_quasi-dist}
    Let $\ckorb$ be a quasi-distinguished nilpotent orbit in $\ckfg^*$. Let $\orb = d(\ckorb)$ and $\widetilde{\orb}$ be the universal cover of $\orb$. Then 
	\begin{enumerate}[label=(\roman*), leftmargin=*]
			\item $\widetilde{\orb}$ is $2$-leafless.
		\item $\gamma_0(\widetilde{\orb}) = \lambda_{\ckorb} = \frac{1}{2}\check{h}_{\ckorb}$.
		\item $I_0(\widetilde{\orb}) = \J_{\ckorb}$.
	\end{enumerate}
\end{Prop}

\begin{proof}
	Part (i) is \cref{thm:codim_X}. For part (ii), the argument is the same as that of \cite[Proposition 9.1.1]{LMBM}, which we briefly recall here. 
	\vskip 1em
	\noindent\emph{Classical and metaplectic case.}
	We can use \cite[Proposition 8.1.3]{LMBM} to compute $\gamma_0(\widetilde{\orb})$. Note that the statement of \cite[Proposition 8.1.3]{LMBM} is only about birationally rigid nilpotent covers equivariant with respect to the linear classical groups $\SO(N)$ and $\Sp(2n)$, but it also works for all $2$-leafless covers without any change of the proof (which is the case for us by part (i)). More precisely, let $X \in \{B, C, D, \mC\}$ be the type of the Lie algebra $\g$ and $(\epsilon, \epsilon') = (\epsilon_X, \epsilon'_X)$ be as defined in \Cref{subsec:special_duality_classical}. Assume $\fg = \fg_{\epsilon}(N)$, so that $\ckfg = \fg_{\epsilon'}(N+\epsilon - \epsilon')$.
	Let $\eta$ be the partition associated to $\ckorb$ in $\ckfg^*$, and $\lambda = d(\eta)$ be the partition associated to the orbit $\orb = d(\ckorb)$ in $\g^*$. Then by \cref{lem:d_quasi-dis}, we have $\lambda = (\eta^\dagger)^*$. The formulas for $\gamma_0(\widetilde{\orb})$ in \cite[Proposition 8.1.3]{LMBM} for all types can be reformulated into the following uniform one:
	\begin{equation}
		\gamma_0(\widetilde{\orb}) = \rho^+((\lambda^*)_\dagger)
	\end{equation}
	where $\rho^+$ is defined as in \cite[Definition 8.6]{LMBM}, and the partition $(\lambda^*)_\dagger = \chi_\dagger$ is obtained from $\chi:=\lambda^* = \eta^\dagger = [\chi_1 \geqslant \chi_2 \geqslant \cdots \geqslant \chi_p \geqslant 0]$ by setting
	\begin{equation} \label{eq:lower_dagger}
		(\chi_\dagger)_i = \chi_i + (-1)^{i+\epsilon}, \quad 1 \leqslant i \leqslant p,
	\end{equation}
	where we again assume that $p$ is the least integer satisfying $p \geqslant \#\chi$ and $p \equiv \epsilon + \epsilon' \bmod 2$, by adding a $0$ to $\chi$ if necessary. To see this, we only need to note that, by property (2) of \cref{defi:quasi-dist_classical} for all parts $a$ of $\chi = \eta^\dagger$, $m(a) \leqslant 2$ and if $m(a) = 2$, then $h(a) \equiv \epsilon \bmod 2$.
	From this, it is easy to see that \eqref{eq:lower_dagger} indeed defines a non-increasing sequence of nonnegative integers and hence a well-defined partition $\chi_\dagger \in \cP_{\epsilon',\,\epsilon}(N + \epsilon - \epsilon')$ and $\chi_\dagger = \eta$. Now $\gamma_0(\widetilde{\orb}) = \rho^+((\lambda^*)_\dagger) = \lambda_{\ckorb} = \frac{1}{2}\check{h}_{\ckorb}$ follows from the explicit formulas for $\slf_2$-triples in \cite[Section 5.2]{CM}.
	\vskip 1em

	\noindent\emph{Exceptional case.}
	Most of the quasi-distinguished nilpotent orbits are distinguished, a case that has already been covered by \cref{prop:dual_dist}. The remaining cases are listed in \cref{tab:qD_exceptional}. A case-by-case analysis shows that $\codim(\partial \orb, \overline{\orb}) \geqslant 4$ for all $\orb$ in these cases, with only one exception when $\g = \mathfrak{e}_6$ and $\ckorb = \orb = D_4(a_1)$. When $\codim(\partial \orb, \overline{\orb}) \geqslant 4$, $I_0(\widetilde{\orb}) = I_0(\orb)$ for any connected cover $\widetilde{\orb}$ of $\orb$ by \cref{rem:2_leafless_covers}.
	For the exceptional case, \cite[Proposition 3.9.5]{MBM} says that $\orb = D_4(a_1)$ is a {\it birationally semi-rigid} orbit (\cite[Definition 3.9.1]{MBM}), meaning that it is not birationally rigid but admits a birationally rigid cover. There is more than one birationally rigid cover in this case, but the unipotent ideals associated to these covers are all equal to $I_0(\widetilde{\orb})$ with $\widetilde{\orb}$ being the universal cover of $\orb$ (which itself, however, is not birationally rigid), again by \cref{rem:2_leafless_covers}. By \cite[Table 11]{MBM}, the infinitesimal character associated to the unipotent ideal of any birationally rigid cover of $\orb = D_4(a_1)$, and hence $\gamma_0(\widetilde{\orb})$, is indeed $\lambda_{\ckorb} = \frac{1}{2}\check{h}_{\ckorb}$.

	For part (iii), we follow the same argument as in the proof of \cite[Proposition 9.2, (iv)]{LMBM}.
\end{proof}

We conclude this subsection by proving the following proposition about special unipotent ideals, which will be used later in the proof of \Cref{lem:AC_special_unip}. We define an anti-involution $\tau_\g: \Ug \to \Ug$ determined by $\tau_\g(\xi) = -\xi$ for all $\xi \in \fg$.

\begin{Prop} \label{prop:J_tau_inv}
	For any $\ckorb \in \check{\cN}_o$ (not necessarily quasi-distinguished), the special unipotent ideal $\J_{\ckorb}$ is invariant under the anti-involution $\tau_\g$ of $\Ug$.
\end{Prop}

\begin{proof}
	Observe that for any primitive ideal $I \subset \Ug$ with infinitesimal character $\chi_\lambda$, $\tau_\g(I)$ is a primitive ideal with infinitesimal character $\chi_{-\lambda}$. Therefore $\tau_\g(\J_{\ckorb})$ is the maximal primitive ideal with infinitesimal character $\chi_{-\lambda_{\ckorb}}$, where $\lambda_{\ckorb} = \check{h}/2$ and $\check{h}$ is the semisimple element in an $\slf_2$-triple $(\check{e}, \check{h}, \check{f})$ for $\ckorb$. Note that $(\check{f}, -\check{h}, \check{e})$ is also an $\slf_2$-triple for $\ckorb$. Therefore $-\check{h}$ is conjugate to $\check{h}$ under the adjoint action of $G$ or the Weyl group $W$ of $\g$, which implies that $\chi_{-\lambda_{\ckorb}} = \chi_{\lambda_{\ckorb}}$ and hence $\tau_\g(\J_{\ckorb}) = \J_{\ckorb}$.
\end{proof}

\subsection{Associated cycles of special unipotent representations} \label{subsec:AC_special_unip}

Before we continue to \cref{sec:orbit_method_real} for the construction of special unipotent representations of real groups by quantization, we will see below that the quantization results for complex groups so far have already imposed strong constraints on the associated cylces of these representations.

We first extend the definition of special unipotent representations to the metaplectic case. Let $\GR = \Mp(2n,\R)$ be the metaplectic group and $\ckfg = \fg = \fmp(2n,\C)$ be of metaplectic type. We can define (metaplectic) special unipotent representations $\UnipOv(\g, K)$ attached to nilpotent orbits $\ckorb$ of $\ckfg^*$ as in \eqref{eq:UnipOv_defn}, with $\lambda_{\ckorb}$ and $\cJ_{\ckorb}$ defined similarly (see \cref{subsec:special_duality}). 

Using the notation in \Cref{subsec:AC}, we have the following description of the associated cycles of special unipotent representations. Note that in this subsection, $\ckorb$ is assumed to be quasi-distinguished, unless otherwise specified.
 
\begin{Lem}\label{lem:AC_special_unip}
	Let $\g$ be simple and $\ckorb \in \check{\cN}_o$. If $\g$ is classical/metaplectic or $\ckorb$ is quasi-distinguished, then the image of the associated cycle map
		\[ \AC: \UnipOv(\g, K) \xrightarrow{\ \ } \AC(\g,K) \]
	lies in $\cK_{>0}\cAOD(\orb, K)$.
\end{Lem}

\begin{proof}
	We apply \cite[Theorem 8.7]{Vogan_AV}. Condition (ii) there follows from \Cref{prop:J_tau_inv}. For condition (i) there, we need to check that the multiplicity $m_{\overline{\orb}}(\J_{\ckorb})$ of $\J_{\ckorb}$ along $\overline{\orb}$ is $1$. When $\ckorb$ is quasi-distinguished, $\J_{\ckorb} = I_0(\widetilde{\orb})$ by \Cref{prop:dual_quasi-dist} (iii), where $\widetilde{\orb}$ is the universal cover of $\orb$. When $\g$ is classical or metaplectic, any equivalence class $\widetilde{\mathsf{D}}(\ckorb)$ as in \cref{rem:BVLS_refined} always contains a member that is equivariant with respect to the linear classical group $G$ with Lie algebra $\g$, and the $G$-equivariant fundamental group $A(\orb)$ is always abelian for classical $\fg$. In both cases, all the relevant covers $\widetilde{\orb}$ are Galois and the conditions in \cite[Proposition 6.4 (iii)]{LMBM} are satisfied, hence $m_{\overline{\orb}}(\J_{\ckorb})=1$.
\end{proof}

Now we consider the situation in the last part of \cref{subsec:basic_settings}: suppose that $\phi: \GR \to \bar{G}_\R$ is a finite surjective homomorphism, which induces a morphism $\phi: (\g,K) \to (\g,\bar{K})$ of symmetric pairs.
Then we have natural embeddings $\UnipOv(\g, \bar{K}) \hookrightarrow \UnipOv(\g, K)$ and $\AOD_{\bar{K}}(\orb) \hookrightarrow \AODK(\orb)$ by inflation. These embeddings are compatible with the associated cycle maps for $(\g, \bar{K})$ and $(\g, K)$ respectively. Note that the subgroup $F:=\ker(K \to \bar{K})$ is a central finite subgroup of $K$ and its adjoint action on $\fg$ factors through $\bar{K}$ and hence is trivial. By Schur's lemma, $F$ acts on any irreducible HC $(\g,K)$-module $M$ by a scalar character $\chi_M: F \to \cm$.

Recall that $\UnipOv(\g,K)^{gen} := \UnipOv(\g,K) \setminus \UnipOv(\g, \bar{K})$ is the subset of $\UnipOv(\g,K)$ consisting of isomorphism classes of genuine (with respect to $\phi: K \to \bar{K}$) special unipotent irreducible HC $(\g,K)$-modules attached to $\ckorb$, that is, those modules with $F$ acting non-trivially. Similarly, let $\cAOD(\orb,K)^{gen}$ be the full abelian subcategory of $\cAOD(\orb,K)$ consisting of genuine objects and 
	\[\AODK(\orb)^{gen} := \AODK(\orb) \setminus \AOD_{\bar{K}}(\orb) = \Irr \cAOD(\orb,K)^{gen}. \]
We use the analogous notation
\[
	\cK\cAOD(\orb,K)^{gen} \quad \text{and} \quad \cK_{>0}\cAOD(\orb, K)^{gen}
\]
for the Grothendieck group of the category $\cAOD(\orb, K)^{gen}$ and its positive semigroup.
The following lemma is standard.

\begin{Lem}\label{lem:AC_gen_F}
	Let $M$ be an HC $(\g,K)$-module such that $F = \ker(K \to \bar{K})$ acts on $M$ by a scalar character $\chi_M: F \to \cm$. Write $\AC(M) = \sum_{j=1}^p [\cV_j] \cdot \orbk^j$. Then $F$ acts on each $\cV_j$ by the same character $\chi_M$.
\end{Lem}

\begin{Cor}\label{cor:AC_special_gen}
	The image of any $M \in \UnipOv(\g, K)$ under the associated cycle map 
	\[ \AC: \UnipOv(\g, K) \xrightarrow{\ \ } \cK_{>0}\cAOD(\orb, K) \]
	from \Cref{lem:AC_special_unip} lies in $\cK_{>0}\cAOD(\orb, K)^{gen}$ if and only if $M$ lies in $\UnipOv(\g, K)^{gen}$.
\end{Cor}

We now consider the case when the covering morphism $K \twoheadrightarrow \bar{K}$ corresponds to the metaplectic double covering map $\Mp(2n,\R) \twoheadrightarrow \Sp(2n,\R)$. Note that in \cite[Section 2.3]{BMSZ:counting}, special unipotent representations of $\Mp(2n,\R)$ are defined to be genuine. The following result shows that this in fact coincides with our definition.

\begin{Cor}\label{cor:metaplectic_genuine}
	Let $\GR = \Mp(2n,\R)$ and $\bar{G}_\R = \Sp(2n,\R)$, and let $\phi: \GR \to \bar{G}_\R$ be the $2$-fold covering map. Then all HC modules in $\UnipOv(\g, K)$ are genuine (with respect to $\phi$).
\end{Cor}

\begin{proof}
	Let $\ckorb$ be any nilpotent orbit in $\ckfg^*$, where $\ckfg = \fg = \fmp(2n)$ is of metaplectic type, and let $\orb = d(\ckorb)$. Then $\orb$ is metaplectic special. By \cref{prop:metaplectic_adm}, any $K$-orbit $\orbk$ in $\orb_{\fk^\perp}$ is metaplectic admissible. Therefore all $K$-admissible vector bundles are genuine, i.e., $\cAOD(\orb, K) = \cAOD(\orb, K)^{gen}$.
	The conclusion now follows from \Cref{lem:AC_special_unip} and \Cref{cor:AC_special_gen}.
\end{proof}

\section{Coadjoint orbit method for real groups} \label{sec:orbit_method_real}

In this section, we first recall the quantization of vector bundles on smooth Lagrangian subvarieties and its Hamiltonian equivariant form. We then enlarge nilpotent $K$-orbit covers inside nilpotent $G$-orbit covers to larger smooth Lagrangian subvarieties, and establish the extension theorem for admissible vector bundles. Using these extensions, we quantize nilpotent $K$-orbit covers to construct Harish-Chandra modules and define the quantization map from admissible orbit data to special unipotent representations. Finally, we prove that this map is bijective by a counting and exhaustion argument, treating classical and metaplectic groups first and then exceptional groups via component-group and \texttt{atlas} computations.

\subsection{Quantization of Lagrangian subvarieties} \label{subsec:quan_lag}

\subsubsection{Definitions} \label{subsubsec:lag_def}

Let $X$ be a smooth symplectic variety and let $Y$ be a smooth closed Lagrangian subvariety of $X$; denote the embedding by $\iota : Y \hookrightarrow X$. There is a natural notion of quantization of a vector bundle over $Y$ (\cite{BGKP}, \cite[Section 1]{BC}) as well as a graded version (\cite[Section 4.3]{LY}).

\begin{defi}
	Suppose $\OO_\hb$ is a formal quantization of $(X, \Omega)$ and $\cE$ is a locally free sheaf of $\OO_Y$-modules over $Y$. A {\it formal quantization} of $\cE$, or rather the sheaf $\iota_* \cE$ of $\OO_X$-modules, is an $\OO_\hb$-module $\cE_\hb$ flat over $\C\series$, complete and separated in the $\hb$-adic topology, with an isomorphism $\varrho: \cE_\hb /\hb \cE_\hb \xrightarrow{\sim} \iota_*\cE$ of $\OO_X$-modules.

	Now assume additionally that $(X, \Omega)$ is graded and $Y$ is a $\cm$-stable Lagrangian subvariety of $X$. Suppose $\OO_\hb$ is a graded quantization of $X$ and $\cE$ is $\cm$-equivariant. A {\it graded (formal) quantization} of $\cE$ is a formal quantization $\cE_\hb$ of $\cE$ with a pro-rational $\cm$-action, such that the module structure map $\OO_\hb \otimes_{\C\series} \cE_\hb \to \cE_\hb$ is $\cm$-equivariant and the isomorphism $\varrho: \cE_\hb /\hb \cE_\hb \xrightarrow{\sim} \iota_*\cE$ is also $\cm$-equivariant.
\end{defi}

\subsubsection{Picard algebroids over Lagrangian subvarieties} \label{subsubsec:lag_pic}

By \cite[Section 5.3]{BGKP} and \cite[Section 4.1]{LY}, we can associate a Picard algebroid over a Lagrangian subvariety $Y$ to a given formal quantization $\OO_\hb$. Namely, let $\J_Y$ denote the ideal subsheaf of $\OO_X$ consisting of local functions vanishing along $Y$. Let $\J_{Y,\hb}$ and $\J_{Y,\hb}'$ denote the preimages of the ideals $\J_Y$ and $\J_Y^2$, respectively, under the projection $\OO_\hb \twoheadrightarrow \OO_X$. Write $\J_{Y,\hb}^2 = \J_{Y,\hb} \cdot \J_{Y,\hb} \subset \OO_\hb$, which is a subsheaf of two-sided ideals of $\OO_\hb$. Then $\hb^{-1} \J_{Y,\hb}^2 \subset \hb^{-1}\J_{Y,\hb}' \subset \hb^{-1}\J_{Y,\hb}$ are all subsheaves of Lie ideals in the sheaf $\hb^{-1}\OO_\hb$ of Lie algebras (with Lie bracket given by the commutator). 

Consider the sheaf $\tatp := \hb^{-1} (\J_{Y,\hb}/\J_{Y,\hb}^2)$, which is a sheaf of Lie algebras supported on $Y$. The left multiplication of $\OO_\hb$ on $\hb^{-1}\J_{Y,\hb}$ factors through a left $\OO_Y$-module structure on $\hb^{-1} (\J_{Y,\hb}/ \J_{Y,\hb}^2)$. These sheaves fit into a short exact sequence of sheaves of Lie algebras and $\OO_Y$-modules
\begin{equation}\label{exsq:J}
	0 \to \hb^{-1} (\J_{Y,\hb}' / \J_{Y,\hb}^2) \to \tatp = \hb^{-1} (\J_{Y,\hb}/ \J_{Y,\hb}^2) \to \hb^{-1} (\J_{Y,\hb} / \J_{Y,\hb}') \to 0.
\end{equation}
By \cite[Lemma 5.4]{BGKP}, there are canonical isomorphisms of sheaves of Lie algebras and $\OO_Y$-modules
\begin{equation} \label{eq:J_isom}
	\hb^{-1} (\J_{Y,\hb} / \J_{Y,\hb}') \simeq \T_Y, \quad \hb^{-1} (\J_{Y,\hb}' / \J_{Y,\hb}^2) \simeq \OO_Y.
\end{equation}
Then one can check that \eqref{exsq:J} and \eqref{eq:J_isom} together make $\tatp$ into a Picard algebroid over $Y$. If $X$ is graded symplectic and $\OO_\hb$ is graded, then $\tatp$ is naturally a graded Picard algebroid.

\begin{Rem}
	$\TT(\OO_\hb, Y)$ can be determined by Proposition 6.4 and Corollary 6.5 in \cite{Yu:period}. See also \cite[(7.6)]{Yu:period} for the case when $Y$ is a nilpotent $K$-orbit.
\end{Rem}

\subsubsection{Existence and uniqueness of quantizations} \label{subsubsec:lag_vec}

Assume the notation in \Cref{subsubsec:lag_def}. Let $\cE_\hb$ be a quantization of an algebraic vector bundle over $Y$. In particular, $\cE_\hb$ is a module over $\OO_\hb$ and we have $\J_{Y,\hb} \cdot \cE_\hb \subset \hb \cE_\hb$. This induces a Lie algebra action of $\hb^{-1} \J_{Y,\hb}$ on $\cE_\hb$ given by $(\hb^{-1} a) \cdot m = \hb^{-1}(a \cdot m)$, by the flatness of $\cE_\hb$ over $\C\series$. This Lie algebra action descends to a Lie algebroid action of $\tatp$ on $\cE_\hb / \hb \cE_\hb \simeq \cE$, which makes $\cE$ into a $\tatp$-module. Note that the $\tatp$-module structure on $\cE$ is completely determined by the $\OO_\hb / \hb^2 \OO_\hb$-module structure on $\cE_\hb / \hb^2 \cE_\hb$. 
If $X$ is graded symplectic and $\OO_\hb$ and $\cE_\hb$ are graded quantizations, then $\cE$ is naturally a ($\OO_Y$-coherent) graded $\tatp$-module, i.e., a weak $(\tatp, \cm)$-module.
We will consider the problem of constructing graded quantizations of a $\tatp$-module $\cE$ that lift the $(\tatp, \cm)$-module structure (for the ungraded case, see \cite[Theorem 1.1]{BC} and also \cite[Theorem 1.4]{BGKP} for the case when $\cE$ is a line bundle).

To this end, we need the notion of strong $\cm$-actions in \cite[Definition 4.3.2]{LY}, which is crucial in the construction of graded quantizations. In order to avoid confusion with the notion of strong $K$-actions in the sense of \Cref{defn:pic_module_action}, we use the terminology \emph{solid} $\cm$-action instead. We review the definition below in a more general setting.

Let $H$ be an algebraic group with $\Lie(H) = \fh$. Let $X$ be a smooth variety with an algebraic action of $H$, let $\TT$ be a weakly $H$-equivariant Picard algebroid, and let $\cE$ be a weak $(\TT,H)$-module over $X$ that is $\OO_X$-coherent. The $\TT$-module structure on $\cE$ induces an algebraic flat connection on the projectivization $\mathbb{P}(\cE)$ of $\cE$, which lifts the differentiation of the $H$-action on $X$ to a Lie algebra $\fh$-action on $\mathbb{P}(\cE)$. On the other hand, the $H$-action on $\cE$ descends to a $H$-action on $\mathbb{P}(\cE)$, which differentiates to a second $\fh$-action. 

\begin{defi} \label{defn:solid}
	The $H$-action on a weak $(\TT, H)$-module $\cE$ is said to be \emph{solid} if the two Lie algebra $\fh$-actions on $\mathbb{P}(\cE)$ defined above coincide.
\end{defi}

\begin{Rem} \label{rem:solid}
	Suppose $\cE$ is an $(H, \kappa)$-equivariant $\TT$-module, where $\kappa \in (\fh^*)^H$. Then the $H$-action on $\cE$ is solid. Also, any weak $H$-action on a line bundle is automatically solid.
	

	Note that this property is not preserved by taking direct sums; hence the category of $\TT$-modules with solid $H$-action is not additive. On the other hand, \cite[Proposition 6.10]{Yu:period} shows that any weak $\cm$-action on a $\TT$-module whose analytification is irreducible is automatically solid.
\end{Rem}

We can now state the main result of this subsection, from \cite[Theorem 4.3.3]{LY} and \cite[Theorem 6.7]{Yu:period}.

\begin{Thm} \label{thm:quan_lag}
	Let $\cE$ be an $\OO_Y$-coherent weak $(\tatp, \cm)$-module over the Lagrangian subvariety $Y$, such that the $\cm$-action on $\cE$ is solid. Then there exists a graded quantization $\cE_\hb$ of $\cE$ which gives rise to the given $(\tatp, \cm)$-module structure on $\cE$. Such a graded quantization is unique up to $\cm$-equivariant isomorphism.
\end{Thm}

\subsubsection{Hamiltonian quantization of Lagrangian subvarieties} \label{subsubsec:ham_lag_quant}

With the setting and notation in \Cref{subsubsec:lag_def}, let us now consider the compatibility of (graded) quantizations of Lagrangian subvarieties with Hamiltonian actions on quantizations of a (graded) symplectic variety $X$. We refer the reader to \cite[Section 7.3]{LY} and \cite[Section 4.4]{LosevYu} for the details. 

Suppose $(X, \Omega)$ is a graded smooth symplectic variety equipped with a Hamiltonian action of an algebraic group $G$ and a classical co-moment map $\phi: \g \to \Gamma(X, \OO_X)$. Let $K$ be another algebraic group with a homomorphism $K \to G$ such that the induced Lie algebra homomorphism $\fk \to \fg$ is injective (e.g., as in \cref{subsec:basic_settings}). Suppose that $Y$ is a smooth closed Lagrangian subvariety in $X$, stable under the $K \times \cm$-action. In addition, we assume that the image of $\fk$ under $\phi$ lies in the ideal subsheaf $\J_Y \subset \OO_X$ defining $Y$. 

Now suppose $(\OO_\hb,\Phi_\hb)$ is a $G$-Hamiltonian quantization of $(X, \Omega)$. The restriction of the $G$-action to $K$ preserves the ideal subsheaf $\J_{Y,\hb} \subset \OO_\hb$ defined in \Cref{subsubsec:lag_pic} and hence induces a weak $K$-action on the sheaf $\tatp=\hb^{-1}(\J_{Y,\hb} / \J_{Y,\hb}^2)$ over $Y$, which is compatible with the $\OO_Y$-module structure. Moreover, the restriction $\Phi_\hb|_{\fk} : \fk \to \Gamma(X, \OO_\hb)$ of the quantum co-moment map to $\fk$ has image lying in $\Gamma(X, \J_{Y,\hb})$. The composition of $\Phi_\hb|_{\fk}$ with the quotient map $\J_{Y,\hb} \twoheadrightarrow \J_{Y,\hb} / \J_{Y,\hb}^2$ induces a strong $K$-action on $\tatp$, which lifts the weak $K$-action and hence makes $\tatp$ into a strongly $K$-equivariant graded Picard algebroid on $Y$.

\begin{defi}\label{defn:equiv_quant_modules}
Suppose $\cE$ is a $\cm$-equivariant vector bundle over $Y$. A {\it $K$-equivariant graded Hamiltonian quantization} of $(Y,\cE)$ is a graded quantization $\cE_\hb$ of $\cE$ equipped with a $\C\series$-linear $K$-action that lifts the $K$-action on $Y$, such that
\begin{enumerate}[label=(\roman*)]
	\item the $K$-action commutes with the $\cm$-action on $\cE_\hb$;
	\item the structure morphism $\OO_\hb \otimes \cE_\hb \to \cE_\hb$ is $K$-equivariant;
	\item the differential of the $K$-action on $\cE_\hb$ is given by
	$\xi \cdot m = \hb^{-1} \Phi_\hb(\xi) \cdot m$
	for any $\xi \in \fk$ and $m \in \cE_\hb$.
\end{enumerate}
\end{defi}

If $\cE_\hb$ is a $K$-equivariant graded Hamiltonian quantization of $\cE$, then this structure induces an algebraic $K$-action on $\cE$ as an $\OO_X$-module commuting with the $\cm$-action on $\cE$. Moreover, this $K$-action and the strong $K$-action on $\tatp$ mentioned above make $\cE$ into a strong $(\tatp, K)$-module with a compatible weak $\cm$-action.

Conversely, the following theorem, proven in \cite[Theorem 6.14]{Yu:period}, is an equivariant version of Theorem \ref{thm:quan_lag}. The work \cite[Theorem 6.2.5]{LY} also addresses the case when $\OO_\hb$ is an even quantization.

\begin{Thm}\label{thm:quan_lag_equiv}
	In the setting above, let $\cE$ be an $\OO_Y$-coherent strongly $K$-equivariant $\tatp$-module over $Y$, with a compatible weak $\cm$-action that is solid. Then for any graded quantization $\cE_\hb$ of $\cE$ as in \Cref{thm:quan_lag}, the $K$-action on $\cE$ lifts uniquely to a $K$-action on $\cE_\hb$, making $\cE_\hb$ into a $K$-equivariant graded Hamiltonian quantization of $\cE$. 
\end{Thm}

We will only be interested in the case when $\OO_\hb$ is an even quantization, which was treated in \cite[Theorem 6.2.5]{LY}. We also assume that $\OO_\hb$ is equipped with a symmetrized quantum Hamiltonian $G$-action with a symmetrized quantum co-moment map $\Phi_\hb$ in the sense of \Cref{defn:quan_moment_symm}.
In this case, we have the following concrete description of the Picard algebroid $\tatp$, by \cite[Lemma 6.2.4]{LY} and the discussion following it. Recall that the $K \times \cm$-action on $Y$ induces a strong $K \times \cm$-action on $\TT^+_Y = \frac{1}{2}\TT(\omega_Y)$.

\begin{Prop} \label{prop:tatp_symm}
	In the setting above, suppose $(\OO_\hb,\epsilon,\beta_\hb,\Phi_\hb)$ is a symmetrized even $G$-Hamiltonian quantization of $(X, \Omega)$. Then there exists a $K \times \cm$-equivariant isomorphism $\tatp \simeq \TT^+_Y = \frac{1}{2}\TT(\omega_Y)$ of strongly $K$-equivariant Picard algebroids. 
\end{Prop}

\subsection{Enlargement of nilpotent \texorpdfstring{$K$}{K}-orbit covers} \label{subsec:enlarge_lag}

Consider the general setting in \cref{subsec:basic_settings}. We make the additional assumption that the complex reductive group $G$ is connected. Let $\bar{K}$ denote the image of the morphism $K \to G^\theta$.
Set $\sigma = -\theta$, where $\theta$ is the Cartan involution of $\g$; then $\sigma$ is a $K$-equivariant anti-Poisson involution of $\g \simeq \g^*$. Let $\orb \in \cN_o$ and assume that $\orb_{\fk^\perp} \neq \varnothing$; then both $\sigma$ and $\theta$ preserve $\orb$. Take any $K$-orbit $\orbk$ in $\orb_{\fk^\perp}$, which is the same as a $\bar{K}$-orbit. 

Let $\rho: \widetilde{\orb} \to \orb$ be a $G$-equivariant Galois cover of $\orb$ with the $G$-equivariant Galois group $\Gamma = \Gamma_{\rho}$. The preimage $\rho^{-1}(\orbk)$ of $\orbk$ in $\widetilde{\orb}$ is in general a disjoint union of $K$-orbits, on which the $K$-action factors through the morphism $K \twoheadrightarrow \bar{K} \subset G^\theta$. The Galois group $\Gamma$ acts transitively on the set of such $K$-orbits. After fixing such a $K$-orbit $\torbk \subset \rho^{-1}(\orbk)$ (which is again the same as a $\bar{K}$-orbit), the restriction 
$\tilde{\rho} := \rho|_{\torbk}$
of $\rho$ to $\torbk$ is a $K$-equivariant cover of $\orbk$. In fact, it is also a Galois cover. Indeed, pick $e \in \orbk$ and $y \in \tilde{\rho}^{-1}(e) = \rho^{-1}(e) \cap \torbk$. Then $\bar{K}_y = G_y \cap \bar{K}$ is normal in $\bar{K}_e = G_e \cap \bar{K}$, hence $\tilde{\rho}$ is a $\bar{K}$-equivariant, hence $K$-equivariant, Galois cover by condition (iv) of \Cref{defn:Galois}. We also see that the $K$-equivariant Galois group $\Aut_K(\tilde{\rho})$ of $\tilde{\rho}$ is canonically identified with the $\bar{K}$-equivariant Galois group $\Aut_{\bar{K}}(\tilde{\rho})$ of $\tilde{\rho}$, and the latter is naturally a subgroup of $\Gamma$. We denote both of them as $\Gamma_{\tilde{\rho}}$. This group can be characterized intrinsically as the subgroup of $\Gamma$ consisting of those automorphisms preserving $\torbk$.  

Let $X = \spec(\C[\widetilde{\orb}])$; then $X$ is a symplectic singularity. Let $X^{reg}$ denote the smooth locus of $X$, so that it is a smooth symplectic variety with the $\cm$-action induced from the Kostant-Brylinski $\cm$-action on $\widetilde{\orb}$ as in \Cref{subsec:orbit_cm}. Let $Y$ be the closure of $\torbk \subset \widetilde{\orb}$ in $X^{reg}$. Then $Y$ is also stable under the $K \times \cm$-action.

\begin{Rem} \label{rem:choice_Y}
	Different choices of $\torbk$ yield different $Y$. However, once we fix a choice of $\torbk$, the $\Gamma$-action interchanges it with other choices, and identifies the set of all $K$-orbits in $\rho^{-1}(\orbk)$ with $\Gamma/\Gamma_{\tilde{\rho}}$ as $\Gamma$-sets. The same holds for $Y$.
\end{Rem}

We will see below in \Cref{prop:Y_lag} that $Y$ is a smooth closed Lagrangian subvariety of $X^{reg}$. We need some preparation before proving this.

Take a universal $K$-equivariant cover $\varphi: Z=\horbk \to \torbk$ of $\torbk$ as in \Cref{subsubsec:AOD_defn}, so that 
$\hat{\rho} = \tilde{\rho} \circ \varphi$ 
is also a universal $K$-equivariant cover of $\orbk$, whose Galois group is $\Gamma_K$. Since $\tilde{\rho}: \torbk \to \orbk$ is a Galois cover, $\Gamma_{\tilde{\rho}}$ is naturally a quotient of $\Gamma_K$. In particular, we have natural homomorphisms $\Gamma_K \twoheadrightarrow \Gamma_{\tilde{\rho}} \hookrightarrow \Gamma$, via which $\Gamma_K$ also acts naturally on $\torbk$ so that $\varphi$ is $K \times \Gamma_K$-equivariant. Also, since the Brylinski-Kostant $\cm$-actions on $K$-equivariant covers of $\orbk$ in \Cref{subsec:orbit_cm} are unique and canonical, $\varphi$ is $K \times \Gamma_K^{\cm}$-equivariant. We also let $\tGamma_K^{\cm}$ act on $\horbk$, $\torbk$ and hence $Y$ via the quotient map $\tGamma_K^{\cm} \to \Gamma_K^{\cm}$.
Thus the relevant covers and groups are summarized by
\[
	Z=\horbk \xrightarrow{\varphi} \torbk \xrightarrow{\tilde{\rho}} \orbk, \qquad
	\hat{\rho}=\tilde{\rho}\circ\varphi,
\]
with Galois groups $\Gamma_K$, $\Gamma_{\tilde{\rho}}$, and natural maps $\Gamma_K \twoheadrightarrow \Gamma_{\tilde{\rho}} \hookrightarrow \Gamma$.

We can lift the anti-Poisson automorphism $\sigma: \orb \to \orb$ to an anti-Poisson automorphism $\tilde{\sigma}: \widetilde{\orb} \to \widetilde{\orb}$, such that $\tilde{\sigma}$ is $K$-equivariant and fixes $y$, and the covering morphism $\rho: \widetilde{\orb} \to \orb$ intertwines $\tilde{\sigma}$ and $\sigma$. This implies that $\tilde{\sigma}$ fixes the $K$-orbit $\torbk = K \cdot y$ pointwise, and $\torbk$ is a union of connected components of the fixed locus $\widetilde{\orb}^{\tilde{\sigma}}$ of $\tilde{\sigma}$. Note that $\tilde{\sigma}$ only depends on $\orbk$ and $\torbk$ and is independent of the choice of $e$ and $y$.
The involution $\tilde{\sigma}$ induces a $K$-equivariant anti-Poisson automorphism on $X = \spec(\C[\widetilde{\orb}])$, which preserves the smooth locus $X^{reg}$ of $X$. The restriction of this involution to $X^{reg}$ is also denoted as $\tilde{\sigma}$. 

\begin{Prop} \label{prop:Y_lag}
	$Y$ is a $K \times \Gamma_K^{\cm}$-stable smooth closed Lagrangian subvariety of $X^{reg}$.
\end{Prop}

\begin{proof}
	By \cite[Proposition 3.4]{Edixhoven}, the $\tilde{\sigma}$-fixed locus $\widetilde{\orb}^{\tilde{\sigma}}$ (resp. $(X^{reg})^{\tilde{\sigma}}$) is a smooth closed subvariety of $\widetilde{\orb}$ (resp. $X^{reg}$). Moreover, since $\tilde{\sigma}$ is an anti-Poisson automorphism, both are Lagrangian subvarieties (cf. \cite[Corollary 5.20]{Vogan_AV}). Since $\torbk \subset (X^{reg})^{\tilde{\sigma}}$, we have $Y \subset (X^{reg})^{\tilde{\sigma}}$. Since $(X^{reg})^{\tilde{\sigma}}$ is smooth, its irreducible components are exactly the same as its connected components. The same holds for $\widetilde{\orb}^{\tilde{\sigma}}$. Moreover, the closure of each irreducible/connected component of $\widetilde{\orb}^{\tilde{\sigma}}$ is an irreducible/connected component of $(X^{reg})^{\tilde{\sigma}}$. Note that $\torbk$ is a closed subvariety of $\widetilde{\orb}^{\tilde{\sigma}}$ of the same dimension, hence it is a union of irreducible/connected components of $\widetilde{\orb}^{\tilde{\sigma}}$.
	Therefore, $Y$ is a union of irreducible/connected components of $(X^{reg})^{\tilde{\sigma}}$. The conclusion follows.
\end{proof}

\subsection{Extensions of admissible vector bundles} \label{subsec:adm_extn}

We assume the setting and notation in \Cref{subsec:enlarge_lag}. We now apply the discussion in \cref{subsec:enlarge_lag} in the following situation.
Suppose that $\ckorb$ is a quasi-distinguished nilpotent orbit in $\ckfg$. Set $\orb = d(\ckorb)$, where $d$ stands for the usual/metaplectic Barbasch-Vogan duality map. Let $\rho: \widetilde{\orb} \to \orb$ be the universal cover of $\orb$ and $X := \spec(\C[\widetilde{\orb}])$. In addition, we assume that $\codim(X^{sing}, X) \geqslant 6$ (cf. \Cref{thm:codim_X}), hence
\begin{equation}\label{eq:codim_Y}
	\codim(\partial Y, Y) \geqslant 3,
\end{equation}
where $\partial Y = Y \setminus \torbk$. We will often write $Y^\diamond$ for $\torbk$ to simplify notation.

We now state the main result (\cref{prop:extn_AOD} below) of this subsection. We consider two cases:
\begin{enumerate}[leftmargin=*, itemindent=*]
	\item $G$ is a connected complex reductive group and $\GR \subset G(\R)$ is a linear reductive group. In this case we take $d$ to be the usual BV duality map;
	\item $\GR = \Mp(2n,\R)$ and $G=\Sp(2n,\C)$. In this case we take $d$ to be the metaplectic BV duality map.
\end{enumerate}
For both cases, let $(\g,K)$ denote the corresponding symmetric pair.
Recall that the Picard algebroid $\TT^+_Y = \frac{1}{2} \TT(\omega_Y)$ is equipped with a strong $K \times \Gamma_K^{\cm}$-action, whose restriction to $Y^\diamond = \torbk$ is canonically isomorphic to $\TT^+_{Y^\diamond}$ as strongly $K \times \Gamma_K^{\cm}$-equivariant Picard algebroids. Therefore we can speak of restrictions of objects in $\Coh(\TT^+_Y, K)$ or $\Coh^\kappa(\TT^+_Y, K \times \Gamma_K^{\cm})$ to $Y^\diamond = \torbk$. 

\begin{Prop}\label{prop:extn_AOD}
	Assume the notation above. Suppose $\cE \in \Coh^\kappa(\TT^+_{Y^\diamond}, K \times \Gamma_K^{\cm})$. Then $\cE$ extends uniquely to an object $\overline{\cE}$ in $\Coh^\kappa(\TT^+_Y, K \times \Gamma_K^{\cm})$. More precisely, this means that: 
	\begin{enumerate}[label=(\roman*), leftmargin=*, itemindent=*]
		\item $\overline{\cE}$ is an $\OO_Y$-coherent $(K \times \Gamma_K^{\cm}, \kappa)$-module equipped with an isomorphism $\varsigma: \overline{\cE}|_{Y^\diamond} \xrightarrow{\sim} \cE$ in $\Coh^\kappa(\TT^+_{Y^\diamond}, K \times \Gamma_K^{\cm})$.
		
		\item If there is another pair $(\overline{\cE}', \varsigma')$ satisfying (i), then there is a (unique) isomorphism $\phi: \overline{\cE} \xrightarrow{\sim} \overline{\cE}'$ in $\Coh^\kappa(\TT^+_Y, K \times \Gamma_K^{\cm})$ such that $\varsigma = \varsigma' \circ (\phi|_{Y^\diamond})$. 
	\end{enumerate}
	This extension defines an equivalence of categories
    	\[ \Coh^\kappa(\TT^+_{Y^\diamond}, K \times \Gamma_K^{\cm}) \xrightarrow{\sim} \Coh^\kappa(\TT^+_Y, K \times \Gamma_K^{\cm}), \quad \cE \mapsto \overline{\cE}, \]
    whose quasi-inverse is given by the restriction functor $\overline{\cE} \mapsto \overline{\cE}|_{Y^\diamond}$.

    Similar statements hold for any object $\cE \in \Coh(\TT^+_{Y^\diamond}, K \times \tGamma_K^{\cm})^{gen}$ (which means that the action of $\tGamma_K^{\cm}$ on $\cE$ is genuine); in particular, we have an equivalence of categories
    \[ \Coh(\TT^+_{Y^\diamond}, K \times \tGamma_K^{\cm})^{gen} \xrightarrow{\sim} \Coh(\TT^+_Y, K \times \tGamma_K^{\cm})^{gen}. \]
	

\end{Prop}

Before proving this proposition, we need some preparation. 
Assume $\orb' \in \cN_o$ satisfies $\orb' \prec \orb$ and $\codim(\orb', \overline{\orb}) = 2$. By \Cref{thm:codim_X}, $\orb'$ is non-special and hence lies in the special piece $\cP(\orb)$ of $\orb$. Suppose $\orbk \subset \orb_{\fk^\perp}$ is such that $\overline{\orb}_K \cap \orb' \neq \varnothing$. Let $\orbk'$ be any $K$-orbit in $\overline{\orb}_K \cap \orb'$. For exceptional $\g$, we list all such pairs $(\orbk, \orbk')$ together with $\orb$ and $\orb'$ in \Cref{tab:qD_real_exceptional}.

Let $e' \in \orbk'$ and complete it into a normal $\slf_2$-triple $(e',h',f')$. We follow the notation in \Cref{subsec:Slodowy} and let $Q'$ and $K_Q' = K \cap Q'$ denote the $G$-centralizer and $K$-centralizer of $(e',h',f')$, with Lie algebras $\fq'$ and $\fk_Q'$ respectively. Define the slices $\cS$, $\cS_\orb = \cS \cap \overline{\orb}$ and $\check{X}$ as in \Cref{subsec:Slodowy}.
Let $\torbk'$ be a $K$-orbit in the intersection of $Y$ with the preimage of $\orbk'$ in $X$. 
Pick $\tilde{e}' \in \torbk'$ so that $\rho(\tilde{e}')=e'$, and take the slice $\check{X}$ containing $\tilde{e}'$ as in \Cref{subsec:Slodowy}. Then $\check{X}$ is $\tilde{\sigma}$-stable and intersects $Y$. It is a transversal slice in $X$ to the $G$-orbit $\widetilde{\orb}' := G \cdot \torbk' \subset X$, which is a cover of $\orb'$ by the restriction of $\rho: \widetilde{\orb} \to \orb$ to $\widetilde{\orb}'$. Recall that we have a moment map $\pr_0: \cS_\orb \to \mathfrak{q}'$. Let $\gamma'$ be the character of $\fk_{e'} = \Lie(K_{e'})$ associated to $\orbk'$ as defined in \Cref{defn:adm_orbit_datum}. It descends to a character of the reductive quotient $\fk_Q'$, still denoted as $\gamma'$.

Let $\cU$ be the union of $\torbk$ and all $\torbk'$ over all such $\orb'$ as above. Then $\cU$ is an open subset of $Y$ and $\codim(Y \backslash \cU, Y) \geqslant 2$. Each $\torbk'$ is a smooth closed divisor in $\cU$. We first need to show that $\cE$ can be extended to $\cU$ using a similar argument as in \cite[Appendix]{LosevYu}, by analyzing the local behaviour of $\cE$ around $\torbk'$ as follows.

We now isolate the local $A_1$ model that controls these codimension-one boundary pieces.
Recall that the Kleinian singularity of type $A_1$ is $\C^2/\Z_2$, or equivalently, the minimal/regular orbit $\orb_{min}(\slf_2)$ in $\slf_2^*$, and hence is a conical symplectic singularity with symplectic form of degree $2$ with respect to the $\cm$-action. It is equipped with a Hamiltonian $\SL_2$-action which factors through $\PSL_2$, so that we have a group isomorphism $\aut_{\cm}(\orb_{min}(\slf_2))^\circ \simeq \PSL_2$. The Lie algebra of $\aut_{\cm}(\orb_{min}(\slf_2))^\circ$ is identified with the degree $2$ component $\C[\orb_{min}(\slf_2)]_2 \simeq \slf_2$ of the coordinate ring $\C[\orb_{min}(\slf_2)]$ of $\orb_{min}(\slf_2)$.
The following proposition identifies $\check{X}$ with $\orb_{min}(\slf_2)$ as Hamiltonian spaces and describes a direct product decomposition of $Q'^\circ$.

\begin{Prop}\label{prop:A1}
	Let $\fg$ be a complex simple Lie algebra and let $G$ be any connected complex simple group with Lie algebra $\fg$. Let $\orb \subset \fg$ be the (metaplectic) Barbasch-Vogan dual orbit of a quasi-distinguished nilpotent orbit in $\ckfg$ and let $\orb'$ be any nilpotent orbit of codimension $2$ in $\overline{\orb}$. Then there exists a $\theta$-stable subgroup $R=R_{\orb'} \simeq \SL_2$ of $Q'^\circ$ with Lie algebra $\rf$ (so that $R^\theta$ is a one-dimensional torus) satisfying the following properties:
	\begin{enumerate}[itemindent=*, leftmargin=*]
		\item 
		There is a graded $R$-equivariant Hamiltonian isomorphism between $\C^2 / \Z_2$ and $\cS_{\orb}$, and a graded $R$-equivariant Hamiltonian isomorphism $\check{X} \simeq \C^2$. Here, $R \simeq \SL_2$ acts on $\cS_{\orb}$ and $\check{X}$ via the $Q'$-action.
		\item 
		There is a direct product decomposition $Q'^\circ = R \times N$, where $N$ is a reductive subgroup of $Q'^\circ$, such that the action of $Q'^\circ$ on $\cS_{\orb}$ and $\check{X}$ factors through $R$.
		\item \label{item:nonadm}
		For any $K$-orbit $\orbk'$ in $\overline{\orb}_K \cap \orb'$, the following conditions are satisfied:
        \begin{enumerate}[label=(\theenumi\alph*), ref=(\theenumi\alph*), itemindent=*, leftmargin=*]
            \item \label{cond:O'_nonadm} 
			If $\orb$ is a special nilpotent orbit in the usual sense, then the restriction of $\gamma'$ to $\fr^\theta$ does not integrate to a character of $R^\theta$.
            \item \label{cond:O'_metaplectic_adm} 
			If $\fg = \fmp(2n,\C)$ and $\orb$ is a metaplectic special nilpotent orbit in $\fg^*$, then the restriction of $\gamma'$ to $\fr^\theta$ integrates to a character of $R^\theta$.
        \end{enumerate}
	\end{enumerate}
\end{Prop}

\begin{proof}
	\noindent\emph{Proof of (1).} We claim that, in all cases, the moment map $\pr_0: \cS_\orb \to \mathfrak{q}'$ induces an isomorphism between $\cS_\orb$ and the closure of the $Q'^\circ$-orbit through a minimal nilpotent element in a simple factor $\rf \simeq \slf_2$ of $\mathfrak{q}'$.  
	For classical $\g$, this has already appeared in the proof of \cite[Proposition 5.12]{MBMY} (the argument is essentially the same as that of \cite[Lemma 4.1]{JLS:Duality} and \cite[Proposition 4.2]{JLS:Duality}). We recall the first part of this argument and extend it to the metaplectic case here: let $\fg=\fg_\epsilon(N)$ and let $\lambda, \lambda' \in \Pe(N)$ be the $\epsilon$-partitions associated to $\orb$ and $\orb'$ respectively. It follows easily from the discussion in \cref{subsec:special_piece_classical} that such $\lambda'$ arises from $\lambda$ by replacing a pair of adjacent members $[\lambda_i, \lambda_{i+1}]$ of $\lambda$, such that $\lambda_i = \lambda_{i+1} + 2$ and $\lambda_i \not\equiv \epsilon \bmod 2$, by $[\lambda_i -1, \lambda_i-1]$. Then $\lambda_i - 1$ has multiplicity $2$ in $\lambda'$. Since $\lambda_i -1 \equiv \epsilon \bmod 2$, the pair $[\lambda_i -1, \lambda_i-1]$ contributes a direct product factor $R = R_i \simeq \Sp(2,\C) \simeq \SL_2$ of $Q'^\circ$ when $G$ is a linear classical group (see \cref{subsec:nil_classical}). Then we can follow the rest of the proof of \cite[Proposition 5.12]{MBMY}, and prove the claim for classical $\g$ and all connected simple groups $G$ with Lie algebra $\g$.
	For exceptional $\g$, we appeal to \cite[Proposition 3.3 and Remark 3.4]{FJLS:generic_sing}.

	Now in all types, since $Q'$ acts on $\cS_{\orb}$ by Hamiltonian automorphisms, we have a group homomorphism 
	\[ \Phi: Q'^\circ \to \Aut_{\cm}(\cS_{\orb}) \simeq \Aut_{\cm}(\orb_{min}(\slf_2))^\circ \simeq \PSL_2, \] 
	whose differential is denoted as $\phi = d\Phi: \q' \to \slf_2$. Then the restriction of $\phi$ to $\rf$ is an isomorphism $\rf \xrightarrow{\sim} \slf_2$. 
	Hence we have a direct sum decomposition 
	\begin{equation} \label{eq:q'_decomp}
		\q' = \rf \oplus \ker \phi 
	\end{equation}
	of $\q'$ into Lie ideals. Moreover, we can make this decomposition $\theta$-invariant. 
	


	Let $R$ be the connected closed subgroup of $Q'^\circ$ with Lie algebra $\rf$; then $R$ is also $\theta$-stable. Note that the restriction of $\Phi$ to $R$ only gives a surjective morphism $R \twoheadrightarrow \PSL_2$. We cannot directly conclude that $R \simeq \SL_2$ from the preceding discussion (except when $\fg$ is classical), so some extra work is required.
	
	Recall that by \Cref{thm:special_piece} or \Cref{thm:codim_X}, $X$ is smooth along the preimage of $\orb'$ in $X$ and hence each $\check{X}$ is also smooth. We can also appeal to \cite[Proposition 4.3]{MBMY} with the conjugacy class $\bar{C}$ there being trivial, which only proves that $X$ is smooth in codimension one.\footnote{This approach avoids the computer computation for exceptional Lie algebras in \cite{FJLS23} which \cref{thm:special_piece} relies on.} Since $\rho: \widetilde{\orb} \to \orb$ is a finite \'{e}tale morphism, so is its restriction $\rho: \check{X} \setminus \{\tilde{e}'\} \to (\C^2 \setminus \{0\}) / \Z_2$ to $\check{X} \setminus \{\tilde{e}'\}$. Since $\check{X} \setminus \{ \tilde{e}' \}$ is connected, it is either isomorphic to $(\C^2 \setminus \{0\}) / \Z_2$ or the double cover $\C^2 \setminus \{0\}$. But $\check{X}$ is normal and the point $\tilde{e}'$ is of codimension $2$, hence $\check{X} = \spec(\C[\check{X} \setminus \{ \tilde{e}' \}])$. Therefore $\check{X} \setminus \{ \tilde{e}' \}$ cannot be isomorphic to $(\C^2 \setminus \{0\}) / \Z_2$, since otherwise $\check{X} \simeq \C^2/\Z_2$, which is singular. Then $\check{X}$ must be isomorphic to the standard symplectic vector space $\C^2$, whose group of graded Poisson automorphisms is $\SL_2$. Therefore $R \simeq \SL_2$. This concludes the proof of (1).
	\vskip 1em

	\noindent\emph{Proof of (2).} Let $N = (\ker\Phi)^\circ$; then $\Lie(N) = \ker \phi$. The group $R$ commutes with $N$ by \eqref{eq:q'_decomp}. Moreover, $R \cap N = \{1\}$ since the restriction of $\Phi$ to $R$ is an isomorphism $R \simeq \SL_2$ by (1). Therefore $Q'^\circ = R \times N$.
	\vskip 1em

	\noindent\emph{Proof of (3).} By (2), we have decompositions 
	\begin{equation} \label{eq:decomp_F}
		(Q'^\circ)^\theta = R^\theta \times N^\theta \quad \text{and} \quad \fk_Q'=(\q')^\theta = \fr^\theta \oplus (\ker \phi)^\theta. 
	\end{equation}  
	We pick a point $e \in (\cS_{\orb})^\sigma$ and complete it into a normal $\slf_2$-triple $(e,h,f)$, and define $Q=Z_G(e,h,f)$ and $K_Q=Z_K(e,h,f)$ as usual. Then we have $N \subset Q^\circ$ since $N \subset \ker\Phi$, hence $(\ker \phi)^\theta \subset \fk_Q$. Let $\gamma$ be the character of $\fk_Q$ associated to $\orbk$ as defined in \Cref{subsec:AOD}. Then by \cite[Proposition 6.3.2, (1)]{LosevYu}\footnote{$\gamma'$ is denoted as $\rho_\omega'$ in \cite[Proposition 6.3.2, (1)]{LosevYu}.}, we have
	\[ \gamma|_{(\ker \phi)^\theta} = \gamma'|_{(\ker \phi)^\theta}. \]
	Since $\orbk$ is linearly admissible by \Cref{prop:adm_dist}, $\gamma|_{(\ker \phi)^\theta}$ can be integrated to a character of $N^\theta$. Therefore by \eqref{eq:decomp_F}, in the non-metaplectic case, the condition \ref{cond:O'_nonadm} holds if and only if $\gamma'$ cannot be integrated to a character of $(K_Q')^\circ$, i.e., $\orbk'$ is not linearly admissible. Now we can apply \Cref{prop:non_adm_codim2} to deduce \ref{cond:O'_nonadm}.

	It remains to consider the case of \ref{cond:O'_metaplectic_adm} for metaplectic type. Recall the description of the partition $\lambda'$ corresponding to $\orb'$ at the start of the proof, in the case when $\epsilon=\epsilon'=1$. In fact, as observed before the proposition, $\orb'$ must lie in $\cP(\orb)$ by \Cref{thm:codim_X}, and hence $h_\lambda(\lambda_i) \equiv \epsilon' \equiv 1 \bmod 2$. Set $c=\lambda_i -1 = \lambda_{i+1}+1$. By \cite[Lemma 7]{Ohta:Adm}, $\gamma'|_{\fr^\theta}$ can be integrated to a character of $R^\theta$ if and only if the number $m_{>c}(\lambda')$ of even parts in $\lambda'$ greater than $c$ is even. But $m_{>c}(\lambda') \equiv h_\lambda(\lambda_i)-1 \equiv 0 \bmod 2$, hence \ref{cond:O'_metaplectic_adm} holds.
\end{proof}

\begin{proof}[Proof of \Cref{prop:extn_AOD}]
	We are in the same setting as \cite[Appendix]{LosevYu}. Let $\cU$ be the union of $\torbk$ and all $\torbk'$ in all $\orb'$ as above. Then $\cU$ is an open subset of $Y$ and $\codim(Y \backslash \cU, Y) \geqslant 2$. We can then follow the discussion in Section A.2.1, especially Proposition A.2.1, and the same argument as in the proof of Theorem A.3.3 in \cite{LosevYu}, as well as \Cref{prop:A1} (especially \ref{item:nonadm}). Specifically, when $\GR$ is linear and $d$ is the usual BV duality map, we are in the case of \cref{prop:adm_classical} and the case \ref{cond:O'_nonadm} of \Cref{prop:A1}, so that $l$ is even, $n=1$ and $k=0$ in \cite[Proposition A.2.1]{LosevYu}; when $\GR = \Mp(2n,\R)$ and $d$ is the metaplectic BV duality map, we are in the case of \cref{prop:metaplectic_adm} and the case \ref{cond:O'_metaplectic_adm} of \Cref{prop:A1}, so that $l$ is odd, $n=2$ and $k=1$ in \cite[Proposition A.2.1]{LosevYu}. Then in all cases, \cite[Proposition A.2.1]{LosevYu} shows that $\cE$ extends to a $\OO_\cU$-coherent $\TT^+_{\cU}$-module $\cE_{min}$, which is nothing but the minimal extension of $\cE$ to $\cU$.

	Now since $\codim(Y \setminus \cU, Y) \geqslant 2$, the zero extension $\overline{\cE}$ of $\cE_{min}$ to $Y$ is also $\OO_Y$-coherent by \cite[Proposition 5.11.1]{EGA4}. The uniqueness statement about $\overline{\cE}$ follows from the uniqueness of the minimal extension. The equivariant structures extend uniquely by the uniqueness of $\overline{\cE}$. The property that the extended action on $\overline{\cE}$ is strong is automatic.
\end{proof}

\subsection{Quantization of nilpotent $K$-orbit covers} \label{subsec:quant_K_orb}

We assume the setting and notation in \Cref{subsec:enlarge_lag} and \Cref{subsec:adm_extn}.
As above, let $\varphi: Z=\horbk \to \torbk$ be a universal $K$-equivariant cover of $\torbk$, so that $\hat{\rho}:=\tilde{\rho} \circ \varphi: \horbk \to \orbk$ is also a universal $K$-equivariant cover of $\orbk$. Since $\tilde{\rho}$ is Galois, we can retain the notations and definitions in \Cref{subsubsec:galois} (see \Cref{rem:strong_cm_action_Galois}). In particular, we have 
\[
	\cW = \varphi_* \omega_Z^{\smallhalf} \in \Coh(\TT^+_{Y^\diamond}, K \times \tGamma_K^{\cm})^{gen}.
\]
By \Cref{prop:extn_AOD}, $\cW$ extends uniquely to an object $\overline{\cW}$ in $\Coh(\TT^+_Y, K \times \tGamma_K^{\cm})$.

As in \Cref{subsubsec:strong_cm_action}, we fix a choice of the character $\varepsilon: \widetilde{T} \to \cm$ and $\kappa = d\varepsilon$, and lift any $V \in \Rep(\tGamma_K)^{gen}$ to an object $V \boxtimes_\C \C_\varepsilon$ in $\Rep(\tGamma_K^{\cm})^{gen}$ (still denoted as $V$) via a functor $\Rep(\tGamma_K)^{gen} \hookrightarrow \Rep(\tGamma_K^{\cm})^{gen}$.
Recall from \eqref{eq:cWV_Y} that we have
\[
	\cWV = (V \otimes_\C \cW)^{\Gamma_\varphi} \in \Coh^\kappa(\TT^+_{Y^\diamond}, K \times \Gamma_{\tilde{\rho}}^{\cm}).
\]
By \Cref{lem:galois_pullback_adm_vb}, specifically \eqref{eq:cWV_FV_pullback}, we have a natural isomorphism 
	\[\cWV \simeq \tilde{\rho}^* \cF_V \quad \text{in} \quad \Coh^\kappa(\TT^+_{Y^\diamond}, K \times \Gamma_{\tilde{\rho}}^{\cm}).\] 
By \Cref{prop:extn_AOD} again, $\cWV$ extends uniquely to an object $\cWVbar$ in $\Coh^\kappa(\TT^+_Y, K \times \Gamma_K^{\cm})$; moreover, we have a unique isomorphism
\begin{equation} \label{eq:isom_cWVbar_Wbar_inv}
	\cWVbar \simeq (V \otimes_\C \overline{\cW})^{\Gamma_\varphi} \quad \text{in} \quad \Coh^\kappa(\TT^+_Y, K \times \Gamma_K^{\cm})
\end{equation} 
which restricts to the identity isomorphism of $\cWV$ by the uniqueness of the minimal extension.

By \Cref{prop:dual_quasi-dist}, $\widetilde{\orb}$ is $2$-leafless. By \Cref{prop:H_vanishing}, the smooth locus $X^{reg}$ is a strongly admissible graded symplectic variety with symplectic form of degree $2$. 
Let $(\OO_\hb, \epsilon, \beta_\hb, \Phi_\hb)$ be an even graded formal quantization of $X^{reg}$. Then the Hamiltonian $G \times \Gamma$-action on $X^{reg}$ lifts to a symmetrized quantum Hamiltonian $G \times \Gamma$-action on $\OO_\hb$ by \Cref{prop:quan_moment_symm}. 
Recall that in \Cref{subsec:quant_conical_symp_sing}, we have defined the canonical filtered quantization $\cA_0 = \Gamma(X, \OO_\hb)^{fin}/(\hb-1)$ of $\C[X^{reg}] = \C[X] = \C[\widetilde{\orb}]$. The quantum co-moment map $\Phi_\hb$ descends to a quantum co-moment map $\Phi: \g \to \cA_0$. We also regard $\cA_0$ as a $G \times \Gamma_K$-equivariant quantization where the $\Gamma_K$-action is the one induced from the $\Gamma$-action on $\cA_0$ and the composition $\Gamma_K \twoheadrightarrow \Gamma_{\tilde{\rho}} \subset \Gamma$.

At this point the construction proceeds in two related steps. First, we quantize the universal extended object $\overline{\cW}$ on $Y$ to obtain $W_{\orbk}$. Then, after tensoring with the chosen genuine representation $V$ and taking $\Gamma_\varphi$-invariants, we quantize the corresponding extended object $\cWVbar$ to obtain $\WV_{\orbk}$.

Let $\iota$ denote the closed embedding $\orbk \hookrightarrow \orb$ and let $\tilde{\iota}$ denote both the closed embeddings $\torbk \hookrightarrow \widetilde{\orb}$ and $Y \hookrightarrow X^{reg}$ by abuse of notation. We can now apply the constructions in \cite{LY} to quantize $\overline{\cW}$ over the Lagrangian subvariety $Y$ to obtain the following proposition. 
 
\begin{Prop} \label{prop:quant_cW}
	There exists an HC $(\cA_0, K \times \tGamma_K)$-module $W_{\orbk}$, with a compatible filtration indexed by $a + \Z$, for some $a \in \frac{1}{2}\Z$, and defined up to degree shift, satisfying the following properties:
    \begin{enumerate}
        \item There is a $K \times \tGamma_K^{\cm}$-equivariant injective homomorphism $\gr W_{\orbk} \hookrightarrow \Gamma(Y,\overline{\cW})$ of finitely generated $\C[X]$-modules, whose localization to $X^{reg}$ is a $K \times \tGamma_K^{\cm}$-equivariant isomorphism $\gr W_{\orbk}|_{X^{reg}} \simeq \tilde{\iota}_* \overline{\cW}$ of $\OO_{X^{reg}}$-modules.
        \item After forgetting the filtration, $W_{\orbk}$ as a HC $(\g, K \times \tGamma_K)$-module is independent of the choice of $K$-orbit $\torbk \subset \rho^{-1}(\orbk)$.
    \end{enumerate}  
\end{Prop}

\begin{proof}
	We follow the same argument as in \cite[Section 7.2]{LY}. Let $\OO_\hb$ be the canonical graded quantization of $X^{reg}$. We endow $\overline{\cW}$ with the $\cm$-action given by that of $\widetilde{T}^\circ$. By \Cref{rem:solid}, this action is solid. Then we can apply \Cref{thm:quan_lag_equiv} and \Cref{prop:tatp_symm} to quantize $\overline{\cW}$ over $Y$ to obtain a $K \times \tGamma_K$-equivariant graded Hamiltonian quantization $\overline{\cW}_\hb$ of $\overline{\cW}$ over $\OO_\hb$, with the graded structure coming from the action of $\widetilde{T}^\circ$. Since the action of $\tGamma_K \times \widetilde{T}$ on $\overline{\cW}$ factors through the quotient group $\tGamma_K^{\cm}$, the action of $\tGamma_K \times \widetilde{T}$ on $\overline{\cW}_\hb$ also factors through $\tGamma_K^{\cm}$.

	Next we take $W_{\orbk} = \Gamma(Y, \overline{\cW}_\hb)^{fin}/(\hb-1)$, where $(\cdot)^{fin}$ means taking the finite part of $\Gamma(Y, \overline{\cW}_\hb)$ with respect to the action of $\widetilde{T}^\circ$. Then $W_{\orbk}$ is naturally a filtered $(\cA_0, K \times \tGamma_K)$-module. When the covering morphism $\widetilde{T}^\circ \twoheadrightarrow T$ is an isomorphism as in case (a) of \Cref{subsubsec:strong_cm_action}, the filtration is indexed by $\Z$. Otherwise, $\widetilde{T} \twoheadrightarrow T$ is a connected $2$-fold covering as in case (b) of \Cref{subsubsec:strong_cm_action}, and the filtration is indexed by $\frac{1}{2} + \Z$.
	
	Thanks to the codimension condition \eqref{eq:codim_Y}, we can apply the same arguments as in the proofs of \cite[Theorem 7.2.3 and Proposition 7.3.1]{LY} to show that $W_{\orbk}$ is a finitely generated module over $\cA_0$ and $\gr W_{\orbk}$, as a module over $\gr\cA_0 \simeq \C[X]$, has support in $X$ equal to $\overline{Y}$. Moreover, we have a $K \times \tGamma_K^{\cm}$-equivariant injective homomorphism $\gr W_{\orbk} \hookrightarrow \Gamma(Y,\overline{\cW})$ of $\C[X]$-modules that localizes to an isomorphism $\gr W_{\orbk}|_{X^{reg}} \simeq \tilde{\iota}_* \overline{\cW}$. Therefore $W_{\orbk}$ satisfies property (1). Property (2) follows from \Cref{rem:choice_Y}.
\end{proof}

We next apply the same quantization construction to the sheaf $\cWVbar$ associated with $V$.

\begin{Prop} \label{prop:quant_cWV}
	For any admissible orbit datum $(\orbk, \cF_V)$ with the corresponding $V \in \Rep(\tGamma_K)^{gen}$ and any lifting of $V$ to an object in $\Rep(\tGamma_K^{\cm})^{gen}$, there exists a filtered HC $(\A, K \times \Gamma_{\tilde{\rho}})$-module $\WV = \WV_{\orbk}$ satisfying the following properties:
    \begin{enumerate}
        \item There is a $K \times \Gamma_{\tilde{\rho}}^{\cm}$-equivariant injective homomorphism $\gr \WV \hookrightarrow \Gamma(Y,\cWVbar)$ of finitely generated $\C[X]$-modules, whose localization to $X^{reg}$ is a $K \times \Gamma_{\tilde{\rho}}^{\cm}$-equivariant isomorphism $\gr \WV|_{X^{reg}} \simeq \tilde{\iota}_* \cWVbar$ of $\OO_{X^{reg}}$-modules.
        \item After forgetting the filtration, $\WV$ as a HC $(\g, K \times \Gamma_{\tilde{\rho}})$-module depends only on the admissible orbit datum $(\orbk, \cF_V)$ and is independent of the choice of the $K$-orbit $\torbk \subset \rho^{-1}(\orbk)$ and the choice of the lifting of $V$ to an object in $\Rep(\tGamma_K^{\cm})^{gen}$.
        \item There is a canonical isomorphism $\WV \simeq (V \otimes_\C W_{\orbk})^{\Gamma_\varphi}$ of filtered $(\g, K \times \Gamma_{\tilde{\rho}})$-modules (after shifting the filtration degrees on both sides appropriately).
    \end{enumerate}  
\end{Prop}
 
\begin{proof}
    The proof of properties (1) and (2) is similar to that of \Cref{prop:quant_cW}. In particular, we can construct a $K \times \Gamma_{\tilde{\rho}}$-equivariant graded Hamiltonian quantization $\cWVbar_\hb$ of $\cWVbar$ over $\OO_\hb$ by applying \Cref{thm:quan_lag_equiv} and \Cref{prop:tatp_symm}, and then take $\WV = \Gamma(Y, \cWVbar_\hb)^{fin}/(\hb-1)$. Different choices of the lifting of $V$ to an object in $\Rep(\tGamma_K^{\cm})^{gen}$ result only in integer shifts of the filtration indices on $W_{\orbk}$.

	To see property (3), we observe that the sheaf $(V \otimes_\C \overline{\cW}_\hb)^{\Gamma_\varphi}$ is also a $K \times \Gamma_{\tilde{\rho}}$-equivariant graded Hamiltonian quantization of $\cWVbar$ over $\OO_\hb$ by \eqref{eq:isom_cWVbar_Wbar_inv}. Finally, note that taking $\Gamma_\varphi$-invariant parts commutes with taking global sections and with taking $\cm$-finite parts.
\end{proof}

\subsection{Quantization and special unipotent representations} \label{subsec:quant_unip}

In this subsection, we will relate special unipotent representations to the quantization of nilpotent $K$-orbit covers discussed in \Cref{subsec:quant_K_orb}. We first set up some notation.

Let $\cUnipOv(\g,K)$ denote the full subcategory of $\MgK$ consisting of $(\g,K)$-modules $M$ that are annihilated by the maximal primitive ideal $\J_{\ckorb}$. Then $\UnipOv(\g,K) = \Irr \cUnipOv(\g,K)$. Let $\cUnipOv^K(\orbk)$ denote the full subcategory of $\cUnipOv(\g,K)$ consisting of $(\g,K)$-modules $M$ that satisfy $\AV(M) = \overline{\orb}_K$. Let $\UnipOv^K(\orbk) := \Irr \cUnipOv^K(\orbk)$ be the set of all equivalence classes of irreducible objects in $\cUnipOv^K(\orbk)$. Then $\UnipOv^K(\orbk)$ is naturally a subset of $\UnipOv(\g,K)$. In fact, we will see in \Cref{thm:Unip_AOD_bij} that $\UnipOv(\g,K)$ is a disjoint union of all $\UnipOv^K(\orbk)$ where $\orbk$ runs over all $K$-orbits in $\orb_{\fk^\perp}$. 

To prepare for this, we will construct a collection of modules in $\UnipOv^K(\orbk)$ via quantization, which will later be shown to exhaust $\UnipOv^K(\orbk)$, as follows.
Retaining the setting and notation in \Cref{subsec:quant_K_orb}, for any $V \in \Rep(\tGamma_K)^{gen}$, we define the HC $(\A^\Gamma, K)$-module
\begin{equation} \label{eq:EV}
    \MV = \MV_{\orbk} := (V \otimes_\C W_{\orbk})^{\Gamma_K}.
\end{equation} 
By \Cref{prop:quant_cWV}(3), we have an isomorphism of HC $(\A^\Gamma, K)$-modules
\begin{equation} \label{eq:MV_WV}
	\MV = [(V \otimes_\C W_{\orbk})^{\Gamma_\varphi}]^{\Gamma_{\tilde{\rho}}} \simeq (\WV_{\orbk})^{\Gamma_{\tilde{\rho}}}.
\end{equation}
Moreover, the good $K \times \Gamma_{\tilde{\rho}}$-stable filtration on $\WV = \WV_{\orbk}$ induces a good $K$-invariant ascending filtration (bounded from below) on $\MV$:
	\[ \MV_{\leqslant i} := (\WV_{\leqslant i})^{\Gamma_{\tilde{\rho}}}, \quad \forall \, i \in \Z. \]

\begin{Prop} \label{prop:MV_AC}
For any $V \in \Rep(\tGamma_K)^{gen}$, the following conclusions hold:
\begin{enumerate}
	\item 
		The HC module $\MV$ only depends on the admissible orbit datum $(\orbk, \cF_V)$ and is independent of the choice of the $K$-orbit $\torbk \subset \rho^{-1}(\orbk)$. 
	\item 
		We have
		\begin{equation}
			\AC(\MV) = [\cF_V] \cdot \orbk \in \AC(\g,K),
		\end{equation}
		where $\cF_V$ is as in \Cref{prop:classification_AOD}. Therefore we have a well-defined additive functor 
		\[ \quan_{\ckorb, \orbk}^K: \Rep(\tGamma_K)^{gen} \simeq \Coh(\TT^+_{\orbk}, K) \to \cUnipOv^K(\orbk), \quad V \mapsto \MV, \]
		called the \emph{quantization functor}.
	\item
		If $V \in \Irr(\tGamma_K)^{gen}$, then $\MV$ has irreducible associated cycle. In particular, $\MV$ is an irreducible $(\g,K)$-module and hence the functor $\quan_{\ckorb, \orbk}^K$ in (2) sends irreducibles to irreducibles.
	\end{enumerate}
\end{Prop}

\begin{proof}
For (1), the claim follows from \Cref{rem:choice_Y}.

For (2), first note that by \Cref{prop:quant_cWV}(1), we have $\AV(\MV) = \overline{\orb}_K$. Next, observe that taking $\Gamma_{\tilde{\rho}}$-invariants is an exact functor and commutes with localizations. Therefore we have a canonical isomorphism
\begin{equation}
	\gr \MV = \gr [(\WV)^{\Gamma_{\tilde{\rho}}}] \simeq (\gr \WV)^{\Gamma_{\tilde{\rho}}}
\end{equation}
of $(S\fg, K)$-modules, for any $V \in \Irr(\tGamma_K)^{gen}$. 
Then the localization $\gr \MV|_{\orb}$ of $\gr \MV$ over $\orb$ is given by
\begin{equation}\label{eq:grMV_loc}
	\gr \MV|_{\orb} \simeq (\gr \WV)^{\Gamma_{\tilde{\rho}}}|_{\orb} \simeq (\gr \WV|_{\orb})^{\Gamma_{\tilde{\rho}}}
\end{equation}
where $\gr \WV$ is regarded as a module over $\C[\orb]$ via the pullback morphism $\rho^*: \C[\orb] \to \C[\widetilde{\orb}] = \C[X]$. 

On the other hand, \Cref{prop:quant_cWV} and base change imply that
\[
	\gr \WV|_{\orb} \simeq (\rho_* \tilde{\iota}_* \cWVbar)\big|_{\orb} \simeq \rho_* \left[ (\tilde{\iota}_* \cWVbar) \big|_{\widetilde{\orb}}\right] \simeq \rho_*\tilde{\iota}_* \left(\cWVbar \big|_{Y^\diamond} \right) \simeq \iota_* \tilde{\rho}_* \cWV.
\]
Therefore by \eqref{eq:grMV_loc}, we have 
\begin{equation}\label{eq:MV_AC}
	\gr \MV|_{\orb} \simeq (\iota_* \tilde{\rho}_* \cWV)^{\Gamma_{\tilde{\rho}}} = \iota_* (\tilde{\rho}_* \cWV)^{\Gamma_{\tilde{\rho}}} \simeq \iota_* \cF_V,
\end{equation} 
which gives the main statement of part (2). 

For (3), by the proof of the second statement of \cite[Proposition 7.3.1]{LY}, the maximality of the ideal $\J_{\ckorb} = I_0(\widetilde{\orb})$ (\Cref{prop:dual_quasi-dist} (iii)), together with the irreducibility of $\cF_V$ as a $K$-equivariant vector bundle over $\orbk$, implies that $\MV$ is irreducible by part (2). 
\end{proof}

\begin{Cor} \label{cor:quant_map}
	For any quasi-distinguished $\ckorb \in \check{\cN}_o$, let $\orb=d(\ckorb)$, let $\widetilde{\orb}$ be its universal cover, and set $X = \spec(\C[\widetilde{\orb}])$. Assume that $\codim(\partial_{sp} \orb, \overline{\orb}) \geqslant 6$. Then we have an injective map (called the \emph{quantization map})
	\[ \quan_{\ckorb}^K: \AODK(\orb) \hookrightarrow \UnipOv(\g, K), \quad (\orbk, \cF_V) \mapsto \MV, \]
	where, for each $\orbk \subset \orb_{\fk^\perp}$, $V \in \Irr(\tGamma_K)^{gen}$ and $\cF_V \in \AODK(\orbk)$ is as in \Cref{prop:classification_AOD}. 

	Moreover, the composite map $\AC \,\circ \,\quan_{\ckorb}^K: \AODK(\orb) \to \AC(\g,K)$ coincides with the inclusion map $\AODK(\orb) \hookrightarrow \AC(\g,K)$ as in \Cref{subsec:AC}. 
\end{Cor}

We now generalize the quantization map to the case of all quasi-distinguished orbits $\ckorb \in \check{\cN}_o$. In view of \cref{thm:codim_X}, it remains to consider the case when $\codim(\partial \orb, \overline{\orb}) \geqslant 4$. 

In this case, we have an equality $I_0(\widetilde{\orb}) = I_0(\orb)$ of maximal primitive ideals by \cref{prop:almost_etale_unip_ideal}. Then $I_0(\orb)$ equals the special unipotent ideal $\J_{\ckorb}$ attached to $\ckorb$ by \cref{prop:dual_quasi-dist} (iii). Combining this with the assumption that $(\g,K)$ corresponds to linear reductive groups $\GR$, \cite[Theorem 1.3.1]{LosevYu} implies that the associated cycle map $\AC: \UnipOv(\g, K) \to \AC(\orb, K)$ is always a bijection onto the subset $\AODK(\orb)$ of $\AC(\orb, K)$ when $\codim(\partial \orb, \overline{\orb}) \geqslant 4$. Moreover, in this case the inverse of the map $\AC: \UnipOv(\g, K) \to \AODK(\orb)$ is given by the same construction of $\quan_{\ckorb}^K$ above applied to $\orb$ or $\widetilde{\orb}$, see \cite[Section 5.7]{LosevYu}. 

Note that it can happen that $\orb$ satisfies both $\codim(\partial_{sp} \orb, \overline{\orb}) \geqslant 6$ and $\codim(\partial \orb, \overline{\orb}) \geqslant 4$, but the discussions above show that in this case the definition of the quantization map coincides with the inverse of $\AC$.
We therefore denote the relevant map uniformly by $\quan_{\ckorb}^K$. We summarize in the following proposition.

\begin{Prop} \label{prop:quant_map_quasi-dist}
	For any quasi-distinguished nilpotent orbit $\ckorb \in \check{\cN}_o$, we have an injective quantization map
	\[ \quan_{\ckorb}^K: \AODK(\orb) \hookrightarrow \UnipOv(\g, K), \]
	such that the composite map $\AC \,\circ \,\quan_{\ckorb}^K: \AODK(\orb) \to \AC(\g,K)$ coincides with the inclusion map $\AODK(\orb) \hookrightarrow \AC(\g,K)$ as in \Cref{subsec:AC}. 

	Moreover, when $\codim(\partial \orb, \overline{\orb}) \geqslant 4$, $\quan_{\ckorb}^K$ is a bijection with inverse given by $\AC$. 
\end{Prop}

We will show in \cref{thm:Unip_AOD_bij} that the quantization map $\quan_{\ckorb}^K$ is always a bijection for any quasi-distinguished nilpotent orbit $\ckorb \in \check{\cN}_o$.

\subsection{Counting and exhaustion} \label{subsec:counting}

Throughout this subsection, we assume that the nilpotent orbit $\ckorb$ in $\ckfg$ is quasi-distinguished. We will show that the number of irreducible special unipotent HC $(\g, K)$-modules attached to $\ckorb$ coincides with the number of admissible orbit data in $\orb = d(\ckorb)$ for $(\g, K)$. Together with the construction in \Cref{subsec:quant_unip}, this shows that all special unipotent irreducible HC $(\g, K)$-modules attached to $\ckorb$ can be obtained from our construction.

\begin{Thm} \label{thm:Unip_AOD_bij}
	Let $G$ be a connected complex simple group with Lie algebra $\g$. Consider the following two cases:
	\begin{enumerate}
		\item $\ckfg$ is the Langlands dual Lie algebra of $\g$, $d$ is the usual Barbasch-Vogan duality map, and $\GR = G(\R)^\circ$ so that $K = (G^\theta)^\circ$;
		\item $\ckfg$ is the metaplectic dual Lie algebra of $\g$, $d$ is the metaplectic Barbasch-Vogan duality map, and $\GR = \Mp(2n,\R)$.
	\end{enumerate}
	Then in both cases, for any quasi-distinguished nilpotent orbit $\ckorb \in \check{\cN}_o$ with $\orb=d(\ckorb)$, the quantization map 
	\[ \quan_{\ckorb}^K: \AODK(\orb) \xrightarrow{\ \scriptstyle\sim \ } \UnipOv(\g, K) \]
	in \Cref{prop:quant_map_quasi-dist} is a bijection, whose inverse is given by the associated cycle map $\AC$ defined in \Cref{subsec:AC}.
\end{Thm}

\begin{proof}
	Without loss of generality, we can assume that $G$ is simple and simply connected. 
	In view of \Cref{prop:quant_map_quasi-dist}, it suffices to show that $|\UnipOv(\g, K)| = |\AODK(\orb)|$ or simply $|\UnipOv(\g, K)| \leqslant |\AODK(\orb)|$, where $|\cdot|$ stands for the cardinality of sets. This will be confirmed in \Cref{cor:counting_classical} for classical $\g$ and \Cref{prop:counting_exceptional} for exceptional $\g$ below.
\end{proof}

\begin{Rem}
	Case (1) of \Cref{thm:Unip_AOD_bij} can be easily extended to the general setting of any linear real reductive group. We omit the details, since the current form of \Cref{thm:Unip_AOD_bij} is enough for the unitarity results in \cref{thm:unitarity_dist} and \cref{cor:unitarity_dist_reductive}.
\end{Rem}

The rest of this subsection verifies this cardinality statement in stages. We first treat classical and metaplectic groups using the counting results of \cite{BMSZ:counting, BMSZ:construction_unitarity}. We then pass to identity components by restriction (for $\mathrm{O}(p,q)$), handle spin groups through genuine admissible data, and finally treat exceptional groups by combining component-group counts with \texttt{atlas} computations.

\subsubsection{Counting for classical groups} \label{subsubsec:counting_classical}

We first treat the case of classical groups by using the counting results in \cite{BMSZ:counting}. Note that \cite{BMSZ:counting} and \cite{BMSZ:construction_unitarity} only deal with $\mathrm{O}(p,q)$ and $\SO(p,q)$ rather than their identity component $\SO_0(p,q)$, so some extra care is needed. 

Let $\g = \fo(N)$ with $N \geqslant 7$ and let $\orb = \orb_\lambda$ be a general nilpotent $\SO(N)$-orbit in $\g^*$ (not necessarily of the form $\orb = d(\ckorb)$ for some quasi-distinguished $\ckorb \in \check{\cN}_o$) that is not very even, i.e., its partition $\lambda$ has at least one odd part. Then $\orb$ is also an 	$\mathrm{O}(N)$-orbit. 
Let $\GR = \mathrm{O}(p,q)$ with $N = p + q$ and $\GR^+=\SO(p,q)$. We take the complexifications of maximal compact subgroups of $\GR$ and $\GR^+$ to be $K = \mathrm{O}(p,\C) \times \mathrm{O}(q,\C)$ and $K^+ = S(\mathrm{O}(p) \times \mathrm{O}(q)) \subset K$ respectively, where $K^+$ consists of all pairs $(g_1, g_2) \in \mathrm{O}(p) \times \mathrm{O}(q)$ such that $\det(g_1)\det(g_2) = 1$. We have the following lemma.

\begin{Lem} \label{lem:res_Opq}
	Let $\orb$ be a nilpotent $\SO(N)$-orbit in $\g = \fso(N)$ that is not very even. Then any $K^+$-orbit $\orb_{K^+}$ in $\orb$ is also a $K$-orbit. Moreover, the restriction of any irreducible $K$-admissible vector bundle over $\orbk$ to $K^+$ is still irreducible, and this gives a $2$-to-$1$ surjective restriction map 
	\[\res_{K^+}^{K}: \AODK(\orb) \to \AOD_{K^+}(\orb).\] 
	In particular, we have 
	\[|\AODK(\orb)| = 2|\AOD_{K^+}(\orb)|.\] 
\end{Lem}

\begin{proof}
	Since $\orb$ is not very even, it is also an $\mathrm{O}(N)$-orbit. 
	One can deduce from this that any $K^+$-orbit $\orbk$ in $\orb$ is also a $K$-orbit (see \cite[Theorem 9.3.4]{CM}). 
	Moreover, for any normal $\slf_2$-triple $(e, h, f)$ of $\orbk$, its $K$-centralizer $K_Q$ is of the form $A \times B$, where $A$ is the direct product of several $\GL(a_i)$-factors corresponding to even parts of $\lambda$ and $B$ is the direct product of several $\mathrm{O}(b_j)$-factors corresponding to odd parts of $\lambda$, and the $K^+$-centralizer $K_Q^+$ is of the form $A \times B^+$, where $B^+$ denotes the subgroup in $B$ consisting of matrices of determinant $1$. 

	By \cite[Lemma 7]{Ohta:Adm}, the fiber $V$ of a $K$- or $K^+$-equivariant admissible vector bundle $\cV$ over $\orbk$ at $e$, as a representation of the Lie algebra $\fk_Q = \Lie(K_Q) = \Lie(K_Q^+)$, is trivial on the $\fo(b_j)$-factors. Therefore, for a $K$-equivariant admissible vector bundle $\cV$, the $K_Q$-representation $V$ factors through a representation of $A \times \pi_0(B)$, where $\pi_0(B) = \pi_0(K_Q) \simeq \fS_2^l$ for some positive integer $l$, and hence $\pi_0(B^+) = \pi_0(K_Q^+) \simeq \fS_2^{l-1}$.
	This implies the remaining statements.
\end{proof}

Now we come back to the case when $\orb=d(\ckorb)$ for quasi-distinguished nilpotent orbit $\ckorb \in \check{\cN}_o$.

\begin{Prop} \label{prop:counting_classical}
	Let $\GR$ be one of the following two types of groups:
	\begin{enumerate}
		\item $\GR = G(\R)$, where $G$ is a (connected) complex linear simple classical group;
		\item $\GR = \Mp(2n,\R)$ is a metaplectic group.
	\end{enumerate}	
	Then for any quasi-distinguished nilpotent orbit $\ckorb \in \check{\cN}_o$ with $\orb=d(\ckorb)$, we have
	\[ |\UnipOv(\g, K)| \leqslant |\AODK(\orb)|.\]
	In particular, the quantization map $\quan_{\ckorb}^K$ is a bijection.
\end{Prop}

\begin{proof}
	In \cite[Section 3.2]{BMSZ:construction_unitarity}, a set of extended parameters $\PBPGext$ is defined for the real classical group $\GR$, based on a set $\PBPG$ of bipartitions, which was defined earlier in \cite{BMSZ:counting}. 
	Then we have the inequality $|\UnipOv(\g, K)| \leqslant |\PBPGext|$ by \cite[Corollary 5.4]{BMSZ:counting} and case-by-case computations in Sections 6--10 of \cite{BMSZ:counting}. See in particular \cite[Corollary 8.7]{BMSZ:counting}. 

	Since $\ckorb$ is quasi-distinguished, we have $|\AODK(\orb)| = |\PBPGext|$ by \cite[Proposition 4.11]{BMSZ:construction_unitarity} and hence the proposition follows. In more detail, \cite[Proposition 4.11]{BMSZ:construction_unitarity} says that, when $\g_\R$ is not $\fso(p,q)$ with $p + q > 0$, there is a bijection $\PBPGext \xrightarrow{\scriptstyle\sim} \AODK(\orb)$. The case of $\SO(p,q)$ with $p + q > 0$ requires a little more elaboration since \cite[Proposition 4.11]{BMSZ:construction_unitarity} deals with $\mathrm{O}(p,q)$ instead of $\SO(p,q)$. Let $K$ and $K^+$ be as defined before \Cref{lem:res_Opq}, then \cite[Proposition 4.11]{BMSZ:construction_unitarity} states that there is a bijection 	
	\[ \PBP_{\SO(p,q)}^{\mathrm{ext}}(\ckorb) \times \Z_2 \xrightarrow{\scriptstyle\sim} \AODK(\orb), \] 
	which implies that $|\AODK(\orb)| = 2|\PBP_{\SO(p,q)}^{\mathrm{ext}}(\ckorb)|$. On the other hand, since $\orb=d(\ckorb)$ for quasi-distinguished $\ckorb$ is not very even by \cref{rem:dist_very_even}, \Cref{lem:res_Opq} implies that $|\AODK(\orb)| = 2|\AOD_{K^+}(\orb)|$, hence we also have $|\AOD_{K^+}(\orb)| = |\PBP_{\SO(p,q)}^{\mathrm{ext}}(\ckorb)|$.
\end{proof}

\begin{Rem} \label{rmk:counting}
	We remark that all the results cited in the proof of \Cref{prop:counting_classical} above only use the theory of cells and the Springer correspondence, and do not rely on any construction of special unipotent representations.
	The inequalities in \Cref{prop:counting_classical} and \cite[Proposition 3.6]{BMSZ:construction_unitarity} are actually equalities by \cite[Theorem 2.28]{BMSZ:counting}, assuming \cite[Conjecture 4.14]{BMSZ:counting}, which is verified for many cases including all classical Lie algebras and metaplectic groups in \cite[Corollary 4.24]{BMSZ:counting}. This requires knowledge of the intricate relationship between Harish-Chandra cells and double cells (see \cite[Section 4.5]{BMSZ:counting}). 

	\Cref{thm:Unip_AOD_bij} for classical types (except for the case of $\g_\R = \fso(p,q)$) and metaplectic type was previously established in \cite[Theorem 5.3, (b)]{BMSZ:construction_unitarity}. There the authors constructed special unipotent representations attached to general nilpotent orbits $\ckorb$ using theta lifting (\cite[Theorem 5.1]{BMSZ:construction_unitarity}). To show that the construction exhausts all the (metaplectic) special unipotent representations, they only use the inequality in \Cref{prop:counting_classical} (or rather \cite[Proposition 3.6]{BMSZ:construction_unitarity}). 
\end{Rem}

The proof of the following corollary uses the compatibility between quantization maps and restriction maps, whose discussion is postponed to \Cref{sec:res_quant}.

\begin{Cor} \label{cor:counting_classical_identity_component}
	Let $\GR$ be one of the following two types of groups:
	\begin{enumerate}
		\item $\GR = G(\R)^\circ$, where $G$ is a (connected) complex linear simple classical group with real points subgroup $G(\R)$;
		\item $\GR = \Mp(2n,\R)$ is a metaplectic group.
	\end{enumerate}	
	Then the quantization map 
	\[\quan_{\ckorb}^{K}: \AOD_{K}(\orb) \xrightarrow{\ \sim \ } \UnipOv(\g, K)\]
	is a bijection.
\end{Cor}

\begin{proof}
	In view of \Cref{prop:counting_classical}, we only need to consider the case when $G=\SO(N)$ and $\GR = \SO_0(p,q)$ is the identity component of $\SO(p,q)$ with $N=p+q$. Let $K' = G^\theta$ and $K = (K')^\circ$, so that $K' \simeq S(\mathrm{O}(p) \times \mathrm{O}(q))$ and $K = (K')^\circ \simeq \SO(p) \times \SO(q)$. We apply \Cref{cor:quant_res_bij} to establish the bijectivity of $\quan_{\ckorb}^{K}$.
\end{proof}

We now consider the case when $G=\Spin(N)$. In \cref{prop:counting_spin} below, we follow the notation in \Cref{subsec:adm_special_orbit} and \Cref{subsec:AC}.

\begin{Prop} \label{prop:counting_spin}
	Assume $\g = \fso(N)$ and let $\g_\R$ be a real form of $\g$ corresponding to a Cartan involution $\theta$. Set $G = \Spin(N)$, $\bar{G} = \SO(N)$, $K = G^\theta$, and $\bar{K} = (\bar{G}^\theta)^\circ$. Then the quantization map 
	\[\quan_{\ckorb}^{K}: \AOD_{K}(\orb) \xrightarrow{\ \sim \ } \UnipOv(\g, K)\]
	is a bijection.
\end{Prop}

\begin{proof}
    Consider the following commutative diagram induced by \Cref{prop:quant_pullback}:
	\[
	\begin{tikzcd}
		\AOD_{\bar{K}}(\orb) \ar[d, hook] \ar[r,hook, "\quan_{\ckorb}^{\bar{K}}"] & \UnipOv(\g, \bar{K}) \ar[d, hook] \\
		\AODK(\orb) \ar[r, hook, "\quan_{\ckorb}^{K}"] & \UnipOv(\g, K) 
	\end{tikzcd}
	\]
	where the vertical maps are defined via inflation. The upper horizontal map is a bijection by \Cref{cor:counting_classical_identity_component}.
	It suffices to show that $\AODK(\orb)^{gen} = \varnothing$ and $\UnipOv(\g, K)^{gen} = \varnothing$. By \Cref{cor:AC_special_gen}, it suffices to show that the restriction of $\quan_{\ckorb}^{K}$ to the subset $\AODK(\orb)^{gen}$ of genuine admissible orbit data 
    \begin{equation} \label{eq:quant_restrict_gen}
        \quan_{\ckorb}^{K}|_{AODK(\orb)^{gen}}: \AODK(\orb)^{gen} \hookrightarrow \UnipOv(\g, K)^{gen} 
    \end{equation}
    is a bijection. There are two cases: 
    \begin{itemize}[itemindent=*, leftmargin=*]
        \item When $\g_\R = \fso(p,q)$, we have $|\AODK(\orb)^{gen}| = |\UnipOv(\g, K)^{gen}|$ by \Cref{prop:adm_Spin_pq} and the counting result \cite[Theorem 3]{BMSZ:spin}, hence bijectivity follows. 
        \item When $\g_\R = \fso^*(2n)$, we have $\AODK(\orb)^{gen} = \varnothing$ by \Cref{prop:adm_SO*(2n)}, which implies that $\UnipOv(\g, K)^{gen}$ is also empty by \Cref{cor:AC_special_gen}. 
    \end{itemize}
    In both cases, \eqref{eq:quant_restrict_gen} is a bijection.
\end{proof}

\begin{Rem}
    When $\ckorb$ is distinguished, \cref{prop:adm_Spin_pq} and \cref{prop:adm_SO*(2n)} imply that $\AODK(\orb)^{gen} = \varnothing$ and hence $\UnipOv(\g, K)^{gen} = \varnothing$, without appealing to the counting result in \cite[Theorem 3]{BMSZ:spin}.
\end{Rem}

Putting \Cref{cor:counting_classical_identity_component} and \Cref{prop:counting_spin} together, we have

\begin{Cor} \label{cor:counting_classical}
	Let $\GR$ be one of the following two types of groups:
	\begin{enumerate}
		\item $\GR = G(\R)$, where $G$ is a simply connected complex simple classical group;
		\item $\GR = \Mp(2n,\R)$ is a metaplectic group.
	\end{enumerate}	
	Then we have
	\[ |\UnipOv(\g, K)| = |\AODK(\orb)|. \]
	In particular, the quantization map $\quan_{\ckorb}^K$ is a bijection.
\end{Cor}

\subsubsection{Counting for exceptional groups} \label{subsubsec:counting_exceptional}

Now we treat the case of exceptional groups. Throughout this subsection, $\g$ is a complex simple exceptional Lie algebra and $\g_\R$ is a real form of $\g$ corresponding to a Cartan involution $\theta$. Let $\bar{G} = \Ad(\g)$ denote the adjoint group of $\g$ and let $\bar{K} := (\bar{G}^\theta)^\circ$, which is the complexification of the identity component of a maximal compact subgroup of the adjoint group $\bar{G}_\R = \Ad(\g_\R)$ of $\g_\R$. Let $G$ be the universal covering group of $\bar{G}$ and $K = G^\theta$.

\begin{Prop} \label{prop:counting_exceptional}
	Let $G$ be a simply connected complex simple exceptional group with Lie algebra $\g$. For any real form $\g_\R$ of $\g$ corresponding to a Cartan involution $\theta$, let $K = G^\theta$. Let $\bar{G}$ be the adjoint group of $\g$ and $\bar{K} = (\bar{G}^\theta)^\circ$. Then all the maps in the following commutative diagram induced by \Cref{prop:quant_pullback}
	\begin{equation} \label{eq:quant_infl_exceptional}
		\begin{tikzcd}
			\AOD_{\bar{K}}(\orb) \ar[d, hook] \ar[r,hook, "\quan_{\ckorb}^{\bar{K}}"] & \UnipOv(\g, \bar{K}) \ar[d, hook] \\
			\AODK(\orb) \ar[r, hook, "\quan_{\ckorb}^{K}"] & \UnipOv(\g, K) 
		\end{tikzcd}
	\end{equation}
	are bijections, where the vertical maps are defined via inflation.
\end{Prop}

The rest of this subsection proves \Cref{prop:counting_exceptional}. In view of \cref{prop:quant_map_quasi-dist}, we only need to consider the case when $\codim(\partial \orb, \overline{\orb}) = 2$.
In \Cref{tab:qD_real_exceptional}, we list, for all real simple exceptional groups, all $K$-orbits $\orbk$ in $\orb_{\fk^\perp}$. Here $\orb$ is a special nilpotent orbit in $\fg$ such that $\orb=d(\ckorb)$, where $\ckorb$ is a quasi-distinguished nilpotent orbit in $\ckfg$, and $\overline{\orb}$ contains at least one nilpotent orbit $\orb'$ of codimension $2$, which must lie in $\cP(\orb)$ and turns out to be unique. We also list all $K$-orbits $\orbk$ in $\orb_{\fk^\perp}$ and $K$-orbits $\orbk'$ lying in $\overline{\orb}_K \cap \orb'$. We follow the labeling of the real forms $\g_\R$ and the $K$-orbits in \cite{Dk1, Dk6, Dk5, Dk2, Dk3, Dk4, Dk7, Dk8}, which contain the closure relations of all the $K$-orbits of all real exceptional groups. 

Note that there are three cases in \Cref{tab:qD_real_exceptional} in which the column of $\orbk'$ is marked as `none' since the closure of $\orbk$ does not intersect with $\orb'$. 
In all these cases, $\overline{\orb}_K$ only intersects smaller orbits in $\overline{\orb}$ of codimension $\geqslant 6$. In these cases, we can also apply the results of \cite{LY}, or directly appeal to \cite[Theorem 1.3.1]{LosevYu}, to obtain exhaustion of $\UnipOv(\g,K)$ without counting.

We also list the $\bar{K}$-component groups $\Z_{\bar{K}}$ of all $\bar{K}$-orbits in \Cref{tab:qD_real_exceptional}. This will be useful in counting the number of irreducible admissible vector bundles over $\orbk$ (see \Cref{lem:adm_count_exceptional} below).
However, \cite{King} only computes the component groups for the adjoint real groups, so extra care is needed. Here $K$-orbits in $\orb_{\fk^\perp}$ coincide with $\bar{K}$-orbits. Let $Z_{\bar{K}}$ denote the $\bar{K}$-component group of a $K$-orbit $\orbk$ (this group is denoted as $Z_K$ in \cite{King}). We have a surjective covering map $K \twoheadrightarrow \bar{K}$, which induces a surjective homomorphism $K_Q \twoheadrightarrow \bar{K}_Q$ between reductive centralizers of a point $e \in \orbk$, and in turn a surjective homomorphism $Z_K \twoheadrightarrow Z_{\bar{K}}$ between component groups. In general, the map $Z_K \twoheadrightarrow Z_{\bar{K}}$ does not have to be an isomorphism. 
The middle column of Table \ref{tab:qD_real_exceptional} records the component group $Z_{\bar{K}}$ of the $\bar{K}$-centralizer of $\orbk$, where $\fD_n$ denotes the dihedral group of order $2n$. 

\begin{Rem}\label{rem:Z_K}
	Note that the $K$-orbit $\# 14$ of $E_{8(8)}$ is marked in red because \cite[Table 11]{King} mistakenly records the $\bar{K}$-component group $Z_{\bar{K}}$ as the trivial group, while the correct answer should be $\fS_2$. We postpone the proof of \Cref{lem:Z_K_correction} to \Cref{sec:correction_component}. The same table row also mistakenly records the component group $A(\orb)$ of $\orb$ as the trivial group, which should be $\fS_2$ by \cite[Table 5]{Alexeevski}. 

	The author has also written a GAP program to recalculate the $\bar{K}$-component groups for adjoint exceptional groups based on the approach of \cite{KingNoel}. The computations found no other mistakes in the tables of \cite{King} for the $K$-orbits in \cref{tab:qD_real_exceptional}.
\end{Rem}

With the data of component groups in place, we now turn to the counting of admissible vector bundles.
We first count the number of $\bar{K}$-equivariant (irreducible) admissible vector bundles for each $\orbk$ in Table \ref{tab:qD_real_exceptional}. We adopt the notation in \Cref{subsec:AOD}. All $K$-orbits $\orbk$ in Table \ref{tab:qD_real_exceptional} are linearly admissible by \cite{Noel1, Noel2}. Recall that by \Cref{lem:bij_irr_gen} and \Cref{prop:classification_AOD}, for any $\orbk \subset \orb_{\fk^\perp}$, we have a bijection $\AOD_{\bar{K}}(\orbk) \simeq \Irr(\widetilde{Z}_{\bar{K}})^{gen}$, where $\widetilde{Z}_{\bar{K}} \simeq \tGamma_{\bar{K}}$ is a central extension of $Z_{\bar{K}} \simeq \Gamma_{\bar{K}}$ by $\Z_2$. It is in general not easy to determine this central extension.
Fortunately, and somewhat surprisingly, it turns out that we only need to know the $\bar{K}$-component group $Z_{\bar{K}}$ in the present situation by the following lemma.

\begin{Lem} \label{lem:adm_count_exceptional}
	For any $K$-orbit $\orbk$ in \cref{tab:qD_real_exceptional}, we have 
	\[ |\AOD_{\bar{K}}(\orbk)| = |\Irr(\widetilde{Z}_{\bar{K}})^{gen}| = |\Irr(Z_{\bar{K}})|. \] 
\end{Lem}

\begin{proof}
The first equality follows from \Cref{prop:classification_AOD}. 

Observe from \cref{tab:qD_real_exceptional} that, if $Z_K$ is not isomorphic to $1$ (the trivial group), $\fS_2$, or $\fS_3$, then there are three possibilities: 
\begin{enumerate}
	\item $\orb=F_4(a_3)$ in $\mathfrak{f}_4$, any $\orbk \subset \orb_{\fk^\perp}$;
	\item $\orb=E_8(a_7)$ in $\mathfrak{e}_8$, any $\orbk \subset \orb_{\fk^\perp}$;
	\item $\orb=D_4(a_1)$ in $\mathfrak{e}_6$, $\g_\R$ = {\bf E I}, $\orbk = \#23$ and $Z_{\bar{K}} = \fS_4$.
\end{enumerate}
In the first two cases $\q = 0$ since the orbits $\orb$ are distinguished, so that $\fk_Q = 0$. In the third case, we also have $\fk_Q=0$ by \cite[Table VIII]{Djokovic:centralizer_outer}. Therefore in all three cases, $\bar{K}_Q = Z_{\bar{K}}$. In these cases, $|\AOD_{\bar{K}}(\orbk)| = |\Irr(Z_{\bar{K}})|$. The case of $Z_{\bar{K}} = 1$ is also trivial. 

Now consider the remaining cases when $Z_{\bar{K}} \simeq \fS_2$ or $\fS_3$. When $Z_{\bar{K}} \simeq \fS_2$, then $\widetilde{Z}_{\bar{K}} \simeq \Z_2 \times \fS_2$ or $\Z_4$. In both cases, $|\Irr(\widetilde{Z}_{\bar{K}})^{gen}| = 2 = |\Irr(Z_{\bar{K}})|$. When $Z_{\bar{K}} \simeq \fS_3$, then either $\widetilde{Z}_{\bar{K}} \simeq \Z_2 \times \fS_3$ or $\widetilde{Z}_{\bar{K}} \simeq \mathfrak{BD}_3$, the binary dihedral group of order $12$. In both cases, $|\Irr(\widetilde{Z}_{\bar{K}})^{gen}| = 3 = |\Irr(Z_{\bar{K}})|$. Indeed, in these situations, there exists a (non-unique) one-dimensional genuine representation of $\widetilde{Z}_{\bar{K}}$. Tensoring with such a character induces a (non-canonical) bijection between $\Irr(Z_{\bar{K}})$ and $\Irr(\widetilde{Z}_{\bar{K}})^{gen}$.
\end{proof}

\begin{proof}[Proof of \Cref{prop:counting_exceptional}]
	We compute $|\UnipOv(\g, K)|$ using the \texttt{atlas} software (see \cite{atlas}) and record the result in the rightmost column of Table \ref{tab:qD_real_exceptional}. 
	By \Cref{lem:adm_count_exceptional}, $|\AOD_{\bar{K}}(\orbk)| = |\Irr(Z_{\bar{K}})|$ for each $K$-orbit $\orbk$ in Table \ref{tab:qD_real_exceptional}. One can sum $|\Irr(Z_{\bar{K}})|$ for various $\orbk \subset \orb$ in each row of \Cref{tab:qD_real_exceptional} and see that the total coincides with $|\UnipOv(\g, K)|$. This forces all sets in \eqref{eq:quant_infl_exceptional} to have equal cardinality and hence all four maps are bijections.
\end{proof}

\begin{Rem} \label{rem:atlas_exceptional}
	One can also find the list and number of all special unipotent representations of real simple exceptional groups on the \texttt{atlas} website \cite{atlas_exceptional}.
\end{Rem}

\begin{Ex}
	For the split real form {\bf E VIII} of $E_8$ with $\ckorb=\orb=E_8(a_7)$, the three $K$-orbits $\#66$, $\#67$, and $\#68$ have $Z_{\bar{K}}=\fD_4$, $\fS_5$, and $\fD_6$, respectively, as shown in \Cref{tab:qD_real_exceptional}; hence the corresponding numbers of admissible vector bundles are $|\Irr(\fD_4)|=5$, $|\Irr(\fS_5)|=7$, and $|\Irr(\fD_6)|=6$, and their sum is $5+7+6=18$, matching the total number of representations found by \texttt{atlas}. 
\end{Ex}

\begingroup
\small
\setlength{\tabcolsep}{3pt}
\renewcommand{\arraystretch}{1.05}
\renewcommand{\multirowsetup}{\centering}
\newlength{\excdistcolw}
\newlength{\excdistcolA}\newlength{\excdistcolB}\newlength{\excdistcolC}\newlength{\excdistcolD}
\newlength{\excdistcolE}\newlength{\excdistcolF}\newlength{\excdistcolG}\newlength{\excdistcolH}
\setlength{\excdistcolw}{\dimexpr(\linewidth-16\tabcolsep-9\arrayrulewidth)/8\relax}
\setlength{\excdistcolA}{\dimexpr 1.00\excdistcolw\relax}
\setlength{\excdistcolB}{\dimexpr 0.90\excdistcolw\relax}
\setlength{\excdistcolC}{\dimexpr 1.05\excdistcolw\relax}
\setlength{\excdistcolD}{\dimexpr 0.60\excdistcolw\relax}
\setlength{\excdistcolE}{\dimexpr 1.00\excdistcolw\relax}
\setlength{\excdistcolF}{\dimexpr 1.20\excdistcolw\relax}
\setlength{\excdistcolG}{\dimexpr 0.80\excdistcolw\relax}
\setlength{\excdistcolH}{\dimexpr 0.65\excdistcolw\relax}
\setlength{\LTleft}{0pt plus 1fill}
\setlength{\LTright}{0pt plus 1fill}
\setlength{\LTcapwidth}{\textwidth}
\begin{longtable}{|V{\excdistcolA}|V{\excdistcolB}|V{\excdistcolC}|V{\excdistcolD}|V{\excdistcolE}|V{\excdistcolF}|V{\excdistcolG}|V{\excdistcolH}|}
	\caption{\raggedright Quasi-distinguished $\ckorb$ in exceptional Lie algebras with $\codim(\partial \orb, \overline{\orb}) = 2$}
	\label{tab:qD_real_exceptional} \\
	\hline
	$\g_\R$ &  $\ckorb$  &  $\orb$  & $\orbk$ & $Z_{\bar{K}}$ & $\orb'$ & $\orbk'$ & $\#$repn  \\ \hline
	\endfirsthead
	\multicolumn{8}{c}{\tablename\ \thetable{} (\textit{continued})} \\
	\hline
	$\g_\R$ &  $\ckorb$  &  $\orb$  & $\orbk$ & $Z_{\bar{K}}$ & $\orb'$ & $\orbk'$ & $\#$repn  \\ \hline
	\endhead
	\hline
	\endfoot
		\multirow{2}{=}{{\bf G I}} & \multirow{2}{=}{$G_2(a_1)$} & \multirow{2}{=}{$G_2(a_1)$}  & $3$  & $\fS_2$ & \multirow{2}{=}{$\widetilde{A}_1$} & \multirow{2}{=}{$2$} & \multirow{2}{=}{5}  \\ \cline{4-5}
		& & & $4$ & $\fS_3$ & & & \\ \hline
		\multirow{3}{=}{{\bf F I}} & \multirow{3}{=}{$F_4(a_3)$} & \multirow{3}{=}{$F_4(a_3)$}  & $16$  & $\fS_4$ & \multirow{3}{=}{$C_3(a_1)$} & \multirow{3}{=}{14, 15} & \multirow{3}{=}{14}  \\ \cline{4-5}
		& & & $18$ & \makecell[c]{$\fS_2 \times \fS_2$} & & & \\ \cline{4-5}
		& & & $17$ & $\fD_4$ & & & \\ \hline
		\multirow{2}{=}{{\bf E I}}  & \multirow{2}{=}{$E_6(a_3)$} & \multirow{2}{=}{$A_2$} & 4 & $\fS_2$ & \multirow{2}{=}{$3A_1$} & \multirow{2}{=}{$3$} & \multirow{2}{=}{4}  \\ \cline{4-5}
		& & & 5 & $\fS_2$ & & & \\ \hline
		\multirow{2}{=}{{\bf E I}}  & \multirow{2}{=}{$D_4(a_1)$} & \multirow{2}{=}{$D_4(a_1)$} & 12 & $\fS_2$ & \multirow{2}{=}{$A_3+A_1$} & \multirow{2}{=}{$15$} & \multirow{2}{=}{7}  \\ \cline{4-5}
		& & & 23 & $\fS_4$ & & & \\ \hline
		\multirow{3}{=}{{\bf E II}} & \multirow{3}{=}{$E_6(a_3)$} & \multirow{3}{=}{$A_2$} & 6 & $\fS_2$ & \multirow{3}{=}{$3A_1$} & \multirow{3}{=}{4, 5} & \multirow{3}{=}{6}  \\ \cline{4-5}
		& & & 7 & $\fS_2$ & & & \\ \cline{4-5}
		& & & 8 & $\fS_2$ & & & \\ \hline
		\multirow{3}{=}{{\bf E II}} & \multirow{3}{=}{$D_4(a_1)$} & \multirow{3}{=}{$D_4(a_1)$} & 20 & $\fS_3$ & \multirow{3}{=}{$A_3+A_1$} & \multirow{3}{=}{$18$, $19$} & \multirow{3}{=}{7}  \\ \cline{4-5}
		& & & 21 & $\fS_2$ & & & \\ \cline{4-5}
		& & & 22 & $\fS_2$ & & & \\ \hline
		{\bf E III} & $E_6(a_3)$ & $A_2$ & 6 & 1 & $3A_1$ & none & 1  \\ \hline
		\multirow{2}{=}{{\bf E V}} & \multirow{2}{=}{$E_7(a_5)$} & \multirow{2}{=}{$D_4(a_1)$} & 26 & $\fS_2$ & \multirow{2}{=}{$(A_3+A_1)'$} & \multirow{2}{=}{25} & \multirow{2}{=}{5} \\ \cline{4-5}
		& & & 27 & $\fS_3$ & & & \\ \hline
		\multirow{2}{=}{{\bf E V}} & \multirow{2}{=}{$E_7(a_3)$} & \multirow{2}{=}{$A_2$} & 6 & $\fS_2$ & \multirow{2}{=}{$(3A_1)'$} & \multirow{2}{=}{5} & \multirow{2}{=}{4} \\ \cline{4-5}
		& & & 7 & $\fS_2$ & & & \\ \hline
		\multirow{3}{=}{{\bf E VI}} & \multirow{3}{=}{$E_7(a_5)$} & \multirow{3}{=}{$D_4(a_1)$} & 19 & $\fS_3$ & \multirow{3}{=}{$(A_3+A_1)'$} & \multirow{3}{=}{17, 18} & \multirow{3}{=}{7} \\ \cline{4-5}
		& & & 20 & $\fS_2$ & & & \\ \cline{4-5}
		& & & 21 & $\fS_2$ & & & \\ \hline
		\multirow{3}{=}{{\bf E VI}} & \multirow{3}{=}{$E_7(a_3)$} & \multirow{3}{=}{$A_2$} & 6 & $\fS_2$ & \multirow{3}{=}{$(3A_1)'$} & \multirow{3}{=}{4, 5} & \multirow{3}{=}{6} \\ \cline{4-5}
		& & & 7 & $\fS_2$ & & & \\ \cline{4-5}
		& & & 8 & $\fS_2$ & & & \\ \hline
		{\bf E VII} & $E_7(a_3)$ & $A_2$ & 10 & 1 & $(3A_1)'$ & none & 1 \\ \hline
		\multirow{3}{=}{{\bf E VIII}} & \multirow{3}{=}{$E_8(a_7)$} & \multirow{3}{=}{$E_8(a_7)$} & 66 & $\fD_4$ & \multirow{3}{=}{$E_7(a_5)$} & \multirow{3}{=}{64, 65} & \multirow{3}{=}{18} \\ \cline{4-5}
		& & & 67 & $\fS_5$ & & & \\ \cline{4-5}
		& & & 68 & $\fD_6$ & & & \\ \hline
		\multirow{3}{=}{{\bf E VIII}} & \multirow{3}{=}{$E_8(b_6)$} & \multirow{3}{=}{$D_4(a_1)+A_2$} & 34 & $\fS_2$ & \multirow{3}{=}{$A_3+A_2+A_1$} & \multirow{3}{=}{31, 32} & \multirow{3}{=}{6} \\ \cline{4-5}
		& & & 35 & $\fS_2$ & & & \\ \cline{4-5}
		& & & 36 & $\fS_2$ & & & \\ \hline
		\multirow{2}{=}{{\bf E VIII}} & \multirow{2}{=}{$E_8(b_5)$} & \multirow{2}{=}{$D_4(a_1)$} & 19 & $\fS_2$ & \multirow{2}{=}{$A_3+A_1$} & \multirow{2}{=}{18} & \multirow{2}{=}{5} \\ \cline{4-5}
		& & & 20 & $\fS_3$ & & & \\ \hline
		\multirow{3}{=}{{\bf E VIII}} & \multirow{3}{=}{$E_8(a_5)$} & \multirow{3}{=}{$2A_2$} & \textcolor{red}{14} & \textcolor{red}{$\fS_2$} & \multirow{3}{=}{$A_2+3A_1$} & \multirow{3}{=}{12, 13} & \multirow{3}{=}{5} \\ \cline{4-5}
		& & & 15 & $\fS_2$ & & & \\ \cline{4-5}
		& & & 16 & $\fS_2$ & & & \\ \hline
		\multirow{2}{=}{{\bf E VIII}} & \multirow{2}{=}{$E_8(a_3)$} & \multirow{2}{=}{$A_2$} & 4 & $\fS_2$ & \multirow{2}{=}{$3A_1$} & \multirow{2}{=}{3} & \multirow{2}{=}{4} \\ \cline{4-5}
		& & & 5 & $\fS_2$ & & & \\ \hline
		{\bf E IX} & $E_8(a_5)$ & $2A_2$ & 14 & 1 & $A_2+3A_1$ & none & 1 \\ \hline
		\multirow{3}{=}{{\bf E IX}} & \multirow{3}{=}{$E_8(a_3)$} & \multirow{3}{=}{$A_2$} & 6 & $\fS_2$ & \multirow{3}{=}{$3A_1$} & \multirow{3}{=}{4, 5} & \multirow{3}{=}{6} \\ \cline{4-5}
		& & & 7 & $\fS_2$ & & & \\ \cline{4-5}
		& & & 8 & $\fS_2$ & & & \\ \hline
\end{longtable}
\endgroup

\section{Unitarity of special unipotent representations}
\label{sec:unitarity}

Throughout this section, $\GR$ denotes a real reductive group as in \cref{subsec:basic_settings}, with the corresponding symmetric pair $(\g,K)$.
The proof has three steps: first we show the relevant unipotent ideals are $\theta$-stable, then we deduce Hermitianness of the modules in $\UnipOv(\g,K)$, and finally we apply the mixed-Hodge unitarity criterion.

\subsection{Hermitian forms}

A Hermitian form $\langle \cdot, \cdot \rangle$ on a $(\fg,K)$-module $M$ is said to be \emph{$(\fg_\R,K_\R)$-invariant} if $K_\R$ acts on $M$ by unitary operators and $\fg_\R$ acts on $M$ by skew-Hermitian operators. We say that $M$ is \emph{Hermitian} if it admits a non-degenerate $(\fg_\R,K_\R)$-invariant Hermitian form. If $M$ is irreducible and $\langle \cdot, \cdot \rangle$ is a non-degenerate $(\fg_\R,K_\R)$-invariant Hermitian form on $M$, then $\langle \cdot, \cdot \rangle$ is unique up to multiplication by a nonzero real number. 

The following criterion says that the existence of invariant Hermitian forms on an irreducible $(\fg,K)$-module of real infinitesimal character is an algebraic condition.

\begin{Prop}[\cite{ALTV}] \label{prop:Hermitian}
	Let $M$ be an irreducible $(\fg,K)$-module of real infinitesimal character. Then $M$ admits a non-degenerate $(\fg_\R,K_\R)$-invariant Hermitian form if and only if there is an isomorphism of $(\fg,K)$-modules $M \cong \theta^*M$, where $\theta^*M$ denotes the $(\fg,K)$-module equal to $M$ as a $K$-representation, with the $\fg$-action twisted by $\theta$.
\end{Prop}

We can uniquely extend the Cartan involution $\theta: \fg \to \fg$ to an involutive automorphism of $\Ug$, still denoted by $\theta$.

\begin{Prop}\label{prop:unipotent_theta_stable}
	Let $\orb$ be a nilpotent orbit in $\fg$ such that $\orb_{\fk^\perp} \neq \varnothing$.
	Let $\widetilde{\orb}$ be any finite connected cover of $\orb$. Then the unipotent ideal $I_0(\widetilde{\orb})$ is preserved by $\theta$. 
\end{Prop}

\begin{proof}
	As discussed in \Cref{subsec:enlarge_lag}, define the anti-Poisson involution $\sigma=-\theta$ of $\fg^*$, then $\orb$ is $\sigma$-stable, and hence $\theta$-stable, since $\orb_{\fk^\perp} \neq \varnothing$. Write $\rho:\widetilde{\orb} \to \orb$ for the covering map. Let $\tilde{\sigma}: \widetilde{\orb} \to \widetilde{\orb}$ be any $\cm$-equivariant anti-Poisson involution lifting $\sigma$, as in \Cref{subsec:enlarge_lag}, so that $\rho \circ \tilde{\sigma} = \sigma \circ \rho$. Let $\gamma: \cm \times \widetilde{\orb} \to \widetilde{\orb}$ denote the Brylinski-Kostant $\cm$-action on $\widetilde{\orb}$. Set 
		\[ \tilde{\theta} := \gamma(\sqrt{-1}) \circ \tilde{\sigma} = \tilde{\sigma} \circ \gamma(\sqrt{-1}). \]
	Then $\tilde{\theta}$ is a $\cm$-equivariant Poisson automorphism of $\widetilde{\orb}$ and $\rho \circ \tilde{\theta} = \theta \circ \rho$. Note that $\tilde{\theta}$ need not be an involution.	

	Let $(\cA_0, \varsigma)$ be the canonical filtered quantization of $\C[\widetilde{\orb}]$ as in \Cref{subsec:quant_conical_symp_sing}. By \cite[Proposition 5.1.1]{LMBM}, $\tilde{\theta}$ lifts to a filtered algebra automorphism $\Theta$ of $\cA_0$ in the sense that $\tilde{\theta} \circ \varsigma = \varsigma \circ \gr \Theta$. Thus $\Theta$ gives an isomorphism between $(\cA_0, \tilde{\theta} \circ \varsigma)$ and $(\cA_0, \varsigma)$ as filtered quantizations of $\C[\widetilde{\orb}]$.
	Now let $\Phi: \Ug \to \cA_0$ be the quantum co-moment map for $(\cA_0, \varsigma)$. Then $\Phi \circ \theta: \Ug \to \cA_0$ is a quantum co-moment map for $(\cA_0, \tilde{\theta} \circ \varsigma)$. Since $(\cA_0, \tilde{\theta} \circ \varsigma)$ and $(\cA_0, \varsigma)$ are isomorphic as filtered quantizations via $\Theta$, we must have $\Phi \circ \theta = \Theta \circ \Phi$ by the uniqueness of the quantum co-moment map (\Cref{lem:Hamiltonian_quant}). It follows that $I_0(\widetilde{\orb}) = \ker(\Phi)$ is preserved by $\theta$.
\end{proof}


\begin{Prop}[{\cite[Proposition 5.17]{DMB-Hodge}}] \label{prop:AC_theta}
	For any $M \in \cM_f(\fg,K)$, $\AC(M) = \AC(\theta^*M)$.
\end{Prop}

\begin{Prop} \label{prop:Hermitian_UnipOv}
	Assume that $G$ is a simply connected complex simple group defined over $\R$ and $\GR=G(\R)$, or $\GR = \Mp(2n, \R)$. Let $(\g,K)$ be the associated symmetric pair. Then for any quasi-distinguished nilpotent orbit $\ckorb$ in $\ckfg^*$, all Harish-Chandra modules in $\UnipOv(\g, K)$ are Hermitian.
\end{Prop}

\begin{proof}
	Let $M \in \UnipOv(\g, K)$. By \Cref{thm:V_AV}, the maximal $K$-orbits in $\AV(M)$ all lie in $\orb_{\fk^\perp} = \orb \cap \fk^\perp$. In particular, $\orb_{\fk^\perp} \neq \varnothing$.
	By \Cref{prop:dual_quasi-dist} (iii) and \Cref{prop:unipotent_theta_stable}, the ideal $\J_{\ckorb}$ is preserved by $\theta$; hence $\Ann(\theta^*M) = \theta \J_{\ckorb} = \J_{\ckorb}$ and $\theta^*M \in \UnipOv(\g, K)$. By \Cref{prop:AC_theta}, we have $\AC(M) = \AC(\theta^*M)$. Therefore \Cref{thm:Unip_AOD_bij} implies that $M \simeq \theta^*M$, which means that $M$ is Hermitian by \Cref{prop:Hermitian}.
\end{proof}

\subsection{Unitarity via mixed Hodge theory}

The following unitarity criterion was proved in \cite[Theorem 5.22]{DMB-Hodge} using the theory of mixed Hodge modules.

\begin{Thm}\label{thm:unitarity_criterion}
	Let $M$ be an irreducible $(\fg,K)$-module such that
	\begin{itemize}
		\item[(i)] $M$ is Hermitian.
		\item[(ii)] $\AC(M)$ is irreducible (\Cref{defn:irred_AC}).
		\item[(iii)] $\Ann(M)$ is a maximal ideal.
		\item[(iv)] $\Ann(M)$ is very weakly unipotent.
	\end{itemize}
	Then $M$ is unitary.
\end{Thm}

We refer the reader to \cite[Definition 4.11]{DMB-Hodge} for the notion of very weakly unipotent $\Ug$-modules or two-sided ideals of $\Ug$. It is a weaker condition than the weak unipotence previously introduced by Vogan (\cite[Definition 8.16]{Vogan:unitarizability}). For our application, we need the following proposition, established in \cite[Theorem 2.19, Theorem 3.16]{MaYu:weak_unipotence}.  

\begin{Prop} \label{prop:special_weakly_unipotent}
	Let $\fg$ be a complex simple Lie algebra and let $\ckfg$ be its Langlands/metaplectic dual Lie algebra.
	Let $\ckorb$ be any nilpotent orbit in $\ckfg^*$ and let $\J_{\ckorb}$ be the (metaplectic) special unipotent ideal of $\Ug$ attached to $\ckorb$. Then $\J_{\ckorb}$ is very weakly unipotent.
\end{Prop}
\begin{Rem}
	\Cref{prop:special_weakly_unipotent} was first established in \cite[Proposition 5.10]{BV85:unipotent} for even nilpotent orbits in $\ckfg^*$. Recall that a nilpotent orbit $\ckorb$ in $\ckfg$ is said to be \emph{even} if $\lambda_{\ckorb} = \check{h}/2$ is integral (i.e., $\exp[\pi i \ad(\check{h})] = \Id_{\ckfg}$) for some (and hence all) $\slf_2$-triple $(\check{e},\check{h}, \check{f})$ of $\ckorb$.  
\end{Rem}

Finally, we prove the main result \Cref{thm:main_unitarity_dist} mentioned in \Cref{subsec:main_results}.

\begin{Thm} \label{thm:unitarity_dist}
	Assume that $G$ is a simply connected complex simple group defined over $\R$ and $\GR=G(\R)$, or $\GR = \Mp(2n, \R)$. Let $(\g,K)$ be the associated symmetric pair. Then for any quasi-distinguished nilpotent orbit $\ckorb$ in $\ckfg^*$, all Harish-Chandra modules in $\UnipOv(\fg,K)$ are unitarizable.
\end{Thm}

\begin{proof}
	We check the conditions in \Cref{thm:unitarity_criterion}. Condition (i) follows from \Cref{prop:Hermitian_UnipOv}, condition (ii) from \Cref{thm:Unip_AOD_bij}, condition (iii) from the definition of (metaplectic) special unipotent ideals, and condition (iv) from \Cref{prop:special_weakly_unipotent}.
\end{proof}

\begin{Cor}\label{cor:unitarity_dist_reductive}
	Assume that $G$ is a connected complex reductive group defined over $\R$ and $\GR=G(\R)$. Let $(\g,K)$ be the associated symmetric pair. Then for any quasi-distinguished nilpotent orbit $\ckorb$ in $\ckfg^*$, all Harish-Chandra modules in $\UnipOv(\fg,K)$ are unitarizable.
\end{Cor}

\begin{proof}
	It is easy to reduce to the case when $G$ is simply connected and the (real) Lie algebra $\g_\R$ of $\GR$ is simple. If $\GR$ is itself complex, this has already been treated in \cite{DMB-Hodge}. In this case, $\fg$ is the product of two isomorphic simple Lie algebras (over $\C$). Otherwise, $\fg$ is a simple Lie algebra over $\C$ and $G$ is a complex simple group. This case follows from \Cref{thm:unitarity_dist}.	
\end{proof}

\appendix

\section{Restriction of quantization} \label{sec:res_quant}

The bijectivity of the quantization map in \Cref{thm:Unip_AOD_bij} is obtained in \Cref{subsec:counting} by reducing to cardinality inequalities. Several cases--notably $\SO_0(p,q)$, Spin groups, and real exceptional groups--require passing from a pair $(\g,K')$ to another pair $(\g,K)$ via a morphism $\phi: (\g,K) \to (\g,K')$ of symmetric pairs as in the last part of \cref{subsec:basic_settings}. As for associated cycles (\Cref{lem:AC_pullback}), we need the quantization functors to be compatible with pullback via $\phi$ along such inclusions. This appendix establishes that compatibility in \Cref{prop:quant_res}, \Cref{prop:quant_pullback} and \Cref{cor:quant_res_bij}; these results are applied in the proofs of \Cref{cor:counting_classical_identity_component}, \Cref{prop:counting_spin}, and \Cref{prop:counting_exceptional}.

We follow the same setting and notation as in \Cref{subsec:quant_K_orb}. We first consider the following special situation: let $G$ be a connected complex reductive group with Cartan involution $\theta$, and let $K \subset K'$ be closed subgroups of finite index in $G^\theta$. Let $\orb$ be a nilpotent $G$-orbit in $\g^*$ as in \Cref{subsec:adm_extn} and let $\orb_{K'} \subset \orb_{\fk^\perp}$ be a $K'$-orbit. Then $\orb_{K'}$ is in general a disjoint union of nilpotent $K$-orbits, say 
	\[ \orb_{K'} = \bigsqcup_{i=1}^r \orb_{K}^i.\]
In analogy with $\cUnipOv^K(\orbk)$ and $\cUnipOv^{K'}(\orb_{K'})$ defined in \Cref{subsec:quant_unip}, we define $\cUnipOv^K(\orb_{K'})$ to be the full subcategory of $\cUnipOv(\g,K)$ consisting of $(\g,K)$-modules $M$ that satisfy $\AV(M) \subset \overline{\orb}_{K'}$. Then for each $\orbk^i$, $\cUnipOv^K(\orbk^i)$ is naturally a full subcategory of $\cUnipOv^K(\orb_{K'})$. On the other hand, we have an obvious equivalence of abelian categories 
\[ \Coh(\TT^+_{\orb_{K'}}, K) \simeq \bigoplus_{i=1}^r \Coh(\TT^+_{\orb_{K}^i}, K). \]
Therefore we can assemble the quantization functors $\quan_{\ckorb, \orbk^i}^K$ as in \Cref{prop:MV_AC} (2) to obtain a quantization functor
\begin{equation} \label{eq:quan_orbK'_K}
	\quan_{\ckorb, \orb_{K'}}^K: \Coh(\TT^+_{\orb_{K'}}, K) \to \cUnipOv^K(\orb_{K'}). 
\end{equation}
When $K=K'$ and $\orb_{K'} = \orb_{K}$, this functor coincides with the quantization functor $\quan_{\ckorb, \orbk}^K$.

As in \Cref{subsec:enlarge_lag} and \Cref{subsec:adm_extn}, we choose a $K'$-orbit in $\rho^{-1}(\orb_{K'})$, denoted by $Y'^\diamond = \widetilde{\orb}_{K'}$, with the corresponding covering morphism $\tilde{\rho}: \widetilde{\orb}_{K'} \to \orb_{K'}$. Let $Y'$ be the closure of $Y'^\diamond$ in $X^{reg}$, which is again a smooth closed Lagrangian subvariety of $X^{reg}$. We have the corresponding groups $\Gamma_{K'}$, $\Gamma_{K'}^{\cm}$, $\tGamma_{K'}$, $\tGamma_{K'}^{\cm}$, etc. 
For each $1 \leqslant i \leqslant r$, set 
\[ Y_i^\diamond = \widetilde{\orb}_K^i := \tilde{\rho}^{-1}(\orb_{K}^i).\]
Let $Y_i$ be the closure of $Y_i^\diamond$ in $X^{reg}$, or equivalently, in $Y'$. Then $Y_i^\diamond$ and $Y_i$ are $K$-invariant. Moreover, each $Y_i$ is both open and closed in $Y'$, and hence are $K \times \cm$-invariant smooth closed Lagrangian subvarieties in $X^{reg}$. We have the decompositions
\[ 
	Y_i^\diamond = \bigsqcup_{j=1}^{s_i} Y_{i,j}^\diamond \quad \text{and} \quad Y_i = \bigsqcup_{j=1}^{s_i} Y_{i,j},
\]
where $s_i \in \Z_+$ for each $1 \leqslant i \leqslant r$, $Y_{i,j}^\diamond$ are $K$-orbits and each $Y_{i,j}$ is the closure of $Y_{i,j}^\diamond$ in $X^{reg}$ (or $Y'$), so that $Y_{i,j}^\diamond$ and $Y_{i,j}$ are $K \times \cm$-invariant open and closed subvarieties in $Y_i^\diamond$ and $Y_i$ respectively. Again, the $Y_{i,j}$ are smooth closed Lagrangian subvarieties in $X^{reg}$. 

For each pair of indices $(i, j)$, let $\tilde{\rho}_{i,j}: Y_{i,j}^\diamond \to \orb_{K}^i$ denote the restriction of $\tilde{\rho}$ to $Y_{i,j}^\diamond$. Then $\tilde{\rho}_{i,j}$ is a $K$-equivariant Galois covering morphism with the Galois group denoted as $\Gamma_{\tilde{\rho}_{i,j}}$. The group $\Gamma_{\tilde{\rho}_{i,j}}$ can be naturally identified with the subgroup of $\Gamma_{\tilde{\rho}}$ consisting of those $K'$-equivariant automorphisms of $\tilde{\rho}: Y'^\diamond \to \orb_{K'}$ which preserve $Y_{i,j}^\diamond$. For a fixed index $i$, $Y_{i,j}^\diamond$ for various $j$ are permuted transitively by $\Gamma_{\tilde{\rho}}$, and the subgroups $\Gamma_{\tilde{\rho}_{i,j}} \subset \Gamma_{\tilde{\rho}}$ are conjugate to each other for each fixed $i$ and all $1 \leqslant j \leqslant s_i$. We can form the variety $\Gamma_{\tilde{\rho}} \times_{\Gamma_{\tilde{\rho}_{i,j}}} Y_{i,j}^\diamond$ equipped with the natural (left) $K \times \Gamma_{\tilde{\rho}}$-action induced by the (left) $K$-action on $Y_{i,j}^\diamond$ and the (left) $\Gamma_{\tilde{\rho}}$-action on $\Gamma_{\tilde{\rho}}$. The following lemma is standard.

\begin{Lem} \label{lem:ind_Y}
	For any $1 \leqslant i \leqslant r$ and $1 \leqslant j \leqslant s_i$, the inclusion $Y_{i,j}^\diamond \subset Y_i^\diamond$ and the $K \times \Gamma_{\tilde{\rho}}$-action on $Y_i^\diamond$ induce a $K \times \Gamma_{\tilde{\rho}}$-equivariant isomorphism
	\[ 
		\Gamma_{\tilde{\rho}} \times_{\Gamma_{\tilde{\rho}_{i,j}}} Y_{i,j}^\diamond \xrightarrow{\sim} Y_i^\diamond,
	\]
	and a similar statement holds for $Y_i$ and $Y_{i,j}$.
\end{Lem}

The restriction functor defined in the paragraph above \Cref{lem:AC_pullback} induces the restriction/forgetful functor 
	\[ \res_{K}^{K'}: \cUnipOv(\g,K') \to \cUnipOv(\g,K) \] 
which induces 
\[\res_{K}^{K'}: \cUnipOv^{K'}(\orb_{K'}) \to \cUnipOv^K(\orb_{K'}). \]
Similarly we have restriction/forgetful functors for equivariant twisted $\sD$-modules 
\[ \res_{K}^{K'}: \Coh(\TT^+_{\orb_{K'}},K') \to \Coh(\TT^+_{\orb_{K'}},K) \]
and
\[ \res_{K}^{K'}: \Coh^\kappa(\TT^+_{\orb_{K'}},K' \times \cm) \to \Coh^\kappa(\TT^+_{\orb_{K'}},K \times \cm), \]
and so on, which are compatible with each other in the sense that they are intertwined by the right vertical functors in \eqref{diag:lift_graded}. The following proposition states that the quantization functors are compatible with these restriction functors, which is parallel to \Cref{lem:AC_pullback}.

\begin{Prop} \label{prop:quant_res}
	Let $K$ be a closed subgroup of finite index in $K'$. Then the following diagram is commutative:
	\begin{equation}
		\begin{tikzcd}[column sep = large]
			\Coh(\TT^+_{\orb_{K'}},K') \arrow[r, "\quan_{\ckorb, \orb_{K'}}^{K'}"] \arrow[d, swap, "\res_{K}^{K'}"] & \cUnipOv^{K'}(\orb_{K'})  \arrow[d, "\res_{K}^{K'}"] \\
			\Coh(\TT^+_{\orb_{K'}},K) \arrow[r, "\quan_{\ckorb, \orb_{K'}}^K"] & \cUnipOv^K(\orb_{K'})
		\end{tikzcd}
	\end{equation}
\end{Prop}

\begin{proof}
	We first fix a fully faithful embedding 
	\[ 
		\Coh(\TT^+_{\orb_{K'}},K') \hookrightarrow \Coh^\kappa(\TT^+_{\orb_{K'}},K' \times \cm) 
	\] 
	as in \Cref{subsubsec:strong_cm_action}, and use the same symbol for the image of any $\cF \in \Coh(\TT^+_{\orb_{K'}},K')$ under this embedding. 
	Set 
	\[ 
		\cF^i := \res_K^{K'}(\cF)|_{\orbk^i} \in \Coh^\kappa(\TT^+_{\orbk^i},K \times \cm). 
	\]
    Then $\res_{K}^{K'}(\cF) = \bigoplus_{i=1}^r \cF^i$, where we do not distinguish $\cF^i$ from its zero extension to $\orb_{K'}$. We follow the same convention below.
	Set 
	\[ \cE := \tilde{\rho}^* \cF \in \Coh^\kappa(\TT^+_{Y'^\diamond}, K' \times \Gamma_{\tilde{\rho}}^{\cm}) \]
	and
	\[ 
		\cE^i := \cE|_{Y_i^\diamond} \in \Coh^\kappa(\TT^+_{Y_i^\diamond}, K \times \Gamma_{\tilde{\rho}}^{\cm}) \quad \text{and} \quad \cE^{i,j} := \cE|_{Y_{i,j}^\diamond} \in \Coh^\kappa(\TT^+_{Y_{i,j}^\diamond}, K \times \Gamma_{\tilde{\rho}_{i,j}}^{\cm}), 
	\]
	where we omit the obvious restriction functors from $K'$ to $K$ to simplify notation. Then $\cE^i \simeq \tilde{\rho}^* \cF^i$ and $\cE^{i,j} \simeq \tilde{\rho}_{i,j}^* \cF^i$.
	
	Let $\overline{\cE}$ be the minimal extension of $\cE$ to $\Coh^\kappa(\TT^+_{Y'}, K' \times \Gamma_{\tilde{\rho}}^{\cm})$ as in \Cref{prop:extn_AOD}. As in the proofs of \Cref{prop:quant_cWV} and \Cref{prop:MV_AC}, we can quantize $\overline{\cE}$ over $Y'$ to obtain a $K' \times \Gamma_{\tilde{\rho}}^{\cm}$-equivariant graded Hamiltonian quantization $\overline{\cE}_\hb$ of $\overline{\cE}$ over the quantization $\OO_\hb$ of $X^{reg}$ as in \cref{subsec:quant_K_orb}. Then, by \Cref{prop:quant_cWV}, \Cref{prop:MV_AC}, \eqref{eq:MV_WV} and \eqref{eq:cWV_FV_pullback}, we have 
		\[\quan_{\ckorb, \orb_{K'}}^{K'}(\cF) \simeq [\Gamma(Y', \overline{\cE}_\hb)^{fin} / (\hb-1)]^{\Gamma_{\tilde{\rho}}}.\] 
	We denote this HC $(\g,K')$-module by $M$ to simplify notations. We write 
		\[ 
			\overline{\cE}^i:= \overline{\cE}|_{Y_i} \in \Coh^\kappa(\TT^+_{Y_i}, K \times \Gamma_{\tilde{\rho}}^{\cm}), 
			\quad \text{and} \quad
			\overline{\cE}^{i,j} := \overline{\cE}|_{Y_{i,j}} \in \Coh^\kappa(\TT^+_{Y_{i,j}}, K \times \Gamma_{\tilde{\rho}_{i,j}}^{\cm})\]
		for any $1 \leqslant i \leqslant r$ and $1 \leqslant j \leqslant s_i$. Then $\overline{\cE}^i$ and $\overline{\cE}^{i,j}$ are the minimal extensions of $\cE^i$ and $\cE^{i,j}$ to $Y_i$ and $Y_{i,j}$ respectively.
		Similarly, we set
		\[
			\overline{\cE}_\hb^i := \overline{\cE}_\hb|_{Y_i} \quad \text{and} \quad \overline{\cE}_\hb^{i,j} := \overline{\cE}_\hb|_{Y_{i,j}}
		\]
	for any $1 \leqslant j \leqslant s_i$. Then $\overline{\cE}_\hb^i$ (resp. $\overline{\cE}_\hb^{i,j}$) is the unique $K \times \Gamma_{\tilde{\rho}}^{\cm}$(resp. $K \times \Gamma_{\tilde{\rho}_{i,j}}^{\cm}$)-equivariant graded Hamiltonian quantization of $\overline{\cE}^i$ (resp. $\overline{\cE}^{i,j}$) over $\OO_\hb$. We have direct sum decompositions
	\[ 
		\overline{\cE} = \bigoplus_{i=1}^r \overline{\cE}^i, \qquad \overline{\cE}^i = \bigoplus_{j=1}^{s_i} \overline{\cE}^{i,j}, 
		\qquad
		\overline{\cE}_\hb = \bigoplus_{i=1}^r \overline{\cE}_\hb^i, \qquad \text{and} \qquad \overline{\cE}_\hb^i = \bigoplus_{j=1}^{s_i} \overline{\cE}_\hb^{i,j},
	\]
	where the summands are implicitly viewed as zero extensions.

	Now define the modules 
		\[ E^i := \Gamma(Y_i, \overline{\cE}_\hb^i)^{fin} / (\hb-1) \in \cM_f(\A, K \times \Gamma_{\tilde{\rho}}), \qquad  M^i := (E^i)^{\Gamma_{\tilde{\rho}}} \in \cM_f(\A^{\Gamma}, K),\]
	and
		\[ E^{i,j} := \Gamma(Y_{i,j}, \overline{\cE}_\hb^{i,j})^{fin} / (\hb-1) \in \cM_f(\A, K \times \Gamma_{\tilde{\rho}_{i,j}}), \qquad  M^{i,j} := (E^{i,j})^{\Gamma_{\tilde{\rho}_{i,j}}} \in \cM_f(\A^{\Gamma}, K).\]
	Then we have $\res_{K}^{K'} M = \bigoplus_{i=1}^r M^i$ as $(\g,K)$-modules, and by \Cref{lem:ind_Y}, we have
	\[ 
		E^i \simeq \mathrm{ind}_{\Gamma_{\tilde{\rho}_{i,j}}}^{\Gamma_{\tilde{\rho}}} E^{i,j} := (\C[\Gamma_{\tilde{\rho}}] \otimes_\C E^{i,j})^{\Gamma_{\tilde{\rho}_{i,j}}}
	\]
	as $(\g,K \times \Gamma_{\tilde{\rho}})$-modules,
	for any $1 \leqslant i \leqslant r$ and $1 \leqslant j \leqslant s_i$. Therefore, by Frobenius reciprocity, we have isomorphisms of $(\g,K)$-modules for any $1 \leqslant i \leqslant r$ and $1 \leqslant j \leqslant s_i$,
		\[ M^i = (E^i)^{\Gamma_{\tilde{\rho}}} \simeq \Hom_{\Gamma_{\tilde{\rho}}}(\mathbbm{1}_{\Gamma_{\tilde{\rho}}}, E^i) \simeq \Hom_{\Gamma_{\tilde{\rho}_{i,j}}}(\mathbbm{1}_{\Gamma_{\tilde{\rho}_{i,j}}}, E^{i,j}) \simeq (E^{i,j})^{\Gamma_{\tilde{\rho}_{i,j}}} = M^{i,j}, \]
	where $\mathbbm{1}$'s stand for the trivial representations of the groups in question. 
	
	On the other hand, by definition of the quantization functors $\quan_{\ckorb, \orbk^i}^K$ (\Cref{prop:MV_AC}) and $\quan_{\ckorb, \orb_{K'}}^K$ in \eqref{eq:quan_orbK'_K}, there are natural isomorphisms 
	\[ 
		\quan_{\ckorb, \orb_{K'}}^K(\cF^i) = \quan_{\ckorb, \orbk^i}^K (\cF^i) \simeq M^{i,j} \simeq M^i
	\] 
	as $(\g,K)$-modules, for any $1 \leqslant i \leqslant r$ and $1 \leqslant j \leqslant s_i$, where on the left $\cF^i$ is viewed as its zero extension to $\orb_{K'}$. Therefore we have
	\[ \quan_{\ckorb, \orb_{K'}}^K \circ \res_{K}^{K'} (\cF) = \bigoplus_{i=1}^r \quan_{\ckorb, \orbk^i}^K (\cF^i) \simeq \bigoplus_{i=1}^r M^i = \res_{K}^{K'} M = \res_{K}^{K'} \circ \quan_{\ckorb, \orb_{K'}}^{K'} (\cF). \]
\end{proof}

We now apply \cref{prop:quant_res} to deduce the following \cref{cor:quant_res_bij} needed in the proof of \Cref{cor:counting_classical_identity_component}. 

Given $M \in \cUnipOv(\g, K)$, we can induce the underlying $K$-representation of $M$ from $K$ to $K'$, obtaining $\Ind_{K}^{K'} M$. Then $\Ind_{K}^{K'} M$ carries a natural $(\g, K')$-module structure. This gives an additive functor 
    \[ \Ind_{K}^{K'}: \cUnipOv(\g, K) \to \cUnipOv(\g, K') \] 
and $\Ind_{K}^{K'}$ is both left and right adjoint to the restriction functor $\res_{K}^{K'}$, since $K$ is of finite index in $K'$.

\begin{Cor} \label{cor:quant_res_bij}
	Suppose $K$ is a subgroup of $K'$ of finite index and the quantization map $\quan_{\ckorb}^{K'}: \AOD_{K'}(\orb) \to \UnipOv(\g, K')$ is a bijection. Then $\quan_{\ckorb}^K: \AODK(\orb) \to \UnipOv(\g, K)$ is also a bijection.
\end{Cor}

\begin{proof}
	By \cref{prop:quant_map_quasi-dist}, $\quan_{\ckorb}^K$ is always injective, so we only need to show that $\quan_{\ckorb}^K$ is surjective. Let $M \in \UnipOv(\g, K)$. By Frobenius reciprocity, $M$ appears as a direct summand in $\res_{K}^{K'} M'$, where $M'$ is an irreducible subquotient of $\Ind_{K}^{K'} M$ and hence lies in $\UnipOv(\g, K')$. By the bijectivity of $\quan_{\ckorb}^{K'}$, we have $M' = \quan_{\ckorb}^{K'}(\orb_{K'}, \cV')$ for some $(\orb_{K'}, \cV') \in \AOD_{K'}(\orb)$. Write
	\[
		\orb_{K'} = \bigsqcup_{i=1}^r \orb_{K}^i 
		\quad \text{and} \quad 
		\res_{K}^{K'}(\orb_{K'}, \cV') = \bigoplus_{i=1}^r (\orbk^i, \cV_i)\]
	for some $(\orbk^i, \cV_i) \in \AODK(\orb)$. By \Cref{prop:quant_res}, we have
		\[ \res_{K}^{K'} M' = \res_{K}^{K'} \circ \quan_{\ckorb}^{K'}(\orb_{K'}, \cV') = \quan_{\ckorb}^K \circ \res_{K}^{K'}(\orb_{K'}, \cV') = \bigoplus_{i=1}^r \quan_{\ckorb}^K(\orbk^i, \cV_i).\]
	Hence $M = \quan_{\ckorb}^K(\orbk^j, \cV_j)$ for some $j$, since each $\quan_{\ckorb}^K(\orbk^j, \cV_j)$ is an irreducible $(\g,K)$-module by \Cref{prop:MV_AC} (3). This establishes the surjectivity of $\quan_{\ckorb}^K$. 
\end{proof}

\Cref{prop:quant_res} has the following generalization.

\begin{Prop} \label{prop:quant_pullback}
	Let $\phi: (\g,K) \to (\g,K')$ be a morphism of symmetric pairs as in the last part of \cref{subsec:basic_settings}. Then the following diagram is commutative:
	\begin{equation} \label{diag:quan_pullback}
		\begin{tikzcd}[column sep = large]
			\Coh(\TT^+_{\orb_{K'}},K') \arrow[r, "\quan_{\ckorb, \orb_{K'}}^{K'}"] \arrow[d, swap, "\phi^*"] & \cUnipOv^{K'}(\orb_{K'})  \arrow[d, "\phi^*"] \\
			\Coh(\TT^+_{\orb_{K'}},K) \arrow[r, "\quan_{\ckorb, \orb_{K'}}^K"] & \cUnipOv^K(\orb_{K'})
		\end{tikzcd}
	\end{equation}
	where $\phi^*$ is defined by pullback via $\phi$.
\end{Prop}

\begin{proof}
	Any $\phi$ as in the statement is the composition of $\phi$, viewed as a finite surjective morphism from $K$ onto $\mathrm{Im}(\phi)$, with the inclusion $\mathrm{Im}(\phi) \hookrightarrow K'$. The statement for the finite surjective case is trivial, while the statement for the inclusion $\mathrm{Im}(\phi) \hookrightarrow K'$ is \cref{prop:quant_res}. Composing the two commutative diagrams of the form \eqref{diag:quan_pullback}, we get the commutative diagram for $\phi$.
\end{proof}

\section{Computation of a component group} \label{sec:correction_component}

In this appendix, we prove the claim in Remark \ref{rem:Z_K}. 

\begin{Lem} \label{lem:Z_K_correction}
	Let $G$ be a complex simple Lie group of type $E_8$ and $\GR$ be the split real form {\bf E VIII} of $G$. Then the nilpotent $K$-orbit $\# 14$ in \cite[Table 11]{King} has $K$-component group $Z_K \simeq \fS_2$. 
\end{Lem}

\begin{proof}
	Let $\theta$ be the Cartan involution of $\fg$ and $G$ corresponding to $\GR$.
	Let $\orb$ be the nilpotent orbit in $\fg$ with Bala-Carter label $2A_2$. As usual, let $Q$ denote the the $G$-centralizer of an $\slf_2$-triple $(e,h,f)$ of a point $e \in \orb$.
	By \cite[Table 5]{Alexeevski}, $Q$ is isomorphic to the semidirect product $\Z_2 \ltimes (H \times H)$, where $H$ is a complex simple Lie group of type $G_2$, and the nontrivial element $\zeta=-1$ of $\Z_2$ acts on $H \times H$ by switching the two $H$-factors. 
	
	Now let $\theta$ be the Cartan involution of $\fg$ and $G$ corresponding to $\GR$. Assume that $e$ lies in the $K$-orbit $\orbk$ $\# 14$ and $(e,h,f)$ is normal. Then $Q$ is $\theta$-stable and the reductive $K$-centralizer $K_Q$ of $e$ equals $Q^\theta$. According to \cite[Table XIV]{Djokovic:centralizer_inner}, $\theta$ restricted to $\fq = \Lie(Q)$ corresponds to the product of two compact real forms of $G_2$, and hence is the identity map of $\fq$. Therefore the identity component $K_Q^\circ$ of $K_Q$ is equal to $(Q^\circ)^\theta = Q^\circ = H \times H$. The induced action of $\theta$ on the component group $A(\orb) = Q/Q^\circ \simeq \fS_2$ is certainly trivial, hence we can assume $\theta(\zeta) = \zeta \cdot g$, where $g = (a,b) \in Q^\circ = H \times H$.  One computes that, for any $(x,y) \in H \times H$,
		\[ \theta[\zeta \cdot (x,y) \cdot \zeta] = \theta[(y,x)] = (y,x),\]
	and 
		\[  \theta[\zeta \cdot (x,y)  \cdot \zeta] = \theta(\zeta)\theta[(x,y)]\theta(\zeta) = \zeta g (x,y) \zeta g  = \zeta (a,b)(x,y)(b,a) \zeta =( bya, axb ).\]
	Compare the above two equations: taking $x = y = 1$, we get $b = a^{-1}$, and hence $ax = xa$ for all $x \in H$, i.e., $a$ is in the center of $H$, which is trivial. This implies that $g=(1,1)$ and hence $\theta$ is the identity automorphism on the entire $Q$. Therefore $K_Q = Q^\theta = Q$ and $Z_K = \pi_0 K_Q = \pi_0 Q = Z \simeq \fS_2$.
\end{proof}

\printbibliography

@article{coh,
  author     = {Biswas, Indranil and Chatterjee, Pralay},
  title      = {On the exactness of {K}ostant-{K}irillov form and the second
                cohomology of nilpotent orbits},
  journal    = {Internat. J. Math.},
  fjournal   = {International Journal of Mathematics},
  volume     = {23},
  year       = {2012},
  number     = {8},
  pages      = {1250086, 25},
  issn       = {0129-167X},
  mrclass    = {17B08 (17B20)},
  mrnumber   = {2949224},
  mrreviewer = {Alexander Isaakovich Shtern},
  doi        = {10.1142/S0129167X12500863},
  url        = {https://doi.org/10.1142/S0129167X12500863}
}

@article{Achar:duality,
  title = {An Order-Reversing Duality Map for Conjugacy Classes in {{Lusztig}}'s Canonical Quotient},
  author = {Achar, Pramod N.},
  year = {2003},
  journaltitle = {Transformation Groups},
  shortjournal = {Transform. Groups},
  volume = {8},
  number = {2},
  pages = {107--145},
  issn = {1083-4362},
  doi = {10.1007/s00031-003-0422-x},
  url = {https://mathscinet.ams.org/mathscinet-getitem?mr=1976456},
  urldate = {2022-02-24},
  mrnumber = {1976456}
}

@article{Brylinski-Kostant,
  author     = {Brylinski, Ranee and Kostant, Bertram},
  title      = {Nilpotent orbits, normality and {H}amiltonian group actions},
  journal    = {J. Amer. Math. Soc.},
  fjournal   = {Journal of the American Mathematical Society},
  volume     = {7},
  year       = {1994},
  number     = {2},
  pages      = {269--298},
  issn       = {0894-0347},
  mrclass    = {22E46 (14L30 22E60 32M05 58F06)},
  mrnumber   = {1239505},
  mrreviewer = {William M. McGovern},
  url        = {https://doi.org/10.2307/2152759}
}

@book{Vogan:unitary,
  author    = {Vogan, D. A. },
  title     = {Unitary representations of reductive Lie groups},
  year      = {1987},
  series    = {Ann. of Math. Stud.},
  volume    = {118},
  publisher = {Princeton University Press}
}

@incollection{Vogan_AV,
  author     = {Vogan, Jr., David A.},
  title      = {Associated varieties and unipotent representations},
  booktitle  = {Harmonic analysis on reductive groups ({B}runswick, {ME},
                1989)},
  series     = {Progr. Math.},
  volume     = {101},
  pages      = {315--388},
  publisher  = {Birkh\"auser Boston, Boston, MA},
  year       = {1991},
  mrclass    = {22E46},
  mrnumber   = {1168491},
  mrreviewer = {D. Mili\v ci\'c}
}

@article{Vogan:unitarizability,
  title   = {Unitarizability of {Certain} {Series} of {Representations}},
  volume  = {120},
  issn    = {0003-486X},
  url     = {https://www.jstor.org/stable/2007074},
  doi     = {10.2307/2007074},
  number  = {1},
  urldate = {2025-04-07},
  journal = {Annals of Mathematics},
  author  = {Vogan, David A.},
  year    = {1984}
}

@article{Vogan:unitary_dual_GL,
  title        = {The Unitary Dual of {{GL}}(n) over an Archimedean Field},
  author       = {Vogan, David A.},
  year         = {1986},
  journal      = {Inventiones mathematicae},
  shortjournal = {Invent Math},
  volume       = {83},
  number       = {3},
  pages        = {449--505},
  issn         = {1432-1297},
  doi          = {10.1007/BF01394418},
  url          = {https://doi.org/10.1007/BF01394418},
  urldate      = {2025-11-03},
  langid       = {english}
}

@online{MaYu:weak_unipotence,
  title       = {Weak Unipotence and {{Langlands}} Duality},
  author      = {Ma, Jia-jun and Yu, Shilin},
  date        = {2025-10-15},
  eprint      = {2510.13523},
  eprinttype  = {arXiv},
  eprintclass = {math},
  doi         = {10.48550/arXiv.2510.13523},
  url         = {http://arxiv.org/abs/2510.13523},
  urldate     = {2025-10-20}
}

@article{BeKaledin,
  author     = {Bezrukavnikov, R. and Kaledin, D.},
  title      = {Fedosov quantization in algebraic context},
  journal    = {Mosc. Math. J.},
  fjournal   = {Moscow Mathematical Journal},
  volume     = {4},
  year       = {2004},
  number     = {3},
  pages      = {559--592, 782},
  issn       = {1609-3321},
  mrclass    = {53D55 (14F05)},
  mrnumber   = {2119140 (2006j:53130)},
  mrreviewer = {Stefan Waldmann}
}

@article{LY,
  author   = {Leung, Naichung Conan and Yu, Shilin},
  title    = {Equivariant deformation quantization and coadjoint orbit
              method},
  journal  = {Duke Math. J.},
  fjournal = {Duke Mathematical Journal},
  volume   = {170},
  year     = {2021},
  number   = {8},
  pages    = {1781--1850},
  issn     = {0012-7094},
  mrclass  = {22E46 (17B08 53D55)},
  mrnumber = {4278663},
  doi      = {10.1215/00127094-2020-0066},
  url      = {https://doi.org/10.1215/00127094-2020-0066}
}

@article{BGKP,
  author     = {Baranovsky, Vladimir and Ginzburg, Victor and Kaledin, Dmitry
                and Pecharich, Jeremy},
  title      = {Quantization of line bundles on lagrangian subvarieties},
  journal    = {Selecta Math. (N.S.)},
  fjournal   = {Selecta Mathematica. New Series},
  volume     = {22},
  year       = {2016},
  number     = {1},
  pages      = {1--25},
  issn       = {1022-1824},
  mrclass    = {53D55 (14D21 53D12)},
  mrnumber   = {3437831},
  mrreviewer = {Rongmin Lu},
  url        = {https://doi.org.ezproxy.library.tamu.edu/10.1007/s00029-015-0181-2}
}

@online{BPW,
  author      = {Braden, Tom and Proudfoot, Nicholas and Webster, Ben},
  title       = {Quantizations of conical symplectic resolutions {I}: local and global structure},
  date        = {date},
  date        = {2022-05-08},
  eprint      = {1208.3863},
  eprinttype  = {arXiv},
  eprintclass = {math},
  year        = {2022}
}

@incollection{BeilinsonBernstein,
  author     = {Beilinson, A. and Bernstein, J.},
  title      = {A proof of {J}antzen conjectures},
  booktitle  = {I. {M}. {G}elfand {S}eminar},
  series     = {Adv. Soviet Math.},
  volume     = {16},
  pages      = {1--50},
  publisher  = {Amer. Math. Soc., Providence, RI},
  year       = {1993},
  mrclass    = {22E47 (14A22 14F05)},
  mrnumber   = {1237825 (95a:22022)},
  mrreviewer = {D. Mili{\v{c}}i{\'c}}
}

@article{BC,
  author  = {Baranovsky, Vladimir and Chen, Taiji},
  title   = {Quantization of Vector Bundles on Lagrangian Subvarieties},
  journal = {International Mathematics Research Notices},
  volume  = {},
  number  = {},
  pages   = {rnx230},
  year    = {2017},
  doi     = {10.1093/imrn/rnx230},
  url     = { + http://dx.doi.org/10.1093/imrn/rnx230},
  eprint  = {/oup/backfile/content_public/journal/imrn/pap/10.1093_imrn_rnx230/3/rnx230.pdf}
}

@article{Vergne,
  author     = {Vergne, Mich\`ele},
  title      = {Instantons et correspondance de {K}ostant-{S}ekiguchi},
  journal    = {C. R. Acad. Sci. Paris S\'er. I Math.},
  fjournal   = {Comptes Rendus de l'Acad\'emie des Sciences. S\'erie I.
                Math\'ematique},
  volume     = {320},
  year       = {1995},
  number     = {8},
  pages      = {901--906},
  issn       = {0764-4442},
  mrclass    = {22E47 (32M15)},
  mrnumber   = {1328708},
  mrreviewer = {J. S. Joel}
}

@article{Sekiguchi:bijection,
  author     = {Sekiguchi, Jir\=o},
  title      = {Remarks on real nilpotent orbits of a symmetric pair},
  journal    = {J. Math. Soc. Japan},
  fjournal   = {Journal of the Mathematical Society of Japan},
  volume     = {39},
  year       = {1987},
  number     = {1},
  pages      = {127--138},
  issn       = {0025-5645},
  mrclass    = {53C35 (17B20)},
  mrnumber   = {867991},
  mrreviewer = {Juan Tirao},
  doi        = {10.2969/jmsj/03910127},
  url        = {http://dx.doi.org/10.2969/jmsj/03910127}
}

@article{Losev:isom_quant,
  author     = {Losev, Ivan},
  title      = {Isomorphisms of quantizations via quantization of resolutions},
  journal    = {Adv. Math.},
  fjournal   = {Advances in Mathematics},
  volume     = {231},
  year       = {2012},
  number     = {3-4},
  pages      = {1216--1270},
  issn       = {0001-8708},
  mrclass    = {53D55 (16G20 17B63)},
  mrnumber   = {2964603},
  mrreviewer = {Chengming Bai},
  url        = {https://doi.org/10.1016/j.aim.2012.06.017}
}

@article{Losev:quan_symp_orbit,
  title    = {Deformations of Symplectic Singularities and Orbit Method for Semisimple {{Lie}} Algebras},
  author   = {Losev, Ivan},
  year     = {2022},
  fjournal = {Selecta Mathematica},
  journal  = {Sel. Math. New Ser.},
  volume   = {28},
  number   = {2},
  pages    = {30},
  issn     = {1420-9020},
  doi      = {10.1007/s00029-021-00754-y},
  url      = {https://doi.org/10.1007/s00029-021-00754-y},
  urldate  = {2025-02-08},
  langid   = {english}
}

@article{Losev:HC_bimodules,
  title   = {Harish-Chandra bimodules over quantized symplectic singularities},
  author  = {Losev, I.},
  date    = {2021-06-01},
  journal = {Transformation Groups},
  volume  = {26},
  number  = {2},
  pages   = {565--600},
  issn    = {1531-586X},
  doi     = {10.1007/s00031-020-09638-5},
  url     = {https://doi.org/10.1007/s00031-020-09638-5}
}

@article{Losev:SRA,
  year      = {2019},
  publisher = {Oxford University Press ({OUP})},
  author    = {Losev, I.},
  pages     = {442--472},
  title     = {Derived Equivalences for Symplectic Reflection Algebras},
  journal   = {Int. Math. Res. Not.}
}

@misc{LosevYu,
  author        = {Losev, Ivan and Yu, Shilin},
  title         = {On {H}arish-{C}handra modules over quantizations of nilpotent orbits},
  year          = {2025},
  note          = {arXiv:2309.11191},
  eprint        = {2309.11191},
  eprinttype    = {arXiv},
  archiveprefix = {arXiv},
  primaryclass  = {math.RT}
}

@article{Beauville1,
  author     = {Beauville, Arnaud},
  title      = {Symplectic singularities},
  journal    = {Invent. Math.},
  fjournal   = {Inventiones Mathematicae},
  volume     = {139},
  year       = {2000},
  number     = {3},
  pages      = {541--549},
  issn       = {0020-9910},
  mrclass    = {14B05 (14E15 32S45)},
  mrnumber   = {1738060},
  mrreviewer = {Marko Roczen},
  doi        = {10.1007/s002229900043},
  url        = {https://doi.org.ezproxy.library.tamu.edu/10.1007/s002229900043}
}

@article{Namikawa1,
  author  = {Namikawa, Yoshinori},
  year    = {2001},
  month   = {02},
  pages   = {},
  title      = {A note on symplectic singularities},
  eprint     = {math/0101028},
  eprinttype = {arXiv},
  journal    = {arXiv:math/0101028}
}

@article{Namikawa2,
  author     = {Namikawa, Yoshinori},
  title      = {Flops and {P}oisson deformations of symplectic varieties},
  journal    = {Publ. Res. Inst. Math. Sci.},
  fjournal   = {Kyoto University. Research Institute for Mathematical
                Sciences. Publications},
  volume     = {44},
  year       = {2008},
  number     = {2},
  pages      = {259--314},
  issn       = {0034-5318},
  mrclass    = {14D15 (14E05 14E30 32G07)},
  mrnumber   = {2426349},
  mrreviewer = {Baohua Fu},
  doi        = {10.2977/prims/1210167328},
  url        = {https://doi.org.ezproxy.library.tamu.edu/10.2977/prims/1210167328}
}

@article{BV85:unipotent,
  author     = {Barbasch, Dan and Vogan, Jr., David A.},
  title      = {Unipotent representations of complex semisimple groups},
  journal    = {Ann. of Math. (2)},
  fjournal   = {Annals of Mathematics. Second Series},
  volume     = {121},
  year       = {1985},
  number     = {1},
  pages      = {41--110},
  issn       = {0003-486X},
  mrclass    = {22E46 (20G05 22E50)},
  mrnumber   = {782556},
  mrreviewer = {Floyd L. Williams},
  doi        = {10.2307/1971193},
  url        = {https://doi-org.ezproxy.library.tamu.edu/10.2307/1971193}
}

@article{Barbasch:unitary_dual_complex_classical,
  author     = {Barbasch, Dan},
  title      = {The unitary dual for complex classical {L}ie groups},
  journal    = {Invent. Math.},
  fjournal   = {Inventiones Mathematicae},
  volume     = {96},
  year       = {1989},
  number     = {1},
  pages      = {103--176},
  issn       = {0020-9910},
  mrclass    = {22E46 (22E45)},
  mrnumber   = {981739},
  mrreviewer = {J. A. Wolf},
  doi        = {10.1007/BF01393972},
  url        = {https://doi.org/10.1007/BF01393972}
}

@incollection{Barbasch:unipotent_dual_pair,
  author     = {Barbasch, Dan},
  title      = {Unipotent representations and the dual pair correspondence},
  booktitle  = {Representation theory, number theory, and invariant theory},
  series     = {Progr. Math.},
  volume     = {323},
  pages      = {47--85},
  publisher  = {Birkh\"{a}user/Springer, Cham},
  year       = {2017},
  mrclass    = {22E46 (22E55)},
  mrnumber   = {3753908},
  mrreviewer = {Karl-Hermann Neeb}
}

@article{Beynon-Spaltenstein,
  author     = {Beynon, W. Meurig and Spaltenstein, Nicolas},
  title      = {Green functions of finite {C}hevalley groups of type
                {$E_{n}$} {$(n=6,\,7,\,8)$}},
  journal    = {J. Algebra},
  fjournal   = {Journal of Algebra},
  volume     = {88},
  year       = {1984},
  number     = {2},
  pages      = {584--614},
  issn       = {0021-8693},
  coden      = {JALGA4},
  mrclass    = {20G40 (20C15)},
  mrnumber   = {747534 (85k:20136)},
  mrreviewer = {Toshiaki Shoji},
  doi        = {10.1016/0021-8693(84)90084-X},
  url        = {http://dx.doi.org/10.1016/0021-8693(84)90084-X}
}

@article{Shoji:Green,
  author   = {Shoji, Toshiaki},
  title    = {Green polynomials of a {C}hevalley group of type {$F_4$}},
  journal  = {Comm. Algebra},
  fjournal = {Communications in Algebra},
  volume   = {10},
  year     = {1982},
  pages    = {505--543}
}

@article{BMSZ:metaplecticBV,
  title   = {On the {{Notion}} of {{Metaplectic Barbasch}}--{{Vogan Duality}}},
  author  = {Barbasch, Dan and Ma, Jia-Jun and Sun, Binyong and Zhu, Chen-Bo},
  year    = {2023},
  month   = oct,
  journal = {International Mathematics Research Notices},
  volume  = {2023},
  number  = {20},
  pages   = {17822--17852},
  issn    = {1073-7928},
  doi     = {10.1093/imrn/rnad097},
  urldate = {2024-03-02}
}

@article{BMSZ:counting,
  author  = {Barbasch, Dan and Ma, Jia-Jun and Sun, Binyong and Zhu, Chen-Bo},
  title   = {Special unipotent representations of real classical groups: Counting and Reduction},
  journal = {J. Eur. Math. Soc.},
  year    = {2025},
  doi     = {10.4171/JEMS/1609},
  url     = {https://ems.press/journals/jems/articles/14298688},
  note    = {Published online first},
  mrclass = {22E46, 22E47}
}

@article{BMSZ:construction_unitarity,
  title       = {Special Unipotent Representations of Real Classical Groups: Construction and Unitarity},
  author      = {Barbasch, Dan and Ma, Jia-Jun and Sun, Binyong and Zhu, Chen-Bo},
  journal     = {J. Amer. Math. Soc.},
  doi         = {10.1090/jams/1082},
  eprint      = {1712.05552},
  eprinttype  = {arXiv},
  eprintclass = {math},
  year        = {2026},
  pubstate    = {To appear}
}

@incollection{BMSZ:spin,
  title     = {Genuine Special Unipotent Representations of Spin Groups},
  booktitle = {Symmetry in Geometry and Analysis, Volume 1: {{Festschrift}} in Honor of Toshiyuki Kobayashi},
  author    = {Barbasch, Dan and Ma, Jia-Jun and Sun, Binyong and Zhu, Chen-Bo},
  editor    = {Pevzner, Michael and Sekiguchi, Hideko},
  year      = {2025},
  pages     = {141--164},
  publisher = {Springer Nature Singapore},
  location  = {Singapore},
  doi       = {10.1007/978-981-97-8449-3_3},
  url       = {https://doi.org/10.1007/978-981-97-8449-3_3},
  isbn      = {978-981-97-8449-3}
}

@article{He,
  author  = {He, H.},
  title   = {Unipotent representations, theta correspondences, and quantum induction},
  journal = {Mem. Amer. Math. Soc.},
  volume  = {299},
  year    = {2024},
  number  = {1496, vii+90 pp.}
}

@book{HTT:D-modules,
  title       = {D-{{Modules}}, {{Perverse Sheaves}}, and {{Representation Theory}}},
  editor      = {Hotta, Ryoshi and Takeuchi, Kiyoshi and Tanisaki, Toshiyuki},
  editora     = {Bass, Hyman and Oesterlé, Joseph and Weinstein, Alan},
  editoratype = {redactor},
  date        = {2008},
  series      = {Progress in {{Mathematics}}},
  volume      = {236},
  publisher   = {Birkhäuser},
  location    = {Boston, MA},
  doi         = {10.1007/978-0-8176-4523-6},
  url         = {http://link.springer.com/10.1007/978-0-8176-4523-6},
  urldate     = {2025-10-09},
  isbn        = {978-0-8176-4363-8 978-0-8176-4523-6},
  langid      = {english}
}

@article{Mason-Brown,
  author  = {Mason-Brown, Lucas},
  year    = {2018},
  month   = {6},
  pages   = {},
  title      = {Unipotent Representations and Microlocalization},
  eprint     = {1805.12038},
  eprinttype = {arXiv},
  journal    = {arXiv:1805.12038}
}

@article{Howe79,
  author    = {Howe, Roger},
  title     = {$\theta$-series and invariant theory},
  journal = {Automorphic Forms, Representations and $L$-functions, Proc. Sympos. Pure Math, vol. 33},
  year      = {1979},
  pages     = {275--285}
}

@article{Howe82,
  author  = {Howe, Roger},
  title   = {On a notion of rank for unitary representations of the classical groups},
  journal = {Harmonic analysis and group representations, Liguori, Naples},
  year    = {1982},
  pages   = {223--331}
}

@article{Howe89,
  author   = {Howe, Roger},
  title    = {Transcending classical invariant theory},
  journal  = {J. Amer. Math. Soc.},
  fjournal = {Journal of the American Mathematical Society},
  volume   = {2},
  year     = {1989},
  number   = {3},
  pages    = {535--552},
  issn     = {0894-0347},
  mrclass  = {22E45},
  mrnumber = {985172},
  doi      = {10.2307/1990942},
  url      = {https://doi-org.ezproxy.library.tamu.edu/10.2307/1990942}
}

@article{Li,
  author     = {Li, Jian-Shu},
  title      = {Singular unitary representations of classical groups},
  journal    = {Invent. Math.},
  fjournal   = {Inventiones Mathematicae},
  volume     = {97},
  year       = {1989},
  number     = {2},
  pages      = {237--255},
  issn       = {0020-9910},
  mrclass    = {22E50 (22E46)},
  mrnumber   = {1001840},
  mrreviewer = {Marko Tadi\'c},
  doi        = {10.1007/BF01389041},
  url        = {https://doi-org.ezproxy.library.tamu.edu/10.1007/BF01389041}
}

@article{HuangLi,
  author     = {Huang, Jing-Song and Li, Jian-Shu},
  title      = {Unipotent representations attached to spherical nilpotent
                orbits},
  journal    = {Amer. J. Math.},
  fjournal   = {American Journal of Mathematics},
  volume     = {121},
  year       = {1999},
  number     = {3},
  pages      = {497--517},
  issn       = {0002-9327},
  mrclass    = {22E47 (20G20)},
  mrnumber   = {1738410},
  mrreviewer = {William M. McGovern},
  url        = {http://muse.jhu.edu.ezproxy.library.tamu.edu/journals/american_journal_of_mathematics/v121/121.3huang.pdf}
}

@incollection{Trapa,
  author     = {Trapa, Peter E.},
  title      = {Special unipotent representations and the {H}owe
                correspondence},
  booktitle  = {Functional analysis {VIII}},
  series     = {Various Publ. Ser. (Aarhus)},
  volume     = {47},
  pages      = {210--229},
  publisher  = {Aarhus Univ., Aarhus},
  year       = {2004},
  mrclass    = {22E45},
  mrnumber   = {2127175},
  mrreviewer = {Tomasz Przebinda}
}

@incollection{Arthur84,
  author     = {Arthur, James},
  title      = {On some problems suggested by the trace formula},
  booktitle  = {Lie group representations, {II} ({C}ollege {P}ark, {M}d., 1982/1983)},
  series     = {Lecture Notes in Math.},
  volume     = {1041},
  pages      = {1--49},
  publisher  = {Springer, Berlin},
  year       = {1984},
  mrclass    = {11F72 (22E45 22E55)},
  mrnumber   = {748504},
  mrreviewer = {A. B. Venkov},
  doi        = {10.1007/BFb0073144},
  url        = {https://doi-org.ezproxy.library.tamu.edu/10.1007/BFb0073144}
}

@article{Arthur89,
  author     = {Arthur, James},
  title      = {Unipotent automorphic representations: conjectures},
  note       = {Orbites unipotentes et repr\'{e}sentations, II},
  journal    = {Ast\'{e}risque},
  fjournal   = {Ast\'{e}risque},
  number     = {171-172},
  year       = {1989},
  pages      = {13--71},
  issn       = {0303-1179},
  mrclass    = {22E50 (11F70)},
  mrnumber   = {1021499},
  mrreviewer = {Marko Tadi\'{c}}
}

@article{Arthur:endoscopic,
  author    = {Arthur, J.},
  title     = {The endoscopic classification of representations. Orthogonal and symplectic groups},
  journal   = {Amer. Math. Soc. Colloq. Publ.},
  volume    = {61},
  year      = {2013},
  publisher = {Amer. Math. Soc.},
  location  = {Providence, RI}
}

@article{MoeglinRenard2020,
  author  = {M{\oe}glin, C. and Renard, D.},
  title   = {Sur les paquets d'Arthur des groupes classiques r\'eels},
  journal = {J. Eur. Math. Soc. },
  volume  = {22},
  number  = {6},
  year    = {2020},
  pages   = {1827--1892}
}

@unpublished{Adams:quasi-dist,
  title = {Nilpotent orbits},
  author = {Adams, J.},
  note = {Unpublished notes},
  year = {2013}
}

@online{AdamsTrapaVogan,
  title  = {Computing Hodge filtrations},
  author = {Adams, J. and Trapa, P. and Vogan, D.},
  url    = {http://www.liegroups.org/papers/atlasHodge.pdf},
  year   = {2019}
}

@book{ABV,
  author     = {Adams, Jeffrey and Barbasch, Dan and Vogan, Jr., David A.},
  title      = {The {L}anglands classification and irreducible characters for
                real reductive groups},
  series     = {Progress in Mathematics},
  volume     = {104},
  publisher  = {Birkh\"auser Boston, Inc., Boston, MA},
  year       = {1992},
  pages      = {xii+318},
  isbn       = {0-8176-3634-X},
  mrclass    = {22-02 (22E47)},
  mrnumber   = {1162533},
  mrreviewer = {Brian E. Blank},
  doi        = {10.1007/978-1-4612-0383-4},
  url        = {https://doi-org.ezproxy.library.tamu.edu/10.1007/978-1-4612-0383-4}
}

@article{ALTV,
  author     = {Adams, Jeffrey D. and van Leeuwen, Marc A. A. and Trapa, Peter
                E. and Vogan, Jr., David A.},
  title      = {Unitary representations of real reductive groups},
  journal    = {Ast\'erisque},
  fjournal   = {Ast\'erisque},
  number     = {417},
  year       = {2020},
  pages      = {x + 174},
  issn       = {0303-1179,2492-5926},
  isbn       = {978-2-85629-918-0},
  mrclass    = {22E46 (17B15 20G05)},
  mrnumber   = {4146144},
  mrreviewer = {William\ M.\ McGovern},
  doi        = {10.24033/ast},
  url        = {https://doi.org/10.24033/ast}
}

@unpublished{AMVV,
  author = {Adams, J. and Miller, S. and van Leeuwen, M. and Vogan, D. A.},
  title  = {Unipotent representations of real exceptional groups},
  note   = {In preparation}
}

@article{Miller,
  author  = {Miller, S. D.},
  title   = {Residual automorphic forms and spherical unitary representations of exceptional groups},
  journal = {Ann. of Math.},
  volume  = {177},
  number  = {2},
  year    = {2013},
  pages   = {1169--1179}
}

@online{AIMV:ArthurPackets,
  title = {The {{Unitarity}} of {{Arthur Packets}} for {{Real Reductive Groups}}},
  author = {Adams, Jeffrey and Ionov, Andrei and Mason-Brown, Lucas and Vogan, David},
  date = {2026-06-01},
  eprint = {2606.01609},
  eprinttype = {arXiv},
  eprintclass = {math.RT},
  doi = {10.48550/arXiv.2606.01609},
  url = {http://arxiv.org/abs/2606.01609},
  urldate = {2026-06-14},
  pubstate = {prepublished},
}

@article{AdamsArancibiaMezo,
  author  = {Adams, Jeffrey and Arancibia, Nicolas Robert and Mezo, Paul},
  title   = {Equivalent definitions of Arthur packets for real classical groups},
  journal = {Mem. Amer. Math. Soc.},
  volume  = {300},
  number  = {1503, v+110 pp.},
  year    = {2024}
}

@article{ArancibiaMezo:unitary,
  title   = {Equivalent definitions of Arthur packets for real unitary groups},
  doi     = {10.1090/ert/707},
  journal = {Represent. Theory},
  author  = {Arancibia Robert, Nicolas and Mezo, Paul},
  year    = {2025},
  pages   = {968--1027}
}

@article{ArancibiaMezo:symplectic_orthogonal,
  title   = {Arthur packets for pure real forms of symplectic and special orthogonal groups},
  doi     = {10.4153/S000843952500027X},
  journal = {Canadian Mathematical Bulletin},
  author  = {Arancibia Robert, Nicolas and Mezo, Paul},
  year    = {2025},
  pages   = {1–22}
}

@misc{GLLS:covering_BV,
  title = {Covering {{Barbasch-Vogan}} Duality and Wavefront Sets of Genuine Representations},
  author = {Gao, Fan and Liu, Baiying and Lo, Chi-Heng and Shahidi, Freydoon},
  year = 2025,
  month = nov,
  number = {arXiv:2511.14750},
  eprint = {2511.14750},
  eprinttype = {arXiv},
  primaryclass = {math},
  publisher = {arXiv},
  doi = {10.48550/arXiv.2511.14750},
  urldate = {2026-01-28},
  archiveprefix = {arXiv}
}

@article{Mok,
  author  = {Mok, C. P.},
  title   = {Endoscopic classification of representations of quasi-split unitary groups},
  journal = {Mem. Amer. Math. Soc.},
  volume  = {235},
  number  = {1108},
  year    = {2015},
  pages   = {vi+248 pp.}
}

@article{Sommers:Duality,
  year      = {2001},
  month     = sep,
  publisher = {Elsevier {BV}},
  volume    = {243},
  number    = {2},
  pages     = {790--812},
  author    = {Eric Sommers},
  title     = {Lusztig{\textquotesingle}s Canonical Quotient and Generalized Duality},
  journal   = {Journal of Algebra}
}

@article{FJLS:generic_sing,
  author     = {Fu, Baohua and Juteau, Daniel and Levy, Paul and Sommers,
                Eric},
  title      = {Generic singularities of nilpotent orbit closures},
  journal    = {Adv. Math.},
  fjournal   = {Advances in Mathematics},
  volume     = {305},
  year       = {2017},
  pages      = {1--77},
  issn       = {0001-8708},
  mrclass    = {14B05 (17B08 37B05)},
  mrnumber   = {3570131},
  mrreviewer = {Arvid Siqveland},
  doi        = {10.1016/j.aim.2016.09.010},
  url        = {https://doi-org.ezproxy.library.tamu.edu/10.1016/j.aim.2016.09.010}
}

@misc{JLS:Duality,
  title         = {Minimal special degenerations and duality},
  author        = {Daniel Juteau and Paul Levy and Eric Sommers},
  year          = {2023},
  eprint        = {2310.00521},
  eprinttype    = {arXiv},
  note          = {arXiv:2310.00521},
  archiveprefix = {arXiv},
  primaryclass  = {math.RT}
}

@article{Wong:Richardson,
  author     = {Wong, Kayue Daniel},
  title      = {Unipotent representations of exceptional {R}ichardson orbits},
  journal    = {J. Lie Theory},
  fjournal   = {Journal of Lie Theory},
  volume     = {33},
  year       = {2023},
  number     = {4},
  pages      = {1087--1111},
  issn       = {0949-5932},
  mrclass    = {22E46 (17B08 22E47)},
  mrnumber   = {4652478},
  mrreviewer = {Zhanqiang\ Bai}
}

@book{Lusztig:book,
  isbn      = {9780691083513},
  url       = {http://www.jstor.org/stable/j.ctt1b9x10c},
  author    = {Lusztig, G.},
  publisher = {Princeton University Press},
  title     = {Characters of Reductive Groups over a Finite Field. (AM-107)},
  urldate   = {2025-10-02},
  year      = {1984}
}

@article{Lusztig:irred_Weyl_I,
  author  = {Lusztig, George},
  title   = {A class of irreducible representations of a {W}eyl group},
  journal = {Indag. Math.},
  volume  = {41},
  year    = {1979},
  pages   = {323--335}
}

@article{Lusztig:irred_Weyl_II,
  author    = {Lusztig, George},
  title     = {A Class of Irreducible Representations of a {{Weyl}} Group. {{II}}},
  year      = {1982},
  journal   = {Indagationes Mathematicae (Proceedings)},
  volume    = {85},
  number    = {2},
  pages     = {219--226},
  publisher = {North-Holland},
  issn      = {1385-7258},
  doi       = {10.1016/S1385-7258(82)80013-9},
  url       = {https://www.sciencedirect.com/science/article/pii/S1385725882800139},
  urldate   = {2025-09-24},
  langid    = {american}
}

@article{Lusztig:green,
  author  = {Lusztig, George},
  title   = {Green polynomials and singularities of unipotent classes},
  journal = {Invent. Math.},
  volume  = {42},
  year    = {1981},
  pages   = {169--178}
}

@article{Lusztig:unipotent,
  author     = {Lusztig, George},
  journal    = {Asian J. Math.},
  title      = {Notes on unipotent classes},
  year       = {1997},
  issn       = {1093-6106},
  number     = {1},
  pages      = {194--207},
  volume     = {1},
  fjournal   = {The Asian Journal of Mathematics},
  mrclass    = {20G05 (20G40)},
  mrnumber   = {1480994 (98k:20078)},
  mrreviewer = {Toshiaki Shoji}
}

@article{JiangLiuSavin,
  author  = {Jiang, D. and Liu, B. and Savin, G.},
  title   = {Raising nilpotent orbits in wave-front sets},
  journal = {Represent. Theory},
  volume  = {20},
  pages   = {419--450},
  year    = {2016}
}

@article{Moeglin:p-adic,
  author  = {M{\oe}glin, C.},
  title   = {Front d'onde des repr\'esentations des groupes classiques $p-$adiques},
  journal = {Amer. J. Math.},
  volume  = {118},
  number  = {6},
  year    = {1996},
  pages   = {1313--1346}
}

@article{Moeglin:multi1,
  author  = {M{\oe}glin, C.},
  title   = {Multiplicit\'e 1 dans les paquets d’Arthur aux places p-adiques},
  journal = {Clay Math. Proc.},
  publisher = {Amer. Math. Soc.},
  address = {Providence, RI},
  volume  = {13},
  year    = {2011},
  pages   = {333--374}
}

@unpublished{FHN,
    author = {Finkelberg, Michael and Hanany, Amihay and Nakajima, Hiraku},
    title = {Coulomb branches of orthosymplectic quiver gauge theories},
    note = {In preparation}
}

@article{Namikawa:covers,
  author     = {Namikawa, Yoshinori},
  journal    = {Selecta Math. (N.S.)},
  title      = {Birational geometry for the covering of a nilpotent orbit closure},
  year       = {2022},
  issn       = {1022-1824,1420-9020},
  number     = {4},
  pages      = {Paper No. 75, 59},
  volume     = {28},
  doi        = {10.1007/s00029-022-00790-2},
  fjournal   = {Selecta Mathematica. New Series},
  mrclass    = {14E15 (17B08)},
  mrnumber   = {4470296},
  mrreviewer = {Calum\ Spicer},
  url        = {https://doi.org/10.1007/s00029-022-00790-2}
}

@book{Spaltenstein,
  author     = {Spaltenstein, Nicolas},
  title      = {Classes unipotentes et sous-groupes de {B}orel},
  series     = {Lecture Notes in Mathematics},
  volume     = {946},
  publisher  = {Springer-Verlag},
  address    = {Berlin},
  year       = {1982},
  pages      = {ix+259},
  isbn       = {3-540-11585-4},
  mrclass    = {14L30 (20G15)},
  mrnumber   = {672610 (84a:14024)},
  mrreviewer = {Klaus Pommerening}
}

@article{Duflo,
  author     = {Duflo, Michel},
  title      = {Th\'{e}orie de {M}ackey pour les groupes de {L}ie alg\'{e}briques},
  journal    = {Acta Math.},
  fjournal   = {Acta Mathematica},
  volume     = {149},
  year       = {1982},
  number     = {3-4},
  pages      = {153--213},
  issn       = {0001-5962},
  mrclass    = {22E50 (22E55)},
  mrnumber   = {688348},
  mrreviewer = {A. Kleppner},
  doi        = {10.1007/BF02392353},
  url        = {https://doi-org.ezproxy.library.tamu.edu/10.1007/BF02392353}
}

@book{Schwartz,
  author    = {Schwartz, James Oliver},
  title     = {The determination of the admissible nilpotent orbits in real classical groups},
  note      = {Thesis (Ph.D.)--Massachusetts Institute of Technology},
  publisher = {ProQuest LLC, Ann Arbor, MI},
  year      = {1987},
  pages     = {(no paging)},
  mrclass   = {Thesis},
  mrnumber  = {2941135},
  url       = {http://gateway.proquest.com.ezproxy.library.tamu.edu/openurl?url_ver=Z39.88-2004&rft_val_fmt=info:ofi/fmt:kev:mtx:dissertation&res_dat=xri:pqdiss&rft_dat=xri:pqdiss:0376526}
}

@article{BCHM,
  author     = {Birkar, Caucher and Cascini, Paolo and Hacon, Christopher D.
                and McKernan, James},
  title      = {Existence of minimal models for varieties of log general type},
  journal    = {J. Amer. Math. Soc.},
  fjournal   = {Journal of the American Mathematical Society},
  volume     = {23},
  year       = {2010},
  number     = {2},
  pages      = {405--468},
  issn       = {0894-0347},
  mrclass    = {14E30 (14E05)},
  mrnumber   = {2601039},
  mrreviewer = {Mark Gross},
  doi        = {10.1090/S0894-0347-09-00649-3},
  url        = {https://doi-org.ezproxy.library.tamu.edu/10.1090/S0894-0347-09-00649-3}
}

@book{CM,
  author     = {Collingwood, David H. and McGovern, William M.},
  title      = {Nilpotent orbits in semisimple {L}ie algebras},
  series     = {Van Nostrand Reinhold Mathematics Series},
  publisher  = {Van Nostrand Reinhold Co., New York},
  year       = {1993},
  pages      = {xiv+186},
  isbn       = {0-534-18834-6},
  mrclass    = {17-02 (17B20 17B25 22E60)},
  mrnumber   = {1251060},
  mrreviewer = {Stephen Slebarski}
}

@article{Dk1,
  author     = {\DJ{okovi\'c}, Dragomir \v{Z}.},
  title      = {The closure diagrams for nilpotent orbits of real forms of {$F_4$} and {$G_2$}},
  journal    = {J. Lie Theory},
  fjournal   = {Journal of Lie Theory},
  volume     = {10},
  year       = {2000},
  number     = {2},
  pages      = {491--510},
  issn       = {0949-5932},
  doi        = {10.5802/jolt.218},
  mrclass    = {14L30 (17B45)},
  mrnumber   = {1774875},
  mrreviewer = {Joyce O'Halloran}
}

@article{Dk2,
  author     = {\DJ{okovi\'c}, Dragomir \v{Z}.},
  title      = {The closure diagrams for nilpotent orbits of real forms of {$E_6$}},
  journal    = {J. Lie Theory},
  fjournal   = {Journal of Lie Theory},
  volume     = {11},
  year       = {2001},
  number     = {2},
  pages      = {381--413},
  issn       = {0949-5932},
  doi        = {10.5802/jolt.238},
  mrclass    = {22E46 (17B20)},
  mrnumber   = {1851797},
  mrreviewer = {William M. McGovern}
}

@article{Dk3,
  author     = {\DJ{okovi\'c}, Dragomir \v{Z}.},
  title      = {The closure diagrams for nilpotent orbits of the real forms {E} {VI} and {E} {VII} of {$E_7$}},
  journal    = {Represent. Theory},
  fjournal   = {Representation Theory. An Electronic Journal of the American
                Mathematical Society},
  volume     = {5},
  year       = {2001},
  pages      = {17--42},
  issn       = {1088-4165},
  doi        = {10.1090/S1088-4165-01-00112-1},
  mrclass    = {22E60 (17B25)},
  mrnumber   = {1826427},
  mrreviewer = {William M. McGovern}
}

@article{Dk4,
  author   = {\DJ{okovi\'c}, Dragomir \v{Z}.},
  title    = {Correction to: ``{T}he closure diagrams for nilpotent orbits of the real forms {E} {VI} and {E} {VII} of {$E_7$}'' [{R}epresent. {T}heory {\bf 5} (2001), 17--42; {MR}1826427 (2002b:22033)]},
  journal  = {Represent. Theory},
  fjournal = {Representation Theory. An Electronic Journal of the American Mathematical Society},
  volume   = {5},
  year     = {2001},
  pages    = {503},
  issn     = {1088-4165},
  doi      = {10.1090/S1088-4165-01-00143-1},
  mrclass  = {22E60 (17B25)},
  mrnumber = {1870600}
}

@article{Dk5,
  author     = {\DJ{okovi\'c}, Dragomir \v{Z}.},
  title      = {The closure diagram for nilpotent orbits of the split real
                form of {$E_7$}},
  journal    = {Represent. Theory},
  fjournal   = {Representation Theory. An Electronic Journal of the American Mathematical Society},
  volume     = {5},
  year       = {2001},
  pages      = {284--316},
  issn       = {1088-4165},
  doi        = {10.1090/S1088-4165-01-00124-8},
  mrclass    = {17B25 (14Q15)},
  mrnumber   = {1857083},
  mrreviewer = {William M. McGovern}
}

@article{Dk6,
  author   = {\DJ{okovi\'c}, Dragomir \v{Z}.},
  title    = {The closure diagram for nilpotent orbits of the real form {EIX} of {$E_8$}},
  journal  = {Asian J. Math.},
  fjournal = {Asian Journal of Mathematics},
  volume   = {5},
  year     = {2001},
  number   = {3},
  pages    = {561--584},
  issn     = {1093-6106},
  doi      = {10.4310/AJM.2001.v5.n3.a9},
  mrclass  = {17B20 (22E47)},
  mrnumber = {1868580}
}

@article{Dk7,
  author     = {\DJ{okovi\'c}, Dragomir \v{Z}.},
  title      = {The closure diagram for nilpotent orbits of the split real form of {$E_8$}},
  journal    = {Cent. Eur. J. Math.},
  fjournal   = {Central European Journal of Mathematics},
  volume     = {1},
  year       = {2003},
  number     = {4},
  pages      = {573--643},
  issn       = {1644-3616},
  doi        = {10.2478/BF02475183},
  mrclass    = {17B25 (17B45)},
  mrnumber   = {2040654},
  mrreviewer = {William M. McGovern}
}

@article{Dk8,
  author   = {\DJ{okovi\'c}, Dragomir \v{Z}.},
  title    = {Corrections for ``{T}he closure diagram for nilpotent orbits of the split real form of {$E_8$}'' [{C}ent. {E}ur. {J}. {M}ath. {\bf 1} (2003), no. 4, 573--643 (electronic); MR2040654]},
  journal  = {Cent. Eur. J. Math.},
  fjournal = {Central European Journal of Mathematics},
  volume   = {3},
  year     = {2005},
  number   = {3},
  pages    = {578--579},
  issn     = {1644-3616},
  doi      = {10.2478/BF02475924},
  mrclass  = {17B25 (17B45)},
  mrnumber = {2152477}
}

@article{Djokovic:centralizer_inner,
  title   = {Classification of nilpotent elements in simple exceptional real Lie algebras of inner type and description of their centralizers},
  journal = {Journal of Algebra},
  volume  = {112},
  number  = {2},
  pages   = {503-524},
  year    = {1988},
  issn    = {0021-8693},
  doi     = {https://doi.org/10.1016/0021-8693(88)90104-4},
  url     = {https://www.sciencedirect.com/science/article/pii/0021869388901044},
  author  = {\DJ{okovi\'c}, Dragomir \v{Z}.}
}

@article{Djokovic:centralizer_outer,
  title = {Classification of Nilpotent Elements in Simple Real {{Lie}} Algebras {{{\emph{E}}}}6(6) and {{{\emph{E}}}}6(-26) and Description of Their Centralizers},
  author = {{\DJ}okovi{\'c}, Dragomir {\v Z}},
  year = 1988,
  month = jul,
  journal = {Journal of Algebra},
  volume = {116},
  number = {1},
  pages = {196--207},
  issn = {0021-8693},
  doi = {10.1016/0021-8693(88)90201-3},
  urldate = {2026-05-30},
}

@article{Djokovic:bijection,
 ISSN = {00029947},
 URL = {http://www.jstor.org/stable/2000858},
 author = {Dragomir Ž. Ðoković},
 journal = {Transactions of the American Mathematical Society},
 number = {2},
 pages = {577--585},
 publisher = {American Mathematical Society},
 title = {Proof of a Conjecture of Kostant},
 urldate = {2026-06-19},
 volume = {302},
 year = {1987}
}

@article{GomezZhu:local_theta_lifting,
  title = {Local Theta Lifting of Generalized {{Whittaker}} Models Associated to Nilpotent Orbits},
  author = {Gomez, Raul and Zhu, Chen-Bo},
  date = {2013-08-28},
  eprint = {1302.3744},
  eprinttype = {arXiv},
  eprintclass = {math},
  doi = {10.48550/arXiv.1302.3744},
  url = {http://arxiv.org/abs/1302.3744},
  urldate = {2025-09-29},
  pubstate = {prepublished}
}

@book{McGovern1994,
  author    = {McGovern, W.},
  title     = {Completely Prime Maximal Ideals and Quantization},
  series    = {Mem. Amer. Math. Soc.},
  publisher = {AMS},
  year      = {1994}
}

@article{Nevins,
  author     = {Nevins, Monica},
  title      = {Admissible nilpotent orbits of real and {$p$}-adic split
                exceptional groups},
  journal    = {Represent. Theory},
  fjournal   = {Representation Theory. An Electronic Journal of the American
                Mathematical Society},
  volume     = {6},
  year       = {2002},
  pages      = {160--189},
  issn       = {1088-4165},
  mrclass    = {20G25 (17B25)},
  mrnumber   = {1915090},
  mrreviewer = {William M. McGovern},
  doi        = {10.1090/S1088-4165-02-00134-6},
  url        = {https://doi-org.ezproxy.library.tamu.edu/10.1090/S1088-4165-02-00134-6}
}

@article{Noel1,
  author     = {No{\"e}l, Alfred G.},
  title      = {Classification of admissible nilpotent orbits in simple
                exceptional real {L}ie algebras of inner type},
  journal    = {Represent. Theory},
  fjournal   = {Representation Theory. An Electronic Journal of the American
                Mathematical Society},
  volume     = {5},
  year       = {2001},
  pages      = {455--493},
  issn       = {1088-4165},
  mrclass    = {17B25 (22E60)},
  mrnumber   = {1870598},
  mrreviewer = {Dmitriy A. Rumynin},
  doi        = {10.1090/S1088-4165-01-00141-8},
  url        = {https://doi-org.ezproxy.library.tamu.edu/10.1090/S1088-4165-01-00141-8}
}

@article{Noel2,
  author     = {No{\"e}l, Alfred G.},
  title      = {Classification of admissible nilpotent orbits in simple real
                {L}ie algebras {$E_{6(6)}$} and {$E_{6(-26)}$}},
  journal    = {Represent. Theory},
  fjournal   = {Representation Theory. An Electronic Journal of the American
                Mathematical Society},
  volume     = {5},
  year       = {2001},
  pages      = {494--502},
  issn       = {1088-4165},
  mrclass    = {17B25 (22E60)},
  mrnumber   = {1870599},
  mrreviewer = {Dmitriy A. Rumynin},
  doi        = {10.1090/S1088-4165-01-00142-X},
  url        = {https://doi-org.ezproxy.library.tamu.edu/10.1090/S1088-4165-01-00142-X}
}

@article{King,
  author     = {King, Donald R.},
  title      = {The component groups of nilpotents in exceptional simple real
                {L}ie algebras},
  journal    = {Comm. Algebra},
  fjournal   = {Communications in Algebra},
  volume     = {20},
  year       = {1992},
  number     = {1},
  pages      = {219--284},
  issn       = {0092-7872},
  mrclass    = {17B25},
  mrnumber   = {1145333},
  mrreviewer = {Dragomir \v{Z}. Djokovi\'{c}},
  doi        = {10.1080/00927879208824339},
  url        = {https://doi.org/10.1080/00927879208824339}
}

@article{KingNoel,
  author     = {King, Donald R. and No{\"e}l, Alfred G.},
  title      = {Component groups of centralizers of nilpotents in complex
                symmetric spaces},
  journal    = {J. Algebra},
  fjournal   = {Journal of Algebra},
  volume     = {232},
  year       = {2000},
  number     = {1},
  pages      = {94--125},
  issn       = {0021-8693},
  mrclass    = {17B45},
  mrnumber   = {1783916},
  mrreviewer = {Dragomir \v Z. Djokovi\'c},
  doi        = {10.1006/jabr.2000.8389},
  url        = {https://doi-org.ezproxy.library.tamu.edu/10.1006/jabr.2000.8389}
}

@incollection{Alexeevski,
  author    = {Alexeevski, A.},
  title     = {Component groups of the centralizers of unipotent elements in
               semisimple algebraic groups},
  booktitle = {Lie groups and invariant theory},
  series    = {Amer. Math. Soc. Transl. Ser. 2},
  volume    = {213},
  pages     = {15--31},
  publisher = {Amer. Math. Soc., Providence, RI},
  year      = {2005},
  isbn      = {0-8218-3733-8},
  mrclass   = {20G15},
  mrnumber  = {2140712},
  doi       = {10.1090/trans2/213/02},
  url       = {https://doi.org/10.1090/trans2/213/02}
}

@article{AK,
  author     = {Auslander, L. and Kostant, B.},
  title      = {Polarization and unitary representations of solvable {L}ie
                groups},
  journal    = {Invent. Math.},
  fjournal   = {Inventiones Mathematicae},
  volume     = {14},
  year       = {1971},
  pages      = {255--354},
  issn       = {0020-9910},
  mrclass    = {22E45 (22E25)},
  mrnumber   = {0293012},
  mrreviewer = {A. J. Coleman},
  doi        = {10.1007/BF01389744},
  url        = {https://doi-org.ezproxy.library.tamu.edu/10.1007/BF01389744}
}

@article{KostantRallis,
  issn      = {00029327, 10806377},
  url       = {http://www.jstor.org/stable/2373470},
  author    = {B. Kostant and S. Rallis},
  journal   = {American Journal of Mathematics},
  number    = {3},
  pages     = {753--809},
  publisher = {The Johns Hopkins University Press},
  title     = {Orbits and Representations Associated with Symmetric Spaces},
  urldate   = {2025-09-10},
  volume    = {93},
  year      = {1971}
}

@book{Matsumura:comm_ring,
  place      = {Cambridge},
  series     = {Cambridge Studies in Advanced Mathematics},
  title      = {Commutative Ring Theory},
  publisher  = {Cambridge University Press},
  author     = {Matsumura, H.},
  year       = {1987},
  collection = {Cambridge Studies in Advanced Mathematics}
}

@incollection{SchmidVilonen,
  author     = {Schmid, Wilfried and Vilonen, Kari},
  title      = {Hodge theory and unitary representations of reductive {L}ie
                groups},
  booktitle  = {Frontiers of mathematical sciences},
  pages      = {397--420},
  publisher  = {Int. Press, Somerville, MA},
  year       = {2011},
  mrclass    = {22E46 (22D10 32C38 32G20 32M05 58A14)},
  mrnumber   = {3050836},
  mrreviewer = {Salah Mehdi}
}

@article{Saito88,
  title        = {Modules de {{Hodge Polarisables}}},
  author       = {Saito, Morihiko},
  date         = {1988-12-31},
  journaltitle = {Publications of the Research Institute for Mathematical Sciences},
  volume       = {24},
  number       = {6},
  pages        = {849--995},
  issn         = {0034-5318},
  doi          = {10.2977/prims/1195173930},
  url          = {https://ems.press/journals/prims/articles/3445},
  urldate      = {2025-10-30},
  langid       = {english}
}

@article{Saito90,
  author     = {Saito, Morihiko},
  title      = {Mixed {H}odge modules},
  journal    = {Publ. Res. Inst. Math. Sci.},
  fjournal   = {Kyoto University. Research Institute for Mathematical
                Sciences. Publications},
  volume     = {26},
  year       = {1990},
  number     = {2},
  pages      = {221--333},
  issn       = {0034-5318},
  mrclass    = {14D07 (14C30 32J25)},
  mrnumber   = {1047415},
  mrreviewer = {Min Ho Lee},
  doi        = {10.2977/prims/1195171082},
  url        = {https://doi.org/10.2977/prims/1195171082}
}

@article{Kirillov_nil,
  author  = {Kirillov, A. A.},
  title   = {Unitary representations of nilpotent {L}ie groups},
  journal = {Russian Mathematical Surveys},
  volume  = {17},
  number  = {4},
  pages   = {53},
  url     = {http://stacks.iop.org/0036-0279/17/i=4/a=R02},
  year    = {1962}
}

@article{Kraft-Procesi:special,
  author  = {Kraft, Hanspeter and Procesi, Claudio},
  title   = {A special decomposition of the nilpotent cone of a classical Lie algebra},
  journal = {Ast\'erisque},
  volume  = {173-174},
  year    = {1989},
  pages   = {271--279}
}

@article{Kraft-Procesi:normal,
  title = {Closures of Conjugacy Classes of Matrices Are Normal},
  author = {Kraft, Hanspeter and Procesi, Claudio},
  date = {1979-10-01},
  journaltitle = {Inventiones mathematicae},
  shortjournal = {Invent Math},
  volume = {53},
  number = {3},
  pages = {227--247},
  issn = {1432-1297},
  doi = {10.1007/BF01389764},
  url = {https://doi.org/10.1007/BF01389764},
  urldate = {2026-06-19},
  langid = {english}
}

@misc{FJLS23,
  author        = {Baohua Fu and Daniel Juteau and Paul Levy and Eric Sommers},
  title         = {Local geometry of special pieces of nilpotent orbits},
  year          = {2024},
  archiveprefix = {arXiv},
  eprint        = {2308.07398},
  eprinttype    = {arXiv},
  primaryclass  = {math.RT}
}

@online{JLSY:SpecialPiece,
  title = {Lusztig's Special Pieces Conjecture},
  author = {Juteau, Daniel and Levy, Paul and Sommers, Eric and Yu, Shilin},
  date = {2026-07-16},
  eprint = {2607.15406},
  eprinttype = {arXiv},
  eprintclass = {math.RT},
  doi = {10.48550/arXiv.2607.15406},
  url = {http://arxiv.org/abs/2607.15406},
  urldate = {2026-07-20}
}

@article{Joseph:associated_variety,
  author  = {Joseph, A.},
  title   = {On the associated variety of a primitive ideal},
  journal = {J. Algebra},
  volume  = {93},
  date    = {1985},
  number  = {2},
  pages   = {509--523},
  issn    = {0021-8693}
}

@article{MBMY,
  title   = {Unipotent representations of complex groups and extended Sommers duality},
  volume  = {130},
  issn    = {1460-244X},
  url     = {https://onlinelibrary.wiley.com/doi/abs/10.1112/plms.70035},
  doi     = {10.1112/plms.70035},
  pages   = {e70035},
  number  = {3},
  journal = {Proceedings of the London Mathematical Society},
  author  = {Mason-Brown, Lucas and Matvieievskyi, Dmytro and Yu, Shilin},
  year    = {2025},
  langid  = {english}
}

@misc{LMBM,
  title         = {Unipotent Ideals and {H}arish-{C}handra bimodules},
  author        = {Losev, I. and Mason-Brown, L. and Matvieievskyi, D.},
  year          = {2021},
  note          = {arXiv:2108.03453},
  eprint        = {2108.03453},
  eprinttype    = {arXiv},
  archiveprefix = {arXiv},
  primaryclass  = {math.RT}
}

@article{MBM,
  author    = {Mason-Brown, Lucas and Matvieievskyi, Dmytro},
  title     = {Unipotent Ideals for Spin and Exceptional Groups},
  year      = {2023},
  month     = feb,
  journal   = {Journal of Algebra},
  volume    = {615},
  pages     = {358--454},
  publisher = {Academic Press},
  issn      = {0021-8693},
  doi       = {10.1016/j.jalgebra.2022.10.011},
  urldate   = {2025-09-08},
  langid    = {american}
}

@misc{DMB-Hodge,
  title         = {Hodge theory, intertwining functors, and the Orbit Method for real reductive groups},
  author        = {Davis, Dougal and Mason-Brown, Lucas},
  year          = {2025},
  note          = {arXiv:2503.14794},
  eprint        = {2503.14794},
  eprinttype    = {arXiv},
  archiveprefix = {arXiv},
  primaryclass  = {math.RT}
}

@article{DavisVilonen:MHM_real_groups,
  title        = {Mixed {{Hodge}} Modules and Real Groups},
  author       = {Davis, Dougal and Vilonen, Kari},
  date         = {2025-06},
  journaltitle = {Advances in Mathematics},
  shortjournal = {Advances in Mathematics},
  volume       = {470},
  pages        = {110255},
  issn         = {00018708},
  doi          = {10.1016/j.aim.2025.110255},
  url          = {https://linkinghub.elsevier.com/retrieve/pii/S0001870825001537},
  urldate      = {2025-10-23},
  langid       = {english}
}

@online{DavisVilonen:unitary,
  title       = {Unitary Representations of Real Groups and Localization Theory for {{Hodge}} Modules},
  author      = {Davis, Dougal and Vilonen, Kari},
  date        = {2025-02-17},
  eprint      = {2309.13215},
  eprinttype  = {arXiv},
  eprintclass = {math},
  doi         = {10.48550/arXiv.2309.13215},
  url         = {http://arxiv.org/abs/2309.13215},
  urldate     = {2025-10-23}
}

@article{EGA4,
     author = {Grothendieck, Alexander},
     title = {\'El\'ements de g\'eom\'etrie alg\'ebrique : {IV.} {\'Etude} locale des sch\'emas et des morphismes de sch\'emas, {Quatri\`eme} partie},
     journal = {Publications Math\'ematiques de l'IH\'ES},
     pages = {5--361},
     year = {1967},
     publisher = {Institut des Hautes \'Etudes Scientifiques},
     volume = {32},
     doi = {10.1007/BF02732123},
     zbl = {0153.22301},
     language = {fr},
     url = {https://www.numdam.org/articles/10.1007/BF02732123/}
}

@article{Edixhoven,
  author    = {Edixhoven, Bas},
  journal   = {Compositio Mathematica},
  language  = {eng},
  number    = {3},
  pages     = {291-306},
  publisher = {Kluwer Academic Publishers},
  title     = {N\'{e}ron models and tame ramification},
  url       = {http://eudml.org/doc/90141},
  volume    = {81},
  year      = {1992}
}

@article{Ohta:Adm,
  title   = {Classification of admissible nilpotent orbits in the classical real Lie algebras},
  journal = {Journal of Algebra},
  volume  = {136},
  number  = {2},
  pages   = {290-333},
  year    = {1991},
  issn    = {0021-8693},
  doi     = {https://doi.org/10.1016/0021-8693(91)90049-E},
  url     = {https://www.sciencedirect.com/science/article/pii/002186939190049E},
  author  = {Takuya Ohta}
}

@online{atlas,
  title  = {Atlas of Lie Groups and Representations software},
  author = {Fokko du Cloux and Marc van Leeuwen},
  year   = {2024},
  url    = {https://www.liegroups.org},
  lastchecked = {2025-05-30}
}

@online{atlas_exceptional,
  title  = {List of all special unipotent representations of real simple exceptional groups},
  author = {ATLAS},
  year   = {2024},
  url    = {https://www.liegroups.org/tables/unipotentExceptional/},
  lastchecked = {2025-05-30}
}

@article{Yu:period,
  title   = {The {{Enhanced Period Map}} and {{Equivariant Deformation Quantizations}} of {{Nilpotent Orbits}}},
  author  = {Yu, Shi Lin},
  year    = {2024},
  journal = {Acta Mathematica Sinica, English Series},
  volume  = {40},
  number  = {3},
  pages   = {885--934},
  issn    = {1439-7617},
  doi     = {10.1007/s10114-023-2215-6},
  url     = {https://doi.org/10.1007/s10114-023-2215-6}
}

@article{Sahi,
  author  = {Sahi, S.},
  title   = {Explicit Hilbert spaces for certain unipotent representations},
  journal = {Invent. Math.},
  volume  = {110},
  number  = {2},
  pages   = {409--418},
  year    = {1992}
}

@article{HuangZhu,
  author  = {Huang, J.-S.},
  author  = {Zhu, C.-B.},
  title   = {On certain small representations of indefinite orthogonal groups},
  journal = {Represent. Theory},
  volume  = {1},
  pages   = {190--206},
  year    = {1997}
}

@article{Brylinski,
  author  = {Brylinski, R.},
  title   = {Dixmier algebras for classical complex nilpotent orbits via Kraft-Procesi models. I},
  journal = {The orbit method in geometry and physics (Marseille, 2000). Progr. Math.},
  volume  = {213},
  pages   = {49--67},
  year    = {2003}
}

@incollection{BarbaschWong:normality,
  author    = {Barbasch, Dan and Wong, Kayue Daniel},
  title     = {Admissible modules and normality of classical nilpotent
               orbits},
  booktitle = {Symmetry in geometry and analysis. {V}ol. 1. {F}estschrift in
               honor of {T}oshiyuki {K}obayashi},
  series    = {Progr. Math.},
  volume    = {357},
  pages     = {165--197},
  publisher = {Birkh\"auser/Springer, Singapore},
  year      = {2025},
  isbn      = {978-981-97-8448-6; 978-981-97-8449-3},
  mrclass   = {22E60 (17B08)},
  mrnumber  = {4867011},
  doi       = {10.1007/978-981-97-8449-3\_4},
  url       = {https://doi.org/10.1007/978-981-97-8449-3_4}
}

@article{PaulTrapa,
  title   = {Some small unipotent representations of indefinite orthogonal groups and the theta correspondence},
  author  = {Paul, A. and Trapa, P.},
  journal = {Univ. Aarhus Publ. Ser.},
  volume  = {48},
  pages   = {103--125},
  year    = {2007}
}

@incollection{Moeglin:unipotent_Howe,
  author     = {M{\oe}glin, Colette},
  title      = {Paquets d'{A}rthur sp\'eciaux unipotents aux places
                archim\'ediennes et correspondance de {H}owe},
  booktitle  = {Representation theory, number theory, and invariant theory},
  series     = {Progr. Math.},
  volume     = {323},
  pages      = {469--502},
  publisher  = {Birkh\"auser/Springer, Cham},
  year       = {2017},
  isbn       = {978-3-319-59727-0; 978-3-319-59728-7},
  mrclass    = {11F70 (22E50)},
  mrnumber   = {3753920},
  mrreviewer = {Anne-Marie\ H.\ Aubert},
  doi        = {10.1007/978-3-319-59728-7\_15},
  url        = {https://doi.org/10.1007/978-3-319-59728-7_15}
}

@article{Przebinda91,
  author  = {Przebinda, T.},
  title   = {Characters, dual pairs, and unipotent representations},
  journal = {J. Funct. Anal.},
  volume  = {98},
  number  = {1},
  pages   = {59--96},
  year    = {1991}
}

@article{Przebinda93,
  author  = {Przebinda, T.},
  title   = {Characters, dual pairs, and unitary representations},
  journal = {Duke Math. J. },
  volume  = {69},
  number  = {3},
  pages   = {547--592},
  year    = {1993}
}

\end{document}